\documentclass[11pt,a4paper]{amsart}
\usepackage[english]{babel}
\usepackage[utf8]{inputenc}
\usepackage[T1]{fontenc}
\usepackage{csquotes} %quot. mark
\usepackage{lmodern}
\usepackage{amssymb}
\usepackage{float}
\usepackage{amscd}
\usepackage{amsthm}
 \usepackage[top=1.7cm, bottom=1.7cm, left=1.7cm, right=1.7cm,twoside=false]{geometry}
\usepackage{array}
\usepackage{longtable}
\usepackage{graphicx}
\usepackage{url} 
\usepackage{wrapfig}
\usepackage{color}
\usepackage[usenames,x11names]{xcolor}
\usepackage[all]{xy}
\usepackage{enumitem}
\setlist[1]{leftmargin=*}
\setlist[enumerate,1]{label=\normalfont{(\alph*)}}
\setlist[enumerate,2]{label=\normalfont{(\roman*)}, ref=\normalfont{(\alph{enumi}.\roman*)}}

\usepackage{multirow}
\usepackage{units} %\nicefrac
\usepackage{yfonts}
\usepackage{mathrsfs}
\usepackage{caption}
\usepackage{subcaption}
\usepackage{multicol}
\usepackage[normalem]{ulem}
\usepackage[export]{adjustbox}
\usepackage{lscape}
\usepackage[hypertexnames=false]{hyperref} 
\hypersetup{linkcolor  =DodgerBlue3, citecolor  = teal, urlcolor   = teal, colorlinks = true, hyperfootnotes =false}

\usepackage[scr=boondoxo]{mathalfa}
\newcommand{\mm}{\mathscr{m}}
\renewcommand{\ll}{\mathscr{l}}
\newcommand{\uu}{\mathscr{u}}

\newcommand{\lts}[2]{\mathcal{T}_{#1}(-\log #2)}
\newcommand{\etop}{e_{\mathrm{top}}}

\renewcommand{\to}{\longrightarrow}
\newcommand{\map}{\dashrightarrow}

\newcommand{\into}{\hookrightarrow}

\def\Spec{\operatorname{Spec}}

\def\Aut{\operatorname{Aut}}
\def\Sing{\operatorname{Sing}}

\newcommand{\F}{\mathbb{F}}

\newcommand{\C}{\mathbb{C}}
\newcommand{\Z}{\mathbb{Z}}
\newcommand{\Q}{\mathbb{Q}}

\newcommand{\J}{\mathcal{J}}
\renewcommand{\P}{\mathbb{P}}

\renewcommand{\O}{\mathcal{O}}
\renewcommand{\tilde}{\widetilde}

\renewcommand{\hat}{\widehat}
\renewcommand{\epsilon}{\varepsilon}
\renewcommand{\phi}{\varphi}
\renewcommand{\theta}{\vartheta}

\newcommand{\id}{\mathrm{id}}
\newcommand{\ind}{\mathrm{ind}}

\newcommand{\ftip}[1]{\mathrm{tip}^{+}(#1)}
\newcommand{\ltip}[1]{\mathrm{tip}^{-}(#1)}

\newcommand{\de}{:=}

\newcommand{\toin}[1]{\overset{#1}{\to}}

\def\NS{\operatorname{NS}}
\def\Exc{\operatorname{Exc}}
\def\Supp{\operatorname{Supp}}
\def\Bk{\operatorname{Bk}}
\newcommand{\redd}{_{\mathrm{red}}}
\newcommand{\rev}[1]{#1^{t}}
\newcommand{\s}{^{\dagger}\ \!\!}

\renewcommand{\leq}{\leqslant}
\renewcommand{\geq}{\geqslant}

\newcommand{\hor}{_{\mathrm{hor}}}
\renewcommand{\vert}{_{\mathrm{vert}}}

\newcommand{\FZa}{\mathcal{FZ}_{1}}  
\newcommand{\cA}{\mathcal{A}} 
\newcommand{\cB}{\mathcal{B}} 
\newcommand{\cC}{\mathcal{C}} 
\newcommand{\cD}{\mathcal{D}} 
\newcommand{\cE}{\mathcal{E}}
\newcommand{\cF}{\mathcal{F}}
\newcommand{\cG}{\mathcal{G}}
\newcommand{\ORa}{\mathcal{OR}_{1}}
\newcommand{\ORb}{\mathcal{OR}_{2}}
\newcommand{\Qa}{\mathcal{Q}_{4}}
\newcommand{\Qb}{\mathcal{Q}_{3}}
\newcommand{\FZb}{\mathcal{FZ}_{2}}
\newcommand{\FE}{\mathcal{FE}}
\newcommand{\cH}{\mathcal{H}}
\newcommand{\cI}{\mathcal{I}}
\newcommand{\cJ}{\mathcal{J}}

\theoremstyle{plain}
\newtheorem{tw}{Theorem}[section]

\theoremstyle{definition}
\newtheorem{dfn}[tw]{Definition}
\newtheorem{lem}[tw]{Lemma}
\newtheorem{prop}[tw]{Proposition}

\newtheorem{ex}[tw]{Example}

\newtheorem{conjecture}[tw]{Conjecture}
\newtheorem{notation}[tw]{Notation}
\newtheorem{rem}[tw]{Remark} 

\theoremstyle{remark}
\newtheorem*{uw}{Remark} 
\newtheorem{claim}{Claim}
\newtheorem*{claim*}{Claim}

\makeatletter \def\subsection{\@startsection{subsection}{3}
	\z@{.5\linespacing\@plus.7\linespacing}{.5\linespacing}
	{\bfseries\itshape}} \makeatother

\makeatletter \def\@tocline#1#2#3#4#5#6#7{\relax \ifnum #1>\c@tocdepth \else \par \addpenalty\@secpenalty\addvspace{#2} \begingroup \hyphenpenalty\@M \@ifempty{#4}{\@tempdima\csname r@tocindent\number#1\endcsname\relax}{\@tempdima#4\relax} \parindent\z@ \leftskip#3\relax \advance\leftskip\@tempdima\relax \rightskip\@pnumwidth plus4em \parfillskip-\@pnumwidth #5\leavevmode\hskip-\@tempdima \ifcase #1 \or\or \hskip 1em \or \hskip 2em \else \hskip 3em \fi #6\nobreak\relax \dotfill\hbox to\@pnumwidth{\@tocpagenum{#7}}\par \nobreak \endgroup  \fi}
\makeatother

\makeatletter \renewenvironment{proof}[1][\proofname]{
	\par\pushQED{\qed}\normalfont
	\topsep6\p@\@plus6\p@\relax
	\trivlist\item[\hskip\labelsep\bfseries#1\@addpunct{.}]
	\ignorespaces}{
	\popQED\endtrivlist\@endpefalse} \makeatother

\def\:{\colon}
\numberwithin{equation}{section}
\def\8{\infty}

\newsavebox{\fmbox}
\newenvironment{fmpage}[1]
{\begin{lrbox}{\fmbox}\begin{minipage}{#1}}
		{\end{minipage}\end{lrbox}\fbox{\usebox{\fmbox}}}

\begin{document}

\title[Planar rational cuspidal curves\ \ \  II. Log del Pezzo models]{Classification of planar rational cuspidal curves \\ II. Log del Pezzo models}
\author{Karol Palka}
\author{Tomasz Pełka}
\address{Institute of Mathematics, Polish Academy of Sciences, ul. Śniadeckich 8, 00-656 Warsaw, Poland}
\email{palka@impan.pl}
\email{tpelka@impan.pl}
\thanks{Research project funded by the National Science Centre, Poland, grant No.\ 2015/18/E/ST1/00562.}
\begin{abstract}Let $E\subseteq \P^2$ be a complex curve homeomorphic to the projective line. The Negativity Conjecture asserts that the Kodaira--Iitaka dimension of $K_X+\frac{1}{2}D$, where $(X,D)\to (\P^{2},E)$ is a minimal log resolution, is negative. We prove structure theorems for curves satisfying this conjecture and we finish their classification up to a projective equivalence by describing the ones whose complements admit no $\C^{**}$-fibration. As a consequence, we show that they satisfy the Strong Rigidity Conjecture of Flenner--Zaidenberg. The proofs are based on the almost minimal model program. The obtained list contains one new series of bicuspidal curves.
\end{abstract}

\maketitle

\section{Main result}

\subsection{Main results and corollaries.} 
We work with complex algebraic surfaces. Let $\bar{E}\subseteq \P^{2}$ be a curve homeomorphic, in the Euclidean topology, to the projective line. Such a curve is rational and \emph{cuspidal}, because its singularities are locally analytically irreducible. In particular, its complement $\P^{2}\setminus \bar{E}$ is $\Q$-acyclic. If this complement is not of log general type then there is a classification, for a summary see, for example, \cite[\S 2.2]{Bodnar_two-pairs} or \cite[Lemma 2.14]{PaPe_Cstst-fibrations_singularities} and the references there. In particular, in this case $\bar{E}$ has at most two cusps \cite{Wakabayashi-cusp} and $\P^{2}\setminus \bar{E}$ has a $\C^{1}$- or a $\C^{*}$-fibration \cite[Proposition 2.6]{Palka-minimal_models}. Therefore, we shall assume that $\P^{2}\setminus \bar{E}$ is of log general type, that is, $\kappa(K_{X}+D)=2$, where $(X,D)$ is a log smooth completion of $\P^{2}\setminus \bar{E}$, and $\kappa$ stands for the Kodaira--Iitaka dimension. As it was shown in loc.\ cit., in this case $\kappa(K_{X}+\tfrac{1}{2}D)$ plays a crucial role and one can study $\bar E$ using the modification of the logarithmic Minimal Model Program, the so-called \emph{almost Minimal Model Program} (see Section \ref{sec:MMP}), applying it to the pair $(X,\frac{1}{2}D)$. The guiding principle is the following conjecture, which strengthens the Coolidge--Nagata conjecture proved recently by M.\ Koras and the first author \cite{Palka-Coolidge_Nagata1,KoPa-CooligeNagata2}.

\begin{conjecture}[The Negativity Conjecture, {\cite[Conjecture 4.7]{Palka-minimal_models}}]\label{conj}
	If $(X,D)$ is a log smooth completion of a smooth $\Q$-acyclic surface then $\kappa(K_{X}+\frac{1}{2}D)=-\infty$. 
\end{conjecture}

For further motivation and evidence toward the Negativity Conjecture see Conjecture 2.5 in loc.\ cit. In \cite{PaPe_Cstst-fibrations_singularities} we have classified, up to a projective equivalence, rational cuspidal curves with complements admitting a $\C^{**}$-fibration, where $\C^{**}=\C^{1}\setminus \{0,1\}$, in which case  Conjecture \ref{conj} holds automatically (see \cite[Lemma 2.4(iii)]{Palka-minimal_models}). The goal of the current article is to complete the classification, up to a projective equivalence, of rational cuspidal curves for which  Conjecture \ref{conj} holds.

\begin{tw}
	\label{thm:main}
	Let $\bar{E}\subseteq \P^{2}$ be a complex curve homeomorphic to $\P^{1}$, such that $\P^{2}\setminus \bar{E}$ is of log general type. Then $\P^{2}\setminus \bar{E}$ satisfies Negativity Conjecture \ref{conj} if and only if either:
\begin{enumerate}
	\item\label{item:thm_C**} $\P^{2}\setminus \bar{E}$ has a $\C^{**}$-fibration, hence $\bar{E}$ is of one of the types listed in \cite[Theorem 1.3]{PaPe_Cstst-fibrations_singularities}, or
	\item\label{item:thm_del-Pezzo} $\bar{E}$ is of one of the types $\Qb$, $\Qa$, $\FE
	(\gamma)$, $\FZb(\gamma)$, $\cH(\gamma)$, $\cI$ or $\cJ(k)$ listed in Definition \ref{def:our_curves}.
\end{enumerate}
Each of the above types is realized by a curve which is unique up to a projective equivalence.
\end{tw}

\noindent For a summary of numerical characteristics of the above curves, see Table \ref{table:nofibrations} at the end of this article and \cite[Table 1]{PaPe_Cstst-fibrations_singularities}. Note that cases \ref{item:thm_C**} and \ref{item:thm_del-Pezzo} correspond to the possible outcome of the birational part of the log MMP for $(X,\tfrac{1}{2}D)$,  which is a log Mori fiber space over a curve in case \ref{item:thm_C**} and a log del Pezzo surface of Picard rank one in case \ref{item:thm_del-Pezzo}, see \cite[Theorem 4.5(4)]{Palka-minimal_models}.

The \emph{multiplicity sequence} of a cusp $q\in \bar{E}$ consists of multiplicities of all proper transforms of the germ of $\bar{E}$ at $q$ under consecutive blowups within the minimal log resolution of $q$ (often one omits $1$'s at the end). We write $(m)_{k}$ for the sequence $( m,\dots, m )$ of length $k$.

\begin{dfn}[Log del Pezzo series]\label{def:our_curves} Let $\bar{E}\subseteq \P^{2}$ be a curve homeomorphic to $\P^1$ (hence rational cuspidal). We say that it is \emph{of type} 
 $\Qb$, $\Qa$, $\FE$, $\FZb$, $\cH$, $\cI$, $\cJ$ if the multiplicity sequences of its cusps are, respectively,
 	\begin{flalign*}
	\Qb &: \quad (2,2),\  (2,2),\  (2,2),  &\\
	\Qa &:\quad   (2,2,2),\  (2),\  (2),\  (2),   & \\
	\FE(\gamma)&:\quad  (3(\gamma-3),(3)_{\gamma-3}),\   ((4)_{\gamma-3},2,2),\  (2) \mbox{ for some integer } \gamma \geq 5,  &\\
	\FZb(\gamma)&:\quad  (2(\gamma-2),(2)_{\gamma-2}),\  ((3)_{\gamma-2}), \  (2) \mbox{ for some integer } \gamma \geq 4,  &\\
	\cH(\gamma)&:\quad   (3(\gamma-1),(3)_{\gamma-1}),\  ((4)_{\gamma-1},2,2,2) \mbox{ for some integer } \gamma \geq 3,  &\\
	\cI&:\quad  (6,6,3,3),\  (8,4,4,2,2),&\\
	\cJ(k)&:\quad   (2k,2k,2k,(2)_{k}),\   (2k,(2)_{k}) \mbox{ for some integer }  k \geq 2. &
 	\end{flalign*}
\end{dfn}

By Lemma \ref{lem:HN-equations}, the integers $\deg \bar{E}$ and $E^{2}$, where $E$ is the proper transform of $\bar{E}$ on $X$, are uniquely determined by the multiplicity sequences, see Table \ref{table:nofibrations} at the end of the article.
 
As a consequence of our classification we infer that all rational cuspidal curves with complements of log general type which satisfy  Negativity Conjecture \ref{conj}  share some unexpected geometric properties. We summarize these properties in Theorems \ref{thm:geometric} and \ref{thm:geom_conseq}, see Section \ref{sec:corollaries} for proofs. We use the names of series $\FZa$ and $\cA-\cG$ from \cite{PaPe_Cstst-fibrations_singularities} and the ones from Definition \ref{def:our_curves} above. For $k\geq 1$ we denote by $\ORa(k)$ and $\ORb(k)$ the curves $C_{4k}$ and $C_{4k}^{*}$, a part of the series constructed by Orevkov \cite{OrevkovCurves} for which the complements are of log general type.

\begin{tw}[Existence of special lines]\label{thm:geometric}
	Let $\bar{E}\subseteq \P^{2}$ be a rational cuspidal curve with a complement of log general type, and which is not one of the Orevkov unicuspidal curves $\ORa$, $\ORb$. Assume that Negativity Conjecture \ref{conj} holds for $\P^{2}\setminus \bar{E}$. Then $\bar E$ has two, three or four cusps and there exists a line $\ll$ through a cusp $q_{1}\in\bar{E}$ with the largest multiplicity sequence (in the lexicographic order) meeting $\bar{E}$ only in two points. Hence, $\bar E$ is the closure of the image of a proper injective morphism $\C^*\to \C^2$ given by $\bar{E}\setminus \ll \subseteq \P^{2}\setminus \ll$.
 		\begin{figure}[h]
 			\begin{subfigure}{0.17\textwidth}\centering
 				\includegraphics[scale=0.25]{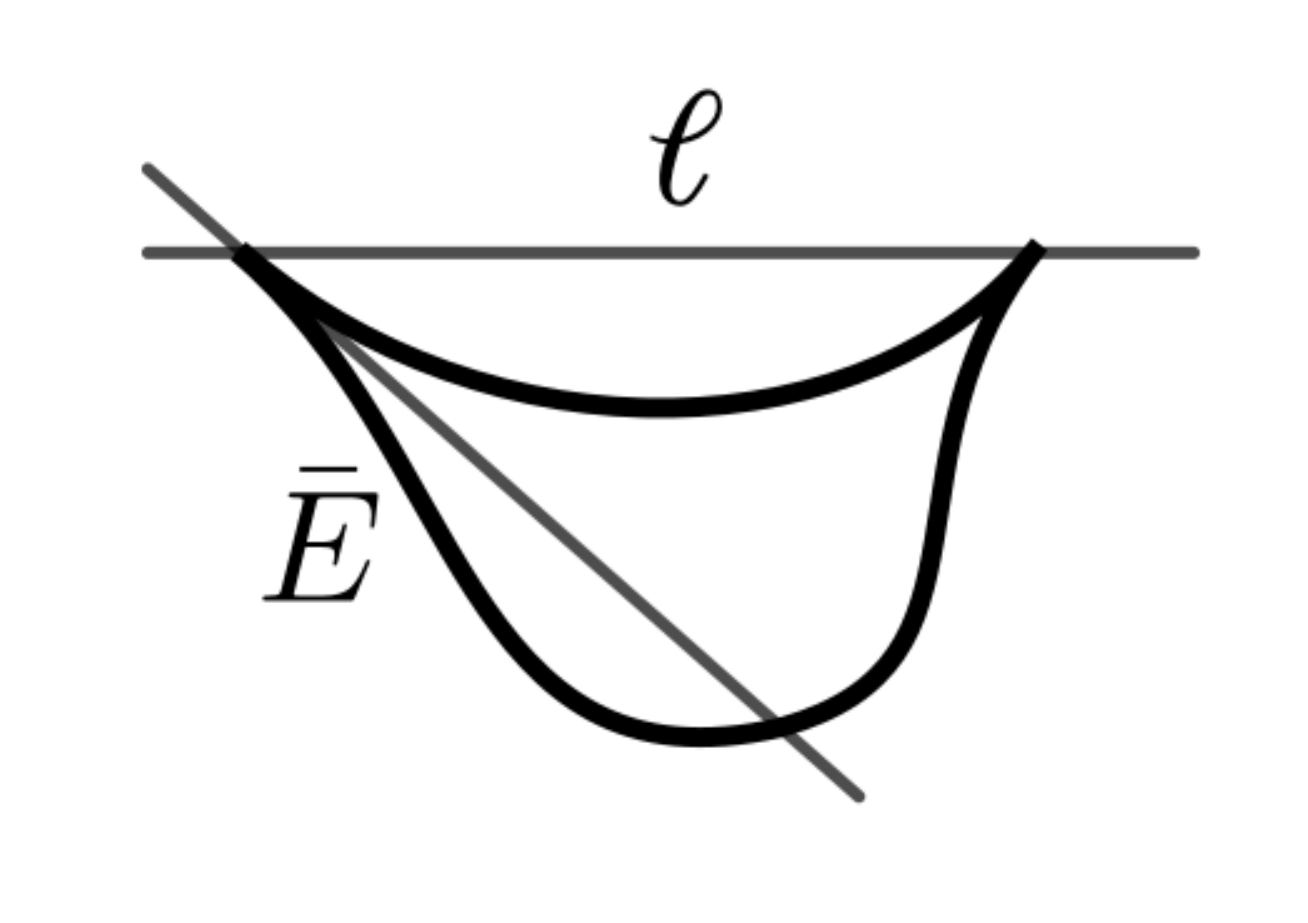}
 				\caption{$\cA,\cB,\cC,\cD,\cH$}
 			\end{subfigure}
 			\hspace{.5em}
 			\begin{subfigure}{0.17\textwidth}\centering
 				\includegraphics[scale=0.252]{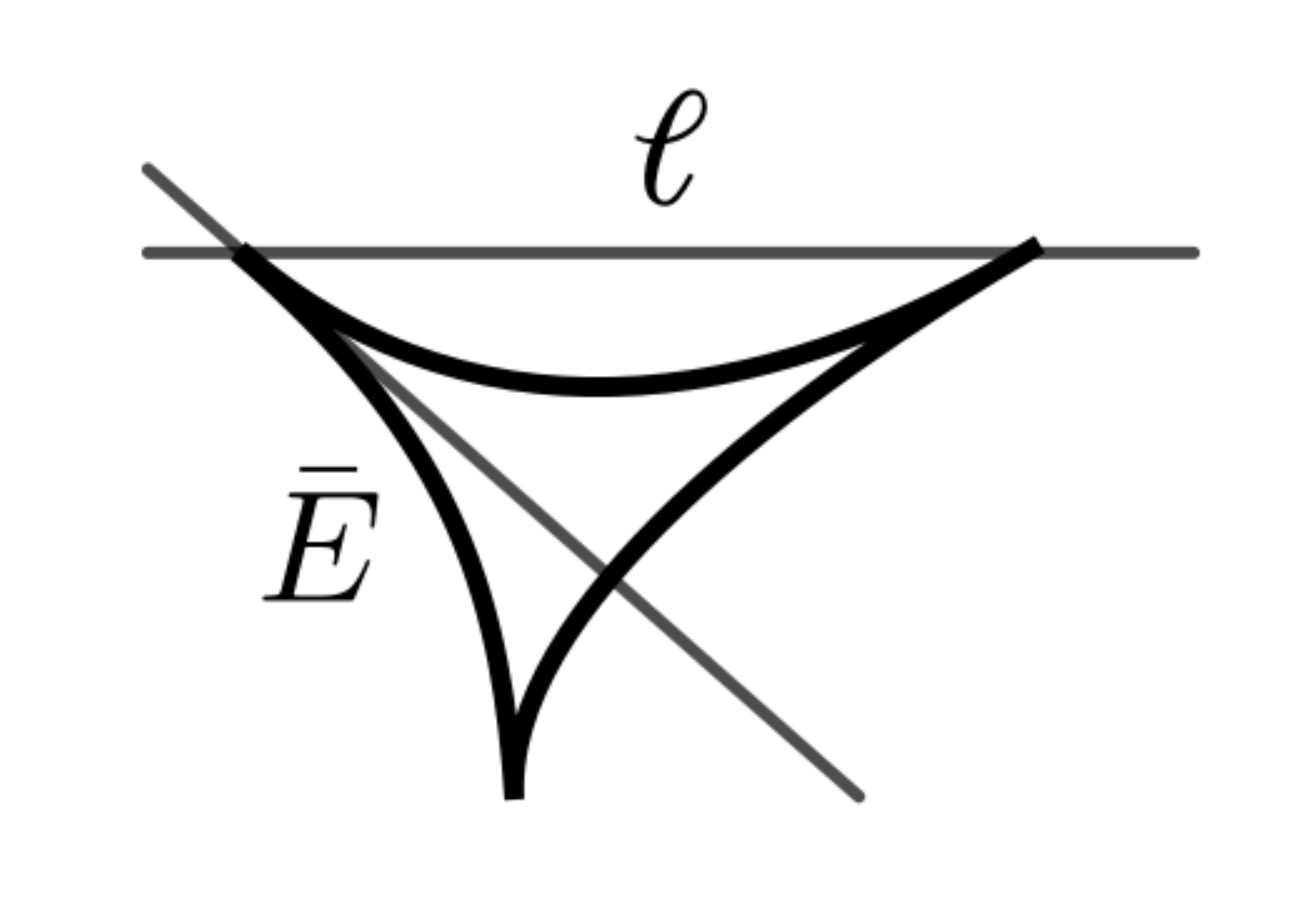}
 				\caption{$\FZa,\FZb,\FE$}
 			\end{subfigure}
 			\hspace{.5em}
 			\begin{subfigure}{0.17\textwidth}\centering
 				\includegraphics[scale=0.25]{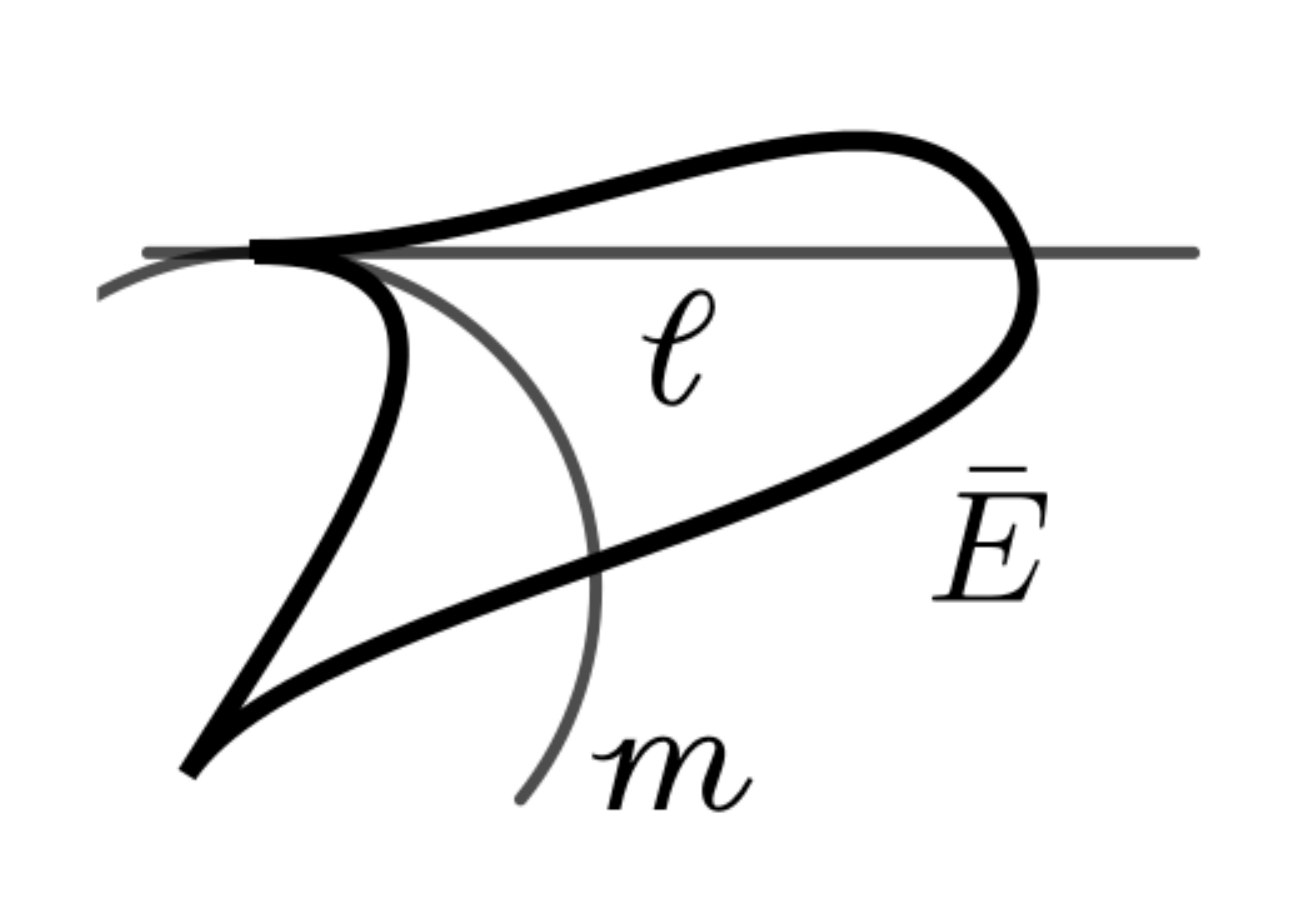}
 				\caption{$\cG$}
 			\end{subfigure}
			\hspace{.5em}
			\begin{subfigure}{0.17\textwidth}\centering
				\includegraphics[scale=0.25]{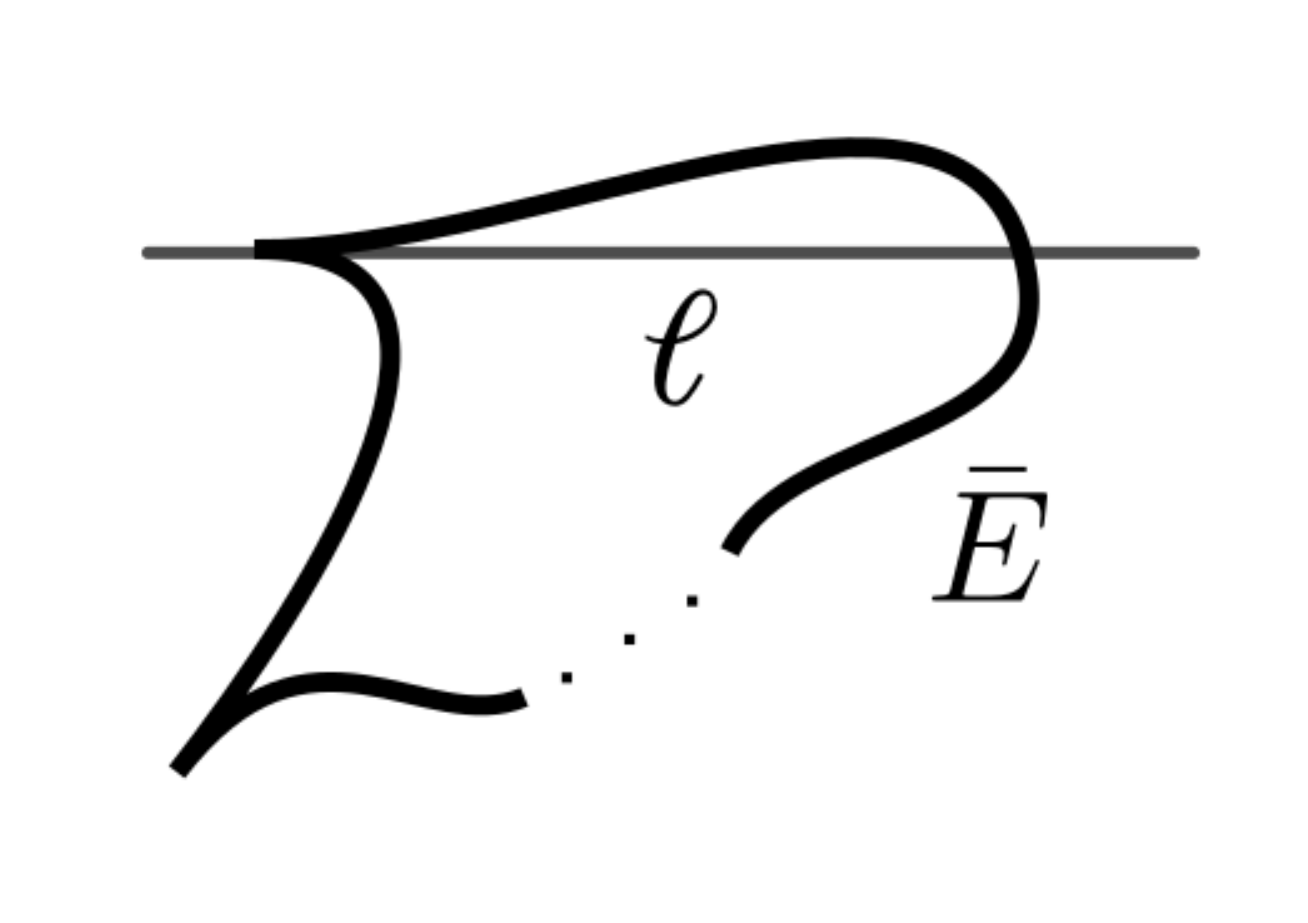}
				\caption{$\Qb,\Qa,\cJ$}
			\end{subfigure}
 				\hspace{.5em}
 			\begin{subfigure}{0.17\textwidth}
 				\centering
 				\includegraphics[scale=0.25]{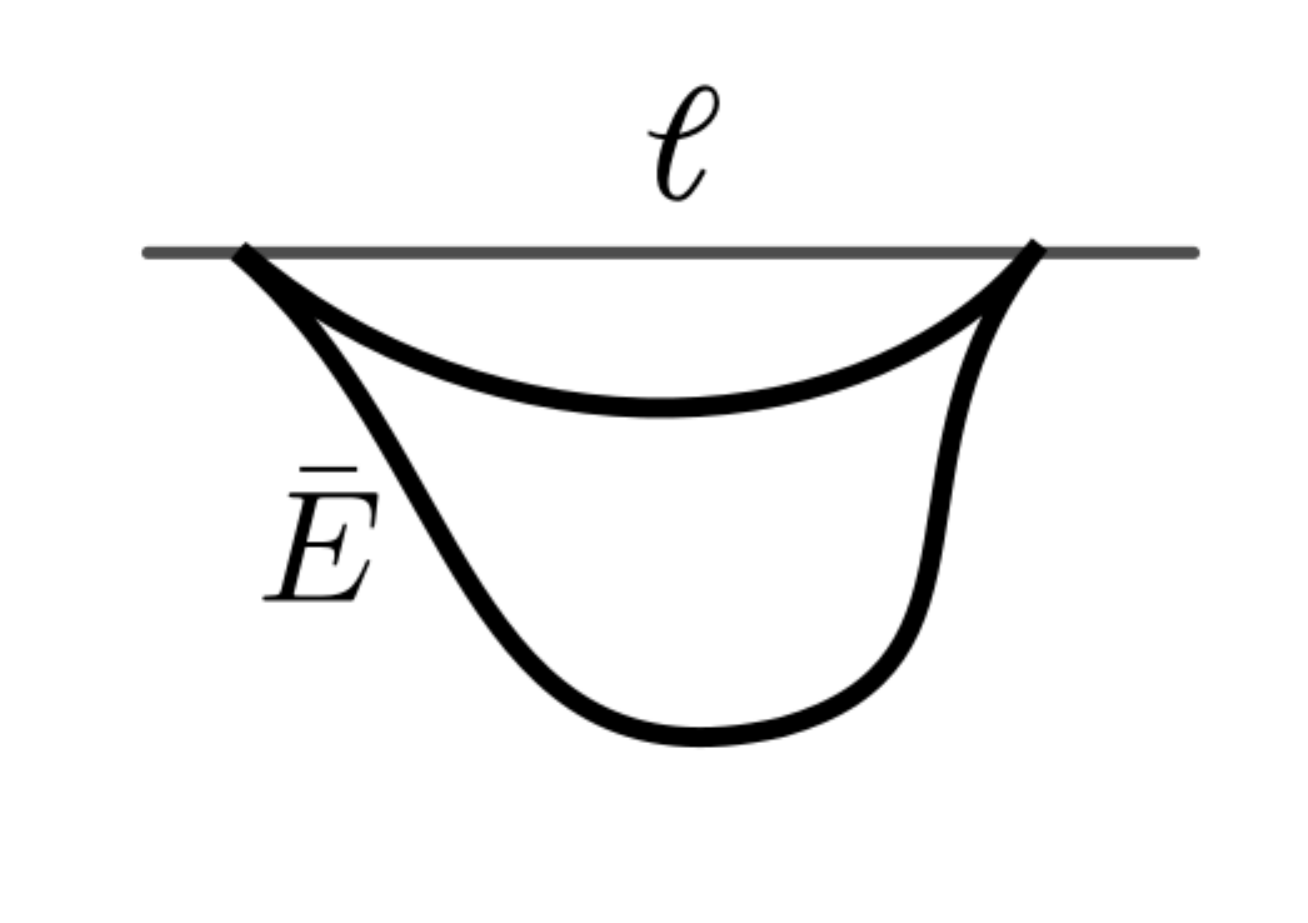}
 				\caption{$\cE,\cF,\cI$}
 			\end{subfigure}
 			\caption{Configurations $(\P^{2}, \bar{E}+\ll)$ from Theorem \ref{thm:geometric}.}
 			\label{fig:thm_geometric}
 		\end{figure}
 	
 	More precisely,	exactly one of one of the following cases \ref{item:cor-asymptote}--\ref{item:cor-no-asymptote} hold, see Figure \ref{fig:thm_geometric}.
	\begin{enumerate}
	\item\label{item:cor-asymptote} There are at least two lines $\ll$ as above. In this case, one of these lines passes through two cusps of $\bar{E}$ and some other meets $\bar{E}\setminus \{q_{1}\}$ transversally. Moreover, $\bar{E}$ is of type $\cA,\cB,\cC,\cD$ or $\cH$ if it has two cusps and is of type $\FZa$, $\FZb$ or $\FE$ otherwise.
	\item\label{item:cor-G} There exists a unique line $\ll$ as above and there exists a  curve $\mm$ such that $\mm\setminus \ll\cong \C^{1}$ and $\mm$ meets $\bar{E}\setminus \ll$ in one point, transversally. In this case, $\ll$ meets $\bar{E}\setminus \{q_{1}\}$ transversally, and there is exactly one curve $\mm$ as above. That $\mm$ is a conic. Moreover, $\bar{E}$ is of type $\cG$.
	\item\label{item:cor-no-asymptote} There is a unique line $\ll$ as above and every curve in $\P^{2}\setminus \ll$ isomorphic to $\C^{1}$ meets $\bar{E}\setminus \ll$ at least twice in the sense of intersection theory. In this case, $\bar{E}$ is of type $\Qb$, $\Qa$ or $\cJ$ if $\ll$ passes through exactly one cusp and is of type $\cE$, $\cF$ or $\cI$ otherwise.
	\end{enumerate}
\end{tw}

\begin{rem}[Existence of $\C^{*}$ in the complement]\label{rem:geometric}\ 
\begin{enumerate}
\item\label{item:C*} 
Let $\bar E\subseteq \P^{2}$ be any rational cuspidal curve. It is known (see \cite[Lemma 2.4]{PaPe_Cstst-fibrations_singularities}) that $\P^2\setminus \bar E$ is of log general type if and only if it does not contain a curve isomorphic to $\C^1$. But in the latter case (under our assumption that the Negativity Conjecture holds) it turns out that $\P^2\setminus \bar E$ always contains a curve isomorphic to $\C^{*}$. Indeed, for the Orevkov curves it is the affine part of a certain nodal cubic with a node at the cusp of $\bar{E}$, see \cite[\S 6]{OrevkovCurves}, and for the remaining types it is  $\ll\setminus \bar{E}$. Another example of $\C^{*}\subseteq \P^{2}\setminus \bar{E}$ is $\ll_{1}\setminus \bar{E}$ or $\mm\setminus \bar{E}$,  where $\ll_{1},\mm$ are the tangent line and the conic from Theorem \ref{thm:geometric}\ref{item:cor-asymptote},\ref{item:cor-G}, respectively (see Remark \ref{rem:AA'} for other examples).
\item\label{item:asymptote} For the bicuspidal curves in Theorem \ref{thm:geometric}\ref{item:cor-asymptote} the line meeting $\bar{E}\setminus\ll$ once  is a \emph{good asymptote} in the sense of \cite[Defintion 1.1]{CKR-Cstar_good_asymptote} for the $\C^{*}$-embedding $\bar{E}\setminus \ll\subseteq \P^{2}\setminus \ll$.
\end{enumerate}
\end{rem}

Using the above classification together with known results for smooth $\Q$-acyclic surfaces we deduce the following important geometric consequences, parts \ref{item:fibr} and \ref{item:c>=4} being results toward the tom Dieck conjecture \cite[Conjecture 2.14]{tDieck_optimal-curves} and the Orevkov--Piontkowski conjecture \cite{Piontkowski-number_of_cusps}, respectively. 
Recall that the logarithmic tangent sheaf of $(X,D)$, denoted by $\mathcal{T}_{X}(-\log D)$, is the sheaf of the $\O_{X}$-derivations which preserve the ideal sheaf of $D$.

\begin{tw}[Geometric properties of planar rational cuspidal curves]\label{thm:geom_conseq}
Let $\bar{E}\subseteq \P^{2}$ be a curve homeomorphic to $\P^{1}$ and let $(X,D)\to (\P^2,\bar E)$ be a minimal log resolution. In case $\P^2\setminus \bar E$ is of log general type assume that it satisfies Negativity Conjecture \ref{conj}. Then the following hold.
\begin{enumerate}
\item\label{item:fibr} \emph{(Fibrations)} $\P^2\setminus \bar E$ has a fibration over $\P^1$ or $\C^1$ with general fiber isomorphic to $\C^1$, $\C^*$, $\C^{**}$ or $\C^{***}$. In the three latter cases one can choose a fibration with no base point on $X$.
\item\label{item:c>=4} \emph{(The number of cusps)} $\bar E$ has at most four singular points (cusps), and if it has exactly four then it is projectively equivalent to the quintic which is the closure of 
	\begin{equation*}\label{eq:q4}
	\C^{1}\ni t\mapsto [t:t^{3}-1:t^{5}+2t^{2}] \in \P^2.
	\end{equation*}	
\item\label{item:C1_or_C*} \emph{(A special line)} If $\bar{E}$ has at least two cusps then there is a line meeting $\bar{E}$ in at most two points.
\item\label{item:rig} \emph{(Strong Rigidity)} We have $H^{2}(\lts{X}{D})=0$. In particular, for planar rational cuspidal curves the Negativity Conjecture implies the Flenner--Zaidenberg Strong Rigidity Conjecture (see Conjecture \ref{conj:rig}).
\end{enumerate}
\end{tw}

\subsection{Discussion of some results in the literature.}
We now comment on the curves from Definition \ref{def:our_curves}. The curves $\Qb$ and $\Qa$ appear as a part of a classification of planar quintics \cite[Theorem 2.3.10]{Namba_geometry_of_curves}. 
The curves $\FZb$ were constructed by Flenner and Zaidenberg \cite{FlZa_cusps_d-3}, who showed that they are the only rational tricuspidal curves with $\mu=\deg \bar{E}-3\geq 3$, where $\mu$ is the maximal multiplicity of a cusp of $\bar{E}$. The curves $\FE$ were constructed by Fenske \cite{Fenske_cusp_d-4}, who showed that they are the only rational tricuspidal curves with $\mu=\deg \bar{E}-4\geq 3$ and $\chi(\lts{X}{D})\leq 0$. For both series projective uniqueness has been settled in these articles.  Note that in general the difference $\deg\bar{E}-\mu$ can be arbitrarily large, for example for the curves $\cJ(k)$ it equals $2k+1$.

The curves $\cH$ and $\cI$ are closures in $\P^2$ of specific proper embeddings of $\C^{*}$ into $\C^{2}$. The classification of such embeddings was initiated by Cassou-Nogues, Koras and Russell \cite{CKR-Cstar_good_asymptote} and will be completed in a forthcoming article of Koras and the first author \cite{KoPa-SporadicCstar2}. Type $\cH$ corresponds to \cite[Theorem 8.2(ii.3)]{CKR-Cstar_good_asymptote}. See Remarks  \ref{rem:BZ_quintics}, \ref{rem:FZb-FE-H}\ref{item:FZb},\ref{item:FE},\ref{item:H_BZ} and \ref{rem:IJ}\ref{item:BZ_IJ} for a comparison with the conditional classification of Borodzik and Żołądek \cite{BoZo-annuli}.

It turns out that almost all curves in our classification have been already discovered. However, in Section \ref{sec:IJ} we construct one new series of rational bicuspidal curves with complements of log general type, the series $\cJ(k)$, depending on a natural number $k\geq 2$. Independently from us, it was recently described by Bodnár \cite[Theorem 3.1(c)]{Bodnar_type_G_and_J}. In fact, as we were told by M.\ Zaidenberg, it was most likely known to T.\ tom Dieck in 1995, see Remark \ref{rem:IJ}\ref{item:Bodnar_IJ}.

\subsection{Scheme of the proof.}
In Section \ref{sec:possible_HN-types} we show that if $\P^{2}\setminus \bar{E}$ satisfies  Negativity Conjecture \ref{conj} but admits no $\C^{**}$-fibration then $\bar{E}\subseteq \P^{2}$ is of one of the types described in Definition \ref{def:our_curves}. The main tool is the \enquote{almost MMP} for the pair $(X_0,\frac{1}{2}D_0)$, where $(X_{0},D_{0})\to (\P^{2},\bar{E})$ is the minimal weak resolution of singularities (see Section \ref{sec:resolutions}). In this article, by a minimal model we mean the outcome of a birational part of the log MMP. Under our assumptions some minimal model (in fact, every minimal model) $(X_{\min},\tfrac{1}{2}D_{\min})$ of $(X_0,\frac{1}{2}D_0)$ is a log del Pezzo surface of Picard rank one, which strongly restricts the geometry of $D_{\min}$. In particular, $D_{\min}$ has at most six components (see \cite[Theorem 4.5(4)]{Palka-minimal_models}). In Section \ref{sec:picture} we find further restrictions by exploiting the fact that the connected components of $D_0-E_0$ can be contracted to smooth points. The proof is then divided into two parts, depending on whether $X_{\min}$ itself is singular or not. The first case is treated in Section \ref{sec:F2}: it turns out that $X_{\min}$ is the quadric cone and $\bar{E}$ is of type $\FE$ or $\cI$. In the second case, treated in Section \ref{sec:P2}, we have $X_{\min}\cong \P^{2}$ and $D_{\min}$ is a simple configuration of at most four curves. These configurations lead to types $\Qb$ or $\Qa$ if $(\P^{2},\tfrac{1}{2}\bar{E})$ is already minimal (see Proposition \ref{prop:n0}) and to types $\FZb$, $\cH$ and $\cJ$ otherwise.

In Section \ref{sec:kappa} we prove that, conversely, if the singularity type of $\bar{E}$ is as in Definition \ref{def:our_curves} then $\P^{2}\setminus \bar{E}$ is a surface of log general type which satisfies Negativity Conjecture \ref{conj} and has no $\C^{**}$-fibration. In Sections \ref{sec:q4} - \ref{sec:IJ} we prove the existence and the projective uniqueness of such curves. In cases $\FZb$, $\FE$ and $\cH$ this result is already known. In the remaining cases we show how to deduce the uniqueness using the original constructions or the almost MMP described in Section \ref{sec:possible_HN-types}. Theorems \ref{thm:geometric} and \ref{thm:geom_conseq} are proved in Section \ref{sec:corollaries}.

\tableofcontents

\section{Preliminaries}\label{sec:preliminaries}

\subsection{Log surfaces.}\label{sec:log_surfaces}

This article is a continuation of \cite{PaPe_Cstst-fibrations_singularities}. We use the terminology and notation introduced there, which we now briefly recall.

Let $D$ be a reduced effective divisor on a smooth projective surface $X$. By a \emph{component} of $D$ we mean an irreducible component. We denote the number of such components by $\#D$. If $D'$ is an effective \emph{subdivisor} of $D$, that is,  $D-D'$ is effective, we define its \emph{branching number} as 
\begin{equation*}
\beta_{D}(D')=D'\cdot (D-D').
\end{equation*}
 We say that a component $T$ of $D$ is a \emph{tip} of $D$ if $\beta_{D}(T)\leq 1$. If $\beta_{D}(T)\geq 3$ we say that $T$ is a branching component of $D$. 

By a \emph{curve} we mean an irreducible and reduced variety of dimension $1$. A smooth complete rational curve $L$ on $X$ with self-intersection number $L^{2}=n$ is called an $n$-curve. For such a curve $K_{X}\cdot L=-n-2$ by adjunction.  We say that $D$ has \emph{simple normal crossings} (is \emph{snc}) if all its components are smooth and meet transversally, at most two in one point. If $D$ is snc, a $(-1)$-curve $L\subseteq D$ is called \emph{superfluous} for $D$ if $0<\beta_{D}(L)\leq 2$ and $L$ meets two different components of $D-L$ in case $\beta_{D}(L)=2$. Note that a $(-1)$-curve $L\subseteq D$ is superfluous if and only if it is not a connected component of $D$ and after its contraction the image of $D$ remains snc.

Let $D_{1},\dots, D_{n}$ be the components of $D$. We say that $D$ is negative definite if its intersection matrix $[D_{i}\cdot D_{j}]_{1\leq i,j \leq n}$ is negative-definite. A \emph{discriminant} of $D$ is defined as 
\begin{equation*}
d(D)=\det [-D_{i}\cdot D_{j}]_{1\leq i,j \leq n} \quad\mbox{ if } D\neq 0 \quad \mbox{and}\quad d(0)=1,
\end{equation*}
see \cite[Section 3]{Fujita-noncomplete_surfaces} for its elementary properties.  
If $D$ has a connected support and for all its components $\beta_D\leq 2$ then we say that $D$ is a \emph{chain} in case at least one inequality is strict and that $D$ is \emph{circular} otherwise. An snc-divisor is a \emph{(rational) tree} if it has a connected support and contains no circular subdivisor (and its components are rational). The components $T_{1},\dots, T_{m}$ of a chain $T$ can be ordered in such a way that $T_{i}\cdot T_{i+1}=1$ for $i=1,\dots, m-1$. Then $T_{1}$ and $T_{m}$ are, respectively, the \emph{first} and the \emph{last} tip of $T$. We denote them by
\begin{equation*}
\ftip{T} \mbox{ - the first tip of }T, \qquad \ltip{T} \mbox{ - the last tip of }T.
\end{equation*}
A \emph{type} of such an ordered chain is the sequence of integers $[-T_{1}^{2},\dots, -T_{m}^{2}]$. We will often abuse the notation and write $T=[-T_{1}^{2},\dots, -T_{m}^{2}]$. We denote by $\rev{T}$ the same chain with an opposite ordering. 
A non-zero ordered chain $T\subseteq D$ is called a \emph{twig} of $D$ if $\ftip{T}$ is a tip of $D$ and the components of $T$ are non-branching in $D$. A twig $T$ of $D$ is \emph{maximal} if it is maximal in the set of twigs of $D$ ordered by inclusion. A rational tree with one branching component and three maximal twigs is called a \emph{fork}.

Let us recall the notion of a \emph{bark}. Let $D$ be a reduced effective divisor with no superfluous $(-1)$-curves and let $T$ be a rational negative definite twig of $D$. The \emph{bark} of $T$ in $D$, denoted by $\Bk_{D}(T)$, is defined in \cite[Section II.3.3]{Miyan-OpenSurf} as a unique $\Q$-divisor supported on $T$ such that for every component $T_{0}$ of $T$ one has
\begin{equation}\label{eq:bark}
T_{0}\cdot \Bk_{D}(T)=\beta_{D}(T_{0})-2,
\end{equation}
see Lemma II.3.3.4 loc.cit. Equivalently, $T_{0}\cdot \Bk_{D}(T)=-1$ if $T_{0}=\ftip{T}$ and $T_{0}\cdot \Bk_{D}(T)=0$ otherwise. One shows that for twigs the coefficients of barks are positive and smaller than $1$. 

If $T$ is a disjoint sum of some twigs of $D$ we define its bark $\Bk_{D}(T)$ as the sum of respective barks. In this article, we will use mostly the barks of $(-2)$-twigs, that is, of twigs whose components are $(-2)$-curves. If $T=T_1+\ldots+T_k$ is a $(-2)$-twig of $D$ then we check that
\begin{equation}\label{eq:bark_of_a_2-twig}
\Bk_{D}(T)=\sum_{i=1}^k\frac{k-i+1}{k+1}T_{i}.
\end{equation}	

\smallskip
We have the following result on chains contractible to smooth points. A similar description was given in \cite[Lemma 3.7]{Palka-Coolidge_Nagata1} and \cite[Proposition 10]{Tono_nie_bicuspidal}.

\begin{lem}[Chains contractible to smooth points]\label{lem:shape_of_contractible_chains}
	For every chain which has a unique $(-1)$-curve and is contractible to a smooth point there is a unique choice of an ordering and unique integers $l\geq 0$, $m_{1}, m_{2},\dots, m_{l},x\geq 0$, such that the type of the ordered chain is:
	\begin{equation*}\begin{split}
	&[(2)_{m_{l}},m_{l-1}+3,\dots,m_{2}+3,(2)_{m_{1}+1},1,m_{1}+3,(2)_{m_{2}},\dots, m_{l}+3,(2)_x], \text{\ \ where\ } 2\nmid l \\ \mbox {or}\quad & [(2)_{m_{l}},m_{l-1}+3,\dots,m_{1}+3,1,(2)_{m_{1}+1},m_{2}+3,(2)_{m_{3}},\dots, m_{l}+3,(2)_x], \text{\ \ where\ } 2\mid l.
	\end{split}\end{equation*}
\end{lem}

\begin{proof}Let $T$ be a chain which has a unique $(-1)$-curve and contracts to a smooth point. We may assume $T\neq [1]$ and $T\neq [1,2]$, for otherwise the type of $T$ is one of the above sequences. The contraction of $T$ can be decomposed into a sequence of contractions of $(-1)$-curves in $T$ and its successive images. Let $T'$ be the image of $T$ after the first contraction. By induction we may assume that with some choice of an ordering the type of $T'$ is one of the above sequences. It contains a unique subsequence $[a,1,b]$ for some $b\geq 2$ and $a\geq 2$ or $a=-\infty$, where we put $[-\8,1]=[1]$. Note that $a=-\infty$ if and only if $l=0$. The type of $T$ can be obtained from the type of $T'$ by replacing $[a,1,b]$ with $[a+1,1,2,b]$ or $[a,2,1,b+1]$. Since the set of the above sequences is closed under such replacements, it contains the type of $T$ for some choice of an ordering on $T$. Moreover, the number $l$ is the number of components of $T$ which are not $(-1)$- or $(-2)$-curves. For $l\neq 0$ we have $m_1+1\geq 1$, so the parity of $l$ determines the side of $1$ on which the nearest $2$ stands in the sequence. It follows that $T$ has a unique ordering such that its type is one of the above sequences.
	
Now, assume that two presentations with $m_1,\ldots,m_l,x$ and $m_1',\ldots,m_l,x'$ as above give the same sequence. By deleting terms equal to $1$ or $2$ and subtracting $3$ from the remaining terms we get $$[m_{l-1},m_{l-3},\ldots,m_2,m_1,\ldots,m_l]=[m_{l-1}',m_{l-3}',\ldots,m_2',m_1',\ldots,m_l']$$ for $l$ odd and $$[m_{l-1},m_{l-3},\ldots,m_1,m_2,\ldots,m_l]=[m_{l-1}',m_{l-3}',\ldots,m_1',m_2',\ldots,m_l']$$  for $l$ even, hence $m_i=m_i'$ for $i=1,\ldots,l$ and then $x=x'$.
\end{proof}

\smallskip
Let $\sigma\colon X \to X'$ be a birational morphism between smooth projective surfaces. The reduced exceptional divisor of $\sigma$ will be denoted by $\Exc\sigma$. A point of $X'$ is called a \emph{center} of $\sigma$ if it is a base point of $\sigma^{-1}$. We define the \emph{rank} of $\sigma$ as $\rho(\sigma)=\rho(X)-\rho(X')$, which is the number of curves contracted by $\sigma$. A \emph{part} of $\sigma$ is any morphism $\sigma'\colon X\to X''$ such that there is a factorization  $\sigma=\sigma''\circ \sigma'$. 

We can write $\sigma$ as a composition of blowups $\sigma=\sigma_{1}\circ \dots \circ \sigma_{z}$ for some $z\geq 0$.  If $L$ is a reduced effective divisor on $X$ then we say that $\sigma$ \emph{touches} $L$ if it is not an isomorphism in every neighborhood of $L$ (in particular, $\sigma$ touches $\Exc\sigma$). We say that $\sigma$ \emph{touches $L$ $n$ times} for some $n\geq 0$ if $L\not\subseteq \Exc\sigma$, exactly $n$ of the blowups $\sigma_{z},\dots, \sigma_{1}$ touch the image of $L$ and each exceptional divisor meets the respective image of $L$ in one point, with multiplicity one. In this case,  $(\sigma_{*}L)^{2}=L^{2}+n$.

Given two divisors $Z_{1}$ and $Z_{2}$ on the same surface we denote by $Z_{1}\wedge Z_{2}$ the divisor which is the sum of their common components. By $Z_{1}\cap Z_{2}$ we denote the intersection of their supports, which may contain components of codimension $2$.

\subsection{Fibrations.}\label{sec:fibrations}
A \emph{fibration} of a smooth surface $X$ is a surjective morphism $X\to B$ onto a curve with a connected, reduced and irreducible scheme-theoretic general fiber. For a given fibration with general fiber $F$, we say that an (irreducible) curve $C$ is \emph{vertical} (resp. \emph{horizontal}) if $C\cdot F=0$ (resp. $C\cdot F\neq 0$). A horizontal curve $C$ with $C\cdot F=n$ is called an \emph{$n$-section}. For a vertical curve $C$ we denote by $\mu(C)$ the multiplicity of $C$ in the fiber containing $C$. Every divisor $T$ can be uniquely decomposed as $T=T\vert+T\hor$, where all components of $T\vert$ are vertical and all components of $T\hor$ are horizontal.

A $\P^{1}$- (respectively, $\C^{1}$-, $\C^{*}$-, $\C^{**}$-) fibration is a fibration with general fiber isomorphic to $\P^{1}$ (respectively, $\C^{1}$, $\C^{*}=\C^{1}\setminus \{0\}$, $\C^{**}=\C^{1}\setminus \{0,1\}$). A fiber non-isomorphic to the general one is called a \emph{degenerate fiber}. Every degenerate fiber of a $\P^{1}$-fibration can be contracted to a $0$-curve by iterated contractions of $(-1)$-curves. By induction one easily gets the following result (see \cite[Section 4]{Fujita-noncomplete_surfaces}).

\begin{lem}[Degenerate fibers]\label{lem:singular_P1-fibers}Let $F$ be a degenerate fiber of a $\P^{1}$-fibration of a smooth projective surface. Then $F$ is a rational tree and its $(-1)$-curves are non-branching in $F\redd$. Furthermore,
	\begin{enumerate}
		\item\label{item:unique_-1-curve} If a $(-1)$-curve $L$ is a component of $F$ and $\mu(L)=1$ then $\beta_{F\redd}(L)=1$ and $F$ contains another $(-1)$-curve.
		\item\label{item:adjoint_chain} If $F$ has a unique $(-1)$-curve $L$ then $F$ has exactly two components of multiplicity one and they are tips of $F$. If these components belong to different connected components of $F\redd - L$ then $F\redd$ is a chain $U+[1]+U^{*}$, where $U^{*}$ is \emph{adjoint} to $U$ \cite[3.9]{Fujita-noncomplete_surfaces}; in particular, $d(U)=d(U^{*})$.
	\end{enumerate}
\end{lem}

\begin{notation}[$\P^{1}$-fibrations]\label{not:fibrations_h_and_nu}
	Let $D$ be a reduced effective divisor on a smooth projective surface $X$ and let $p\colon X\to B$ be a $\P^{1}$-fibration. Let $\nu$ be the number of fibers contained in $D$. For $b\in B$, let $\sigma(F_{b})$ denote the number of components of $F_{b}=p^{-1}(b)$ which are not contained in $D$.
\end{notation}

\begin{lem}[{\cite[4.16]{Fujita-noncomplete_surfaces}}, cf. {\cite[Lemma 2]{Palka-AMS_LZ}}]\label{lem:fibrations-Sigma-chi}
	Fix the notation as above and put 
	$B^{*}=\{b\in B: F_{b}\not\subseteq D\}$. Then
	\begin{equation*}\#D\hor+\nu+\rho(X)=\#D+2+\sum_{b\in B^{*}}(\sigma(F_{b})-1).\end{equation*}
\end{lem}

\medskip\subsection{Log resolutions of rational cuspidal curves.}\label{sec:resolutions}

We now fix some notation for the remaining part of the article. By $\bar{E}\subseteq \P^{2}$ we denote a rational cuspidal curve. Let 
\begin{equation*}
\pi_{0}\colon (X_{0},D_{0})\to (\P^{2},\bar{E})
\end{equation*}
be the minimal weak resolution of singularities, that is, a composition of a minimal number of blowups such that the proper transform $E_{0}$ of $\bar{E}$ on $X_{0}$ is smooth. We will also use the minimal log resolution $$\pi\colon (X,D)\to (\P^{2},\bar{E}),$$ that is, a composition of a minimal number of blowups such that $D=(\pi^{*}\bar{E})\redd$ is snc. It factors as $\pi=\pi_0\circ \psi_{0}$. We put $E=(\pi^{-1})_*\bar E\subseteq X$. We assume that the surface $X\setminus D$ is of log general type. Of course, $X\setminus D=X_{0}\setminus D_{0}=\P^{2}\setminus \bar{E}$. 
By the Poincar\'{e}-Lefschetz duality $\P^{2}\setminus \bar{E}$ is \emph{$\Q$-acyclic}, that is, $b_{i}(\P^{2}\setminus \bar{E})=0$ for $i>0$. We will frequently use the following consequence of the logarithmic version of the Bogomolov-Miyaoka-Yau inequality:

\begin{lem}[No affine lines, \cite{MiTs-lines_on_qhp}]\label{lem:Qhp_has_no_lines}
	A smooth $\Q$-acyclic surface of log general type contains no curve isomorphic to $\C^{1}$.
\end{lem}

By $E_0$ we denote the proper transform of $\bar E$ on $X_0$. For simplicity assume for now that $\bar E$ has only one cusp $q\in \bar E$ and let $Q$ be the reduced preimage of $q$ on $X_0$. We have a unique decomposition $\pi_{0}=\sigma_1\circ\ldots\circ \sigma_k$, where $\sigma_i$ are blowdowns (note the order of indices), or equivalently
\begin{equation}
\label{eq:Q-ordering} \pi_{0}^{-1}=\sigma_k^{-1}\circ\ldots\circ \sigma_1^{-1}. 
\end{equation}
The latter decomposition orders linearly the components of $Q$ as exceptional divisors of the successive blowups (the first component is the one created by the first blowup, or equivalently, contracted last by the resolution morphism). Similarly, the decomposition of $\pi^{-1}$ orders linearly the components of the exceptional divisor of the minimal log resolution over $q$ and this order extends the one on the proper transform of $Q$. If $Z\neq 0$ is a reduced snc-divisor then a blowup of a point on $Z$ is called \emph{inner for $Z$} if it is centered at a singular point of $Z$, otherwise it is called \emph{outer}. The first blowup over $q$ is neither outer nor inner. 

\begin{figure}[ht]
	\begin{subfigure}{\textwidth}\centering
		\includegraphics[scale=0.35]{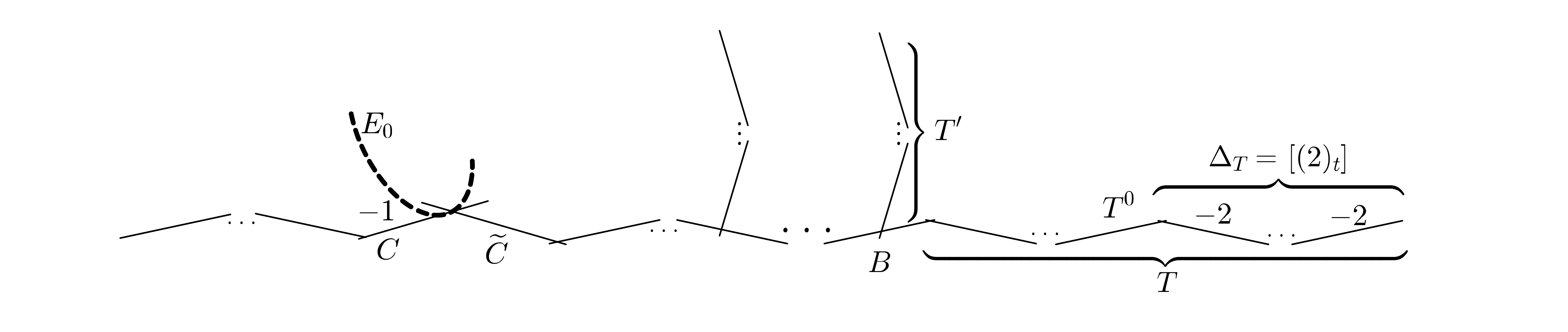}
	\end{subfigure}
	\caption{Exceptional divisor $Q$ of the minimal weak resolution of a cusp $q\in \bar{E}$.}
	\label{fig:treeQ}
\end{figure}

Recall that a \emph{maximal $(-2)$-twig} is a $(-2)$-twig which is not properly contained in any other $(-2)$-twig. In particular, a maximal $(-2)$-twig is not necessarily a maximal twig.

\begin{notation}[The geometry of the minimal weak resolution]\label{not:graphs} We define the following quantities describing the geometry of the cusp $q\in \bar E$ (see Figure \ref{fig:treeQ}):
	\begin{enumerate}
		\item\label{item:not_C} $C$ is the last component of $Q$, that is, the unique $(-1)$-curve in $Q$.	
		\item\label{item:not_tau} $\tau:=C\cdot E_0\geq 2$.
		\item\label{item:not_tildeC} The component of $Q-C$ meeting $E_{0}$ is denoted by $\tilde C$. We put $\tilde C=0$ if there is no such. 
		\item\label{item:not_s} $s=1$ if $\tilde C=0$ and $s=0$ otherwise.
		\item\label{item:not_B} $B$ is the proper transform of the exceptional curve of the last blowup for which the total exceptional divisor over $q$ is still a chain.
		\item\label{item:not_CT} $T$ is the twig of $Q$ meeting (and not containing) $B$ which contains the first component of $Q$. We put $T=0$ if there is no such (then $Q=B=C$).
		\item\label{item:not_Delta_T} If $T$ contains no $(-2)$-twigs of $D$ we put $\Delta_{T}=0$, otherwise we denote by $\Delta_T$ the maximal $(-2)$-twig of $D_{0}$ contained in $T$. We put $t=\#\Delta_T$.
		\item\label{item:not_T^0} $T^{0}$ is the exceptional curve of the $(t+1)$-st blowup.
		\item\label{item:not_T'} $T'$ is the second (not contained in $T$) twig of $Q$ meeting $B$. We put $T'=0$ if there is no such (then $B$ is a tip of $Q$).
	\end{enumerate}
\end{notation}

It is known, and easy to see by induction, see for example \cite[Lemmas 2.11 and 2.12]{PaPe_Cstst-fibrations_singularities}, that the multiplicity sequence of $q\in \bar{E}$ uniquely determines and is determined by the weighted graph of $\pi^{-1}(q)$; or, equivalently, by the weighted graph of $Q$ and the numbers $\tau$, $s$.

The exceptional divisor $Q$ contracts to a smooth point and has a unique $(-1)$-curve $C$. If $Q\neq C$ then the contraction of $C$ leads to a divisor with the same properties. By induction on $\#Q$ it follows that $\beta_{Q}(C)\leq 2$ and $\beta_{Q}(G)\leq 3$ for every component $G$ of $Q$, and if $\beta_{Q}(G)=3$ then $G$ meets a twig of $Q$. Similarly, we see that if we remove from $Q$ the maximal twigs of $D_0$ contained in $Q$ then what remains is a chain (see Figure \ref{fig:treeQ}). Components of this chain are of particular interest, because they are not contracted by $\psi$. Indeed, $\Exc \psi \wedge D_{0}$ is the sum of proper transforms of $\Exc\psi_{i}\wedge D_{i}$, the latter being contained in the sum of twigs of $D_{i}$ by Lemma \ref{lem:MMP-properties}\ref{item:Exc-psi_i}. Here is a list of some elementary properties of $Q$.

\begin{lem}[The geometry of the minimal weak resolution]\label{lem:notation}
	With the above notation the following hold:
	\begin{enumerate}
		\item\label{item:first_exc} If $\#Q>1$ then $T\neq 0$ and $\ftip{T}$ is the first component of $Q$.
		\item\label{item:B} If $Q$ is a chain then $B=C$, otherwise $B$ is the first branching component of $Q$. Every component of $Q-B$ meets at most one twig of $D_{0}$.
		\item\label{item:T^0=C} We have $T^{0}=C$ if and only if  $Q=[(2)_{t},1]$ and $\tilde C=0$, equivalently, if and only if $q\in \bar{E}$ has multiplicity sequence $(\tau)_{t+1}$. If $T^{0}\neq C$ then $T^{0}\subseteq T$, so $\ftip{T}=\ftip{\Delta_{T}+T^{0}}$.
		\item\label{item:blowups} The first blowup over $q$ is (by definition) neither inner nor outer, the next $t$ blowups are outer. If $T^{0}\neq C$ then the $(t+2)$-nd blowup is also outer, and the $(t+3)$-rd one, if occurs, is the first inner one.
	\end{enumerate}
\end{lem}

\begin{proof}
	Parts \ref{item:first_exc}, \ref{item:B} follow from the inductive structure of $Q$ described above. For the proof of \ref{item:T^0=C}, \ref{item:blowups} note that after the first $t+1$ blowups over $q\in\bar{E}$ the exceptional divisor, which is an image of $\Delta_{T}+T^{0}$, is a chain $[(2)_{t},1]$, so all these blowups (except for the first one) are outer. Because $\Delta_{T}$ is zero or a $(-2)$-twig, it is not touched by the remaining part of the resolution, so the proper transform of $\bar{E}$ meets this chain only in the last component. If there are no more blowups then $T^{0}=C$, so $E_{0}\cdot T^{0}=\tau$, and since each blowup is outer, each contracted curve meets the image of $E_{0}$ with multiplicity $\tau$, so $q\in\bar{E}$ has multiplicity sequence $(\tau)_{t+1}$. On the other hand, if $T^{0}\neq C$ then the $(t+2)$-nd blowup is outer, so $T^{0}\subseteq T$ meets $\Delta_{T}$.  This shows \ref{item:T^0=C}. Because $T^{0}$ is not a part of a $(-2)$-twig of $D_{0}$, it equals $\tilde{C}$ or is touched at least once more, so the proper transform of $\bar{E}$ meets the exceptional divisor in a common point of the images of $T^{0}$ and the next component of $Q$. It follows that the $(t+3)$-rd blowup, if occurs, is inner. This shows \ref{item:blowups}.
\end{proof}

\smallskip 
From now on we denote the cusps of $\bar{E}$ by $q_{1},\dots, q_{c}$ and we write $Q_{j}, C_{j}, T_{j}, t_{j},\dots$ for the quantities $Q,C,T,t,\dots$ as above defined for the cusp $q_{j}\in \bar{E}$, $j\in \{1,\dots,c\}$.

\begin{ex}[Semi-ordinary cusps]\label{ex:ordinary_cusp}A cusp of multiplicity $2$ is called \emph{semi-ordinary}. It is locally analytically isomorphic to the singular point of $\{x^{2}=y^{2m+3}\}$ at $(0,0)\in \Spec \C[x,y]$ for some $m\geq 0$. Its multiplicity sequence equals $(2)_{m+1}$. The exceptional divisor of its minimal log resolution is a chain $[(2)_{m},3,1,2]$. Hence, $Q=[(2)_{m},1]$, $T=\Delta_{T}$, $t=m$, $T^{0}=B=C$, $T'=0$, $\tau=2$ and $s=1$. A semi-ordinary cusp with $m=0$ (type $A_2$) is called \emph{ordinary}. 
\end{ex}

\begin{ex}
	\begin{figure}[ht]
		\includegraphics[scale=0.23, trim={0 1.5cm 0 1.5cm}, clip]{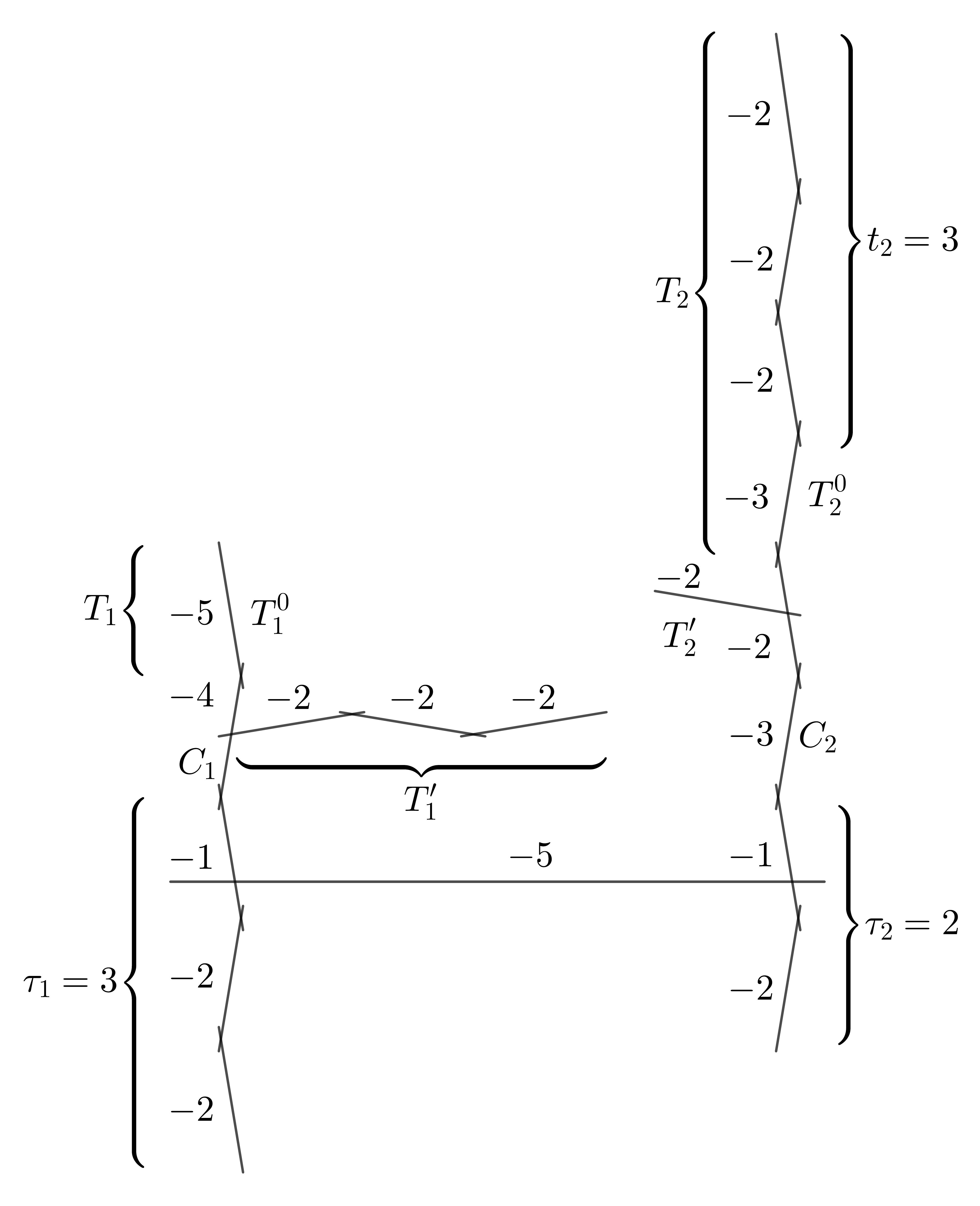}
		\caption{The graph of $D$ for a curve of type $\cH(5)$}
		\label{fig:exH}
	\end{figure}
	
	Let $\bar{E}$ be of type $\cH(5)$ (see Definition \ref{def:our_curves}). Its cusps have multiplicity sequences $(12,(3)_{4})$ and $((4)_{4},(2)_{3})$, respectively. The divisor $D$ is shown in Figure \ref{fig:exH}.  We have $t_{1}=0$ ($T_{1}=[5]$ has no $(-2)$-twig), $\tau_{1}=3$ and $s_{1}=1$. For the second cusp we have $t_{2}=3$, $\tau_{2}=2$ and $s_{2}=1$. The standard HN-pairs of the cusps of $\bar{E}$ (see \cite[Section 2D]{PaPe_Cstst-fibrations_singularities}) are $\binom{15}{12}\binom{3}{1}$ and $\binom{18}{4}\binom{2}{3}$, respectively.
\end{ex}

We recall the relation between integers $\deg\bar{E}$ and $E^{2}$ in terms of multiplicity sequences of the cusps of $\bar{E}$. For a cusp $q\in \bar{E}$ we define $M(q)$ as the sum of all terms of the multiplicity sequence of $q$ (including $1$'s in the end) and $I(q)$ as the sum of their squares. 

\begin{lem}[Equations for multiplicity sequences]\label{lem:HN-equations}
	Let $\bar{E}\subseteq \P^{2}$ be a rational cuspidal curve with cusps $q_{1},\dots, q_{c}\in \bar{E}$. Then
	\begin{align}
	3\deg\bar{E}-E^{2}-2&=\sum_{j=1}^{c} M(q_{j}),\tag{a}\label{eq:3deg}\\
	(\deg\bar{E})^{2}-E^{2}&=\sum_{j=1}^{c} I(q_{j}),\tag{b}\label{eq:deg^2}\\
	(\deg\bar{E}-1)(\deg\bar{E}-2)&=\displaystyle\sum_{i=1}^{c}(I(q_{j})-M(q_{j}))\tag{c}\label{eq:genus-degree}
	\end{align}
	Given multiplicity sequences, these formulas 
	determine $\deg\bar{E}$ and $E^{2}$ uniquely provided $\deg\bar{E}\geq 3$, that is, when $\bar{E}$ is singular.
\end{lem}
\begin{proof} Let $B$ be a curve on a smooth projective surface $Y$ and let $Y'\to Y$ be a blowup of a point of multiplicity $\mu$ on $B$. Denote by $B'$ the proper transform of $B$ on $Y'$. Then 
\begin{equation*}
K_{Y'}\cdot B'-K_{Y}\cdot B=\mu\quad \mbox{and} \quad B^{2}-(B')^{2}=\mu^{2}.
\end{equation*}
We have $K_{X}\cdot E-K_{\P^{2}}\cdot \bar{E}=-(E^{2}+2)+3\deg\bar{E}$ and $\bar{E}^{2}-E^{2}=(\deg \bar{E})^{2}-E^{2}$, so by induction one gets respectively \eqref{eq:3deg} and \eqref{eq:deg^2}. Part \eqref{eq:genus-degree} is their direct consequence.
	
If two rational cuspidal curves of degrees $d,d'\geq 3$ have the same multiplicity sequences of their cusps then \eqref{eq:genus-degree} gives $0=(d-1)(d-2)-(d'-1)(d'-2)=(d-d')(d+d'-3)$. Since $d+d'\neq 3$, it follows that $d=d'$. Thus multiplicity sequences of the cusps of $\bar{E}$ determine uniquely $\deg \bar{E}$, and hence $E^{2}$ by \eqref{eq:3deg}.
\end{proof}

\begin{lem}[Upper bounds on $E^{2}$, cf.\ \cite{Tono-cusps_self-intersections}, {\cite[Lemma 2.16]{PaPe_Cstst-fibrations_singularities}}]\label{lem:Tono_E2}\ \\
	Assume that $\P^{2}\setminus \bar{E}$ is of log general type and has no $\C^{**}$-fibration. Then:
	\begin{enumerate}
		\item\label{item:c=1} if $c=1$ then $E^{2}\leq -3$,
		\item\label{item:c=2} if $c=2$ then $E^{2}\leq -2$,
		\item\label{item:c=2,tau=1} if $c=2$ and $(\tau_{j},s_{j})=(2,1)$ for some $j\in \{1,2\}$ then $E^{2}\leq -3$.
	\end{enumerate}
\end{lem}
\begin{proof}
	\ref{item:c=1},\ref{item:c=2} Suppose the contrary.  Let $C_{j}'$ be the last exceptional curve over $q_{j}\in \bar{E}$ on the minimal log resolution. Blow up over $C_{1}'\cap E$ until the proper transform $\hat{E}$ of $E$ has self-intersection $-2$ if $c=1$ and $-1$ if $c=2$. Call $\hat{C}$ the exceptional curve of the last blowup, or put $\hat{C}=C_{1}'$ if no blowups were needed. Then $\hat{C}$ meets some $(-2)$-curve $U$ which is a non-branching component of the total transform of $Q_{1}$. If $c=1$ then $|\hat{E}+2\hat{C}+U|$ induces a $\P^{1}$-fibration which restricts to a $\C^{1}$-, $\C^{*}$- or a  $\C^{**}$-fibration of $\P^{2}\setminus \bar{E}$. Similarly, if $c=2$ then so does $|\hat{C}+\hat{E}|$. Since $\C^{**}$ is excluded by assumption, we get $\kappa(\P^{2}\setminus \bar{E})\leq 1$ by Iitaka's Easy Addition Theorem; a contradiction.
	
	\ref{item:c=2,tau=1} We have $E^{2}\leq -2$ by \ref{item:c=2}. Suppose that $E^{2}=-2$. The assumption $\tau_{j}=2$, $s_{j}=1$ means that $C_{j}'$ meets a twig $U=[2]$ of $D$. Then $|E+2C_{j}'+U|$ induces a $\P^{1}$-fibration of $X$ which restricts to a $\C^{**}$-fibration of $\P^{2}\setminus \bar{E}$; a contradiction.
\end{proof}

\smallskip

\subsection{Almost minimal models with half-integral boundaries.}\label{sec:MMP}

Given a log surface $(X,B)$ we say that an irreducible curve $\ll$ is \emph{log exceptional} if $\ll^2<0$ and $(K_X+B)\cdot \ll<0$. The contraction of such curves leads to a model with strong properties given by Mori theory. Below we recall a definition of an \emph{almost log exceptional curve} and the construction of an almost minimal model of the pair $(X_{0},\frac{1}{2}D_{0})$  given in \cite[Definition 3.6]{Palka-minimal_models}, which allows to avoid introducing singularities. The notation introduced here is used also in Sections \ref{sec:consequences_of_noCstst} and \ref{sec:possible_HN-types}. 

\begin{dfn}[The morphism $\psi_{A}$] \label{def:psi_A}
Let $D$ be a reduced connected divisor on a smooth projective surface $X$ such that $\kappa(X\setminus D)=2$. Assume that $A\subseteq X$ is a $(-1)$-curve such that 
	\begin{equation}\label{eq:bubble}
	A\not\subseteq D,\quad A\cdot D=2 \quad\mbox{and $A$ meets $D$ in two different components,}
	\end{equation}
	so $A$ is a superfluous $(-1)$-curve in $A+D$. For such $A$ we denote by $\psi_{A}$ the composition of contractions of $A$ and all superfluous $(-1)$-curves in the subsequent images of $D$ which pass through the image of $A$.
\end{dfn} 	
	The morphism $\psi_{A}$ is well defined, that is, uniquely determined by $A$. Indeed, if after some number of contractions the image of $A$ is a common point of two superfluous curves in the image of $D$ then it is the only common point of these $(-1)$-curves, so the linear system of their sum induces a $\C^{*}$-fibration of an open subset of $X\setminus D$, contrary to the fact that $\kappa(X\setminus D)=2$. Therefore, in each step the contracted $(-1)$-curve is unique.
	
	Note also that since $D$ is connected, the center of each blowup in the decomposition of $\psi_{A}$ is a common point of two components of the image of $D$.
	\smallskip

We now return to the study of $(X_{0},D_{0})$, that is, of the minimal weak resolution of $(\P^{2},\bar{E})$. Following \cite[Section 3]{Palka-minimal_models} we define inductively a sequence of contractions between smooth projective surfaces
\begin{equation}\label{eq:MMP}
(X_{0},\tfrac{1}{2}D_{0})\toin{\psi_{1}}(X_{1},\tfrac{1}{2}D_{1})\toin{\psi_{2}}\dots \toin{\psi_{n}}(X_{n},\tfrac{1}{2}D_{n}).
\end{equation}
First, we define inductively the following divisors on $X_{i}$. Recall that a \emph{maximal $(-2)$-twig} is a twig consisting of $(-2)$-curves which is not properly contained in any other twig consisting of $(-2)$-curves. 
\begin{notation}[{\cite[Notation 3.3]{Palka-minimal_models}}]\label{not:MMP}
	Let $(X_{i},D_{i})$ be as in \eqref{eq:MMP}. Assume that $X_{i}$ is smooth (cf.\ Lemma \ref{lem:MMP-properties}\ref{item:Xi_smooth}). 
	\begin{enumerate}
		\item $\Delta_{i}$ is the sum of all maximal $(-2)$-twigs of $D_{i}$.
		\item\label{item:not_Ups} $\Upsilon_{i}$ is the sum of all $(-1)$-curves $U$ in $D_{i}$ such that $\beta_{D_{i}}(U)=3$ and $U\cdot \Delta_{i}=1$ or $\beta_{D_{i}}(U)=2$ and $U$ meets exactly one component of $D_{i}$ (see Figure \ref{fig:Upsilon}).
		\item\label{item:Delta^-} $\Delta_{i}^{+}$ is the sum of all $(-2)$-twigs of $D_{i}$ meeting $\Upsilon_{i}$; and
		$\Delta_{i}^{-}:=\Delta_{i}-\Delta_{i}^{+}$.
		\item\label{item:Upsilon^0} $\Upsilon_{i}^{0}$ is the sum of those components of $\Upsilon_{i}$ which do not meet $\Delta_{i}^{+}$.
		\item $R_{i}=D_{i}-\Delta_{i}-\Upsilon_{i}$.
		\item $E_{i}$ is the proper transform of $\bar{E}$ on $X_{i}$.
		\item $D_{i}^{\flat}=D_{i}-\Upsilon_{i}-\Delta_{i}^{+}-\Bk_{D_{i}}(\Delta_{i}^{-})$, see \eqref{eq:bark_of_a_2-twig}.
	\end{enumerate}	
\end{notation}
\begin{figure}[ht]
	\begin{subfigure}{0.3\textwidth}\centering
		\includegraphics[scale=0.25]{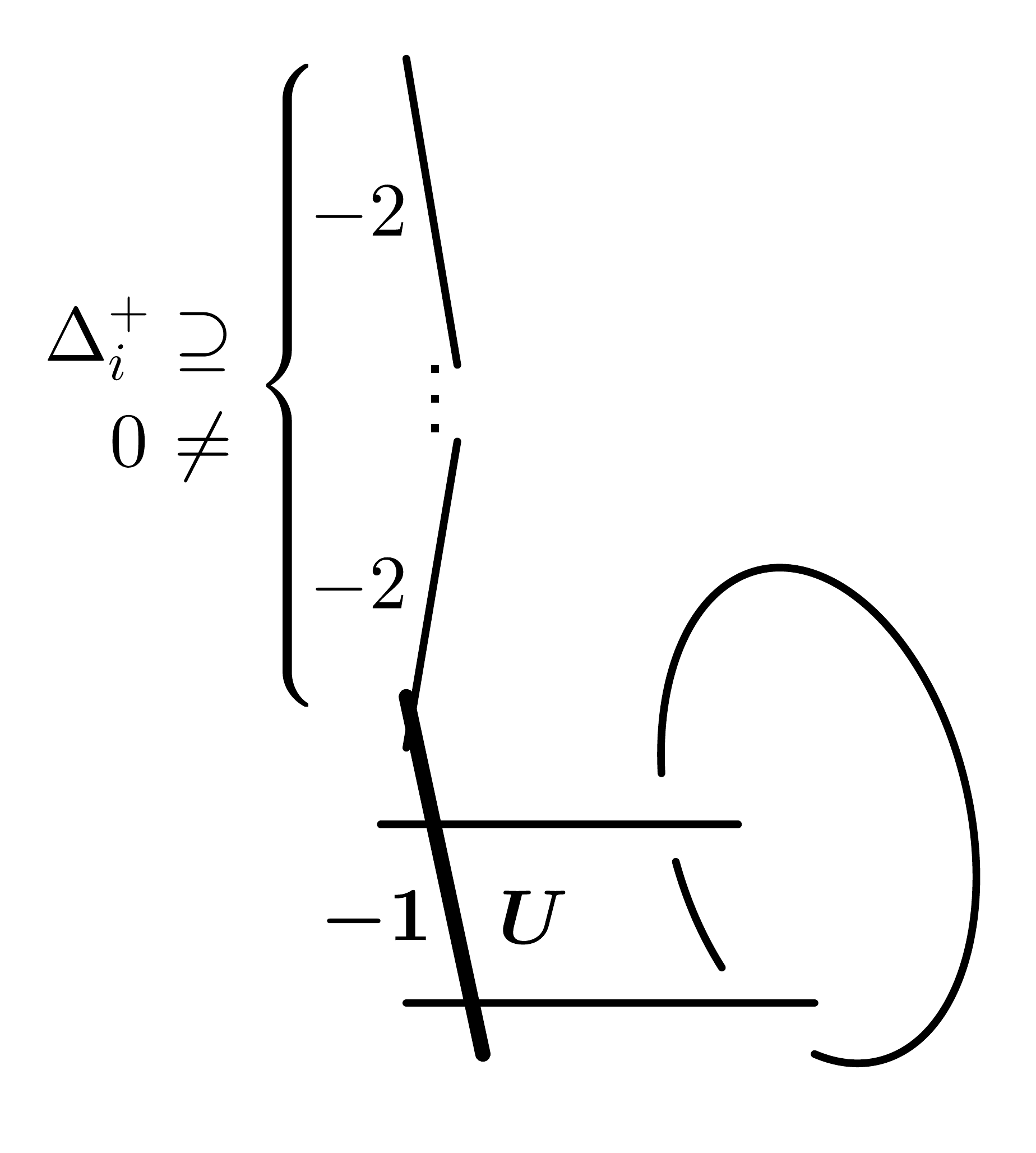}
	\end{subfigure}
	\begin{subfigure}{0.3\textwidth}\centering
		\includegraphics[scale=0.25]{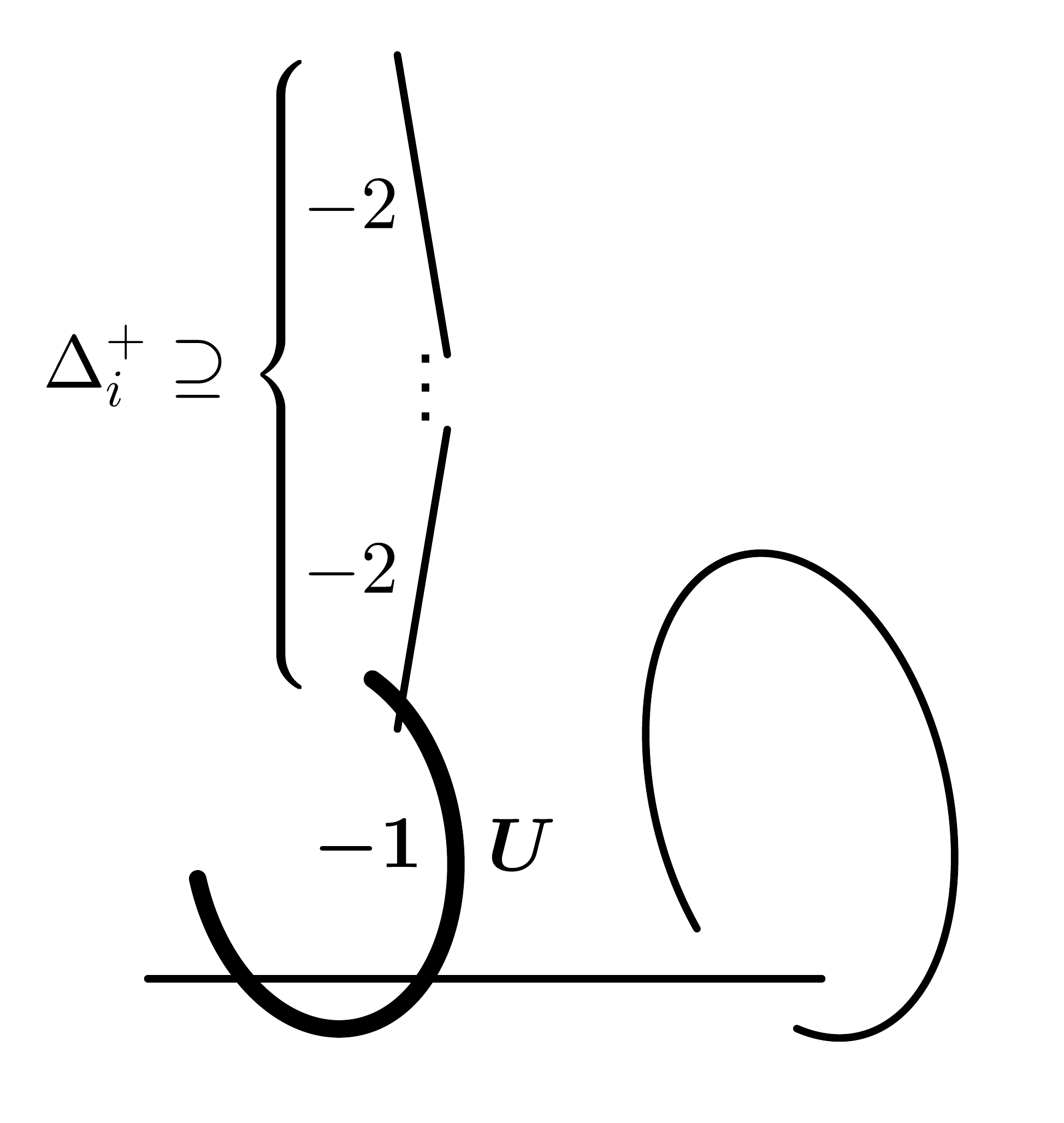}
	\end{subfigure}
	\begin{subfigure}{0.3\textwidth}\centering
		\includegraphics[scale=0.25]{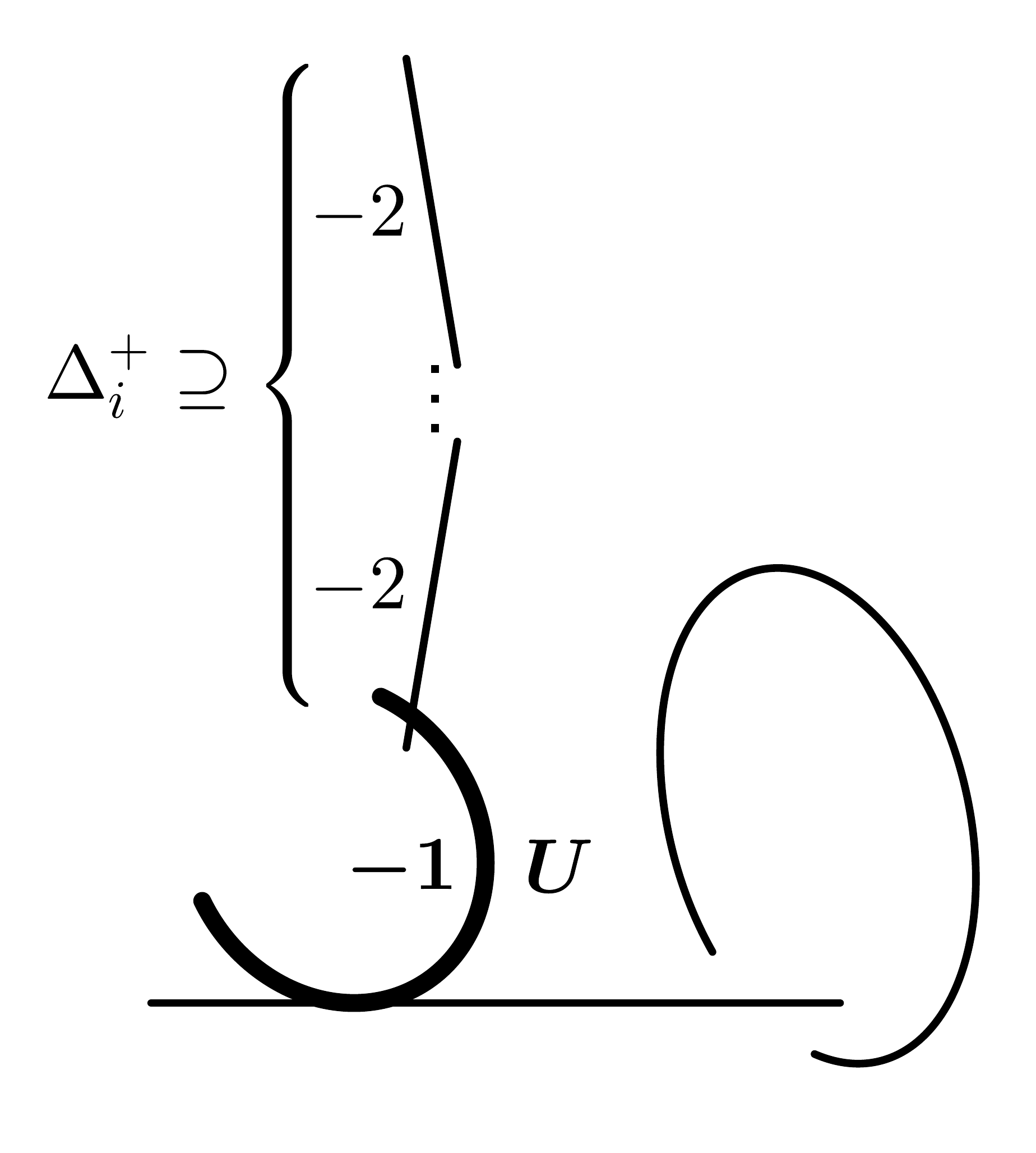}
	\end{subfigure}
	\caption{Possible arrangements of a component $U\subseteq\Upsilon_{i}$ inside $D_{i}$.}
	\label{fig:Upsilon}
\end{figure}

\begin{dfn}[Almost log exceptional curves]
	\label{def:ale}
	A $(-1)$-curve $A$ on $(X_{i},D_{i})$ satisfying \eqref{eq:bubble} is an \emph{almost log exceptional curve} on $(X_{i},\frac{1}{2}D_{i})$ if 
	\begin{equation}\label{eq:ale}
	A\cdot \Delta_{i}=1,  \quad A \mbox{ meets a tip of } \Delta_{i}\quad \mbox{ and } \quad  A\cdot (\Upsilon_{i}+\Delta_{i}^{+})=0.
	\end{equation}
\end{dfn}

\begin{uw}[see Figure \ref{fig:placement-of-A}] Let $A$ be almost log exceptional on $(X_{i},\tfrac{1}{2}D_{i})$ and let $\Delta_{A}$ be the maximal $(-2)$-twig of $D_{i}$ meeting $A$. If the tip of $\Delta_{A}$ meeting $A$ is not a tip of $D_{i}$ then $A=\Exc \psi_{A}$. In the other case, the inclusion $A\subseteq \Exc\psi_{A}$ is strict unless the component of $D_{i}$  meeting $\ftip{\Delta_A}$ meets $A$, too.
\end{uw}

\begin{figure}[ht]
	\begin{subfigure}{0.45\textwidth}\centering
		\includegraphics[scale=0.25]{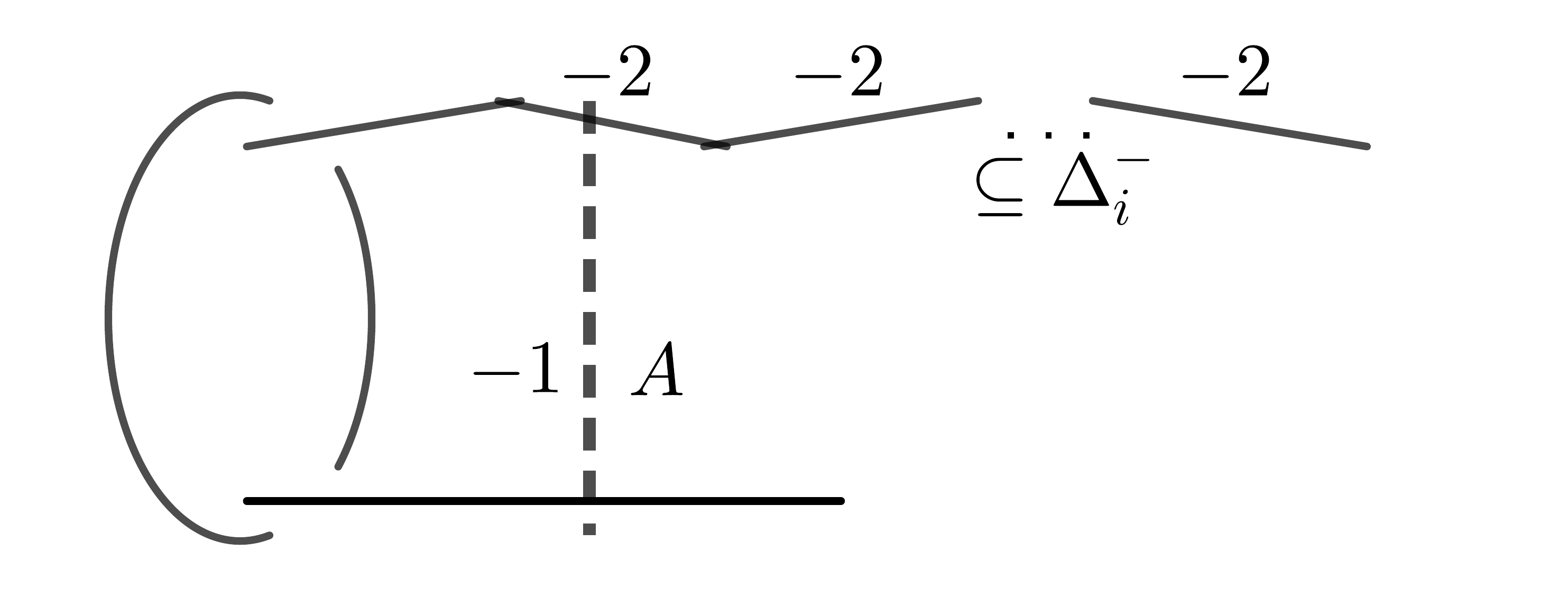}
		\caption{$A=\Exc\psi_{A}$}
	\end{subfigure}
	\hfill
	\begin{subfigure}{0.45\textwidth}\centering
		\includegraphics[scale=0.25]{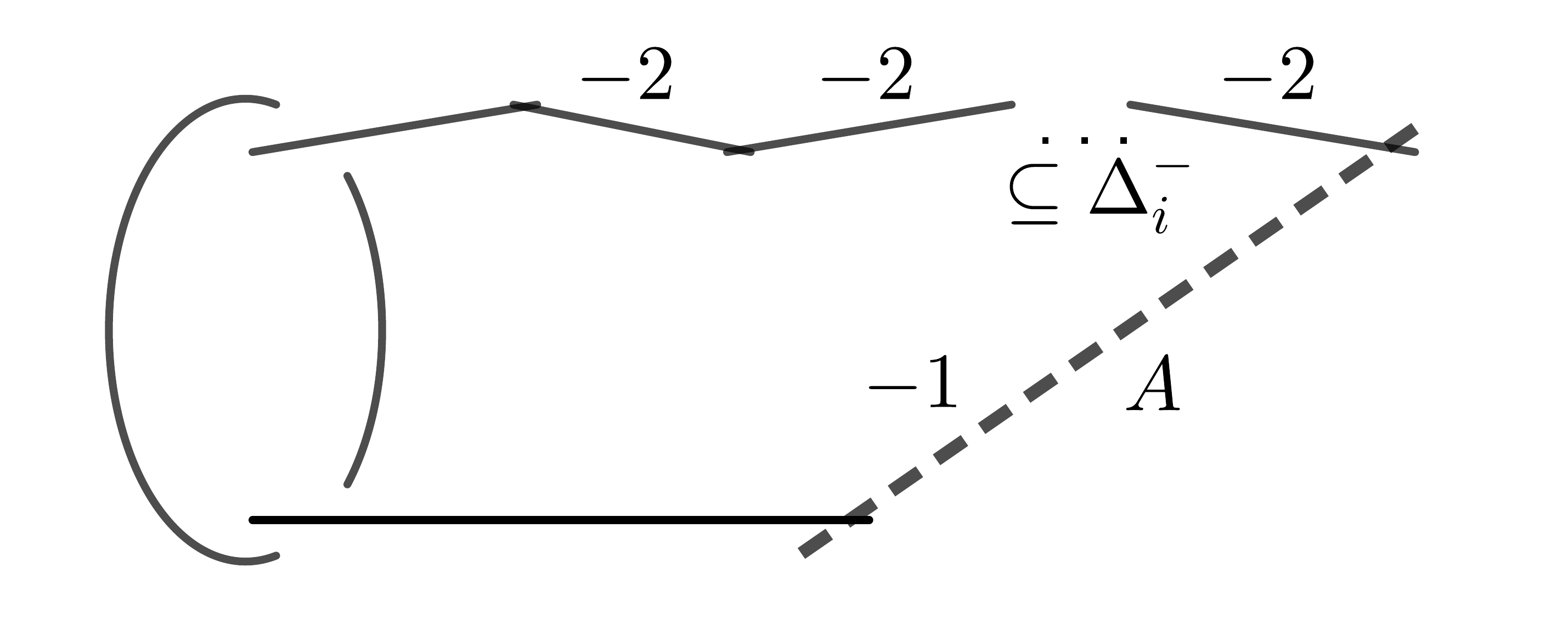}
		\caption{$A\subsetneq \Exc\psi_{A}$}
	\end{subfigure}
	\caption{Possible arrangements of an almost log exceptional curve $A$ on $(X_{i},\tfrac{1}{2}D_{i})$.}
	\label{fig:placement-of-A}
\end{figure}

\smallskip If $X_{i}$ contains an almost log exceptional curve, say  $A_{i+1}\subseteq X_{i}$, we put 
\begin{equation*}
\psi_{i+1}=\psi_{A_{i+1}}\colon X_{i}\to X_{i+1}\quad\mbox{and}\quad D_{i+1}=(\psi_{i+1})_{*}D_{i}.
\end{equation*}
Knowing that $X_{i+1}$ is smooth (see Lemma \ref{lem:MMP-properties}\ref{item:Xi_smooth}), we proceed by induction. The final pair $(X_{i},\frac{1}{2}D_{i})$, the one with $i=n$, is the first such that there is no almost log exceptional curve on $(X_{i},\frac{1}{2}D_{i})$. We call $(X_{n},\frac{1}{2}D_{n})$ an \emph{almost minimal model} of $(X_{0},\frac{1}{2}D_{0})$ and we put 
\begin{equation}
\psi\de \psi_{n}\circ\dots \circ \psi_{1}.
\end{equation}
We call the number $n$ the \emph{length} of the process ($\psi$) of almost minimalization. It equals the number of irreducible curves not contained in $D_{0}$ contracted by $\psi$.

\begin{uw}
	We do not claim that $\psi$ or $n$ is uniquely determined by $(X_{0},\tfrac{1}{2}D_{0})$. In general, they depend on the choices of the curves $A_{i}$, see Example \ref{ex:different_weak}. We will simply work with a fixed choice of $\psi$.
\end{uw}

For $(X_{i},\frac{1}{2}D_{i})$ as above we define the \emph{peeling morphism}
\begin{equation}\label{eq:peeling}
\alpha_{i}\colon\, (X_{i},\tfrac{1}{2}D_{i})\to (Y_{i},\tfrac{1}{2}D_{Y_{i}})
\end{equation}
where $D_{Y_{i}}=(\alpha_{i})_{*}D_{i}$, as the contraction of $\Delta_{i}+\Upsilon_{i}$ (it exists by Lemma \ref{lem:MMP-properties}\ref{item:Ui_disjoint}). 

\begin{uw}
	The motivation behind Definition \ref{def:ale} is \cite[Corollary 3.5]{Palka-minimal_models} which says that there are no log exceptional curves on $(Y_{i},\frac{1}{2}D_{Y_{i}})$ contained in $D_{Y_{i}}$ and that $A\subseteq X_{i}$ is almost log exceptional on $(X_{i},\frac{1}{2}D_{i})$ if and only if $\alpha_{i}(A)$ is log exceptional on $(Y_{i},\frac{1}{2}D_{Y_{i}})$. In particular, $(X_{\min},\tfrac{1}{2}D_{\min})\de (Y_{n},\tfrac{1}{2}D_{Y_{n}})$ is a minimal model of $(X_{0},\tfrac{1}{2}D_{0})$.
\end{uw}

As a consequence of basic theorems of the log minimal model program and of the construction of an almost minimal model for $(X_{0},\tfrac{1}{2}D_{0})$ in \cite[Section 3]{Palka-minimal_models}, we have the following result.

\begin{prop}[Properties of minimal models]\label{prop:MMP}
	Let $(X_{0},\frac{1}{2}D_{0})\toin{\psi_{1}} \dots \toin{\psi_{n}} (X_{n},\frac{1}{2}D_{n})$ be some almost minimalization of $(X_{0},\frac{1}{2}D_{0})$ as defined above and let $\alpha_{n}\colon (X_{n},\frac{1}{2}D_{n})\to (X_{\min},\frac{1}{2}D_{\min})$ be a peeling morphism. Then $\kappa(K_{X_{i}}+\frac{1}{2}D_{i})=\kappa(K_{X}+\frac{1}{2}D)$ and the following hold: 
	\begin{enumerate}
		\item\label{item:MMP_nef} If $\kappa(K_{X}+\frac{1}{2}D)\geq 0$ then $K_{X_{\min}}+\frac{1}{2}D_{\min}$ is numerically effective.
		\item\label{item:MMP_MFS} If $\kappa(K_{X}+\frac{1}{2}D)=-\infty$ then $X\setminus D$ has a $\C^{**}$-fibration or $ (X_{\min},\frac{1}{2}D_{\min})$ is a log del Pezzo surface of Picard rank one, that is,  $-(K_{X_{\min}}+\frac{1}{2}D_{\min})$ is ample and $\rho(X_{\min})=1$.
	\end{enumerate}
\end{prop}

\begin{rem}[The shape of $\Upsilon_{0}+\Delta_{0}^{+}$]\label{rem:semi-ordinary}\
	\begin{enumerate}
		\item\label{item:Ei_Ri} Let us note that $E_{i}\subseteq R_{i}$. This fact is implicitly used in the proof of \cite[Theorem 4.5(6)]{Palka-minimal_models} (in the form of the equality $c_{0}'=\#\Upsilon_{0}^{0}$). Let us give a proof. Clearly $E_{i}\not \subseteq \Delta_{i}$ because $E_{0}$, and hence each $E_{i}$ for $i\in \{1,\dots, n\}$, is not contained in any twig of $D_{i}$. Suppose that $E_{i}\subseteq \Upsilon_{i}$. Then, since $\beta_{D_{0}}(E_{0})=\sum_{j=1}^{c}(\tau_{j}+1-s_{j})\geq 2$, we get $E_{0}\subseteq \Upsilon_{0}$, so  $(c,\tau_{1},s_{1})=(1,2,1)$, that is, the multiplicity sequence of the unique cusp $q_{1}\in \bar{E}$ consists of even terms followed by $(1,1)$ at the end, and $E_{0}^{2}=-1$, so $E^{2}=-3$. Then Lemma \ref{lem:HN-equations}\eqref{eq:deg^2} gives $(\deg\bar{E})^{2}\equiv 3\pmod{4}$; a contradiction.
		
		\item\label{item:semiordinary}By \ref{item:Ei_Ri}, $\Upsilon_0\subseteq C_1+\cdots+C_c$, so it is easy to see that the divisor $\Upsilon_{0}+\Delta_{0}^{+}$ equals the sum of exceptional divisors over the semi-ordinary cusps of $\bar{E}$ (see Example \ref{ex:ordinary_cusp}). By Definition \ref{def:ale} its push-forward on $X_{i}$ does not meet almost log exceptional curves. Consequently, if all cusps of $\bar{E}$ are semi-ordinary then $n=0$, that is, $(X_{0},\frac{1}{2}D_{0})$ is already almost minimal.
	\end{enumerate}
\end{rem}

\begin{lem}[Properties of $(X_{i},\frac{1}{2}D_{i})$, \cite{Palka-minimal_models}]\label{lem:MMP-properties} 
	Let $(X_{i},D_{i})$ be as above. Then
	\begin{enumerate}
		\item\label{item:Xi_smooth} $X_{i}$ is smooth, $E_{i}$ is smooth and $D_{i}-E_{i}$ is snc.
		\item\label{item:Ui_disjoint} The components of $\Upsilon_{i}$ are disjoint.
		\item\label{item:K_min-formula} $\alpha_{i}^{*}(K_{Y_{i}}+\frac{1}{2}D_{Y_{i}})=K_{X_{i}}+\frac{1}{2}D_{i}^{\flat}$.
		\item\label{item:Exc-psi_i} $\Exc \psi_{i}$ is a chain and $\Exc \psi_{i}-A_{i}$ is contained in (at most two) maximal twigs of $D_{i-1}$. 
		\item\label{item:center-psi_i} The point $\psi_{i}(A_{i})$ is a point of normal crossings of two components of $D_{i}$. The set $\Supp D_{i} \setminus \psi_{i}(A_{i})$ is connected.
		\item\label{item:Ui_not_touched} 
		$(\psi_i)_*(\Upsilon_{i-1})\subseteq \Upsilon_{i}$ and $(\psi_i)_*(\Delta_{i-1}^{+})\subseteq \Delta_{i}^{+}.$
		\item\label{item:Delta_pr-tr} The morphism $\psi_{i}$ does not touch the proper transforms of the twigs of $D_{i}$. In particular, $(\psi_{i}^{-1})_{*}\Delta_{i}\subseteq \Delta_{i-1}$ and $(\psi_{i}^{-1})_{*}\Delta_{i}^{-}\subseteq \Delta_{i-1}^{-}$.
		\item\label{item:R_pr-tr} $(\psi_{i}^{-1})_{*}R_{i}\subseteq R_{i-1}$.
	\end{enumerate}
\end{lem}
\begin{proof}
	Parts \ref{item:Xi_smooth},\ref{item:Ui_disjoint} and \ref{item:K_min-formula} are proved in \cite[Proposition 4.1(i) and Lemma 3.4(i)]{Palka-minimal_models}. Parts \ref{item:Exc-psi_i} and \ref{item:center-psi_i} follow from the fact that $D_{i-1}$ is connected and $A_{i}\cdot D_{i-1}=2$, so $A_{i}$ lies in a circular subdivisor of $D_{i-1}+A_{i}$, and $\psi_{i}$ contracts superfluous $(-1)$-curves, which are non-branching in the images of $D_{i-1}+A_{i}$. 
	By  \cite[Proposition 4.1(iii)]{Palka-minimal_models} $(\psi_i)_*(\Upsilon_{i-1})\subseteq \Upsilon_{i}$, so $\Upsilon_{i-1}$ is not touched by $\psi_i$. Then $\Delta_{i-1}^{+}$ is not touched by $\psi_i$, because $A_i\cdot \Delta_{i-1}^+=0$. This gives  \ref{item:Ui_not_touched}.
	Part \ref{item:Delta_pr-tr} is a direct consequence of \ref{item:center-psi_i} and \ref{item:Ui_not_touched}.
	
	For the proof of \ref{item:R_pr-tr} let $G$ be a component of $D_{i-1}$ such that $G\not\subseteq \Exc\psi_{i}$ and $\psi_i(G)\subseteq R_i$. By \ref{item:Ui_not_touched}, $G$ is not a component of $\Upsilon_{i-1}$. Suppose that it is contained in some $(-2)$-twig $\Delta_G$ of $D_{i-1}$. The morphism $\psi_i$ touches but does not contract $G$, so $A_i$ meets $\Delta_G$. There is no component of $D_{i-1}-\Delta_G$ meeting both $G$ and $A_i$, because otherwise we get $\psi_i(G)\subseteq \Upsilon_i$, contrary to the assumption. It follows that $A_i$ meets $G$. But then either $\psi_i$ contracts $G$ or again $\psi_i(G)\subseteq \Upsilon_i$; a contradiction.
\end{proof}

\smallskip
\subsection{Consequences of non-existence of a $\C^{**}$-fibration.}\label{sec:consequences_of_noCstst}
As before, assume $\bar E\subseteq \P^2$ is a rational cuspidal curve for which  Negativity Conjecture \ref{conj} holds. Let $\pi\:(X,D)\to (\P^2,\bar E)$ and $\pi_0\:(X_0,D_0)\to (\P^2,\bar E)$ denote respectively the minimal log resolution and the minimal weak resolution. We have $\kappa(K_X+\frac{1}{2}D)=\kappa(K_{X_0}+\frac{1}{2}D_0)$ by Proposition \ref{prop:MMP}. As discussed above, by the existing classification results we may reformulate our assumptions as:
\begin{equation}\label{eq:assumption}
\parbox{.9\textwidth}{\bfseries
	$\bar{E}\subseteq \P^{2}$ is a rational cuspidal curve such that $\kappa(\P^{2}\setminus \bar{E})=2$, $\kappa(K_{X_{0}}+\tfrac{1}{2}D_{0})=-\infty$ and $\P^{2}\setminus \bar{E}$ has no $\C^{**}$-fibration}
\end{equation}

By Proposition \ref{prop:MMP}\ref{item:MMP_MFS}, $(X_{\min},\frac{1}{2}D_{\min})$ is a log del Pezzo surface of Picard rank $1$, hence the number $(2K_{X_{\min}}+D_{\min})^{2}$ is positive. This number is computed in \cite[Lemma 4.4]{Palka-minimal_models} and it is shown in Theorem 4.5(6) loc.\ cit.\ that its positivity bounds the number of components of $D_{n}$. The explicit formula for this bound is important for us. It can be conveniently formulated in terms of contributions $\lambda_{j}$, $j\in\{1,\dots, c\}$ of cusps, see \eqref{eq:lambdaineq}, which we define as (see Notation \ref{not:graphs} and \ref{not:MMP}):
\begin{equation}\label{eq:deflambda}
\lambda_{j}=
\tau_{j}-s_{j}+\# (\psi_{*}Q_{j}- \psi_{*}Q_{j}\wedge \Upsilon_{n}^{0})-b_{0}(\psi_{*}Q_{j}\wedge \Delta_{n}).
\end{equation}

\begin{rem}[Properties of $\lambda_{j}$]\label{rem:lambda>1} The following hold:
	\begin{enumerate}
		\item\label{item:lambda_not-ordinary} If $q_{j}\in\bar{E}$ is not ordinary, then
		\begin{equation}\label{eq:lambda>=tau}
		\lambda_{j}\geq \tau_{j}-s_{j}+1\geq \tau_{j}\geq 2.
		\end{equation}
		\item\label{item:lambda_semi-ordinary} If $q_{j}\in \bar{E}$ is semi-ordinary then $\lambda_{j}=1+t_{j}$. In particular, $\lambda_{j}=1$ if and only if $q_{j}\in\bar{E}$ is ordinary.
		\item\label{item:lambda_not-touched} If $q_{j}\in \bar{E}$ is not ordinary and $\psi$ does not touch $Q_{j}$ then $\lambda_{j}=\tau_{j}-s_{j}+\#Q_{j}-b_{0}(Q_{j}\wedge \Delta_{0}).$
	\end{enumerate}
\end{rem}	

\begin{proof}
	\ref{item:lambda_not-ordinary} Assume that $q_{j}$ is not an ordinary cusp. Then $\psi(C_{j})\not\subseteq \Upsilon_{n}^{0}+\Delta_{n}$, so $\#(\psi_{*}Q_{j}-\psi_{*}Q_{j}\wedge (\Upsilon_{n}^{0}+\Delta_{n}))\geq 1$ and hence $\lambda_{j}\geq \tau_{j}-s_{j}+1$, which proves \eqref{eq:lambda>=tau} since $s_{j}\leq 1$ and $\tau_{j}\geq 2$ by definition.

	\ref{item:lambda_semi-ordinary} Assume now that $q_{j}\in \bar{E}$ is a semi-ordinary cusp, that is, $(\tau_{j},s_{j})=(2,1)$ and $Q_{j}=[(2)_{t_{j}},1]$ (see Example \ref{ex:ordinary_cusp}). By Remark \ref{rem:semi-ordinary}\ref{item:semiordinary}, $\psi$ does not touch $Q_{j}$, so $\#\psi_{*}Q_{j}=t_{j}+1$. Moreover, if $t_{j}=0$ then we have $\psi_{*}Q_{j}\wedge \Upsilon_{n}^{0}=\psi(C_{j})$ and $b_{0}(\psi_{*}Q_{j}\wedge \Delta_{n})=0$; otherwise  $\psi_{*}Q_{j}\wedge \Upsilon_{n}^{0}=0$ and $b_{0}(\psi_{*}Q_{j}\wedge \Delta_{n})=1$.
		
	\ref{item:lambda_not-touched} If $\psi$ does not touch $Q_{j}$ then by Lemma \ref{lem:MMP-properties}\ref{item:Ui_not_touched}, $\psi_{*}Q_{j}\wedge \Delta_{n}=\psi_{*}(Q_{j}\wedge \Delta_{0})$ and $\psi_{*}Q_{j}\wedge \Upsilon_{n}^{0}=\psi_{*}(Q_{j}\wedge \Upsilon^{0}_{0})$. The latter divisor is zero if $q_{j}\in\bar{E}$ is not ordinary.
\end{proof}

\begin{lem}[The basic inequality] Let the assumptions be as in \eqref{eq:assumption}. Then
	\begin{equation}\label{eq:lambdaineq}
	\lambda_1+\cdots+\lambda_c \leq 6.
	\end{equation}	
\end{lem}

\begin{proof}
	By Lemma \ref{lem:MMP-properties}\ref{item:K_min-formula}, we have $0<(2K_{X_{\min}}+D_{X_{\min}})^{2}=(2K_{X_{n}}+D_{n}^{\flat})^{2}$. Let $(X_n',D_n')\to (X_n,D_n)$ be the minimal log resolution of $(X_{n},D_{n})$. By \cite[Lemma 4.4]{Palka-minimal_models}:
	\begin{equation*}
	(2K_{X_{n}}+D_{n}^{\flat})^{2}=3h^{0}(2K_{X}+D)+8+b_{0}(\Delta_{n}')+\#\Upsilon_{n}^{0}-\rho(X_{n}')-n-\sum_{T}\frac{1}{d(T)},
	\end{equation*}
	where the last sum runs over all connected components $T$ of $\Delta_{n}^{-}$ and $\Delta_{n}'$ is the sum of the $(-2)$-twigs of $D_{n}'$. Lemma \ref{lem:MMP-properties}\ref{item:center-psi_i} implies that $\psi$ is an isomorphism near the non-nc points of $D_{0}$, so $(X_{n}',D_{n}')\to (X_{n},D_{n})$ resolves the tangency points of $E_{n}$ and $\psi_{*}C_{j}$. The exceptional divisor over such point contains $\tau_{j}$ components and $s_{j}$ maximal $(-2)$-twigs of $D_{n}'$ (see Section \ref{sec:resolutions}). Thus $b_{0}(\Delta_{n}')=b_{0}(\Delta_{n})+\sum_{j=1}^{c}s_{j}$ and $\rho(X_{n}')+n=\#D_{n}'= 1+\sum_{j=1}^{c}(\tau_{j}+\#\psi(Q_{j})).$ Therefore,
	\begin{equation*}
	0<(2K_{X_{\min}}+D_{X_{\min}})^{2}=3h^{0}(2K_{X}+D)+7-\sum_{j=1}^{c}\lambda_{j}-\sum_{T}\frac{1}{d(T)}\leq3h^{0}(2K_{X}+D)+ 7-\sum_{j=1}^{c}\lambda_{j}.
	\end{equation*}
	Eventually, $h^{0}(2K_{X}+D)=0$ by our assumptions, so \eqref{eq:lambdaineq} follows.
\end{proof}

\begin{lem}[Pullbacks of almost log exceptional curves]\label{lem:ale_pr-tr}
Let $X_k$, $k\in \{0,\ldots, n-1\}$ be one of the surfaces in Proposition \ref{prop:MMP}. For $i\geq k+1$ denote by $A_i'$ the proper transform of $A_{i}$ on $X_k$. Then
	\begin{enumerate}
		\item\label{item:ale_pr-tr} $A_{i}'$ is almost log exceptional on $(X_k,\tfrac{1}{2}D_k)$.
		\item\label{item:Exc-psi_i-disjoint} Total transforms on $X_{k}$ of the divisors $\Exc\psi_i$, $i\geq k+1$ are equal to their proper transforms, hence are pairwise disjoint.
		\item\label{item:Exc-psi_i_pr-tr} If \eqref{eq:assumption} holds then the proper transform of $\Exc\psi_{i}$ on $X_k$ equals $\Exc\psi_{A_{i}'}$.
	\end{enumerate}
\end{lem}
\begin{proof}
 By induction we may assume $i=k+2$. Put $A=A'_{k+2}\subseteq X_{k}$.
	
	\ref{item:ale_pr-tr} Lemma \ref{lem:MMP-properties}\ref{item:center-psi_i} implies that $A_{k+2}$ does not pass through $\psi_{k+1}(A_{k+1})$, so $\psi_{k+1}$ does not touch $A$. It follows that $A$ is as in \eqref{eq:bubble}. Let $W\subseteq \Delta_{k+1}^-$ be the maximal $(-2)$-twig of $D_{k+1}$ meeting $A_{k+2}$ and let $W'=(\psi_{k+1}^{-1})_{*}W$.  By Lemma \ref{lem:MMP-properties}\ref{item:Ui_not_touched},\ref{item:Delta_pr-tr} 
	\begin{equation}\label{eq:A_Ups}
	A\cdot (\Upsilon_{k}+\Delta_{k}^+)\leq A_{k+2}\cdot (\Upsilon_{k+1}+\Delta_{k+1}^{+})=0\quad \mbox{and}\quad A\cdot\Delta_{k}\geq A\cdot W'=1.
\end{equation}
Suppose that \ref{item:ale_pr-tr} fails. Suppose further that $W'$ is not a maximal $(-2)$-twig. Then there is a component $W_0'\subseteq \Delta_{k}$  meeting $W'$, such that $\psi_{k+1}(W_0')\nsubseteq \Delta_{k+1}$. By Lemma \ref{lem:MMP-properties}\ref{item:Exc-psi_i} it follows that $A_{k+1}$ meets $W_0'$ and hence $\Exc \psi_{k+1}=A_{k+1}$. But then $\psi_{k+1}(W_0')\subseteq \Upsilon_{k+1}$ and $W\subseteq \Delta_{k+1}^+$; a contradiction. Thus $W'$ is a maximal $(-2)$-twig of $D_{k}$. Since \ref{item:ale_pr-tr} fails, we have $A\cdot \Delta_{k}\geq 2$ by \eqref{eq:A_Ups}. Then $A$ meets a different maximal $(-2)$-twig $W''\subseteq D_{k}$. Since $A_{k+2}$ is almost log exceptional, $(\psi_{k+1})_*(W'')$ is not a $(-2)$-twig. Then $A_{k+1}$ meets $W''$. Since $\psi_{k+1}$ does not touch $A$, it does not contract the component of $W''$ meeting $A$, hence the latter becomes necessarily a component of $\Upsilon_{k+2}$. But then $A_{k+2}$ is not almost log exceptional; a contradiction.
	
\ref{item:Exc-psi_i-disjoint}, \ref{item:Exc-psi_i_pr-tr} By Lemma \ref{lem:MMP-properties}\ref{item:Exc-psi_i} $\Exc\psi_{k+2}-A_{k+2}$ is contained in the sum of twigs of $D_{k+1}$, so by Lemma \ref{lem:MMP-properties}\ref{item:center-psi_i} it does not pass through the image of $A_{k+1}$. Consequently, the proper and the total transforms of $\Exc \psi_{k+2}$ on $X_k$ are equal. Denote them by $T$. The self-intersection and branching numbers of the components of $\Exc \psi_{k+2}$ and of their proper transforms are the same, so $T\subseteq \Exc\psi_{A}$. Suppose that the inclusion is proper and let $\sigma$ be the contraction of $T$ (it is a part of $\psi_{A}$). Then there is a component $V$ of $\Exc \psi_A-T$ such that $\sigma(V)$ is a superfluous $(-1)$-curve in $\sigma_{*}D_{k}$ and the chain $\sigma_{*}(\Exc \psi_{k+1})$ meets $V$. Now contract the $(-1)$-curves in the subsequent images of $\sigma_{*}(\Exc \psi_{k+1})$ until the image of $\sigma(V)$, say $V'$, becomes a $0$-curve. The branching number of $V'$ in the image of $D_{k}$ equals at most $\beta_{\sigma_{*}(D_{k}+A_{k+1})}(\sigma_{*}V)\leq \beta_{\sigma_{*}D_{k}}(\sigma_{*}V)+1\leq 3$. Thus $|V'|$ induces a $\P^{1}$-fibration which restricts to a $\C^{**}$-fibration of some open subset of $\P^{2}\setminus \bar{E}$; a contradiction with \eqref{eq:assumption}.
\end{proof}

\begin{rem}[Changing the process of almost minimalization]
	Lemma \ref{lem:ale_pr-tr}\ref{item:Exc-psi_i_pr-tr} implies that given a process of almost minimalization $\psi$ of $(X_0,\frac{1}{2}D_0)$ as above and two almost log exceptional curves $A_i$, $A_j$, $j>i$ in this process, there exist another process of almost minimalization $\psi$ which agrees with $\psi$ until it reaches $X_{i-1}$ but at the $i$-th step contracts $A$, the proper transform of $A_j$, instead of $A_i$, and moreover that the respective $\Exc\psi_A$ is the proper transform of $\Exc \psi_{j}$. However, after such reordering it may happen that the push-forward of $A_i$ on $X_{i}$ is no longer almost log exceptional (hence at the end we may get a non-isomorphic almost minimal model). Therefore, we cannot freely change the order of contractions of almost log exceptional curves. This is illustrated by the example below.
\end{rem}

\begin{ex}[Non-uniqueness of almost minimal models]\label{ex:different_weak}
	Let us look at a curve of type $\cJ(2)$ (the situation for other $\cJ(k)$ is similar, see Propositions \ref{prop:J}--\ref{prop:n3}). It is constructed in Section \ref{sec:IJ}. 
	Figure \ref{fig:ex2} illustrates two possible courses ($\psi_{3}\circ \psi_{2}\circ \psi_{1}$ and $\tilde\psi_{2}\circ \psi_{1}$) of the almost MMP for the pair $(X_{0},\frac{1}{2}D_{0})$.
	
	\end{ex}
		\begin{figure}[htbp]
			\begin{displaymath}
			\xymatrix @C=5pc{ 
				\begin{subfigure}{0.45\textwidth}
				\includegraphics[scale=0.4]{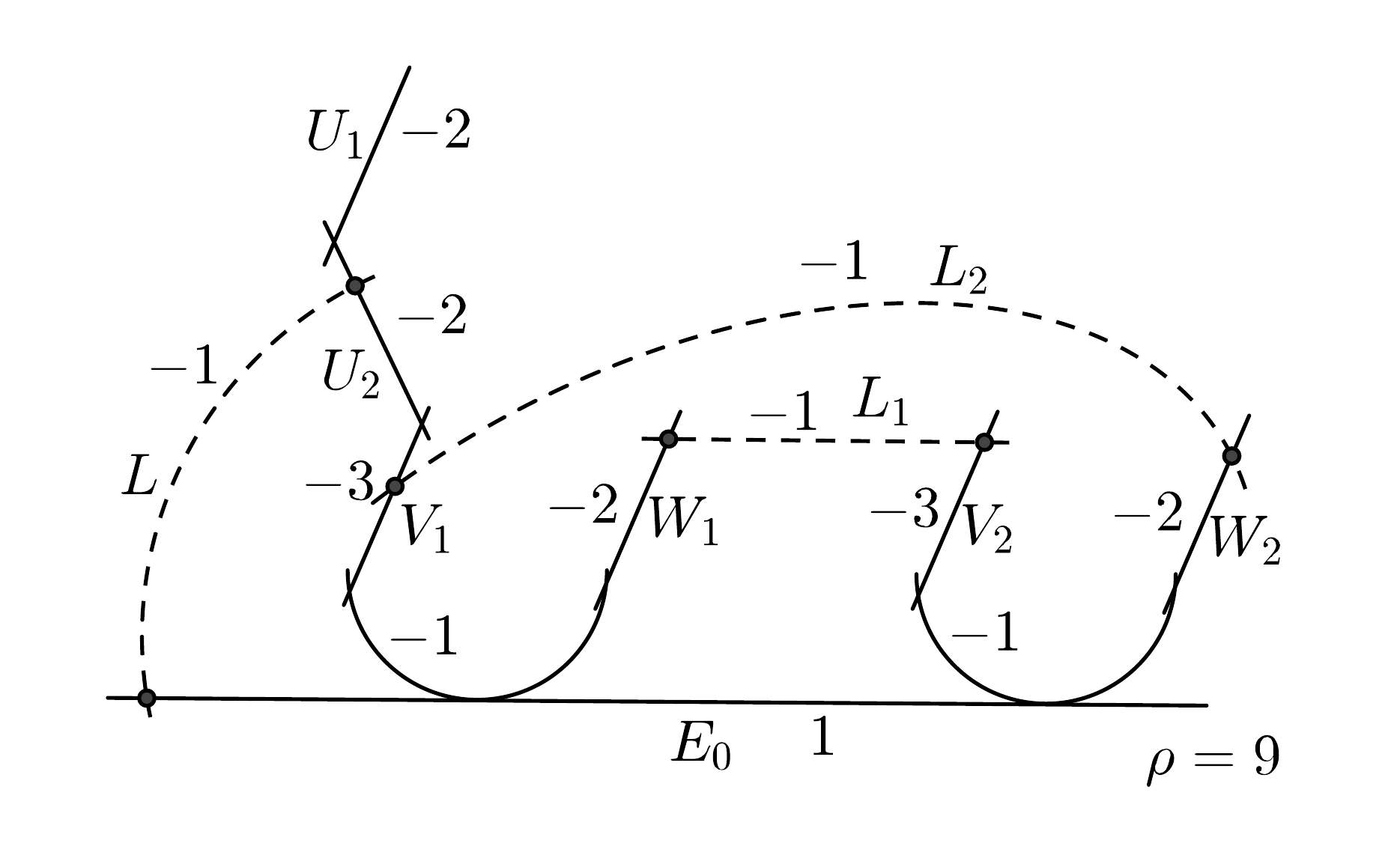}
				\end{subfigure}
				\ar^{\Exc \psi_{1}=L_{1}+W_{1}+V_{2}}[d]
				&
				\\ 
				\begin{subfigure}{0.45\textwidth}
				\includegraphics[scale=0.4]{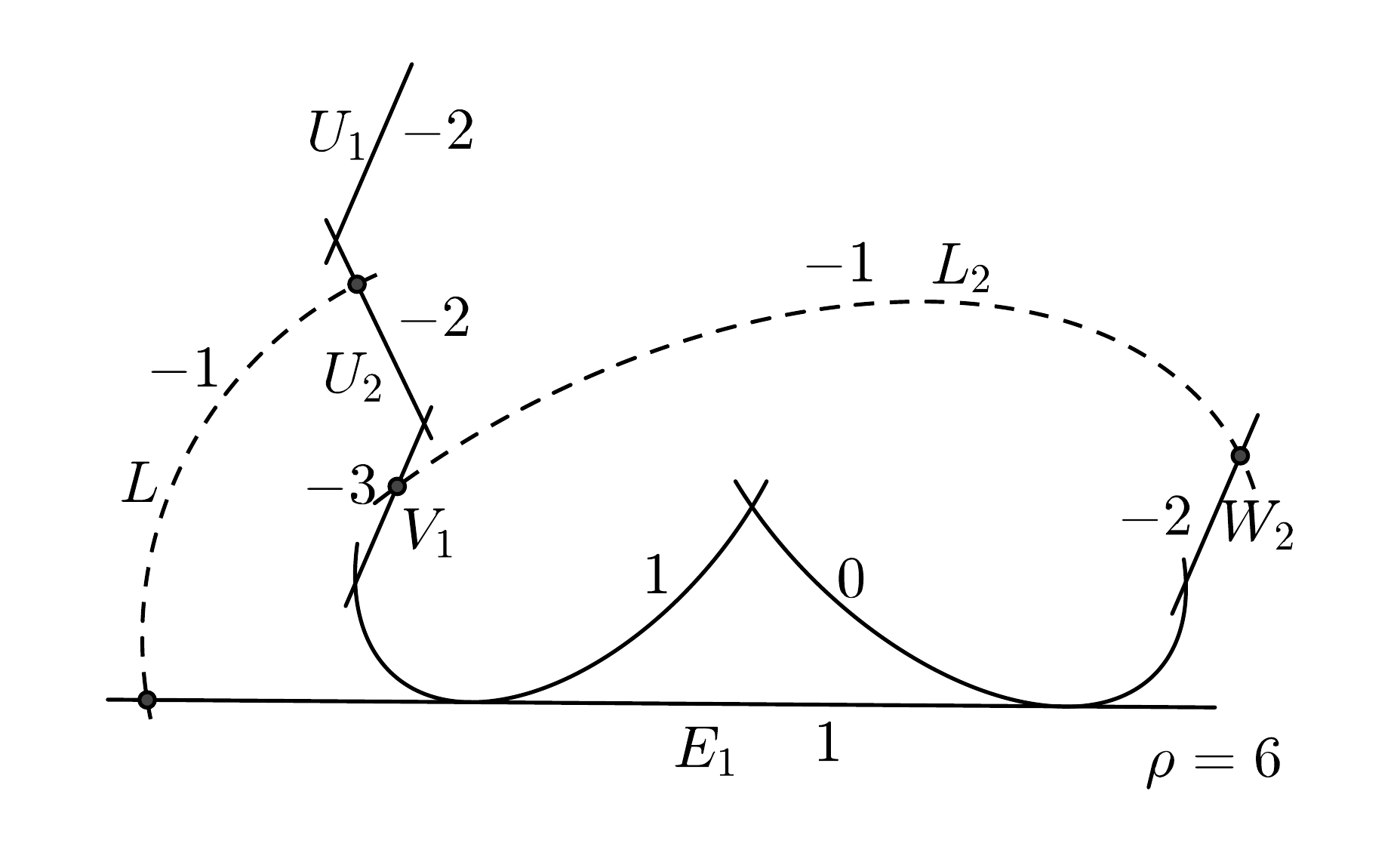}
				\end{subfigure}
				\ar^{\Exc \tilde{\psi}_{2}=L_{2}+W_{2}}[d]
				\ar_{\Exc\psi_{2}=L}[r]
				&
				\begin{subfigure}{0.45\textwidth}
				\includegraphics[scale=0.4]{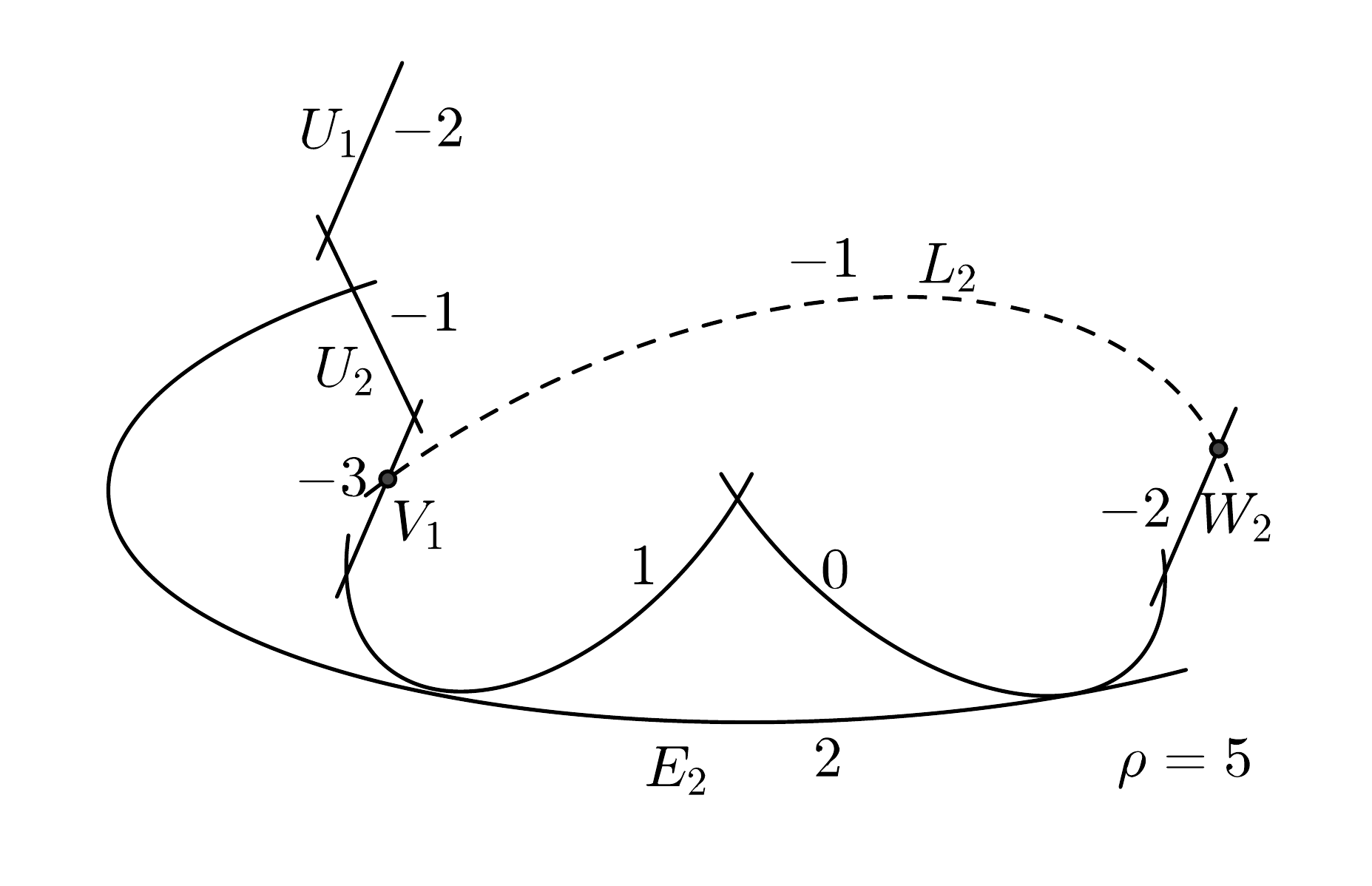}
				\end{subfigure}
				\ar^{\Exc \psi_{3}=L_{2}+W_{2}}[d]
				\\
				\begin{subfigure}{0.45\textwidth}
				\includegraphics[scale=0.4]{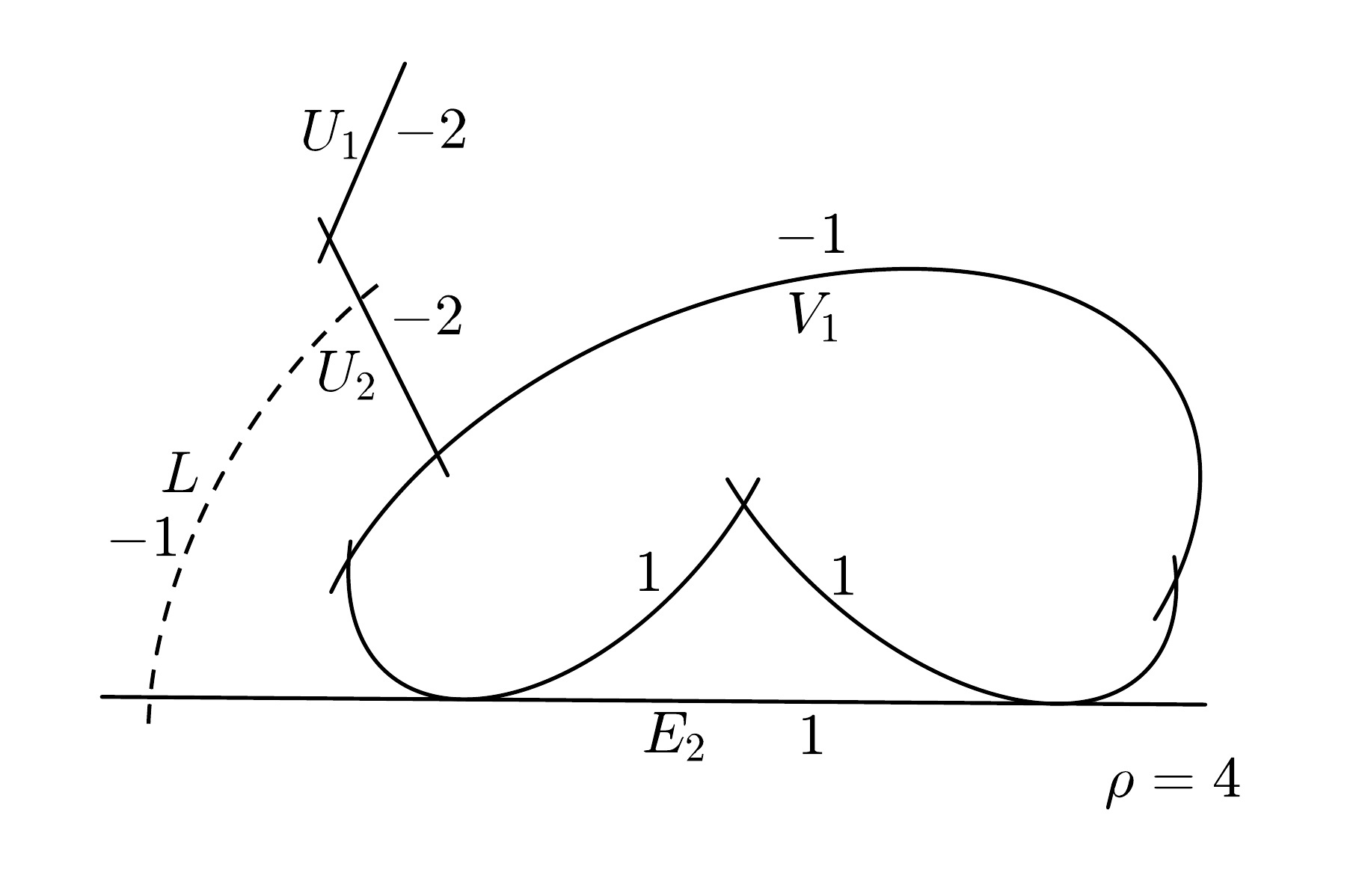}
				\end{subfigure}
				\ar^(.6){\Exc \alpha_{2}=V_{1}+U_{2}+U_{1}}[d]
				&
				\begin{subfigure}{0.45\textwidth}
				\includegraphics[scale=0.4]{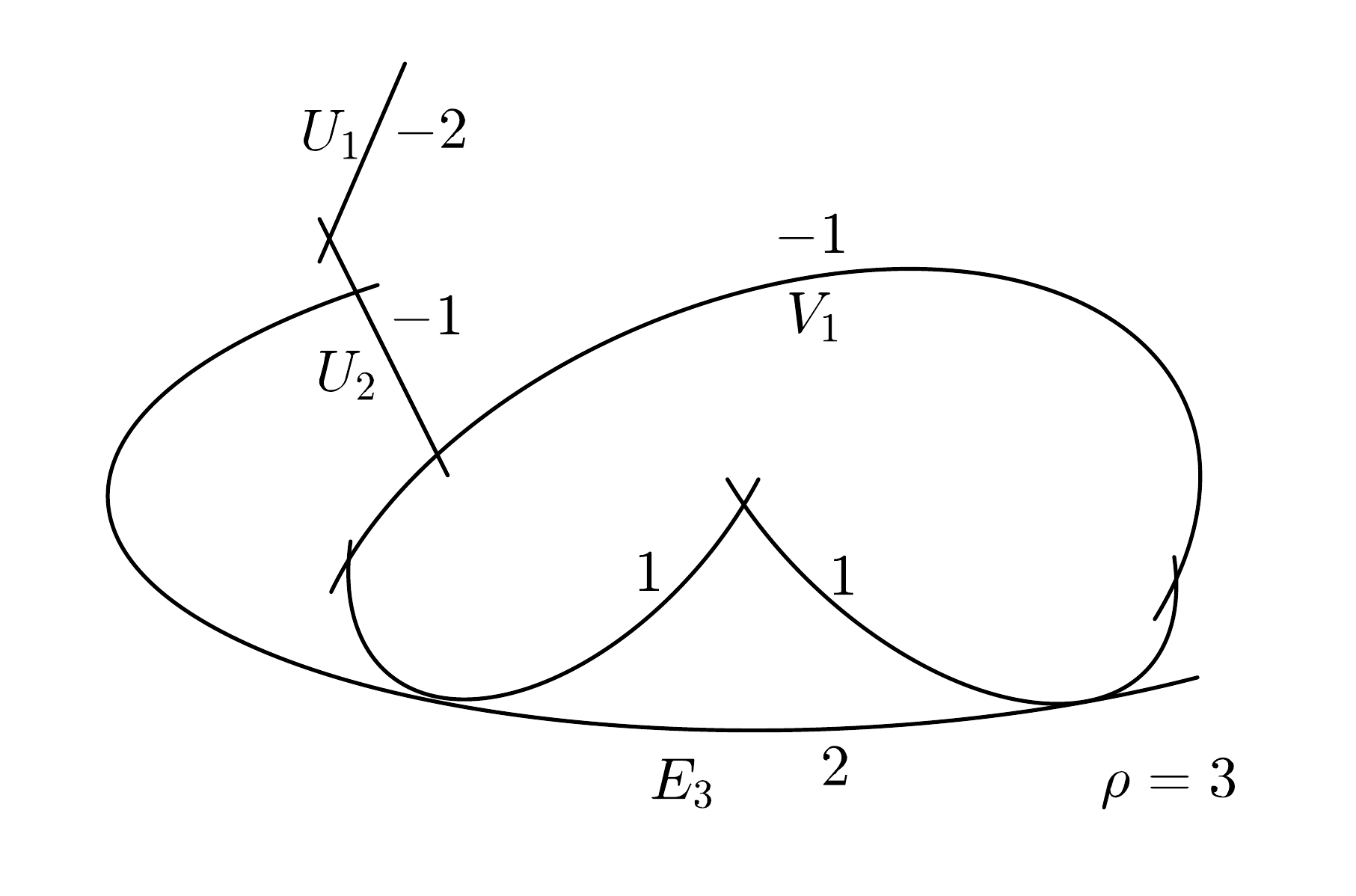}
				\end{subfigure}
				\ar^(.55){\Exc \alpha_{3}=U_{2}+U_{1}}[d]
				\\
				\begin{subfigure}{0.45\textwidth}\centering
				\includegraphics[scale=0.4]{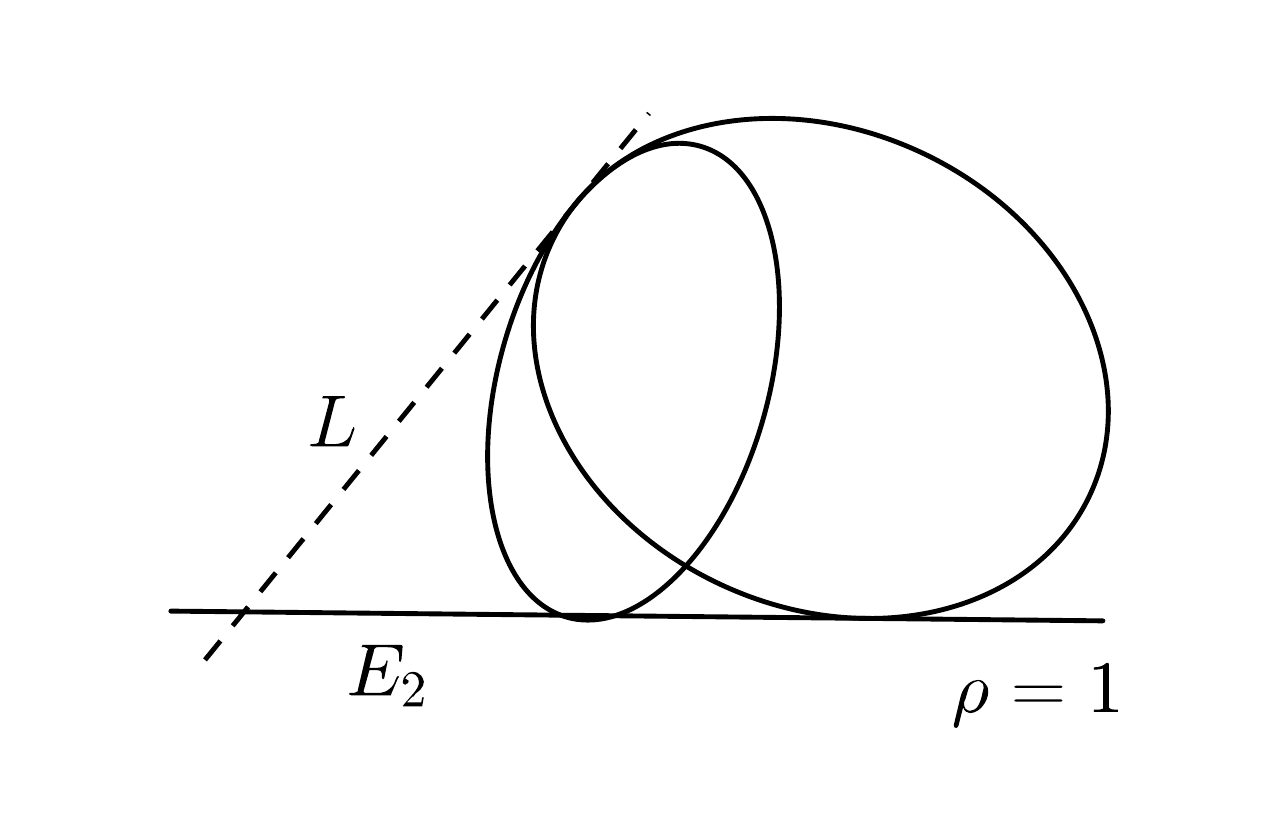}
				\end{subfigure}
				&
				\begin{subfigure}{0.45\textwidth}\centering
				\includegraphics[scale=0.4]{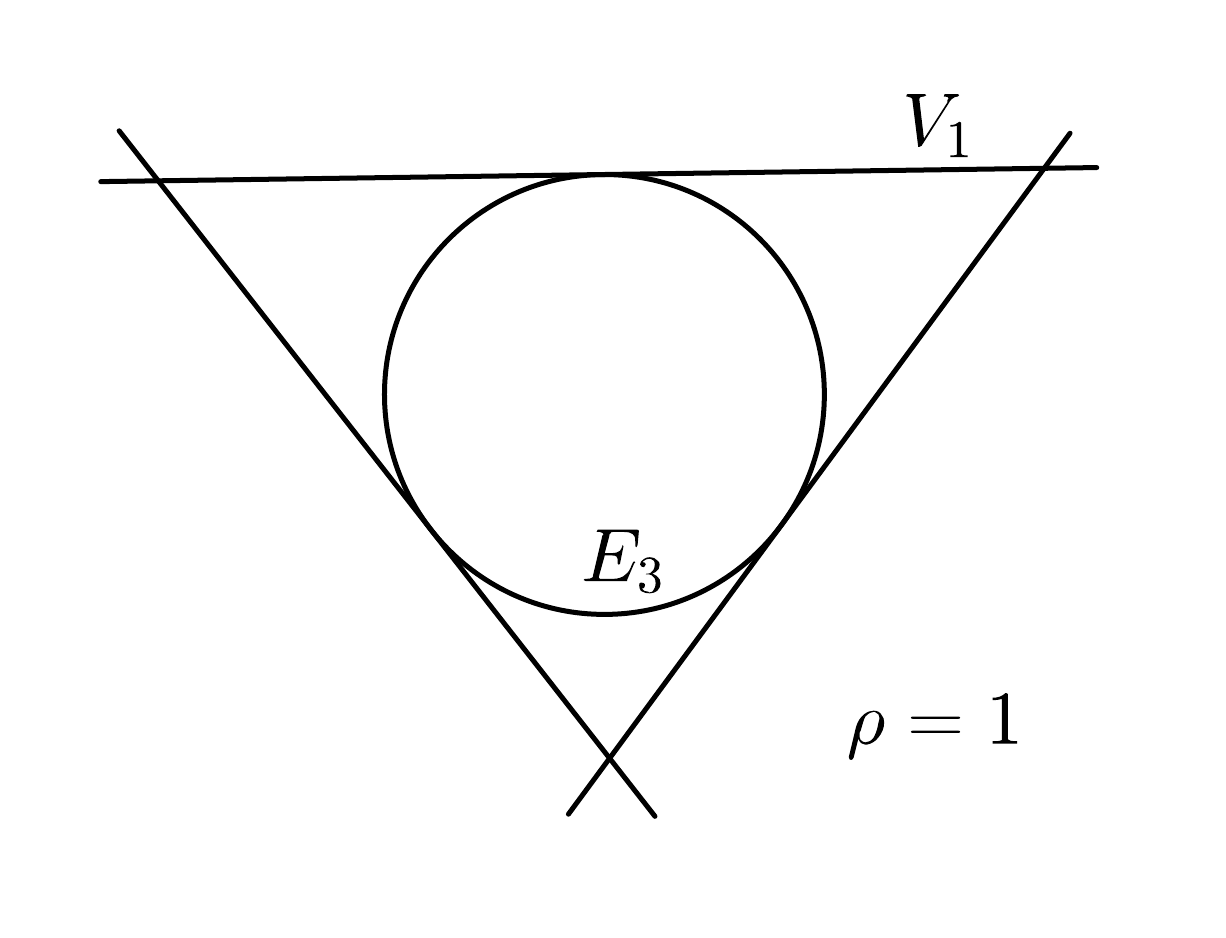}
				\end{subfigure}
			}
			\end{displaymath}	
			\caption{Two minimal models of $(X_0,\frac{1}{2}D_0)$ for $\bar{E}\subseteq \P^{2}$ of type $\cJ(2)$.}
			\label{fig:ex2}
		\end{figure}
	
	First, we find relevant almost log  exceptional curves as in Figure \ref{fig:ex2}. For $j\in\{1,2\}$ denote by $V_{j}$, $W_{j}$ respectively the $(-3)$-curve and the $(-2)$-curve meeting $C_{j}$, and denote the first and the last tip of the long $(-2)$-twig of $Q_{1}$ by $U_{1}$ and $U_{2}$, respectively. Let $L$ be the proper transform on $X_{0}$ of the line $\ll$ tangent to $\bar{E}$ at $q_{1}$. The multiplicity sequence of $q_{1}$ is $(4,4,4,2,2)$, hence the inequality $4< (\ll\cdot \bar{E})_{q_{1}}\leq \deg \bar{E}=9$ implies that $(\ll\cdot \bar{E})_{q_{1}}=8$, so by the projection formula $L\cdot U_{2}=1$ and $\ll\cdot \bar{E}-(\ll\cdot \bar{E})_{q_{1}}=1$. It follows that $L\cdot (D_{0}-U_{2})=1=L\cdot E_{0}$. We now show the existence of $(-1)$-curves $L_{j}$, $j\in \{1,2\}$ satisfying $L_{j}\cdot D_{0}=2$ and meeting $D_{0}$ on $W_{j}$ and $V_{3-j}$. To this end, let $\eta\colon X_{0}\to \tilde {X}_{0}$ be the contraction of $L+(D_{0}-E_{0}-V_{1})$. Then $\rho(\tilde{X}_{0})=1$, so $\tilde X_{0}\cong \P^{2}$ and $\eta(E_{0})$ is a cuspidal curve with cusps $p_{1}$, $p_{2}$ whose multiplicity sequences are $(2,2)$ and $(4,2,2)$, respectively. Moreover, $\eta(V_{1})$ is a line meeting $\eta(E_{0})$ at $\eta(L)$ with multiplicity $2$ and at $p_{1}$ with multiplicity $4$. Thus $\deg \eta(E_{0})=6$. Let $\ll_{2}$ be the line tangent to $\eta (E_{0})$ at $p_{2}$. Then $\ll_{2}\cdot \eta(E_{0})=6$, so $\ll_{2}$ does not meet $\eta(E_{0})$ in any other point and meets the line $\eta(V_{1})$ off $\eta(E_{0})$. Let $\ll_{1}$ be the line joining the two cusps of $\eta(E_{0})$. Then $6=\ll_{1}\cdot \eta(E_{0})\geq (\ll_{1}\cdot \eta(E_{0}))_{p_{1}}+(\ll_{1}\cdot \eta(E_{0}))_{p_{2}}\geq 2+4=6$, so $\ll_{1}$ meets $\eta(E_{0})$ only at $p_{1}$, $p_{2}$ with multiplicities $2$ and $4$. It follows that the proper transforms $L_{1}$, $L_{2}$ of $\ll_{1}$, $\ll_{2}$ are as claimed.
	
	As the first step of the construction of an almost minimal model we take $A_{1}\de L_{1}$, then $\Exc \psi_{1}=W_{1}+L_{1}+V_{2}$ and the images of $L$, $L_{2}$ are both almost log exceptional on $(X_{1},\tfrac{1}{2}D_{1})$. Now we have a choice: we can either take $A_{2}\de \psi_{1}(L)$ or $A_{2}\de \psi_{1}(L_{2})$. To emphasize that, we denote $\psi_{2}$ by $\tilde{\psi}_{2}$ in the second case.
	
	Consider the first choice. Then $\Exc \psi_{2}=\psi_{1}(L)$ and the image of $L_{2}$ is almost log exceptional on $(X_{2},\tfrac{1}{2}D_{2})$. We then take $A_{3}\de (\psi_{2}\circ \psi_{1})(L_{2})$, so $\Exc \psi_{3}=(\psi_{2}\circ \psi_{1})_{*}(L_{2}+W_{2})$. Now $\Delta_{3}^{-}=0$, so $(X_{3},\tfrac{1}{2}D_{3})$ has no almost log exceptional curves,  and we conclude that $n=3$. The peeling morphism \eqref{eq:peeling} contracts the image of $U_{1}+U_{2}$. Thus $X_{\min}\cong \P^{2}$ and $D_{\min}$ is a conic inscribed in a triangle.
	
	Consider the second choice, that is, $A_{2}\de \psi_{1}(L_{2})$. Then  $\Exc \tilde{\psi}_{2}=(\psi_{1})_{*}(L_{2}+W_{2})$. But $(\tilde \psi_{2}\circ\psi_{1})_{*}(V_{1}+U_{2}+U_{1})=\Upsilon_{2}+\Delta_{2}^{+}$, so the image of $L$ meets $\Delta_{2}^{+}$ and thus is not almost log exceptional on $(X_{2},\tfrac{1}{2}D_{2})$. In fact, $\Delta_{2}^{-}=0$, so we conclude that $n=2$. The peeling morphism \eqref{eq:peeling} contracts  the image of $V_{1}+U_{2}+U_{1}$. Again $X_{\min}\cong \P^{2}$, but now $D_{\min}$ consists of two conics meeting with multiplicities $3$ and $1$ and a line tangent to both of them.
	
	Note that although $\psi_{1}(L)$ is almost log exceptional, its push-forward via $\tilde{\psi}_{2}$ is not, even though $\tilde{\psi}_{2}$ does not touch $\psi_{1}(L)$ or the components of $D_{1}$ meeting it. Moreover, the curve $V_{1}$ can become a component of $\Upsilon_{n}$ or not, depending on the choice of $\psi$. Note also that the two almost minimal models constructed above have different Picard ranks: $3$ in the first case and $4$ in the second case.

 \bigskip
\section{Possible types of cusps}\label{sec:possible_HN-types}
In this section we prove the following proposition, which is the main part of Theorem \ref{thm:main}.

\begin{prop}\label{prop:possible_cusp_types} Let $\bar{E}\subseteq \P^{2}$ be as in \eqref{eq:assumption}. Then $\bar{E}$ has one of the singularity types listed in Definition \ref{def:our_curves}.
\end{prop}

We use Notation \ref{not:graphs} for the minimal weak resolution $\pi_{0}\colon (X_{0},D_{0})\to (\P^{2},\bar{E})$ and notation from Section \ref{sec:MMP} for a fixed process of almost minimalization
\begin{equation*}
\psi=\psi_{n}\circ\dots\circ\psi_{1} \colon\, (X_{0},\tfrac{1}{2}D_{0}) \to (X_{n},\tfrac{1}{2}D_{n}).
\end{equation*}
In particular, we have a decomposition
\begin{equation*}
\psi_{*}D_{0}=D_{n}=R_{n}+\Upsilon_{n}+\Delta_{n}.
\end{equation*}
The divisors $\Upsilon_{n}$ and $\Delta_{n}$, which are contained in the fixed loci of all systems $|m(2K_{n}+D_{n})|$, $m\geq1$,  consist of $(-1)$-curves and $(-2)$-curves respectively and their geometry is clear; see Notation \ref{not:MMP}. 
Their sum is the exceptional divisor of the peeling morphism \eqref{eq:peeling}
\begin{equation*}
\alpha_{n}\colon(X_{n}, \tfrac{1}{2}D_{n})\to(X_{\min},\tfrac{1}{2}D_{\min}).
\end{equation*}
The goal of Section \ref{sec:picture} is to investigate the geometry of the divisor $R_{n}$ and to impose restrictions on the shape of $D_n$. Recall that the number of components of $D_{\min}=(\alpha_{n})_{*}D_{n}$ is bounded by \eqref{eq:lambdaineq}. In Sections \ref{sec:F2}--\ref{sec:P2} we perform a case-by-case study of possible pairs $(X_{\min},D_{\min})$ and we recover from them the pairs $(X_{0},D_{0})$.

Recall that by Lemma \ref{lem:MMP-properties}\ref{item:Exc-psi_i} all components of $D_{0}$ contracted by a process of almost minimalization are contained in the maximal twigs of $D_{0}$. Let us note that by \cite[Corollary 1.3]{Palka-minimal_models} $D_{0}$ has at most 20 components which are not contained in such twigs. A bound of this type cannot exist if $\kappa(\P^{2}\setminus \bar{E})\neq 2$, see for example,  \cite{Kashiwara,Tono_doctoral_thesis}.

 By Iitaka's Easy Addition Theorem, \eqref{eq:assumption} implies that no open subset of $\P^{2}\setminus \bar{E}$ admits a $\C^{1}$-, $\C^{*}$- or a $\C^{**}$-fibration. In particular, $X_n\setminus D_n$ does not admit such a fibration.

\medskip \subsection{Restrictions on the geometry of almost minimal models.}\label{sec:picture}

Recall that $Q_{j}\subseteq D_{0}$ denotes the exceptional divisor of $\pi_{0}$ over the cusp $q_{j}\in\bar{E}$, $j\in\{1,\dots, c\}$. It contains a unique $(-1)$-curve $C_{j}$, which is the exceptional curve of the last blowup  over $q_{j}$. For more details and notation see Section \ref{sec:resolutions} and Figure \ref{fig:treeQ}. In particular, we use Notation \ref{not:graphs}. 
Additionally, we decompose $\alpha_{n}$ as $\alpha_{n}^{-}\circ\alpha_{n}^{+}$, where $$\alpha_{n}^{+}\colon (X_{n},\tfrac{1}{2}D_{n})\to (Z,\tfrac{1}{2}D_{Z})$$ contracts $\Upsilon_{n}+\Delta_{n}^{+}$ and $$\alpha_{n}^{-}\colon (Z,\tfrac{1}{2}D_{Z})\to (X_{\min},\tfrac{1}{2}D_{\min})$$ contracts the image of $\Delta_{n}^{-}$. We put $\psi^{+}=\alpha_{n}^{+}\circ\psi$.	
We have a factorization $\alpha_{n}\circ \psi=\alpha_{n}^{-}\circ \psi^{+}$, where $\alpha_{n}^{-}\colon Z\to X_{\min}$ is a minimal resolution of singularities:
\begin{equation}\label{eq:theta}
\xymatrix{
	(X_{0},\tfrac{1}{2}D_{0}) \ar[r]^-{\psi} \ar@{->}@/_2pc/[rr]_-{\psi^{+}} & (X_{n},\tfrac{1}{2}D_{n}) \ar[r]^-{\alpha_{n}^{+}} \ar@{->}@/^2pc/[rr]^-{\alpha_{n}} & (Z,\tfrac{1}{2}D_{Z}) \ar[r]^-{\alpha_{n}^{-}} & (X_{\min},\tfrac{1}{2}D_{\min})
	}
\end{equation}

\begin{lem}
	[Properties of $R_{n}$]
	\label{lem:beta_flat}\  
	\begin{enumerate}
		\item\label{item:Ups} Every component of $\Upsilon_{n}-\psi_{*}\Upsilon_{0}$ meets $R_{n}$ in two points, exactly one of which is a center of $\psi$.
		\item\label{item:Rn_connected} $\psi^{-1}_{*}R_{n}$ is a connected subdivisor of $R_{0}$.
		\item\label{item:Rn=n+1} $\#R_{n}=\#D_{\min}=n+1$.
		\item\label{item:Dmin_meet_each_other} Any two distinct components of $D_{\min}$ meet.
		\item \label{item:centers_on_En} For each component $G$ of $D_0-E_0$ we have $\psi^+(G)\cdot \psi^{+}(E_{0})\leq G\cdot E_0+1$. If the equality holds then the unique connected component of $\Exc \psi^+$ meeting $G$ and $E_{0}$ is a chain equal to $A+\Delta_A$, where $\Delta_A$ is a maximal $(-2)$-twig of $D_0$ meeting $G$ and $A$ is an almost log exceptional curve meeting $E_0$, see Figure \ref{fig:Delta_A}.
			\begin{figure}[ht]
				\begin{subfigure}{0.4\textwidth}\centering
					\includegraphics[scale=0.22]{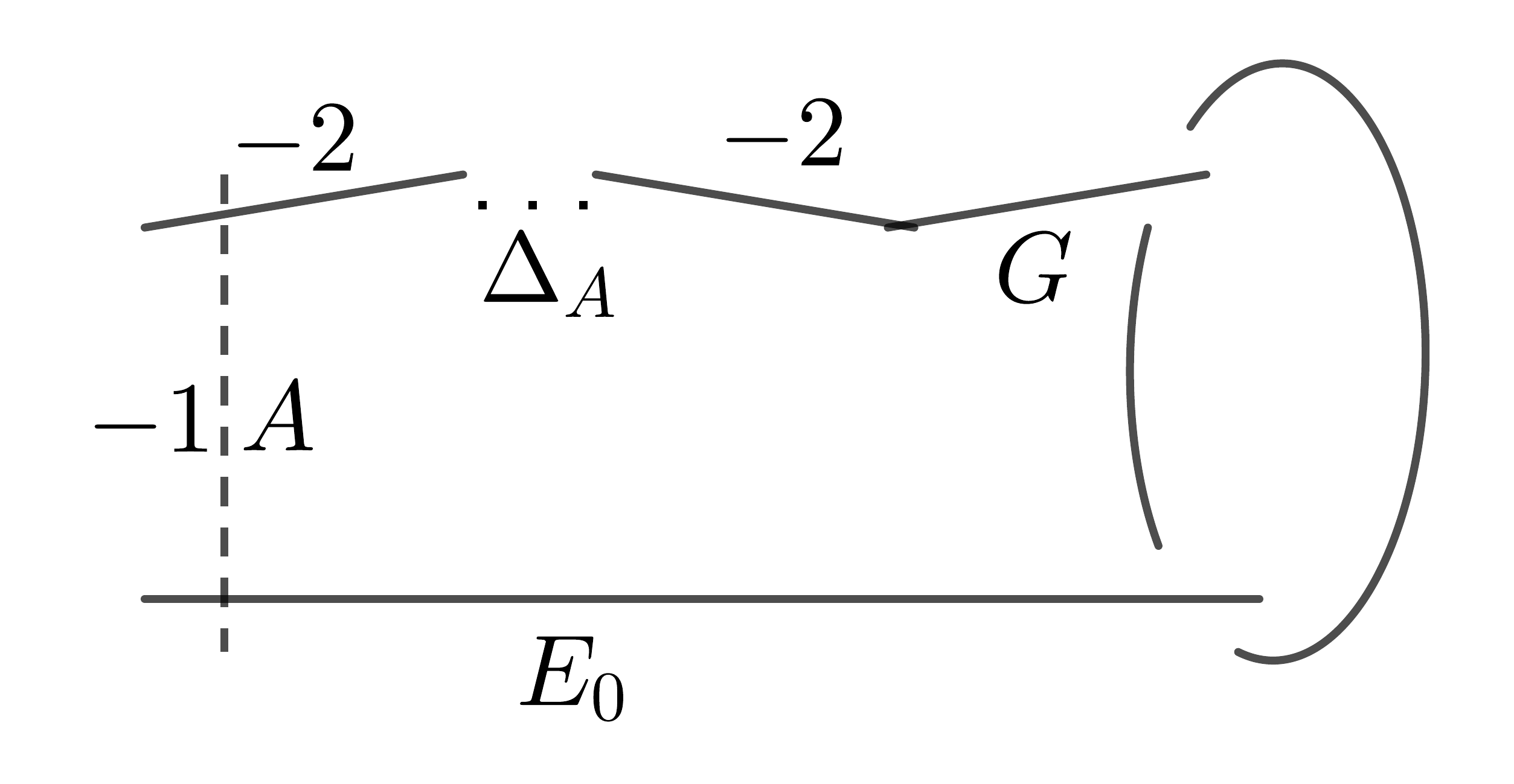}
					\caption{Case $A+\Delta_{A}\subseteq \Exc\psi$.}
				\end{subfigure}
					~\hfill
				\begin{subfigure}{0.4\textwidth}\centering
					\includegraphics[scale=0.22]{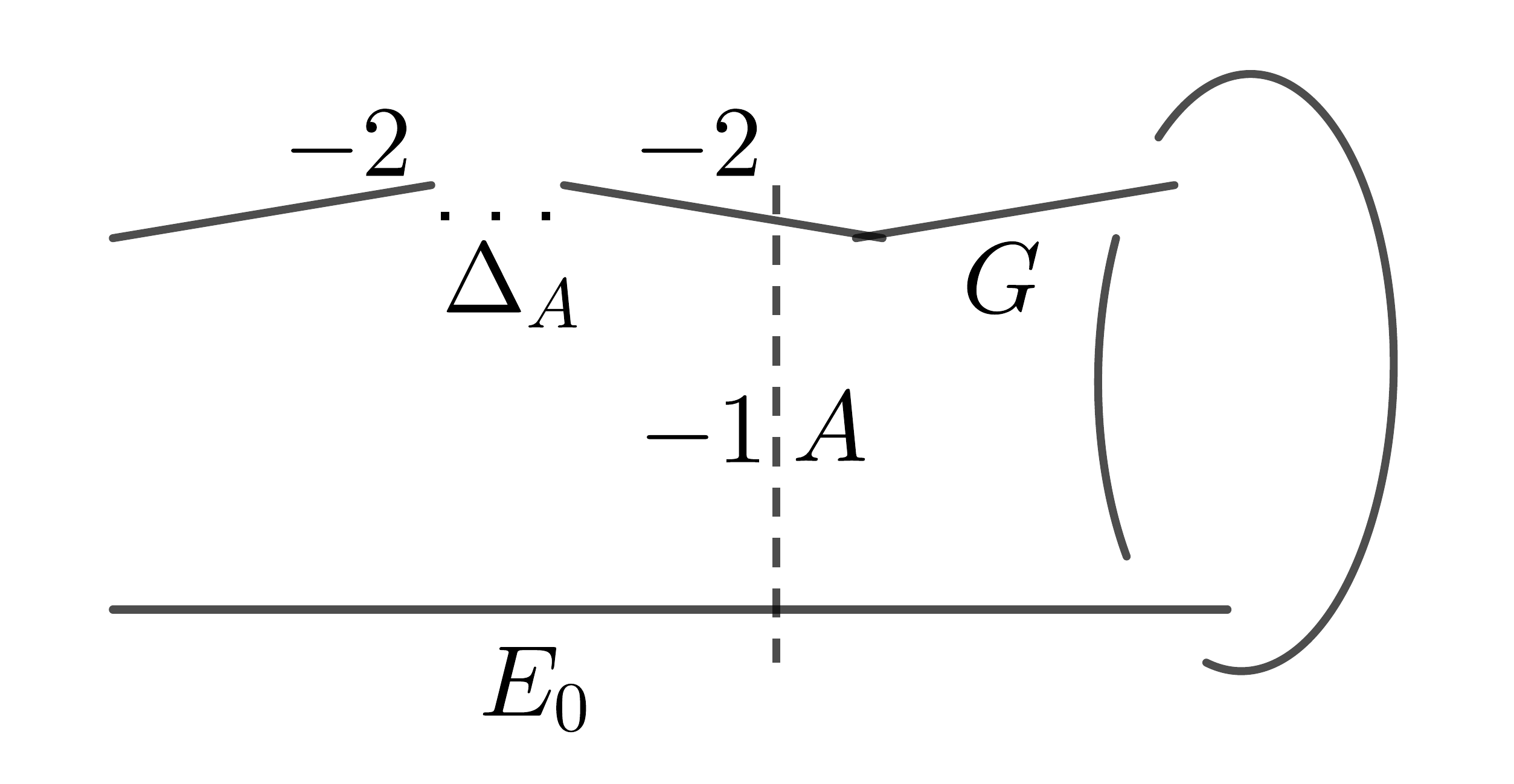}
					\caption{Case $A+\Delta_{A}\not\subseteq \Exc\psi$.}
				\end{subfigure}
				\caption{The preimage on $X_{0}$ of a center of $\psi^{+}$ on $E_{n}$ (Lemma \ref{lem:beta_flat}\ref{item:centers_on_En}).}
				\label{fig:Delta_A}
			\end{figure}
		\item\label{item:meeting_E0} Every component of $\psi^{-1}_{*}R_{n}-E_{0}$ meets $E_{0}+\Delta_{0}-\psi^{-1}_{*}\Delta_{n}$.
		\item\label{item:centers_psi+} The centers of $\psi^{+}$ are the semi-ordinary cusps of $\psi^{+}(E_{0})$ and the images of $A_{1}, \dots, A_{n}$. At each of the latter points exactly two analytic branches of $D_{\min}$ meet.
		\item\label{item:n0} $n=0$ if and only if all cusps of $\bar{E}$ are semi-ordinary.
	\end{enumerate}
\end{lem}
\begin{proof}
	\ref{item:Ups} Let $U$ be a component of $\psi^{-1}_{*}\Upsilon_{n}-\Upsilon_{0}$. Because $U\not\subseteq \Upsilon_{0}$, the morphism $\psi$ touches $U$, so $\psi(U)$ contains a center of $\psi$. By Lemma \ref{lem:MMP-properties}\ref{item:Ui_disjoint},\ref{item:Delta_pr-tr}, this center is not contained in $\Delta_{n}+(\Upsilon_{n}-\psi(U))$, so by Lemma \ref{lem:MMP-properties}\ref{item:center-psi_i} it is a point of normal crossings of $\psi(U)$ and $R_{n}$. We have $\psi(U)\cdot R_{n}=2$ by the definition of $\Upsilon_{n}$ and by Lemma \ref{lem:MMP-properties}\ref{item:Ui_disjoint}, so $\psi(U)$ meets $R_{n}$ in exactly two points. Moreover, one of these points is not a center of $\psi$, for otherwise by Lemma \ref{lem:MMP-properties}\ref{item:center-psi_i} $D_{0}$ would not be connected.
	
	\ref{item:Rn_connected} Lemma \ref{lem:MMP-properties}\ref{item:center-psi_i} implies that $\psi^{-1}_{*}D_{n}$ is connected. By Lemma \ref{lem:MMP-properties}\ref{item:Delta_pr-tr}  $\psi^{-1}_{*}\Delta_{n}$ is contained in the sum of twigs of $\psi^{-1}_{*}D_{n}$, hence $ \psi^{-1}_{*}(D_{n}-\Delta_{n})=\psi^{-1}_{*}(R_{n}+\Upsilon_{n})$ is connected. Let $U\subseteq X_{0}$ be a proper transform of a component of $\Upsilon_{n}$. If $U\subseteq \Upsilon_{0}$ then $U$ is the unique $(-1)$-curve over a semi-ordinary cusp (see Remark \ref{rem:semi-ordinary}), and if $U\not\subseteq \Upsilon_{0}$ then $U\cdot \psi^{-1}_{*}R_{n}=1$ by \ref{item:Ups}. In any case, $U$ meets $\psi^{-1}_{*}(D_{n}-\Delta_{n})$ in one point, so $\psi^{-1}_{*}(D_{n}-\Delta_{n})-U$ is connected. Thus $\psi^{-1}_{*}R_{n}=\psi^{-1}_{*}(D_{n}-\Delta_{n}-\Upsilon_{n})$ is connected and $\psi^{-1}_{*}R_{n}\subseteq R_{0}$ by Lemma \ref{lem:MMP-properties}\ref{item:R_pr-tr}.

\ref{item:Rn=n+1} The minimal model $(X_{\min}, \frac{1}{2}D_{\min})$ of $(X_{0},\frac{1}{2}D_{0})$ is a log del Pezzo surface of Picard rank one.  Hence, $$\#D_{\min}-1=\#D_{\min}-\rho(X_{\min})=\#D_{0}-\rho(X_{0})+n=\#\bar E-\rho(\P^2)+n=n.$$ Recall that $R_{n}$ is the proper transform of $D_{\min}$ on $X_{n}$ (see Proposition \ref{prop:MMP}). Therefore, $\#R_{n}=\#D_{\min}=n+1$, as claimed.
	
\ref{item:Dmin_meet_each_other} Because $\rho(X_{\min})=1$, every curve on $X_{\min}$ is numerically equivalent to a positive multiple of an ample divisor on $X_{\min}$, hence every two distinct curves on $X_{\min}$ have a positive intersection number.

\ref{item:centers_on_En} If $G\cdot E_0<\psi(G)\cdot E_n$ then we are done by \cite[Lemma 4.1(vii)]{Palka-minimal_models} and Lemma \ref{lem:ale_pr-tr}\ref{item:ale_pr-tr}, so assume $G\cdot E_0=\psi(G)\cdot E_n<\psi^+(G)\cdot \alpha_n(E_n)$. Then $\psi(G)$ meets a component  $U\subseteq \Upsilon_n-\psi_{*}\Upsilon_0$ for which $U\cdot E_n>0$. The former implies that $\psi^{-1}_{*}U\cdot E_0=0$, so we are done by applying \cite[Lemma 4.1(vii)]{Palka-minimal_models} to $U$.

\ref{item:meeting_E0} This follows immediately from \ref{item:Dmin_meet_each_other} and \ref{item:centers_on_En}.

\ref{item:centers_psi+} This follows from Lemma \ref{lem:MMP-properties}\ref{item:Exc-psi_i} and from the definition of $\Upsilon_{n}$ (see Figure \ref{fig:Upsilon} and Remark \ref{rem:semi-ordinary}).

\ref{item:n0} If $\bar{E}$ has only semi-ordinary cusps then $\Delta_{0}^{-}=0$, so there is no almost log exceptional curve on $(X_{0},\frac{1}{2}D_{0})$ (see \eqref{eq:ale}), hence $n=0$. Conversely, if $n=0$ then \ref{item:Rn=n+1} gives $\#R_{0}=1$, so $R_{0}=E_{0}$. Then $C_{1},\dots, C_{c}\subseteq \Upsilon_{0}$, so $q_{1},\dots, q_{c}$ are semi-ordinary. 
\end{proof}

The following lemma is a step towards Lemma \ref{lem:line}. For the definition of $T_{j}$, $T_{j}^{0}$ and $T_{j}'$ see Notation \ref{not:graphs}\ref{item:not_CT},\ref{item:not_T^0} and \ref{item:not_T'}, respectively.

\setcounter{claim}{0}
\begin{lem}[A restriction on loops in $A_i+Q_j$]\label{lem:orevkov_ending} If a proper transform $A$ of some $A_{i}$, $i\in\{1,\dots, n\}$ meets $\ftip{T_{j}}$ for some $j\in\{1,\dots,n\}$ then $A\cdot (T_{j}^{0}+\ftip{T_{j}'})=0$.
\end{lem}

\begin{proof}
	We may assume $j=1$. By Lemma \ref{lem:ale_pr-tr}\ref{item:ale_pr-tr} $A$ is an almost log exceptional curve on $(X_{0},\tfrac{1}{2}D_{0})$. Since $A$ meets $\ftip{T_{1}}$, we have $T_1\neq 0$, $A\cdot\ftip{T_{1}}=1$ and $A\cdot (D_{0}-\ftip{T_{1}})=1$ by \eqref{eq:ale}. Suppose that  $A$ meets $T_{1}^{0}+\ftip{T_{1}'}$. Let $P$ be the circular subdivisor of $D_{0}+A$ and let $V$ be the component of $P$ meeting $D_{0}+A-P$, see Figure \ref{fig:bubble_ind}. Because $D_{n}-R_{n}$ is a sum of disjoint chains, $\psi_{*}P$ contains a component of $R_{n}$. Hence, $V\subseteq \psi^{-1}_{*}R_{n}$ by Lemma \ref{lem:beta_flat}\ref{item:Rn_connected}. The only possible $(-2)$-twig of $D_{0}$ meeting $V$, namely $\Delta_{T_{1}}$ (see Notation \ref{not:graphs}\ref{item:not_Delta_T}), is contained in $P$, so by Lemma \ref{lem:beta_flat}\ref{item:Dmin_meet_each_other},\ref{item:centers_on_En}, $V\cdot E_{0}=\psi^{+}(V)\cdot \psi^{+}(E_{0})>0$. Hence, $V$ equals $C_{1}$ or $\tilde{C}_{1}$.
	\begin{figure}[ht]
		\begin{subfigure}{0.4\textwidth}\centering
			\includegraphics[scale=0.25]{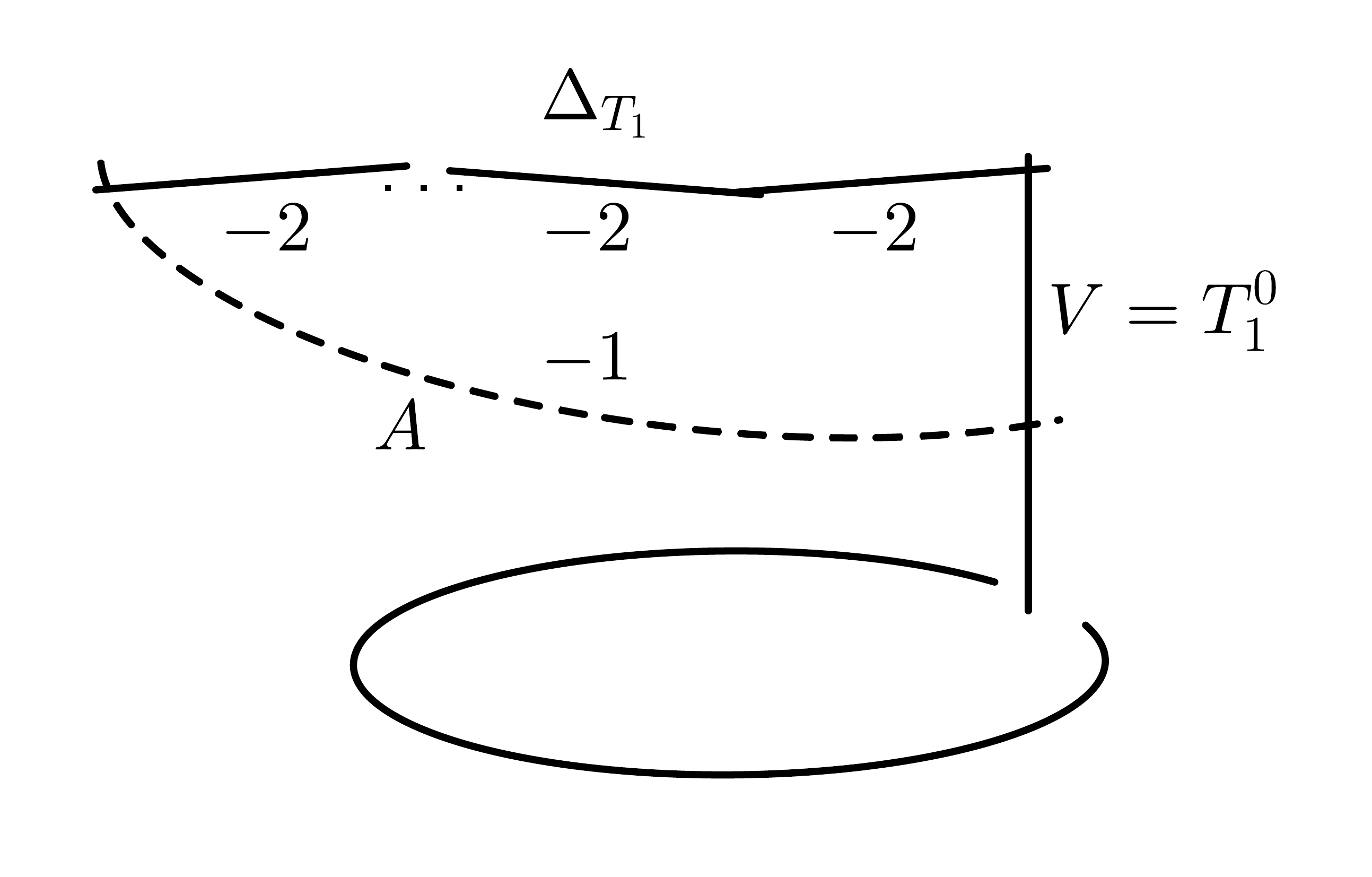}
			\caption{Case $A\cdot T_{1}^{0}=1$.}
		\end{subfigure}
			~\hfill
		\begin{subfigure}{0.4\textwidth}\centering
			\includegraphics[scale=0.25]{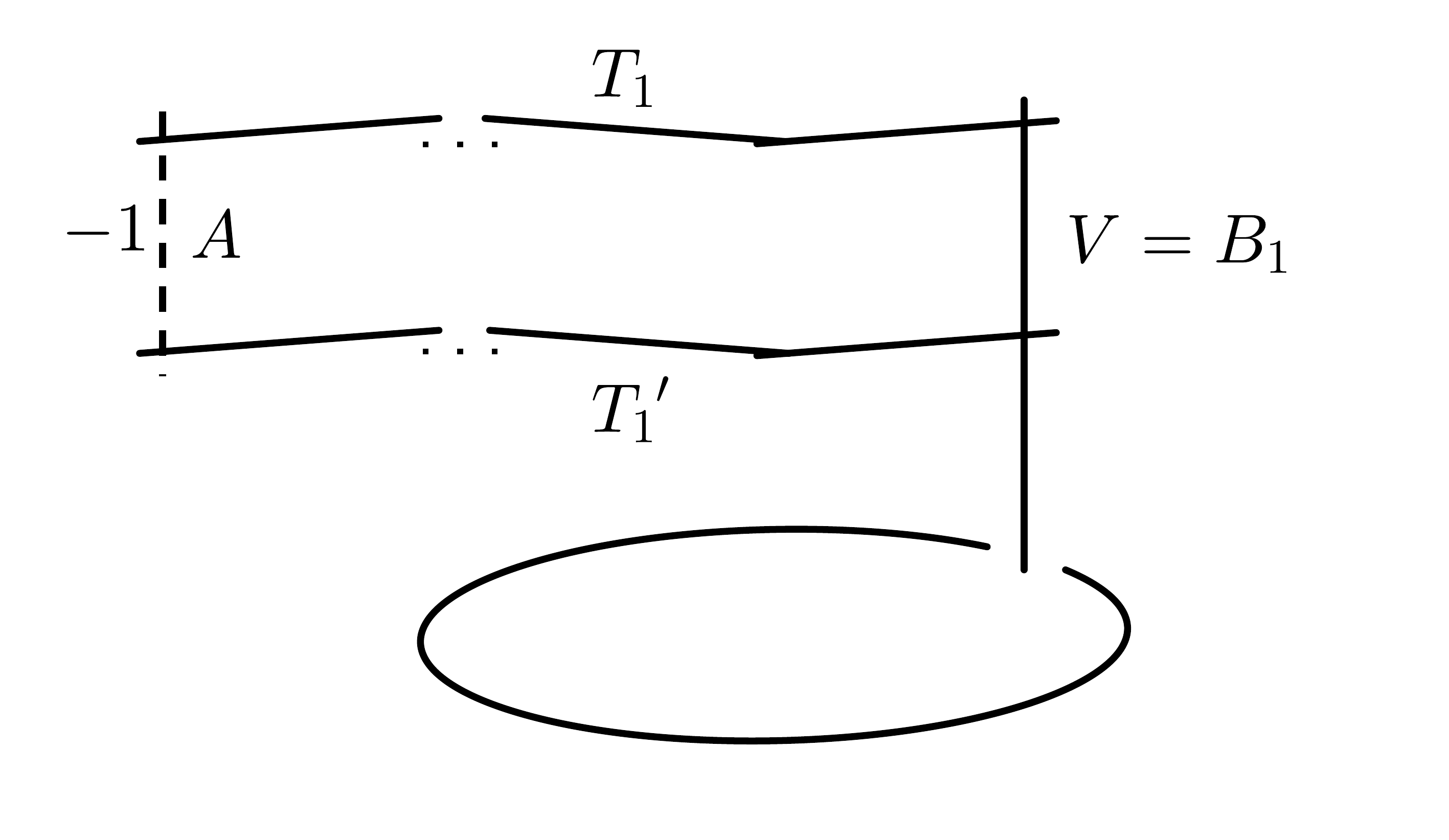}
			\caption{Case $A\cdot \ftip{T_{1}'}=1$.}
		\end{subfigure}
		\caption{The curve $A$ from the proof of Lemma \ref{lem:orevkov_ending}.}
		\label{fig:bubble_ind}
	\end{figure}
	
	Consider the case $V=C_{1}$. Since $C_{1}$ is non-branching in $Q_{1}$, Lemma \ref{lem:notation}\ref{item:B},\ref{item:T^0=C} gives $Q_{1}=P-A$. Let $\sigma$ be the contraction of $P-C_{1}+(Q_{2}+\dots+Q_{c})$. The image of $X_{0}$ is a smooth surface of Picard rank one, so it is isomorphic to $\P^{2}$. The image of $C_{1}$ has one singular point,  which is an ordinary double point, so $\sigma(C_{1})$ is a nodal cubic, and $\sigma(C_{1})\cdot\sigma(E_{0})=C_{1}\cdot E_{0}=\tau_{1}$. Hence, $3\deg\sigma(E_{0})=\tau_{1}$. We have $\tau_{1}\leq \lambda_{1}\leq 6$ by \eqref{eq:lambdaineq}, so $\deg\sigma(E_{0})\leq 2$. In particular, $\sigma(E_{0})$ is smooth, so $c=1$. Moreover, $\sigma$ does not touch $E_{0}$, so 
	\begin{equation*}
	E^{2}=E_{0}^{2}-\tau_{1}=\sigma(E_{0})^{2}-\tau_{1}=(\deg \sigma(E_{0}))^{2}-3\deg\sigma(E_{0})= -2,
	\end{equation*}
	a contradiction with Lemma \ref{lem:Tono_E2}\ref{item:c=1}.
	
	Consider the case $V=\tilde{C}_{1}$. By the definition of $T_{1}^{0}$ the components of $P-A$ are the last components of $Q_{1}$ contracted by $\pi_{0}$, so $Q_{1}=(P-A)+C_{1}+\Delta_{C_{1}}$, where $\Delta_{C_{1}}$ is zero or a $(-2)$-twig of $D_{0}$ meeting $C_{1}$, say $\Delta_{C_{1}}=[(2)_{r-1}]$ for some $r\geq 1$. Let $\sigma$ be the contraction of $(P-\tilde{C}_{1})+(C_{1}+\Delta_{C_{1}})+(Q_{2}+\dots+Q_{c})$. As in the previous case, we see that $\sigma(X_{0})\cong \P^{2}$ and $\sigma(\tilde{C}_{1})$ is a nodal cubic. The curve $\sigma(E_{0})$ has a cusp with multiplicity sequence $((\tau_{1})_{r})$ at $\sigma(C_{1})$, so $3\deg\sigma(E_{0})=\sigma(\tilde{C}_{1})\cdot\sigma(C_{1})=\tau_{1}r+1$. Note that $\psi_{*}Q_{1}$ contains $\psi(C_{1})+\psi(\tilde{C}_{1})$, so 
	\begin{equation*}
	\#\psi_{*}Q_{1}-b_{0}(\Delta_{n}\wedge \psi_{*}Q_{1})-\#\Upsilon^{0}_{n}\wedge \psi_{*}Q_{1}\geq 2.
	\end{equation*}
	We have $s_{1}=0$ (see Notation \ref{not:graphs}\ref{item:not_s}), so $\lambda_{1}\geq \tau_{1}+2$. Thus $\tau_{1}\leq \lambda_{1}-2\leq 4$ by \eqref{eq:lambdaineq}. It follows that $3\deg\sigma(E_{0})\leq 4r+1$. Because $\sigma(E_{0})$ is singular, we have $\deg \sigma(E_{0})\geq 3$, so $r\geq 2$. The intersection multiplicity of a cusp with its tangent line equals the sum of some initial terms of its multiplicity sequence, so in our case $2\tau_{1}\leq \deg\sigma(E_{0})$. Thus $6\tau_{1}\leq 3\deg\sigma(E_{0})=\tau_{1}r+1$, which gives $r\geq 6$. It follows that $\psi$ contracts some components of $\Delta_{C_{1}}$: indeed, otherwise 
	\begin{equation*}
	\#\psi_{*}Q_{1}-b_{0}(\Delta_{n}\wedge \psi_{*}Q_{1})-\#\Upsilon^{0}_{n}\wedge \psi_{*}Q_{1}\geq \#(C_{1}+\tilde{C}_{1}+\Delta_{C_{1}})-1\geq 6,
	\end{equation*}
	so $\lambda_{1}\geq \tau_{1}+6\geq 8$, contrary to \eqref{eq:lambdaineq}. Thus there exists another almost log exceptional curve $A_{i'}$, $i'\neq i$, whose proper transform $A'\subseteq X_{0}$ meets $\ftip{\Delta_{C_{1}}}$. Lemma \ref{lem:ale_pr-tr} implies that $A'$ meets $D_{0}$ in $\ftip{\Delta_{C_{1}}}$ and in 
	\begin{equation*}
	D_{0}-\Delta_{C_{1}}-(\Exc\psi_{A}-A)=(C_{1}+\tilde{C}_{1})+(E_{0}+Q_{2}+\dots+Q_{c}).
	\end{equation*}
	It follows that $A'$ meets $C_{1}+\tilde{C}_{1}$: indeed, otherwise $1=\sigma(\tilde{C}_{1})\cdot\sigma(A')=3\deg\sigma(A')$, which is impossible. Let $\sigma'$ be the contraction of $P-\tilde{C}_{1}+(A'+\Delta_{C_{1}})+(Q_{2}+\dots+Q_{c})$. Again, $\sigma'(X_{0})\cong \P^{2}$ and $\sigma'(\tilde{C}_{1})$ is a nodal cubic. The curve $A'$ meets $C_{1}$, because otherwise $A'$ meets $\tilde{C}_{1}$ and $\sigma'(C_{1})^{2}=C_{1}^{2}+1=0$, which is impossible. Hence, $\sigma'(C_{1})$ is also a nodal cubic, so $9=\sigma'(C_{1})\cdot \sigma'(\tilde{C}_{1})=C_{1}\cdot\tilde{C}_{1}=1$; a contradiction.
\end{proof}

\begin{rem}[Orevkov curves]\label{rem:Orevkov}
The proof of Lemma \ref{lem:orevkov_ending} uses the assumption \eqref{eq:assumption} that $\P^{2}\setminus \bar{E}$ has no $\C^{**}$-fibration, which is necessary and was used to restrict the shape of $T_{1}+T_{1}'$. To see the necessity of this assumption let $\bar{E}\subseteq \P^{2}$ be the Orevkov curve $\ORa(k)$ or $\ORb(k)$ (originally, in \cite{OrevkovCurves} denoted by $C_{4k}$, $C_{4k}^{*}$). The minimal weak resolution of such curve is shown in Figure \ref{fig:Ore}, cf.\ Figures 20--21 in \cite{PaPe_Cstst-fibrations_singularities}. The existence of an almost log exceptional curve $A$ meeting $\ftip{T_{1}}$ and $T_{1}^{0}+\ftip{T_{1}'}$, see Figure \ref{fig:Ore}, is shown, for example, in \cite[Lemma 15]{Tono-on_Orevkov_curves} (see \cite[Example  3.2]{PaPe_Cstst-fibrations_singularities} for the case $k=0$). Here, the twigs $T_{1}$, $T_{1}'$ are complicated (in particular can be arbitrarily long) and $\psi=\psi_{A}$ contracts them.
\begin{figure}[ht]
	\includegraphics[scale=0.25]{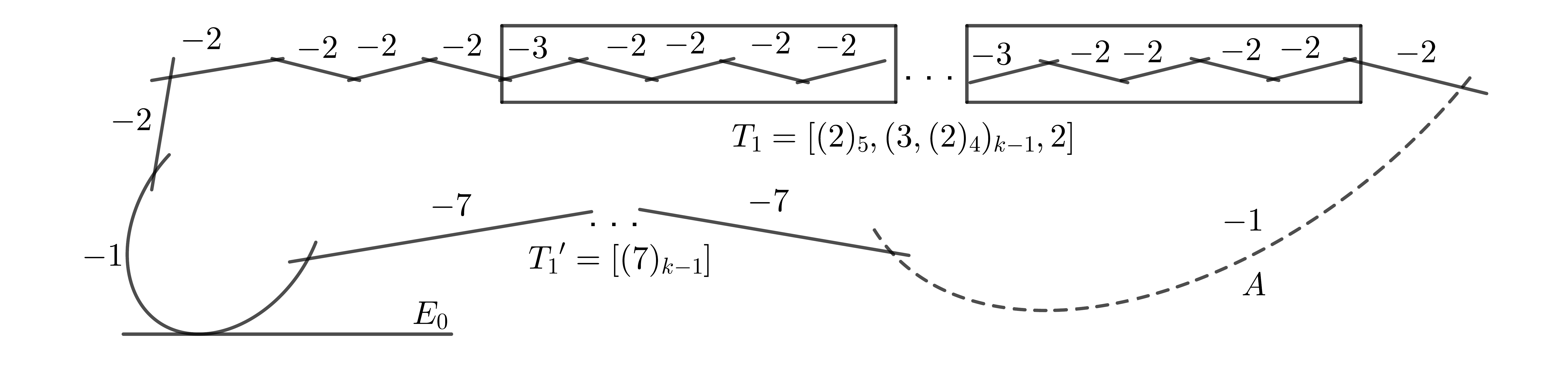}
	\caption{An almost log exceptional curve $A$ excluded by Lemma \ref{lem:orevkov_ending}.}
	\label{fig:Ore}
\end{figure}
\end{rem}

The following lemma will be used to show the existence of a line $\ll\subseteq \P^{2}$ as in Theorem \ref{thm:geometric}.
Note that by Lemma \ref{lem:notation}\ref{item:T^0=C} $T_{j}=\Delta_{T_j}$ (see Notation \ref{not:graphs}\ref{item:not_Delta_T}) if and only if $Q_{j}=[2,\dots,2,1]$ and $\tilde{C}_j=0$, equivalently if and only if $T_j^0=C_{j}$. Note also that if $T_j$ is zero or a twig of $D_0$ then $\tilde{C}_{j}\nsubseteq T_j$, so then the condition $\tilde{C}_j=0$ in the latter equivalence can be omitted.
\begin{lem}[Some almost log exceptional curves give special lines]\label{lem:line}	
	Let $(\mu_{j},\mu_{j}',\dots)$  be the multiplicity sequence of $q_{j}\in \bar{E}$, $j\in\{1,\ldots, c\}$. Assume that for some $j$
\begin{equation}\label{eq:line}
\psi \mbox{ contracts } T_{j} \quad  \mbox{and} \quad T_{j}\neq \Delta_{T_{j}}\neq 0.
\end{equation}
Then we can renumber the cusps of $\bar{E}$ so that the following hold:
\begin{enumerate}
		\item\label{item:A_meets_T1T2} $j=1$ and the proper transform $A\subseteq X_{0}$ of some $A_{i}$ meets $\ftip{T_{1}}$ and $\ftip{T_{2}}$.
		\item\label{item:t1>0} $t_{2}=0$  and $\ftip{T_{2}}\subseteq \Exc\psi_{A}$.
		\item\label{item:Exc-psi_1} $\Exc\psi_{A}=T_{1}+A+T_{2}\wedge (D_{0}-\tilde{C}_{2})$.
		\item\label{item:A_line} $\pi_{0}(A)$ is a line meeting $\bar{E}$ only at $q_{1}$ and $q_{2}$ with multiplicities $\mu_{1}$, $\mu_{2}$, respectively. 
		\item\label{item:deg} $\deg\bar{E}=\mu_{1}+\mu_{2}$.
		\item\label{item:not_semiord} The cusps $q_{1},q_{2}\in \bar{E}$ are not semi-ordinary.
		\item\label{item:tangent} Let $A'\subseteq X_{0}$ be the proper transform of $A_{i'}$ for some $i'\neq i$. Then $A'$ meets $Q_{2}$. If $A'\cdot Q_{1}=0$ then $A'\cdot E_{0}=A'\cdot\ftip{T_{2}'}=1$ and 
		$\deg\bar{E}=\mu_{2}+\mu_{2}'+1$. If additionally $s_{1}=s_{2}=1$ then $\{\tau_{1},\tau_{2}\}=\{2,3\}$.
		\item\label{item:j>3} $\psi$ does not touch $Q_{3}+\dots + Q_{c}$.
\end{enumerate}
Moreover, denoting by $\epsilon$ the number of cusps for which \eqref{eq:line} holds, we have
\begin{equation}\label{eq:eps}
b_{0}(\Delta_{n}^{-})\leq \epsilon\leq 1.
\end{equation}
\end{lem}

\begin{proof}
Without loss of generality we may assume $j=1$.
	
 \ref{item:A_meets_T1T2} Because $T_{1}\subseteq\Exc \psi$, by Lemma \ref{lem:MMP-properties}\ref{item:Exc-psi_i} $T_{1}$ is a twig of $D_{0}$ and the proper transform of some $A_{i}$, say $A$, meets $\ftip{T_{1}}$. Then $T_{1}\subseteq \Exc\psi_{A}$ by Lemma \ref{lem:ale_pr-tr}\ref{item:Exc-psi_i_pr-tr}. By assumption, $T_{1}$ contains a twig of $D_{0}$ of type $[(2)_{t_{1}},a]$ for some $t_{1}\geq 1$, $a\geq 3$, so $A$ meets a tip $W\not\subseteq T_{1}$ of $D_{0}$ contracted by $\psi_{A}$. Hence, $W^{2}=-(t_{1}+2)\leq -3$. We may assume that $W\subseteq Q_{k}$ for some $k\in \{1,2\}$. Contract $C_{k}$ and subsequent $(-1)$-curves in the images of $Q_{k}$ as long as $W$ is not touched and denote this morphism by $\sigma$, see Figure \ref{fig:35_phi}. Because $Q_{k}$ contracts to a smooth point, the divisor $\sigma_{*}Q_{k}$ has a twig $T_{W}=[t_{1}+2,1,(2)_{t_{1}}]$ such that $\sigma(W)=\ftip{T_{W}}$. Clearly, $W\neq \ftip{T_{1}}$.
 
  \begin{figure}[ht]
  	\includegraphics[scale=0.26]{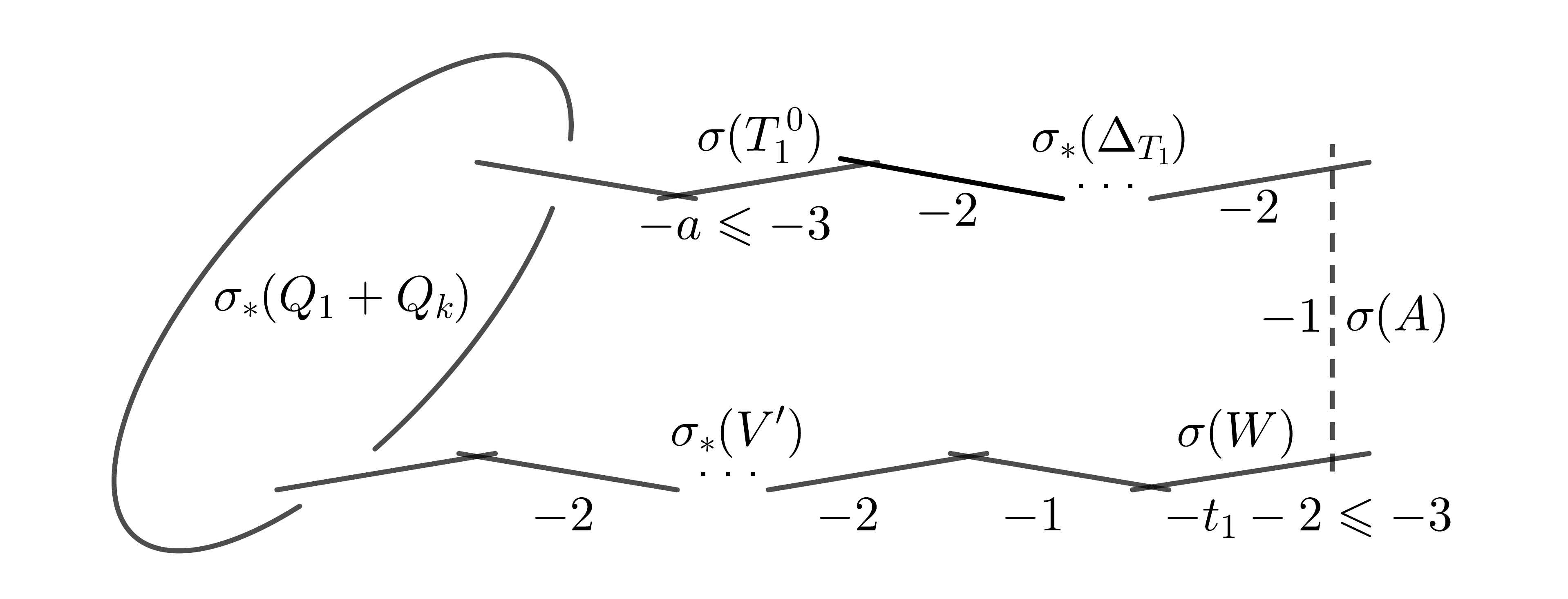}
  	\caption{The divisor $\sigma_{*}(Q_{k}+Q_{1})$ from the proof of Lemma \ref{lem:line}\ref{item:A_meets_T1T2}.}
  	\label{fig:35_phi}
  \end{figure}
  
 Suppose that  $W\neq \ftip{T_{2}}$. Then $T_{W}$ is a proper subdivisor of $\sigma_{*}Q_{k}$. It follows that the proper transform $V'\subseteq X_{0}$ of the subchain $[(2)_{t_{1}}]$ of $T_{W}$ is disjoint from $\Delta_{0}$ and contains no branching component of $Q_{k}$.

 We claim that $V'$ has a component $V$ disjoint from $E_{0}$. Suppose otherwise. Then $V'=\tilde{C}_{k}$ and $W$ meets $C_{k}$, so $\sigma=\id$ and $t_{1}=1$. We have now  $\Exc\psi_{A}=W+A+T_{1}=[3,1,2,3,(2)_{b}]$ for some $b\geq 0$, hence $\psi_{A}$ touches $C_{k}$ exactly $b+2$ times. Let $\eta\colon X_{0}\map X'$ be a composition of $\psi_{A}$ and $b+1$ blowups on the proper transforms of $C_{k}$, of which the first $\min\{\tau_{k},b+1\}$ ones are centered on the proper transforms of $E_{0}$, too. 
 Let $D'$ be the reduced total transform of $D_{0}$. Then $|\eta_{*}C_{k}|$ induces a $\P^{1}$-fibration of $X'$ and, since $X'\setminus D'$ is isomorphic to an open subset of $X_{0}\setminus D_{0}$, \eqref{eq:assumption} implies that $$4\leq \eta_{*}C_{k}\cdot D'=2+\eta_{*}E_{0}\cdot \eta_{*}C_{k}=2+\max\{\tau_{k}-b-1,0\}.$$ Hence, $\tau_{k}\geq 3+b$. The divisor $\psi_{*}Q_{k}-(\Delta_{n}+\Upsilon^{0}_{n})\wedge \psi_{*}Q_{k}$ contains the images of $C_{k}$ and $\tilde{C}_{k}$, so $\lambda_{k}\geq \tau_{k}+2\geq 5+b$. By \eqref{eq:lambda>=tau}, $\lambda_{1}\geq 2$, so \eqref{eq:lambdaineq} implies that $k=1$. The divisor $Q_{1}$ contains a branching component $B_{1}$, because  otherwise $Q_{1}$ is a chain and $W=\ftip{T_{1}'}$, contrary to Lemma \ref{lem:orevkov_ending}. Moreover, $B_{1}\neq \tilde{C}_{1}$ because $\tilde{C}_{1}$ is contained in the twig $T_{W}$ of $Q_{1}$. Hence, $\psi_{*}Q_{1}-(\Delta_{n}+\Upsilon^{0}_{n})\wedge \psi_{*}Q_{1}$ contains the images of $C_{1}$, $\tilde{C}_{1}$ and $B_{1}$, so $\lambda_{1}\geq 3+\tau_{1}\geq 6+b$. Now by  \eqref{eq:lambdaineq} $c=1$, $b=0$, $\tau_{1}=3$ and 
 \begin{equation*}
 \psi_{*}Q_{1}-(\Delta_{n}+\Upsilon^{0}_{n})\wedge \psi_{*}Q_{1} =\psi(C_{1})+\psi(\tilde{C}_{1})+\psi(B_{1}) =R_{n}-E_{n}.
 \end{equation*} The latter implies that $B_{1}$ is the only branching component of $Q_{1}$ and meets $\tilde{C}_{1}\subseteq T_{W}$, so $Q_{1}=B_{1}+T_{1}+T_{1}'+T_{W}$ with $T_{W}=[3,1,2]$ and $T_{1}=[2,3]$. Then $T_{1}'=[2]$ and $B_{1}^{2}=-3$,  since $Q_{1}$ contracts to a smooth point. It follows that $q_{1}\in \bar{E}$ has multiplicity sequence $(22,22,11,11,7,4,3)$, hence $I(q_{1})-M(q_{1})=1287-83=1204$.  This is in contradiction with Lemma \ref{lem:HN-equations}\eqref{eq:genus-degree}.

Thus $V'$ has a component $V$ disjoint from $E_{0}$. Since $V'$ is disjoint from $\Delta_{0}$, Lemma \ref{lem:beta_flat}\ref{item:meeting_E0} gives $V\not\subseteq \psi^{-1}_{*}R_{n}$, so $V\subseteq \Exc\psi$ or $\psi(V)\subseteq \Upsilon_{n}^{0}$. In the second case, $\beta_{D_{n}}(\psi(V))\leq 2$ and $\psi(V)$ contains a center of $\psi$, so its proper transform meets a connected component of $\Exc\psi$. In any case, Lemma \ref{lem:MMP-properties}\ref{item:Exc-psi_i} implies that $V$ is contained in a twig of $D_{0}$ whose first tip meets a proper transform $A'$ of some almost log exceptional curve $A_{i'}$, $i'\neq i$. It follows that $\sigma(W)$ and $\sigma_{*}(Q_{k}-W)\supseteq \sigma(V)$ are twigs of $\sigma_{*}Q_{k}$, so $\sigma_{*}Q_{k}$ is a chain. We have $W\neq\ftip{T_{k}}$ by assumption, so $W=\ftip{T_{k}'}$, hence $V\subseteq T_{k}$ and so $\ftip{T_{k}}$ meets $A'$. Lemma \ref{lem:orevkov_ending} applied to $A$ implies that $k\neq 1$, so $k=2$. Then $\pi_{0}(A)$ is smooth and $\pi_{0}(A)^{2}=A^{2}+1+(t_{2}+2)>1$, so $\pi_{0}(A)$ is a conic. We have $\pi_{0}(A)\cap\pi_{0}(A')\subseteq \{q_{1},q_{2}\}$. Since $A'$ meets the first exceptional curve over $q_{2}$, we have $(\pi_{0}(A')\cdot\pi_{0}(A))_{q_{2}}=1$. Similarly, since $A$ meets the first exceptional curve over $q_{1}$, we have $(\pi_{0}(A')\cdot \pi_{0}(A))_{q_{1}}=\mu$, where $\mu\geq 0$ is the multiplicity of $\pi_{0}(A')$ at $q_{1}$. Therefore, $2\deg\pi_{0}(A')=\pi_{0}(A)\cdot \pi_{0}(A')=\mu+1$, hence $\mu\geq 1$, that is, $q_{1}\in \pi_{0}(A')$. Since $T_{1}\cdot A'=0$, we have $\pi_{0}(A')^{2}>A^2+t_1+2\geq 1$, so $\pi_{0}(A')$ is not a line. Denoting by $\ll_{1}$ the line tangent to $\pi_{0}(A')$ at $q_{1}$, we obtain $\deg\pi_{0}(A')\geq (\pi_{0}(A')\cdot \ll_{1})_{q_{1}}\geq \mu+1=2\deg\pi_{0}(A')$; a contradiction.

\ref{item:t1>0} We have $t_{2}=0$ by \ref{item:A_meets_T1T2}, because $A\cdot \Delta_{0}=1$ by Lemma \ref{lem:ale_pr-tr}\ref{item:ale_pr-tr}. Moreover, since $T_1\neq \Delta_{T_1}$, we have $A+T_1\subsetneq \Exc \psi_A$, so $\psi_A$ contracts $\ftip{T_{2}}$.

\ref{item:Exc-psi_1} We have $\ftip{T_{2}}\subseteq \Exc\psi_{A}$ by \ref{item:t1>0}, and since by definition $T_{2}$ does not contain $C_{2}$, the maximal twig of $D_{0}$ containing $\ftip{T_{2}}$ equals $T_{2}\wedge (D_{0}-\tilde{C}_{2})$. Thus $\Exc\psi_{A}\subseteq T_{1}+A+T_{2}\wedge (D_{0}-\tilde{C}_{2})$, and by \eqref{eq:line} for the opposite inclusion it suffices to show that $T_{2}\wedge (D_{0}-\tilde{C}_{2})\subseteq \Exc\psi$. Suppose that  a component $V$ of $T_{2}\wedge (D_{0}-\tilde{C}_{2})$ is not contracted by $\psi$. Since $V$ is disjoint from $\Delta_0+E_0$, by  Lemmas \ref{lem:MMP-properties}\ref{item:Delta_pr-tr} and \ref{lem:beta_flat}\ref{item:meeting_E0} $\psi(V)\subseteq \Upsilon_{n}^0$. But since $A\cdot Q_1=A\cdot Q_2=1$, the component $\psi(V)$ meets two different components of $D_n-\psi(V)$; a contradiction.

\ref{item:A_line}, \ref{item:deg} Since $A$ meets $D_{0}$ only at $\ftip{T_{j}}$ for $j\in\{1,2\}$, we have $(\pi_{0}(A)\cdot\bar{E})_{q_{j}}=\mu_{j}$ and $\pi_{0}(A)^{2}=A^{2}+2=1$, so $\mu_{1}+\mu_{2}=\pi_{0}(A)\cdot\bar{E}=\deg\bar{E}$.

\ref{item:not_semiord} holds because by Remark \ref{rem:semi-ordinary}\ref{item:semiordinary} $\psi$ does not touch the exceptional divisors over semi-ordinary cusps. 	

\ref{item:tangent} Assume $A'\cdot Q_{j}=0$ for some $j\in\{1,2\}$. Since $\pi_{0}(A)$ is a line, we have $\deg\pi_{0}(A')=\pi_{0}(A)\cdot\pi_{0}(A')=\mu$, where $\mu$ is the multiplicity of $\pi_{0}(A')$ at $q_{3-j}$. In particular, $q_{3-j}\in \pi_{0}(A')$, so $A'$ meets $Q_{3-j}$. Furthermore,  $\pi_{0}(A')$ is a line, because otherwise the line tangent to some branch of $\pi_{0}(A')$ at $q_{3-j}\in\pi_{0}(A')$ meets $\pi_{0}(A')$ with multiplicity bigger than $\mu$, which is impossible, because $\deg\pi_{0}(A')=\mu$. Thus $\pi_{0}(A')^{2}-(A')^{2}=2$, so $\pi_{0}$ touches $A'$ twice. We have $A'\cdot \Exc\psi_{A}=0$ by Lemma \ref{lem:ale_pr-tr}\ref{item:Exc-psi_i-disjoint}, so by \ref{item:t1>0} $A'\cdot (T_{1}+\ftip{T_{2}})=0$. Hence, $A'$ meets $D_{0}-E_{0}$ only in the second component of $Q_{3-j}$, say $V$, and $A'\cdot V=1$. By Lemma \ref{lem:ale_pr-tr}\ref{item:ale_pr-tr}  $A'$ is almost log exceptional on $(X_{0},\tfrac{1}{2}D_{0})$, so $1=A'\cdot (D_{0}-V)=A'\cdot E_{0}$ and $A'\cdot \Delta_{0}=1$, so $V\subseteq \Delta_{0}$. Because $T_{2}$ is disjoint from $\Delta_{0}$ by \ref{item:t1>0}, we have $V\not\subseteq T_{1}+T_{2}$. Hence, $V=\ftip{T_{3-j}'}$ and $t_{3-j}=0$, so $j=1$ by assumption \eqref{eq:line}. It follows that $\pi_{0}(A')$ is a line meeting $\bar{E}$ transversally in the image of $A'\cap E_{0}$ and with multiplicity $\mu_{2}+\mu_{2}'$ at $q_{2}$. The Bezout theorem gives $\deg\bar{E}=\pi_{0}(A')\cdot\bar{E}=\mu_{2}+\mu_{2}'+1$.

Part \ref{item:deg} implies now that $\mu_{1}=\mu_{2}'+1$, so $\mu_{2}'$ and $\mu_{1}$ are coprime.  Note that the inequality $\#Q_{j}\geq 2$ implies that $\mu_{j},\mu_{j}'>1$ for $j\in\{1,2\}$. Assume $s_{1}=s_{2}=1$. Then $Q_{j}\cdot E_{0}=C_{j}\cdot E_{0}=\tau_{j}$, for $j\in\{1,2\}$. It follows that all terms in the multiplicity sequence of $q_{j}\in \bar{E}$, except the $1$'s at the end, are divisible by $\tau_{j}$.  Thus $\tau_{1}$ and $\tau_{2}$ are coprime. By \eqref{eq:lambda>=tau} and \eqref{eq:lambdaineq} $\tau_{1}+\tau_{2}\leq \lambda_{1}+\lambda_{2}\leq 6$, so $\{\tau_{1},\tau_{2}\}=\{2,3\}$.

\ref{item:j>3} Part \ref{item:tangent} implies that for every $i\in\{1,\dots, n\}$, the almost log exceptional curve $A_{i}$ meets the image of $Q_{1}+Q_{2}+E_{0}$ twice, so it does not meet the image of $Q_{3}+\dots +Q_{c}$, hence by Lemma \ref{lem:MMP-properties}\ref{item:Exc-psi_i} the latter is not touched by $\psi_{i}$.
\smallskip

For the proof of \eqref{eq:eps}, note first that $\epsilon \leq 1$. Indeed, if $\epsilon\neq 0$ then, numbering the cusps as above, we have $T_{j}\not\subseteq \Exc\psi$ for $j\geq 3$ by \ref{item:j>3} and $\Delta_{T_{2}}=0$ by \ref{item:t1>0}, so \eqref{eq:line} holds at most for $j=1$. Let $\hat{\Upsilon}$ be the sum of those components of $D_{0}-\Delta_{0}-\Upsilon_{0}^{0}$ whose image lies in  $\Upsilon_{n}$. 

\setcounter{claim}{0}
\begin{claim}\label{cl:Ups_hat}
$\psi_{*}\hat{\Upsilon}\subseteq \Upsilon_{n}-\Upsilon_{n}^{0}$.
\end{claim}
\begin{proof}
Suppose that $U$ is a component of $\hat{\Upsilon}$ such that $\psi(U)\subseteq \Upsilon_{n}^{0}$. Then there is a unique component $V$ of $D_{0}$ such that $\psi(U)$ meets $\psi(V)$ and by Lemma \ref{lem:beta_flat}\ref{item:Ups} one of the points of $\psi(U)\cap \psi(V)$ is a center of $\psi$, and the other is not. By Lemma \ref{lem:MMP-properties}\ref{item:Exc-psi_i} the preimage of the former is a chain $\rev{T_{U}}+A+T_{V}$, where $A_{U}$ is an almost log exceptional curve and $T_{U}$, $T_{V}$ are zero or twigs of $D_{0}$ meeting $U$ and $V$, respectively. In particular, $T_{U}+U$ is a twig of $D_{0}$. We have $T_{V}\neq 0$, for otherwise $T_{U}\subseteq \Delta_{0}$ and $\psi$ touches $U$ once, so $U^{2}=\psi(U)^{2}-1=-2$, that is, $U\subseteq \Delta_{0}$, contrary to the definition of $\hat{\Upsilon}$. By Lemma \ref{lem:notation}\ref{item:B}, $V=B_{j}$ and $\{U+T_{U},T_{V}\}=\{T_{j},T_{j}'\}$ for some $j\in \{1,\dots, c\}$. This is a contradiction with Lemma \ref{lem:orevkov_ending}.
\end{proof}
	
\begin{claim}\label{cl:b0_1}
$b_{0}(\Delta_{0})-b_{0}(\Delta_{n}^{-})=n+\#\hat \Upsilon$.
\end{claim}
\begin{proof}
By definition, $\psi$ touches exactly $n$ connected components of $\Delta_{0}$. Let $W$ be a connected component of $\Delta_{0}$ not touched by $\psi$ and such that $\psi_{*}W$ is not a connected component of $\Delta_{n}^{-}$. Then $\psi_{*}W$ is a connected component of $\Delta_{n}^{+}$, so $W$ meets a unique component $U$ of $D_{0}-\Delta_{0}-\Upsilon_{0}^{0}$ such that $\psi(U)\subseteq \Upsilon_{n}$, that is, $U\subseteq \hat{\Upsilon}$. Conversely, if $U\subseteq \hat{\Upsilon}$ then by Claim \ref{cl:Ups_hat}, $\psi(U)\subseteq \Upsilon_{n}-\Upsilon_{n}^{0}$, so by Lemma \ref{lem:MMP-properties}\ref{item:Delta_pr-tr}, $U$ meets a connected component $W$ of $\Delta_{0}$ such that $\psi_{*}W$ is a connected component of $\Delta_{n}^{+}$ and, because $U\cdot\Delta_{0}\leq 1$, such $W$ is unique. Hence the number of  connected components $W$ as above equals $\#\hat{\Upsilon}$.
\end{proof}

Let $P$ be the sum of components of $D_0$ contained in $\Delta_0$ or contracted by $\psi$. Then 
\begin{equation*}
D_{0}-P=\psi^{-1}_{*}D_{n} \wedge (D_{0}-\Delta_{0})=\psi^{-1}_{*}(R_{n}+\Upsilon_{n}+\Delta_{n}) \wedge (D_{0}-\Delta_{0}).
\end{equation*}
The divisor $\hat{R}=\psi^{-1}_{*}R_{n}$ is contained in $R_{0}$ by Lemma \ref{lem:beta_flat}\ref{item:Rn_connected}, so it has no common component with $\Delta_{0}$. On the other hand, by Lemma \ref{lem:MMP-properties}\ref{item:Delta_pr-tr},  $\psi^{-1}_{*}\Delta_{n}\subseteq \Delta_0$. We get 
\begin{equation*} 
D_0-P=\hat{R}+\psi^{-1}_{*}(\Upsilon_{n}) \wedge (D_{0}-\Delta_{0})=\hat{R}+\hat{\Upsilon}+\Upsilon^{0}_{0}.
\end{equation*}
By Lemma \ref{lem:MMP-properties}\ref{item:Exc-psi_i}, $P$ is a sum of some twigs of $D_{0}$. In particular, it is disjoint from $E_0+\Upsilon_{0}^{0}$ (see Remark \ref{rem:semi-ordinary}\ref{item:semiordinary}), hence every connected component of $P$ meets $\hat{R}-E_{0}+\hat{\Upsilon}$. 

\begin{claim}\label{cl:b0_2}
$b_{0}(\Delta_{0})\leq \#(\hat{R}-E_{0}+\hat{\Upsilon})+\epsilon$.
\end{claim}
\begin{proof}
Let $W$ be a connected component of $\Delta_0$. It is contained in a unique connected component $P_W$ of $P$ (which is a twig of $D_0$), and the latter meets a unique component $V_W$ of $\hat{R}-E_{0}+\hat{\Upsilon}$. Assume $V_{W}=V_{W'}$ for some $W\neq W'$. Then by Lemma \ref{lem:notation}\ref{item:B}, $V_{W}=B_{j}$ and  $\{P_{W},P_{W'}\}=\{T_{j},T_{j}'\}$ for some $j\in \{1,\dots, c\}$. 
It follows that $\ltip{T_{j}}\subseteq P-\Delta_0\subseteq \Exc \psi$, hence by Lemma \ref{lem:MMP-properties}\ref{item:Exc-psi_i} $\psi$ contracts $T_j$ and so \eqref{eq:line} holds for $j$.
\end{proof}

We have $\#(\hat{R}-E_{0})=n$ by Lemma \ref{lem:beta_flat}\ref{item:Rn=n+1}, so Claims \ref{cl:b0_1} and \ref{cl:b0_2} give
\begin{equation*}
b_{0}(\Delta_{0})-\epsilon
\leq
\#(\hat{R}-E_{0}+\hat{\Upsilon})
=
n+\#\hat{\Upsilon}
=
b_{0}(\Delta_{0})-b_{0}(\Delta_{n}^{-}),
\end{equation*}
which proves $b_{0}(\Delta_{n}^{-})\leq \epsilon$ and ends the proof of \eqref{eq:eps}.
\end{proof}

\begin{rem}[Relations with proper $\C^{*}$-embeddings into $\C^2$, cf.\ Theorem \ref{thm:geometric}]\label{rem:line} \ 
	\begin{enumerate}
	\item If $A$ is as in Lemma \ref{lem:line}\ref{item:A_line} then $\bar{E}\setminus \pi_{0}(A)\subseteq \P^{2}\setminus \pi_{0}(A)$ is the image of a proper injective morphism $\C^{*}\to \C^{2}$.
	Such images are classified in  \cite{CKR-Cstar_good_asymptote,KoPaRa-SporadicCstar1} in case when they are smooth and in \cite{BoZo-annuli} under some regularity conditions.
	\item If $A'$ is as in 	Lemma \ref{lem:line}\ref{item:tangent} then $\pi_{0}(A')$ is a line which is a good asymptote in the sense of \cite{CKR-Cstar_good_asymptote} for the above $\C^{*}$-embedding.
	\end{enumerate}
\end{rem}

\smallskip
\subsection{Types with a singular minimal model.}\label{sec:F2}
In this section, we prove the following result on types with a singular minimal model.
\begin{prop}\label{prop:F2}
	If $X_{\min}$ is singular then $\bar{E}$ is of type $\FE$ or $\cI$. 
\end{prop}

Throughout this section, we assume that $X_{\min}$ is singular. Recall from \eqref{eq:peeling} that each singular point of $X_{\min}$ is the image of a connected component of $\Delta_{n}^{-}$, which is a maximal $(-2)$-twig of $D_{n}$. By \eqref{eq:eps} $\Delta_{n}^{-}$ is connected, so $X_{\min}$ has only one singular point. We denote its preimage on $X_{0}$ by
\begin{equation*}
\hat{\Delta}^{-}\de \psi^{-1}_{*}\Delta_{n}^{-}.
\end{equation*}
Moreover, \eqref{eq:eps} implies that the condition \eqref{eq:line} of Lemma \ref{lem:line} is satisfied, so we can, and will, number the cusps of $\bar{E}$ so that Lemma \ref{lem:line}\ref{item:A_meets_T1T2}--\ref{item:j>3} holds. In particular, we have an almost log exceptional curve $A$ on $(X_{0},\tfrac{1}{2}D_{0})$ meeting $\ftip{T_{1}}$ and $\ftip{T_{2}}$. If $n>1$ then there exists an almost log exceptional curve contracted by $\psi$ other than $A$. In this case we pick one and we denote its proper transform on $X_0$ by $A'$. Such $A'$ is almost log exceptional on $(X_{0},\tfrac{1}{2}D_{0})$ by Lemma \ref{lem:ale_pr-tr}\ref{item:ale_pr-tr}. We denote by $W$ the unique maximal $(-2)$-twig of $D_0$ meeting $A'$.

Recall (see \eqref{eq:theta}) that $(Z,\tfrac{1}{2}D_{Z})$ denotes the image of $(X_{n},\tfrac{1}{2}D_{n})$ after the contraction of $\Upsilon_{n}+\Delta_{n}^{+}$. In particular, $Z$ is smooth.
For $m\geq 0$ we denote by $\F_{m}$ the Hirzebruch surface $\P(\O_{\P^{1}}\oplus\O_{\P^{1}}(m))$ and by $F$ a general fiber of the unique $\P^1$-fibration $p_{\F_{m}}\colon \F_{m}\to \P^{1}$.

\setcounter{claim}{0}
\begin{lem}[The geometry of $D_{0}+A+A'$ in case of a singular minimal model]\label{lem:F2_core}Let $A$, $\hat{\Delta}^{-}$ and $(Z,D_{Z})$ be as above. Then:
\begin{enumerate}
	\item\label{item:F2_semiord}The cusps $q_{j}\in \bar{E}$ for $j\geq 3$ are semi-ordinary and $q_{1},q_{2}\in \bar{E}$ are not.
	\item\label{item:F2_core} $R_{n}=E_{n}+\psi(C_{1})+\psi(C_{2})$.
	\item\label{item:F2_s2} $s_{j}=1$ for $j\in\{1,\dots, c\}$ (see Notation \ref{not:graphs}\ref{item:not_s}).
	\item\label{item:F2_Exc_psi_A} $\Exc\psi_{A}=\rev{T_{1}}+A+T_{2}$.
	\item\label{item:F2_n2}	$n=2$.
	\item\label{item:F2_A'} $A'+W$ is disjoint from $T_{1}+A+T_{2}$.
	\item\label{item:F2_Q1} $Q_{1}=T_{1}+C_{1}+\hat{\Delta}^{-}=[(2)_{t_{1}},3,1,2]$.
	\item\label{item:F2_Q2}  Either $Q_{2}=T_{2}+C_{2}+W=[t_{1}+2,1,(2)_{t_{1}}]$ or $Q_{2}$ is a fork with maximal twigs $T_{2}=[t_{1}+2]$, $W+C_{2}$ and $\hat{\Delta}^{+}=[(2)_{t_{1}}]$, where $\psi_{*}\hat{\Delta}^{+}\subseteq \Delta_{n}^{+}$.
	\item\label{item:F2_Z=F2} $Z\cong \F_{2}$, $\psi^{+}(\hat{\Delta}^{-})$ is the negative section and $D_{Z}\cdot F=4$.
	\item\label{item:F2_hor} $D_{Z}=\psi^{+}(E_{0})+\psi^{+}(C_{1})+\psi^{+}(C_{2})+\psi^{+}(\hat{\Delta}^{-})$ and  $D_{Z}-\psi^{+}(C_{1})$ is horizontal.
	\end{enumerate}	
\end{lem}
\begin{proof}
\ref{item:F2_semiord} The cusps $q_{1}$, $q_{2}$ are not semi-ordinary by Lemma \ref{lem:line}\ref{item:not_semiord}. Suppose that $c\geq 3$ and $q_{3}\in \bar{E}$ is not ordinary. By Remark \ref{rem:lambda>1} $\sum_{j=1}^{3}\lambda_{j}\geq \sum_{j=1}^{3}\tau_{j}\geq 6$, so by \eqref{eq:lambdaineq} the equalities hold. Therefore, $\lambda_{3}=2$, hence $\tau_{3}=2$ and $\psi_{*}Q_{3}\subseteq \psi(C_{3})+\Upsilon_{n}+\Delta_{n}$. Lemma \ref{lem:line}\ref{item:tangent} implies that $\psi$ does not touch $Q_{3}$, so by Lemma \ref{lem:beta_flat}\ref{item:Ups} $\psi_{*}Q_{3}\subseteq \psi(C_{3})+\Delta_{n}$, and thus $Q_{3}\subseteq C_{1}+\Delta_{0}$ by Lemma \ref{lem:MMP-properties}\ref{item:Delta_pr-tr}. Lemma \ref{lem:notation}\ref{item:T^0=C} shows that $q_{3}$ is semi-ordinary.
\smallskip

Before we prove the remaining parts of Lemma \ref{lem:F2_core}, we make the following observation.

\begin{claim}\label{cl:F2}
If $\hat{\Delta}$ is irreducible then \ref{item:F2_Z=F2} holds and $D_{Z}$ contains at most one fiber of $p_{\F_{2}}$.
\end{claim}
\begin{proof}
	The surface $Z$ is smooth, rational and $\rho(Z)=\rho(X_{\min})+\#\hat{\Delta}^{-}=2$, so $Z$ is a Hirzebruch surface. It is $\F_{2}$, because it contains the $(-2)$-curve $\psi^{+}(\hat{\Delta}^{-})=\alpha_{n}^{+}(\Delta_{n}^{-})$. If $D_{Z}\cdot F\leq 3$ then $Z\setminus D_{Z}$, which is an open subset of $\P^{2}\setminus \bar{E}$, admits a $\C^{1}$-, $\C^{*}$- or a $\C^{**}$-fibration. Thus by \eqref{eq:assumption} $D_{Z}\cdot F\geq 4$. Since $-(2K_{X_{\min}}+D_{\min})$ is ample, we have 
	\begin{equation*}
	0>\alpha_{n}^{-}(F)\cdot (2K_{X_{\min}}+D_{\min})=F\cdot (2K_{Z}+D_{Z}-\tfrac{1}{2}\alpha_{n}^{+}(\Delta_{n}^{-}))=-4+F\cdot D_{Z}-\tfrac{1}{2},
	\end{equation*}	
	so $F\cdot D_{Z}\leq 4$ and \ref{item:F2_Z=F2} follows. Since $\Delta_{n}^{-}$ is a twig of $D_{n}$, it meets a unique component of $D_{n}-\Delta_{n}^{-}$, so $\alpha_{n}^{+}(\Delta_{n}^{-})$ meets a unique component of $D_{Z}-\alpha_{n}^{+}(\Delta_{n}^{-})$. Because $\alpha_{n}^{+}(\Delta_{n}^{-})$ is a section, $D_{Z}$ contains at most one fiber of $p_{\F_{2}}$.
\end{proof}

\ref{item:F2_core} By \ref{item:F2_semiord}, $q_{1},q_{2}\in\bar{E}$ are not semi-ordinary, so 
\begin{equation*}
E_{n}+\psi(C_{1})+\psi(C_{2})\subseteq R_{n}\subseteq \psi_{*}(Q_{1}+Q_{2}).
\end{equation*}
Suppose that  the first inclusion is strict. Let $G$ be a component of $\psi^{-1}_{*}R_{n}-C_{1}-C_{2}$. Then  by \ref{item:F2_semiord}, $G\subseteq Q_{1}+Q_{2}$. We claim that $\tau_{1}-s_{1}+\tau_{2}-s_{2}\geq 3$. Indeed, if $s_{j}=0$ for some $j\in \{1,2\}$ then this follows from the inequalities $\tau_{1},\tau_{2}\geq 2$. Assume $s_{1}=s_{2}=1$. Then $G\cdot E_{0}=0$, so Lemma \ref{lem:beta_flat}\ref{item:Dmin_meet_each_other},\ref{item:centers_on_En} implies that the proper transform on $X_{0}$ of some $A_{i}$ meets $E_{0}$, and thus $\tau_{1}+\tau_{2}=5$ by Lemma \ref{lem:line}\ref{item:tangent}.

The divisor $\psi_{*}(Q_{1}+Q_{2})\wedge R_{n}$ contains images of $C_{1}$, $C_{2}$ and $G$, so $$\sum_{j=1}^{c}\lambda_{j}\geq 3+\tau_{1}-s_{1}+\tau_{2}-s_{2}+\#\Delta_{n}-b_{0}(\Delta_{n})\geq 6.$$ Hence \eqref{eq:lambdaineq} implies that the equalities hold. In particular, $c=2$, $\#\Delta_{n}=b_{0}(\Delta_{n})$ and if $s_{j}=0$ for some $j\in \{1,2\}$ then $s_{3-j}=1$ and $\tau_{1}=\tau_{2}=2$. 

From the equality $\#\Delta_{n}=b_{0}(\Delta_{n})$ we get $\hat{\Delta}^{-}=[2]$, so by Claim \ref{cl:F2}, $Z\cong \F_{2}$ with negative section $\psi^{+}(\hat{\Delta}^{-})$. The divisor $D_{Z}$ contains images of $E_{0}$, $C_{1}$, $C_{2}$, $G$ and $\hat{\Delta}^{-}$. Again from Claim \ref{cl:F2} we infer that $\#D_{Z}=5$ and $D_{Z}$ consists of a fiber and four $1$-sections. We have $\psi^{+}(C_{j})\cdot \psi^{+}(E_{0})\geq C_{j}\cdot E_{0}\geq 2$ for $j\in\{1,2\}$, which implies that $\psi^{+}(C_{j})$ and $\psi^{+}(E_{0})$ are horizontal. Thus $\psi^{+}(G)$ is a fiber, so it meets $\psi^{+}(\hat{\Delta}^{-})$. Because $\psi^{+}$ does not touch $\hat{\Delta}^{-}$, the latter is a unique $(-2)$-twig of $D_{0}$ meeting $G$ (see Lemma \ref{lem:notation}\ref{item:B}). It follows from Lemma \ref{lem:beta_flat}\ref{item:meeting_E0} that $G\cdot E_{0}>0$, so $G=\tilde{C}_{j}$ for some $j\leq 2$. The components of $D_{Z}-\psi^{+}(\hat{\Delta}^{-})-\psi^{+}(G)$ are $1$-sections disjoint from the negative section $\psi^{+}(\hat{\Delta}^{-})$, so they are linearly equivalent to $2\psi^{+}(G)+\psi^{+}(\hat{\Delta}^{-})$. We obtain that $\psi^{+}(E_{0})^{2}=2$ and $\psi^{+}(E_{0})\cdot \psi^{+}(C_{k})=2$ for $k\in\{1,2\}$. The latter implies that $\psi^{+}(E_{0})\cdot \psi^{+}(V)=E_{0}\cdot V$ for every component $V$ of $D_{0}-E_{0}$ not contracted by $\psi^{+}$. Hence by Lemma \ref{lem:MMP-properties}\ref{item:center-psi_i} $\psi^{+}$ does not touch $E_{0}$. Therefore,  $E^{2}=\psi^{+}(E_{0})^{2}-(\tau_{1}+\tau_{2})=-2$. Because $s_{3-j}=1$ and $\tau_{1}=\tau_{2}=2$, this is a contradiction with Lemma \ref{lem:Tono_E2}\ref{item:c=2,tau=1}.

\ref{item:F2_s2} This follows from \ref{item:F2_core}, because $R_{n}$ does not contain $\psi(\tilde{C}_{j})$ for $j\in\{1,\dots,c\}$.

\ref{item:F2_Exc_psi_A} This follows from \ref{item:F2_s2} and Lemma \ref{lem:line}\ref{item:Exc-psi_1}.

\ref{item:F2_n2} This follows from \ref{item:F2_core} and Lemma \ref{lem:beta_flat}\ref{item:Rn=n+1}.

\ref{item:F2_A'} We have $W+A'\subseteq \Exc\psi_{A'}$ and, by \ref{item:F2_Exc_psi_A}, $T_{1}+A+T_{2}=\Exc\psi_{A}$. These divisors are disjoint by Lemma \ref{lem:ale_pr-tr}\ref{item:Exc-psi_i-disjoint}.
\smallskip

Before we prove \ref{item:F2_Q1}, we need some preparation. If $C_{j}$ is not a tip of $Q_{j}$ for some $j\in \{1,2\}$ then,  because $Q_{j}$ contracts to a smooth point, \ref{item:F2_s2} implies that $C_{j}$ meets a twig of $D_{0}$ other than $T_{j}$. Denote this twig by $V_{j}$ and put $V_{j}=0$ if $C_{j}$ is a tip of $Q_{j}$. Fix $j\in \{1,2\}$. We need the following claims.

\begin{claim}\label{cl:Qj_fork}
If $Q_{j}$ is not a chain then $Q_{3-j}$ is a chain, $\psi(B_{j})\subseteq \Upsilon_{2}$ and $Q_{j}$ is a fork with maximal twigs $T_{j}$, $V_{j}+C_{j}$ and some $\hat{\Delta}^{+}$ for which $\psi_{*}\hat{\Delta}^{+}\subseteq \Delta_{2}^{+}$ (see Figure \ref{fig:picture}).
\end{claim}
\begin{proof}
Lemma \ref{lem:notation}\ref{item:B} gives $B_{j}\neq C_{j}$, hence $\psi(B_{j})\not\subseteq R_{2}$ by \ref{item:F2_core}. We have $\psi(B_{j})\not\subseteq \Delta_{2}$ by Lemma \ref{lem:MMP-properties}\ref{item:Delta_pr-tr}, so $\psi(B_{j})\subseteq \Upsilon_{2}$. Since $B_{1}$, $B_{2}$ meet $T_{1}$, $T_{2}$, respectively, from Lemma \ref{lem:line}\ref{item:Exc-psi_1} we infer that $\psi(B_{1})$ meets $\psi(B_{2})$. Then by Lemma \ref{lem:MMP-properties}\ref{item:Ui_disjoint},  $\psi(B_{3-j})\not\subseteq \Upsilon_{2}$, so by \ref{item:F2_core},  $B_{3-j}=C_{3-j}$, that is, $Q_{3-j}$ is a chain. Moreover, $\psi(B_{j})\not\subseteq \Upsilon_{2}^{0}$, because $\psi(B_{j})$ meets $\psi(B_{3-j})\subseteq \psi_{*}Q_{3-j}$ and $\psi_{*}(Q_{j}-B_{j})$. Thus $B_{j}$ meets $\psi^{-1}_{*}\Delta_{2}^{+}$, which proves the claim.
\end{proof}	

	\begin{figure}[htbp]
		\centering
		\begin{subfigure}{0.49\textwidth}
			\centering
			\includegraphics[scale=0.2]{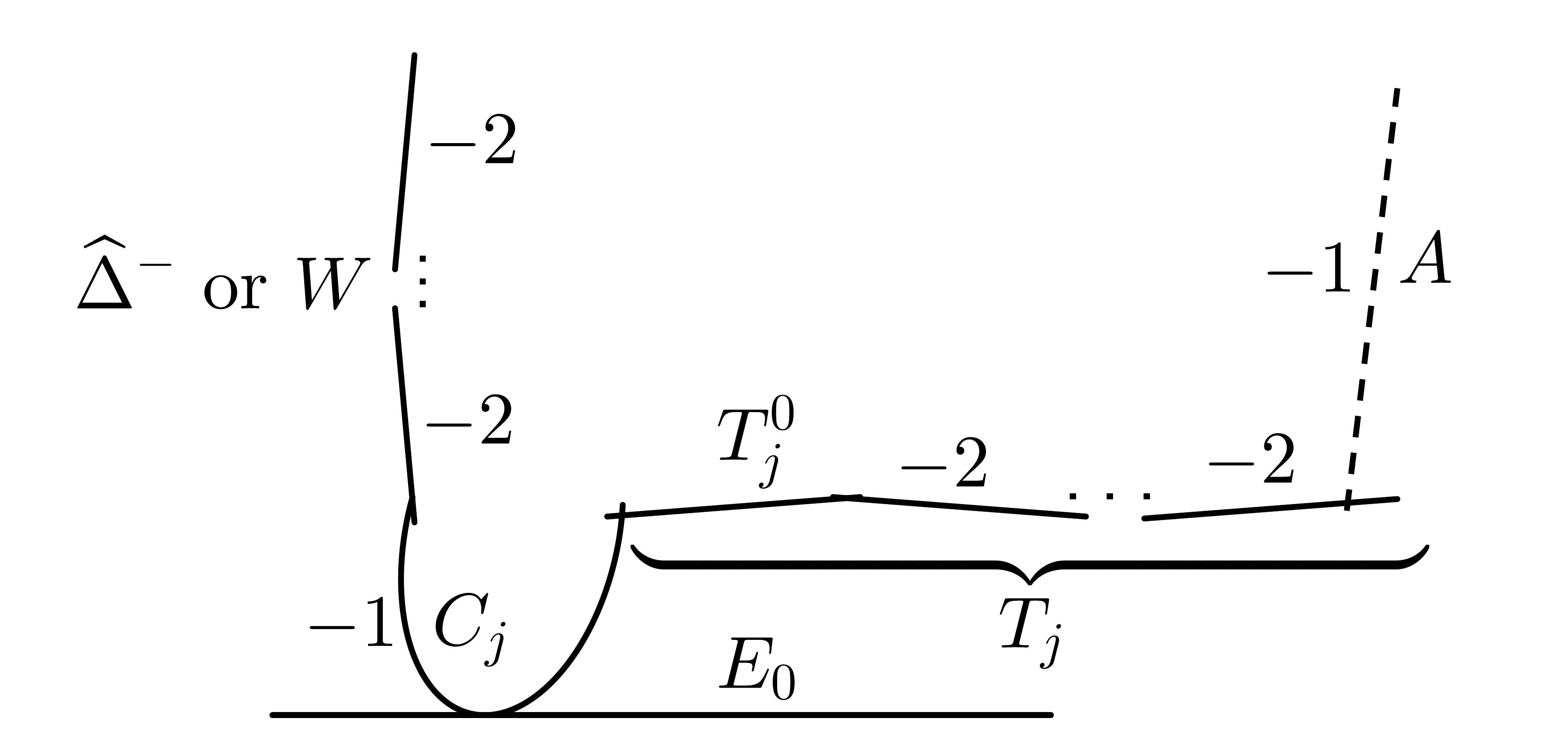}
			\caption{$Q_{j}$ - chain}
		\end{subfigure}
		\hfill
		\begin{subfigure}{0.49\textwidth}
			\centering
			\includegraphics[scale=0.2]{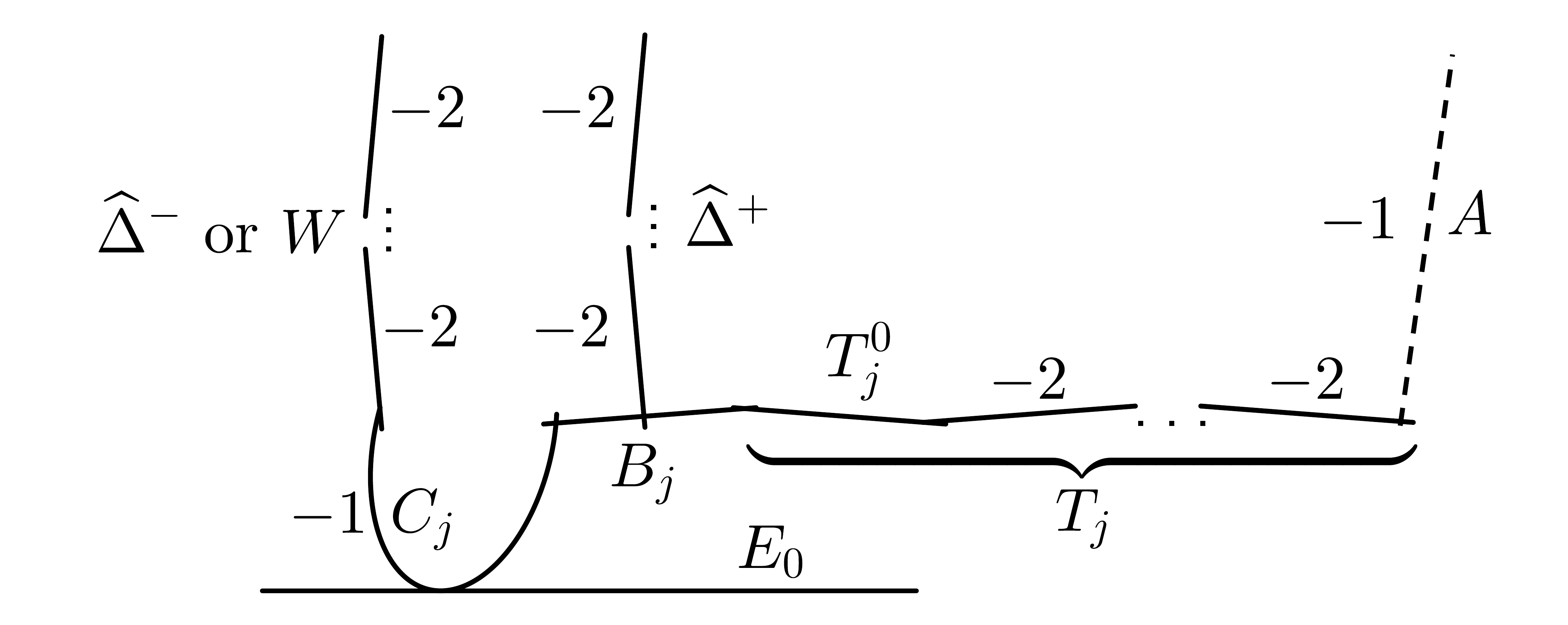}
			\caption{$Q_{j}$ - fork}
		\end{subfigure}
		\caption{Two possible shapes of $Q_{j}$ according to Claims \ref{cl:Qj_fork} and \ref{cl:Cj'} in the proof of Lemma \ref{lem:F2_core}.}
		\label{fig:picture}
	\end{figure}

\begin{claim}\label{cl:Cj'}
The divisor $Q_{1}+Q_{2}$ contains four or five maximal twigs of $D_{0}$, namely $T_{1}$, $T_{2}$, $\hat{\Delta}^{-}$, $W$ and (possibly) $\hat{\Delta}^{+}$.
\end{claim}
\begin{proof}
	By Claim \ref{cl:Qj_fork}, $Q_{1}+Q_{2}$ contains at most five maximal twigs of $D_{0}$. Among them there are $T_{1}$, $T_{2}$, which are not $(-2)$-twigs by Lemma \ref{lem:line}\ref{item:t1>0}, and $(-2)$-twigs: $\hat{\Delta}^{-}$, the $(-2)$-twig $W\neq \hat{\Delta}^{-}$ meeting $A'$ and, in case there are exactly five twigs, also  $\hat{\Delta}^{+}\neq W, \hat{\Delta}^{-}$.
\end{proof}

\begin{claim}\label{cl:T1T2}
	$T_{1}=[(2)_{t_{1}},3]$ and $T_{2}=[t_{1}+2]$.
\end{claim}
\begin{proof}
By Claim \ref{cl:Cj'}, $T_{j}'\subseteq \Delta_{0}$ for $j\in\{1,2\}$ (see Notation \ref{not:graphs}\ref{item:not_T'}), so because $T_{j}+[1]+T_{j}'$ contracts to a smooth point, we get $T_{j}=[(2)_{t_{j}},\#T_{j}'+2]$. We have $t_{1}>0$ and $t_{2}=0$ by Lemma \ref{lem:line}\ref{item:t1>0}, so \ref{item:F2_Exc_psi_A} gives $\Exc\psi_{A}=[\#T_{1}'+2,(2)_{t_{1}},1,\#T_{2}'+2]$. Because $\Exc\psi_{A}$ contracts to a smooth point, we get $\#T_{1}'=1$ and $\#T_{2}'=t_{1}$.
\end{proof}

We return to the proof of Lemma \ref{lem:F2_core}.
\smallskip 

\ref{item:F2_Q1}  Suppose that $Q_{1}$ is not a chain. By Claim \ref{cl:Qj_fork}, $Q_{1}$ is a fork and $\psi(B_{1})\subseteq \Upsilon_{2}$. Part \ref{item:F2_Exc_psi_A} and Claim \ref{cl:T1T2} imply that $\psi_{A}$ touches $B_{1}$ once. By Lemma \ref{lem:beta_flat}\ref{item:Ups} the point $\psi(A)$ is the unique center of $\psi$ on $\psi(B_{1})$, so $B_{1}^{2}=\psi(B_{1})^{2}-1=-2$. Since by Claim \ref{cl:Qj_fork}, $C_{1}$ meets $B_{1}$, the contractibility of $Q_{1}$ to a smooth point implies that $C_{1}$ is a tip of $Q_{1}$, contrary to Claim \ref{cl:Cj'}.

Thus $Q_{1}$ is a chain. Because $\ftip{T_{1}}$ contracts last among the components of $Q_{1}$, Claim \ref{cl:T1T2} shows that $Q_{1}=[(2)_{t_{1}},3,1,2]$. It remains to prove that $V_{1}=\hat{\Delta}^{-}$. Suppose the contrary. Then by Claim \ref{cl:Cj'}, $V_{1}=W$, so $Q_{1}$ meets $A'$. By Lemma \ref{lem:line}\ref{item:tangent}, $A'$ meets $Q_{2}$. We have $A'\cdot T_{2}=0$ by \ref{item:F2_A'} and $A'\cdot V_{2}=0$, because $V_{2}=\hat{\Delta}^{-}$ by Claim \ref{cl:Cj'}. Claim \ref{cl:Qj_fork} implies that $Q_{2}-C_{2}-T_{2}-V_{2}$ is either zero or equal to $B_{2}+\hat{\Delta}^{+}$, which is a proper transform of a connected component of $\Upsilon_{2}+\Delta_{2}^{+}$ containing the point $\psi(A)$. We infer from Lemma \ref{lem:beta_flat}\ref{item:Ups} that $A'\cdot (Q_{2}-C_{2}-T_{2}-V_{2})=0$, so $A'$ meets $C_{2}$. By \ref{item:F2_Exc_psi_A} and \ref{item:F2_A'} the curve $A$ meets $D_{0}+A'$ only in the components of $Q_{1}$ and $Q_{2}$ which are contracted last by $\pi_{0}$. Hence, $\pi_{0}(A)$ is a line meeting $\pi_{0}(A')$ only at $q_{1}$, $q_{2}$ with the least possible multiplicity, which equals respectively $1$ and, say, $\mu$ for some $\mu\geq 2$. Thus $\deg \pi_{0}(A')=\pi_{0}(A)\cdot\pi_{0}(A')=\mu+1$, so the intersection number at $q_{2}$ of $\pi_{0}(A')$ and its tangent line equals at most $\mu+1$. But this number is the sum of at least two initial terms of the multiplicity sequence of $q_{2}\in \pi_{0}(A')$, so this sequence equals $(\mu,1,\dots)$. It follows that $Q_{2}=[\mu+1,1,(2)_{\mu-1}]$. Claim \ref{cl:T1T2} implies that $t_{1}=\mu-1$, so the contraction of $Q_{1}$ touches $A'$ exactly $\mu+1$ times. Therefore, $$(\mu+1)^{2}=\pi_{0}(A')^{2}=(A')^{2}+(\mu+1)+(\mu+\mu^{2})=(\mu+1)^{2}-1;$$ a contradiction.

\ref{item:F2_Q2} Part \ref{item:F2_Q1} implies that $\hat{\Delta}^{-}\not\subseteq Q_{2}$, so by Claim \ref{cl:Cj'}, $V_{2}=W$. Claim \ref{cl:T1T2} gives $T_{2}=[t_{1}+2]$, so $T_{2}'=[(2)_{t_{1}}]$ by the contractibility of $Q_{2}$ to a smooth point. This proves \ref{item:F2_Q2} if $Q_{2}$ is a chain. If $Q_{2}$ is a fork then \ref{item:F2_Q2} follows from Claim \ref{cl:Qj_fork}.

\ref{item:F2_Z=F2} This follows from Claim \ref{cl:F2}, because by \ref{item:F2_Q1}, $\hat{\Delta}^{-}=[2]$.

\ref{item:F2_hor} The first statement follows from \ref{item:F2_core}. To prove the second one, recall that $\psi^{+}$ does not touch $\hat{\Delta}^{-}$, so $\psi^{+}(\hat{\Delta}^{-})\cdot \psi^{+}_{*}(D_{0}-C_{1})=\hat{\Delta}^{-}\cdot (D_{0}-C_{1})=0$ by \ref{item:F2_Q1}. The curve $\psi^{+}(\hat{\Delta}^{-})$ is a section by Claim \ref{cl:F2}, so $\psi^{+}_{*}(D_{0}-C_{1})$ contains no fiber.
\end{proof}

\begin{figure}[htbp]
	\begin{tabular}{c c c}
		\multirow{3}{*}{
			\begin{subfigure}{0.45\textwidth}\centering
				\vspace{-2.8cm}
				\includegraphics[scale=0.25]{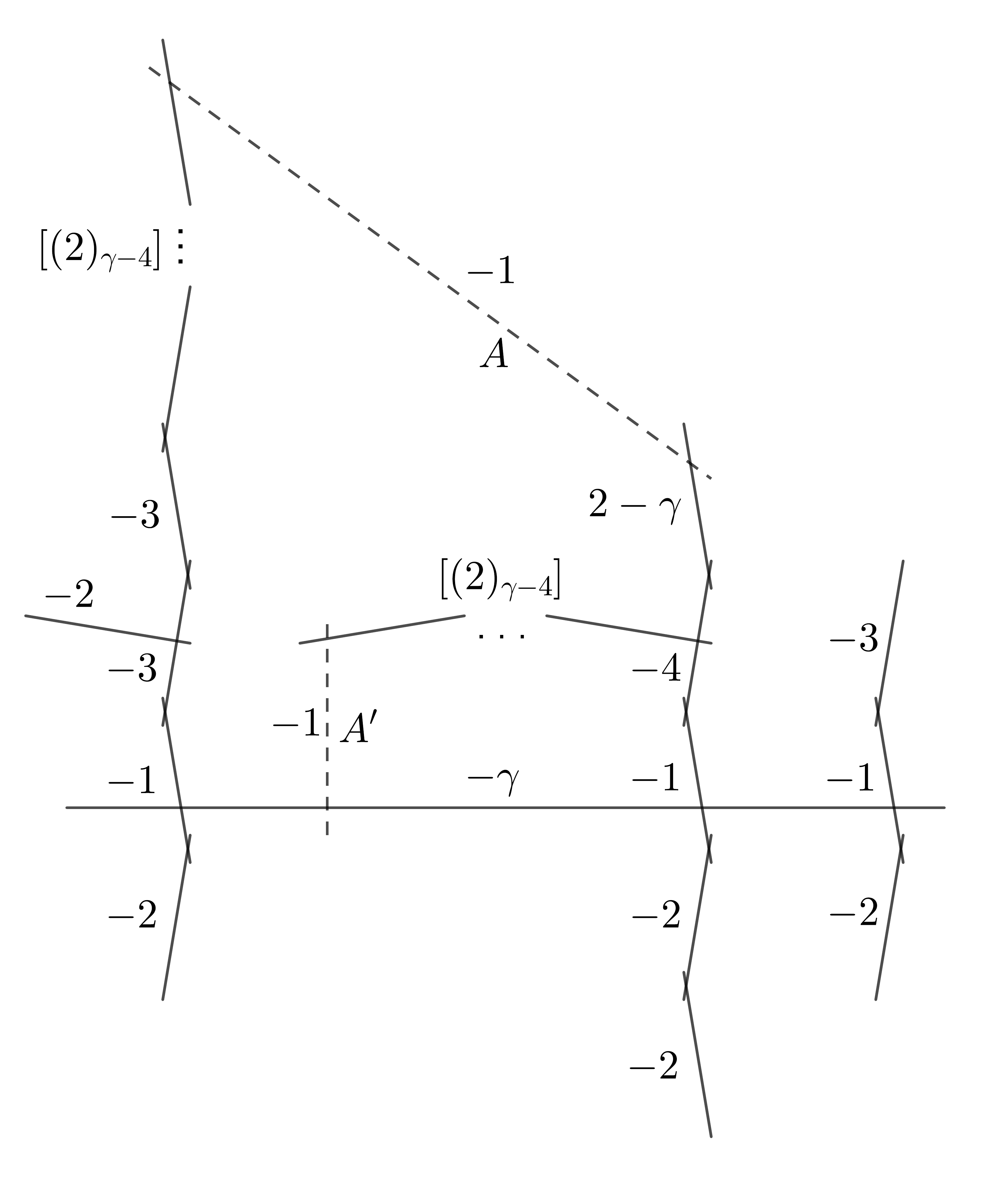}
				\vspace{-1.5cm}
				\caption{$(X,D)$}
				\vspace{.5cm}
				\begin{flushright} \begin{fmpage}{.5\textwidth} \begin{equation*}\begin{split}
						\boldsymbol{\circ}& = \mbox{ image of } A\\
						\boldsymbol{\diamond}&= \mbox{ image of } A'
						\end{split}\end{equation*} \end{fmpage} \end{flushright}					
			\end{subfigure}
		}
		&
		$\xrightarrow{\quad \displaystyle{ \psi\circ\psi_{0}}\quad }$
		& 
		\begin{subfigure}{0.35\textwidth}\centering
			\includegraphics[scale=0.35]{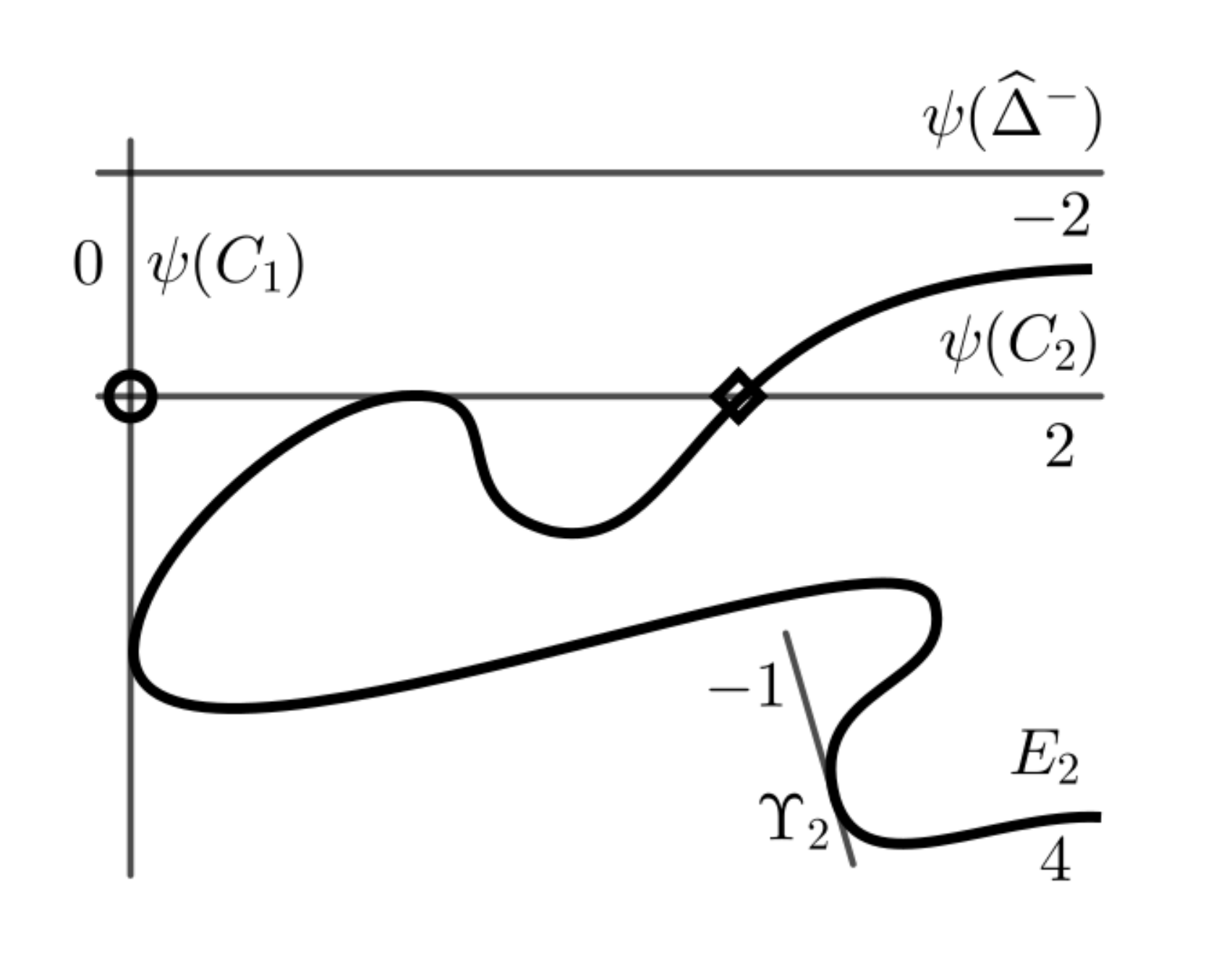}
			\vspace{-.5cm}
			\caption{$(X_{2},D_{2})$}
		\end{subfigure}
		\vspace{0.5cm}
		\\
		&&
		$
		\Bigg\downarrow \alpha_{2}^{+}
		$
		\\
		&&
		\begin{subfigure}{0.35\textwidth}\centering
			\includegraphics[scale=0.35]{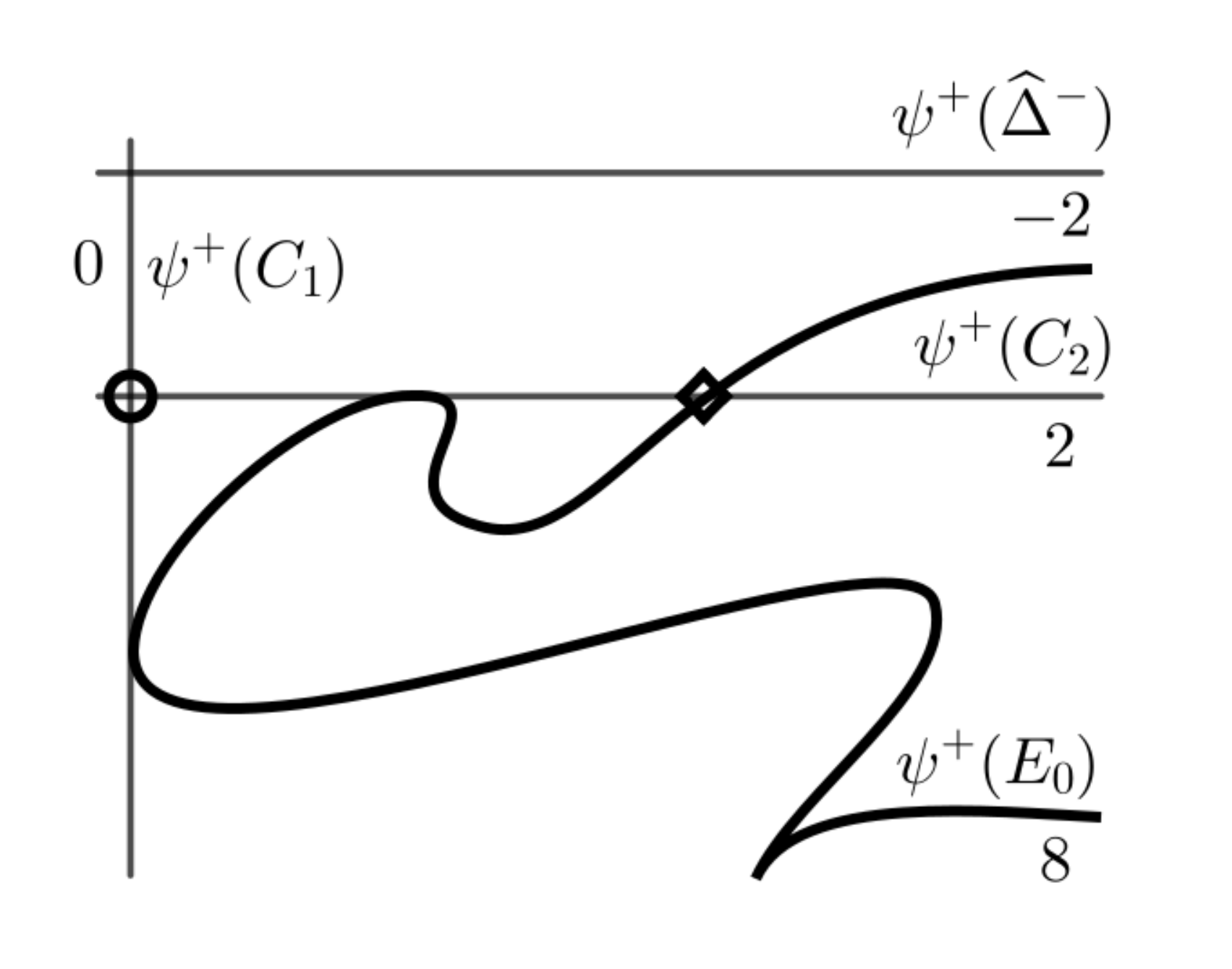}
			\vspace{-1em}
			\caption{$(Z,D_Z)$}
		\end{subfigure}
	\end{tabular}
	\caption{Type $\FE(\gamma)$, $\gamma \geq 5$ (in Def.\ \ref{def:our_curves}, $q_{1},q_{2}\in \bar{E}$ are in the opposite order).}
	\label{fig:FE}
\end{figure}
\begin{prop}\label{prop:F2_FE}
	If $D_{Z}$ contains a fiber then $\bar{E}$ is of type $\FE$ (see Figure \ref{fig:FE}).
\end{prop}
\begin{proof}
	By Lemma \ref{lem:F2_core}\ref{item:F2_hor} the unique fiber in $D_{Z}$ equals $\psi^{+}(C_{1})$. Because $C_{1}^{2}=-1=\psi^{+}(C_{1})^{2}-1$ and $\psi_{A}$ touches $C_{1}$ once, we have $A'\cdot C_{1}=0$. In fact, $A'\cdot Q_{1}=0$, because $Q_{1}-C_{1}=T_{1}+\hat{\Delta}^{-}$ by Lemma \ref{lem:F2_core}\ref{item:F2_Q1} and $T_{1}\cdot A'=0$ by Lemma \ref{lem:F2_core}\ref{item:F2_A'}. By Lemma \ref{lem:line}\ref{item:tangent} $A'$ meets $\ftip{T_{2}'}$ and $E_{0}$, so $W=T_{2}'$. Lemma \ref{lem:F2_core}\ref{item:F2_Q2} implies that $Q_{2}=[t_{1}+2,1,(2)_{t_{1}}]$. It follows from Lemma \ref{lem:F2_core}\ref{item:F2_core} that 
	\begin{equation*}
	4=\psi^{+}(C_{1})\cdot D_{Z}=\psi^{+}(C_{1})\cdot (\psi^{+}(\hat{\Delta}^{-})+\psi^{+}(C_{2})+\psi^{+}(E_{0}))=2+\tau_{1},
	\end{equation*}
	so $\tau_{1}=2$. Lemma \ref{lem:line}\ref{item:tangent} gives $\tau_{2}=3$. We have $\psi^{+}(E_{0})\cdot \psi^{+}(C_{1})=\tau_1=2$ and $\psi^{+}(E_{0})\cdot \psi^{+}(\hat{\Delta}^{-})=0$, which by numerical properties of $\F_{2}$ gives $\psi^{+}(E_{0})\sim -K_{\F_{2}}$, hence $p_{a}(\psi^{+}(E_{0}))=1$. It follows (see the proof of Lemma \ref{lem:HN-equations}) that $\Sing \psi^{+}(E_{0})$ is an ordinary cusp, thus $c=3$ and $q_{3}$ is ordinary. Since $\tau_{1}=2$, $\tau_{2}=3$ and $s_{1}=s_{2}=1$, we see that $\bar{E}$ is of type $\FE(\gamma)$, where $\gamma=t_{1}+4\geq 5$ by Lemma \ref{lem:line}. We have $E^{2}=-\gamma$ by Lemma \ref{lem:HN-equations}, see Figure \ref{fig:FE}. Note that the order of cusps $q_{1},q_{2}\in \bar{E}$ in Figure \ref{fig:FE} is different than the one in Definition \ref{def:our_curves}.
\end{proof}

\begin{figure}[htbp]
	\begin{tabular}{c c c}
		\multirow{2}{*}{
			\begin{subfigure}{0.35\textwidth}\centering
				\vspace{-2cm}
				\includegraphics[scale=0.25]{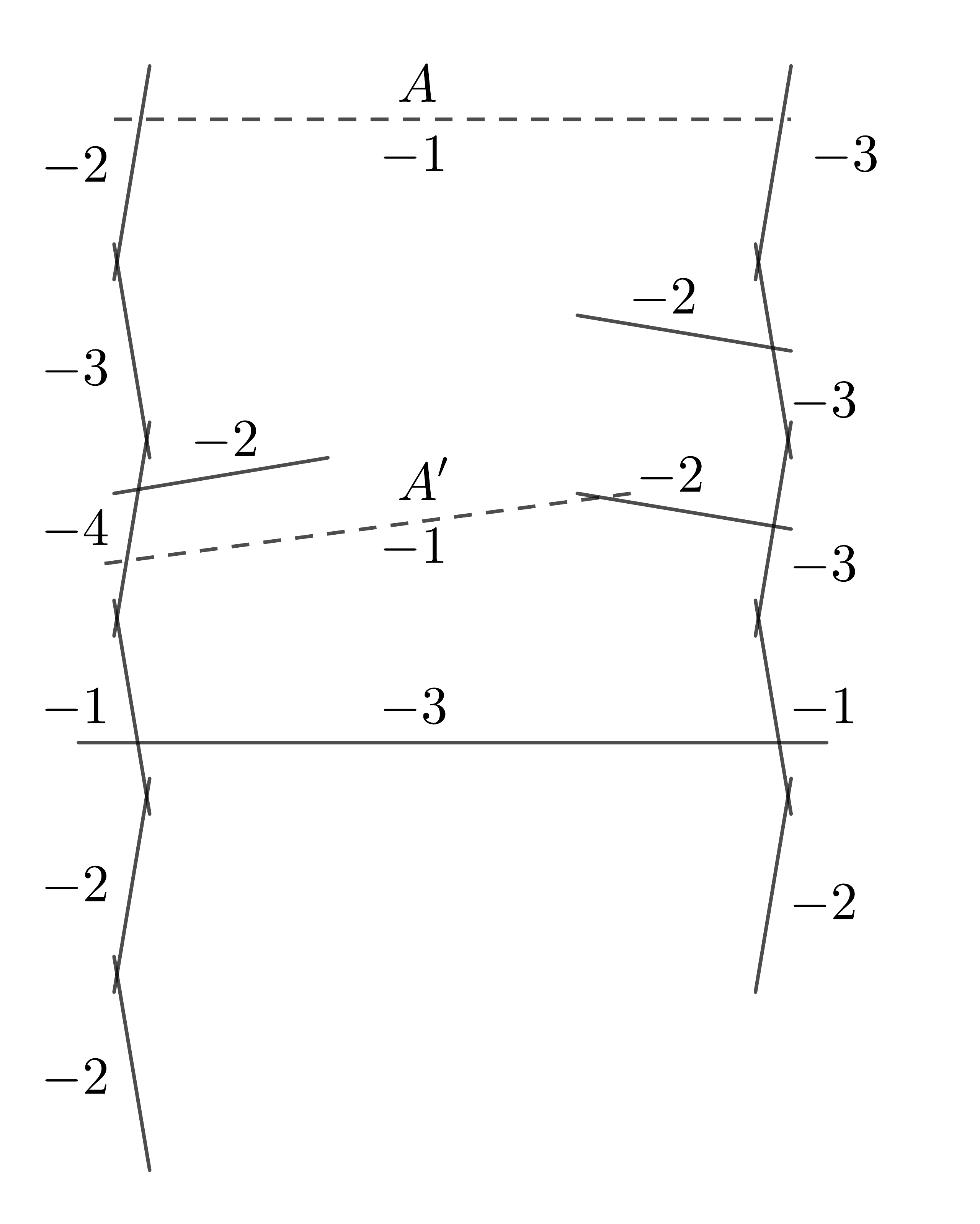}
				\vspace{-1cm}
				\caption{$(X,D)$}
				\vspace{.5cm}
				\begin{flushright} \begin{fmpage}{.5\textwidth} \begin{equation*}\begin{split}
						\boldsymbol{\circ}& = \mbox{ image of } A\\
						\boldsymbol{\diamond}&= \mbox{ image of } A'
						\end{split}\end{equation*} \end{fmpage} \end{flushright}	
			\end{subfigure}
		}
		&
		$\xrightarrow{\quad \displaystyle{ \psi\circ\psi_{0}}\quad }$
		& 
		\begin{subfigure}{0.35\textwidth}\centering
			\includegraphics[scale=0.22]{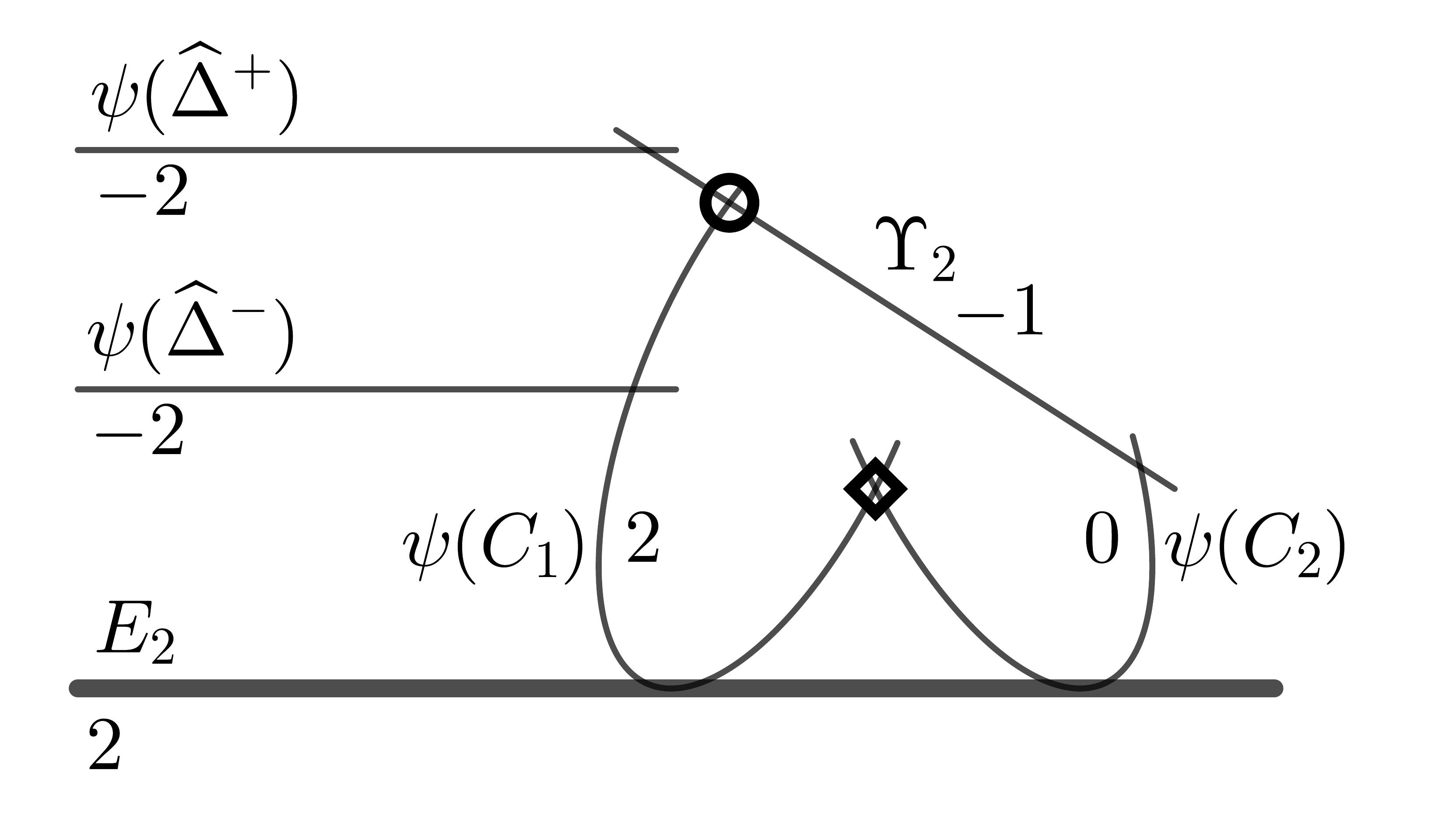}
			\caption{$(X_{2},D_{2})$}
		\end{subfigure}
		\vspace{0.5cm}
		\\
		&&
		$
		\Bigg\downarrow \alpha_{2}^{+}
		$
		\\
		&&
		\begin{subfigure}{0.35\textwidth}\centering
			\includegraphics[scale=0.22]{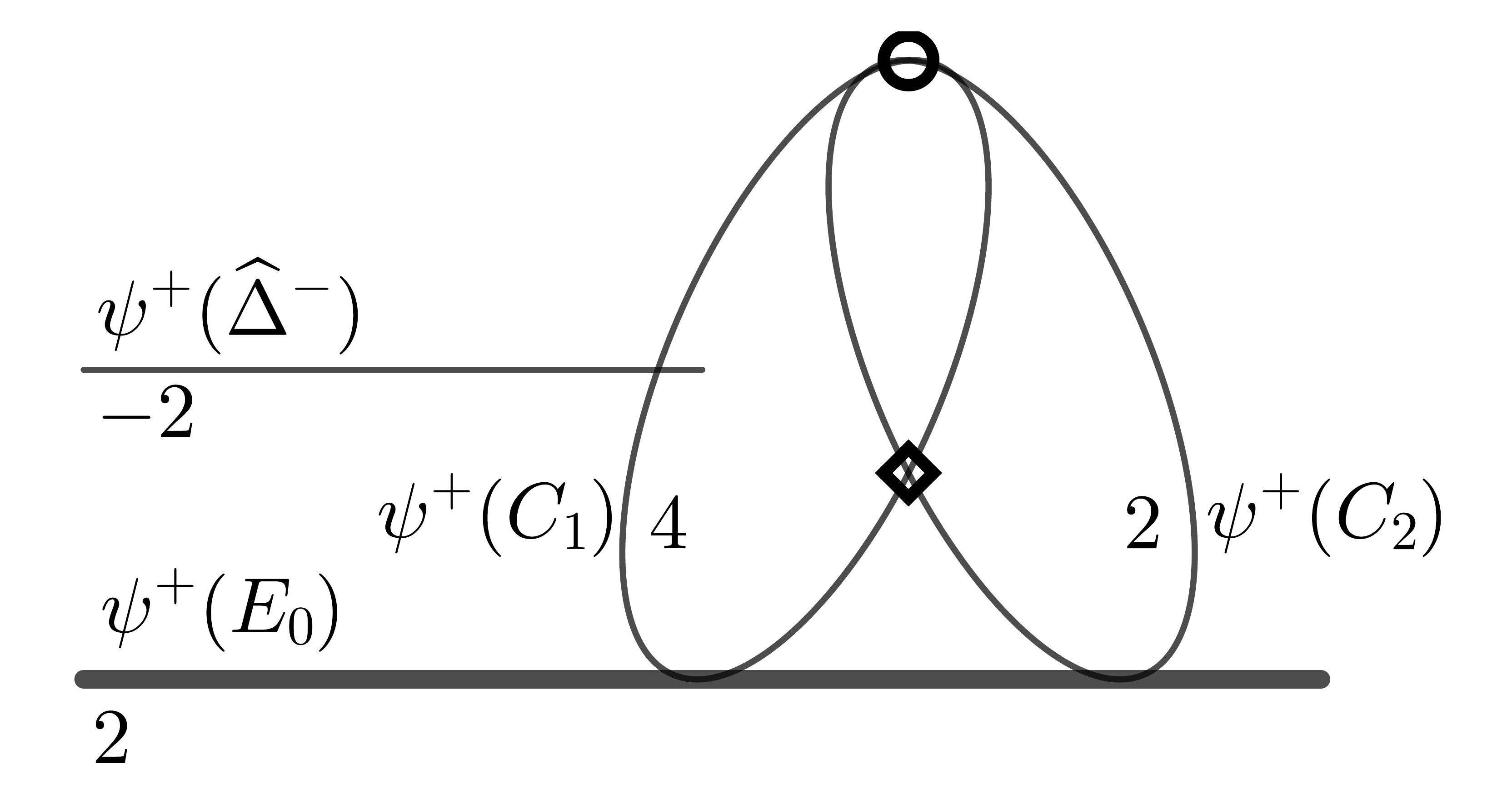}
			\caption{$(Z,D_{Z})$}
		\end{subfigure}
	\end{tabular}
	\caption{Type $\cI$.}
	\label{fig:I}
\end{figure}

\begin{prop}\label{prop:F2_I}
	If $D_{Z}$ is horizontal then $\bar{E}$ is of type $\cI$ (see Figure \ref{fig:I}).
\end{prop}
\begin{proof}
	Lemma \ref{lem:F2_core}\ref{item:F2_Z=F2},\ref{item:F2_hor} shows that $D_{Z}$ consists of four $1$-sections. In particular, $\psi^{+}(E_{0})$ is smooth, so $c=2$ by Lemma \ref{lem:F2_core}\ref{item:F2_semiord} and Remark \ref{rem:semi-ordinary}\ref{item:semiordinary}. Since $\psi^{+}$ does not touch $\hat{\Delta}^{-}$, by Lemma \ref{lem:F2_core}\ref{item:F2_Q1} we have $\psi^{+}(C_{1})\cdot \psi^{+}(\hat{\Delta}^{-})=1$ and $\psi^{+}(C_{2})\cdot\psi^{+}(\hat{\Delta}^{-})=\psi^{+}(E_{0})\cdot\psi^{+}(\hat{\Delta}^{-})=0$. Using elementary numerical properties of Hirzebruch surfaces, we compute that 
	\begin{equation*}
	\psi^{+}(C_{1})^{2}=4,\ \ \psi^{+}(C_{1})\cdot\psi^{+}(C_{2})=\psi^{+}(C_{1})\cdot\psi^{+}(E_{0})=3\ \ \text{and}\ \ \psi^{+}(C_{2})\cdot\psi^{+}(E_{0})=2.
	\end{equation*}
	If $\psi^{+}$ touches $E_{0}$ then we infer that $\tau_{1}<3$ or $\tau_{2}<2$ and that, by Lemma \ref{lem:beta_flat}\ref{item:centers_on_En}, $A'$ meets $E_{0}$. This is impossible by Lemma \ref{lem:line}\ref{item:tangent}. Hence, $\tau_{1}=2$, $\tau_{2}=3$ and  
	both centers of $\psi^{+}$ are contained in $\psi^{+}(C_{1})+\psi^{+}(C_{2})$. We have $\psi^{+}(C_{1})\cdot\psi^{+}(C_{2})=3$ and $C_{1}\cdot C_{2}=0$, so from Lemma \ref{lem:MMP-properties}\ref{item:center-psi_i} we infer that $\psi^{+}(C_{1})$ and $\psi^{+}(C_{2})$ meet in two points, with multiplicity $1$ and $2$. The latter point is the image of  $\Upsilon_{2}+\Delta_{2}^{+}=[1,2]$. Suppose that $Q_{2}$ is a chain. By Lemma \ref{lem:F2_core}\ref{item:F2_Q1},\ref{item:F2_Q2},  $\psi^{-1}_{*}(\Upsilon_{2}+\Delta_{2}^{+})=W=[(2)_{t_{1}}]$, so $t_{1}=2$. Then $Q_{1}=[2,2,3,1,2]$, $Q_{2}=[4,1,2,2]$ and $A'$ meets $D_{0}$ only in $C_{1}$ and in the third component of $Q_{2}$. We compute that $\pi_{0}(A')^{2}=-1+(2+3\cdot 2^{2})+(2+2^{2})=19$; a contradiction. Thus $Q_{2}$ is a fork as in Lemma \ref{lem:F2_core}\ref{item:F2_Q2}. We have $\hat{\Delta}^{+}=\psi^{-1}_{*}\Delta_{2}^{+}=[2]$, so $t_{1}=1$. Now $\psi(B_{2})\subseteq \Upsilon_{2}$ and $\psi_{A}$ touches $B_{2}$ twice by Lemma \ref{lem:F2_core}\ref{item:F2_Exc_psi_A}. Since $\psi_{A'}$ does not touch $B_{2}$ by Lemma \ref{lem:beta_flat}\ref{item:Ups},  $B_{2}^{2}=\psi(B_{2})^{2}-2=-3$. Because $B_{2}+C_{2}+W$ contracts to a smooth point, we get $W=[2]$. This determines the weighted graph of $Q_{1}+Q_{2}$. Since $\tau_{1}=3$, $\tau_{2}=2$ and $s_{1}=s_{2}=1$, we see that $\bar{E}$ is of type $\cI$. Note that $E^{2}=-3$ by Lemma \ref{lem:HN-equations}, see Figure \ref{fig:I}.
\end{proof}

\subsection{Types with a smooth minimal model.}\label{sec:P2}
In this section we prove the following result on types with a smooth minimal model.

\begin{prop}\label{prop:P2}
	If $X_{\min}$ is smooth then $\bar{E}$ is of type $\Qb$, $\Qa$, $\FZb$, $\cH$ or $\cJ$.
\end{prop}	
	
Throughout this section, we assume that $X_{\min}$ is smooth. The peeling morphism \eqref{eq:peeling} between the almost minimal model $(X_{n},\tfrac{1}{2}D_{n})$ and the minimal model $(X_{\min},\tfrac{1}{2}D_{\min})$ is now just the contraction of $\Upsilon_{n}+\Delta_{n}^{+}$. Equivalently, $\Delta_{n}^{-}=0$.

We collect some properties of $D_{\min}$ which are consequences of \eqref{eq:assumption}.

\begin{lem}[The geometry of $D_{\min}$ in case of a smooth minimal model]\label{lem:P2_lem}
Let $\psi^{+}$ and $D_{\min}$ be as above (see \eqref{eq:theta}).  Then:
\begin{enumerate}
\item\label{item:P2-deg} $X_{\min}\cong \P^{2}$ and $\deg D_{\min}=5$.
\item\label{item:P2-s_j} $s_{j}=1$ for $j\in\{1,\dots, c\}$.
\item\label{item:P2-semiordinary} The only singularities of $\psi^{+}(E_{0})$ are images of semi-ordinary cusps of $\bar{E}$ through $\psi^{+}\circ \pi_{0}^{-1}$, which is an isomorphism on some neighborhood of those cusps.
\item\label{item:P2-smooth} Components of $D_{\min}-\psi^{+}(E_{0})$ are smooth.
\item\label{item:P2_nodes} Components of $D_{\min}$ meet in $2n$ points, in each point exactly two of them.
\item\label{item:P2-n} $n=\#D_{\min}-1\in \{0,2,3\}$
\item\label{item:P2_deg-E} If $n\neq 0$ then $\deg\psi^{+}(E_{0})\leq 3$.
\item\label{item:P2-conics} If $D_{\min}$ contains two conics then they meet in two points, with multiplicities $3$ and $1$. The remaining part of $D_{\min}$ is a line tangent to those conics off their common points.
\end{enumerate}
\end{lem}
\begin{proof}
\ref{item:P2-deg},\ref{item:P2-s_j} Proposition \ref{prop:MMP}\ref{item:MMP_MFS} and \eqref{eq:assumption} imply that $(X_{\min},\tfrac{1}{2}D_{\min})$ is log del Pezzo surface of Picard rank one. Since $X_{\min}$ is smooth, $X_{\min}\cong \P^{2}$. Since $-(2K_{X_{\min}}+D_{\min})$ is ample, $0>\deg (2K_{X_{\min}}+D_{\min})=-6+\deg D_{\min}$, that is, $\deg D_{\min}\leq 5$. Suppose the inequality is strict or that $s_{j}= 0$ for some $j\in \{1,\dots, c\}$. Let $\mu_{p}$ be the multiplicity of a point $p\in D_{\min}$. If $\deg D_{\min}\in\{3,4\}$ then, because the components of $D_{\min}$ are rational, we can choose $p$ with $\mu_{p}\geq 2$. If $s_{j}=0$ then we have $\mu_{p}\geq 3$ for $\{p\}=\psi^{+}(C_{j}\cap E_{0} \cap \tilde{C}_{j})$. In any case, we have $\deg D_{\min}-\mu_{p}\leq 2$, so the pencil of lines through $p$ induces a $\C^{1}$-, $\C^{*}$- or a $\C^{**}$-fibration of $X_{\min}\setminus D_{\min}$. The latter is isomorphic to an open subset of $\P^{2}\setminus \bar{E}$. This is a contradiction with \eqref{eq:assumption}.

\ref{item:P2-semiordinary},\ref{item:P2-smooth} Let $G$ be a component of $D_{0}$ such that $\psi^{+}(G)$ is singular. Because $\psi(G)$ is smooth, $\Sing \psi^{+}(G)$ consists of images of those components of $\Upsilon_{n}$ which meet  $R_{n}$ only in $\psi(G)$. Let $U$ be a proper transform on $X$ of such a component. If $U\subseteq \Upsilon_{0}$ then by Remark \ref{rem:semi-ordinary}\ref{item:semiordinary}, $G=E_{0}$ and $\psi^{+}(U)$ is an image of a semi-ordinary cusp of $\bar{E}$. Assume $U\not\subseteq \Upsilon_{0}$. We infer from Lemma \ref{lem:beta_flat}\ref{item:Ups} that one of the points of $\psi(G)\cap\psi(U)$ is the image of the point $G\cap U$, and the other is the center of $\psi$ contained in $\psi(U)$. By Lemma \ref{lem:MMP-properties}\ref{item:Exc-psi_i}, the preimage of the latter point is a chain $\rev{T_{U}}+A+T_{G}$, where $A$ is the proper transform of some $A_{i}$ and $T_{U}$, $T_{G}$ are zero or twigs of $D_{0}$ meeting $U$ and $G$, respectively.
Since $\psi(U)\subseteq \Upsilon_{n}$, either $\beta_{D_{0}}(U)=2$, or $U$ meets a connected component of $\psi^{-1}_{*}\Delta_{n}^{+}$, say $\Delta_{U}$. By Lemma \ref{lem:MMP-properties}\ref{item:Delta_pr-tr}, $\Delta_{U}\subseteq \Delta_{0}$. If $G=E_{0}$ then, since $E_{0}$ meets no twigs of $D_{0}$, we have $T_{E_{0}}=0$, so $0\neq T_{U}\subseteq \Delta_{0}$, and $\beta_{D_{0}}(U)\geq 3$. But then $U$ meets two $(-2)$-twigs of $D_{0}$, which is impossible. Hence, $G\neq E_{0}$. In particular, we proved \ref{item:P2-semiordinary}.

Suppose that $D_{0}-E_{0}-G$ contains a component $V$ not contracted by $\psi^{+}$. Because $\psi^{+}(G)$ is singular, $\deg\psi^{+}(G)\geq 3$. We infer from \ref{item:P2-deg} that $\#D_{\min}=3$, $\deg\psi^{+}(G)=3$ and that  $\psi^{+}(E_{0})$ and $\psi^{+}(V)$ are lines. Hence, $V\cdot E_{0}\leq \psi^{+}(V)\cdot\psi^{+}(E_{0})=1$, so $V\neq C_{j}$ for $j\in\{1,\dots, c\}$. Part \ref{item:P2-s_j} implies that $\psi^{+}(V)\cap\psi^{+}(E_{0})$ is a center of $\psi^{+}$, other than $\Sing \psi^{+}(G)$. Moreover, since $\psi^{+}(V)\cdot \psi^{+}(G)=3>V\cdot G$, the set $\psi^{+}(V)\cap\psi^{+}(G)$ contains another center of $\psi^{+}$. Hence, $n\geq 3$. But $n=\#D_{\min}-1=2$ by Lemma \ref{lem:beta_flat}\ref{item:Rn=n+1}; a contradiction.

Thus, $G=C_{1}$ and $D_{\min}=\psi^{+}(E_{0})+\psi^{+}(C_{1})$. Lemma \ref{lem:beta_flat}\ref{item:Rn=n+1} gives $n=1$,  hence
\begin{equation*}
D_{1}=E_{1}+\psi(C_{1})+\psi(U)+\psi_{*}\Delta_{U}+\psi_{*}(\Upsilon_{0}+\Delta_{0}^{+}).
\end{equation*}
Put $b=\#\Delta_{U}$. We have $b\geq 1$, because otherwise  $Q_{1}=T_{C_{1}}+C_{1}+U+\rev{T_{U}}$ is a chain and $A$ meets it in tips, contrary to Lemma \ref{lem:orevkov_ending}. Because $Q_{1}$ contracts to a smooth point, we have $T_{C_{1}}=[(2)_{-U^{2}-2}]\subseteq \Delta_{0}$ and either $T_{U}=0$ or $T_{U}=[(2)_{t_{1}},b+2]$. Hence, 
\begin{equation*} 
\Exc\psi=\rev{T_{U}}+A+T_{C_{1}}=[1,(2)_{-U^{2}-2}]{\ \ \text{or}\ \ } [b+2,(2)_{t_{1}},1,(2)_{-U^{2}-2}].
\end{equation*}
In both cases, $\psi$ touches $C_{1}$ at most twice, so $\psi(C_{1})^{2}\leq C_{1}^{2}+2=1$.

The contraction of each component of $\psi_{*}(U+\Delta_{U})$ increases the arithmetic genus (respectively, the self-intersection number) of the image of $\psi(C_{1})$ by $1$ (respectively, by $4$). Hence,
\begin{equation*}
b+1=p_{a}(\psi^{+}(C_{1}))=\tfrac{1}{2}\deg(\psi^{+}(C_{1}))(\deg\psi^{+}(C_{1})-3)+1\quad\mbox{and}\quad 4(b+1)=(\deg\psi^{+}(C_{1}))^{2}-(\psi(C_{1}))^{2}.
\end{equation*}
Because $b\geq 1$, the first equation gives $\deg\psi^{+}(C_{1})\geq 4$, so by \ref{item:P2-deg}, $\deg\psi^{+}(C_{1})=4$ and $b=2$. Now the second equation gives $(\psi(C_{1}))^{2}=4$; a contradiction.

\ref{item:P2_nodes} For a reduced effective divisor $V$ denote by $\nu(V)$ the number of points where exactly two components of $V$ meet. Let $R$ be the proper transform of $D_{\min}$ on $X_{0}$. We have $R=\psi^{-1}_{*}R_{n}$ (see \eqref{eq:peeling}), so by Lemma \ref{lem:beta_flat}\ref{item:Rn_connected}, $R$ is a connected subdivisor of $R_{0}$. Hence the graph of $R$ has no loops, and it follows from \ref{item:P2-s_j} that no three components of $R$ meet at the same point. As a consequence, $\nu(R)=\#R-1$. From Lemma \ref{lem:beta_flat}\ref{item:Rn=n+1} we obtain $\nu(R)=n$.

Parts \ref{item:P2-semiordinary}, \ref{item:P2-smooth} imply that the components of $D_{\min}$ have no singularities but cusps. Hence by Lemma \ref{lem:beta_flat}\ref{item:centers_psi+} the centers of $\psi^{+}$ are the cusps of $\psi^{+}(E_{0})$ and some $n$ points, at each of which exactly two components of $D_{\min}$ meet. Therefore, $\nu(D_{\min})=\nu(R)+n=2n$.

\ref{item:P2-n} We have $n=\#D_{\min}-1$ by Lemma \ref{lem:beta_flat}\ref{item:Rn=n+1}. If $n\geq 4$ then by \ref{item:P2-deg},  $n=4$ and $D_{\min}$ is a union of five lines, which by \ref{item:P2_nodes} meet in eight points, exactly two at each point. This is impossible, hence $n\leq 3$. Suppose that $n=1$. Then $D_{\min}=\psi^{+}(E_{0})+\psi^{+}(C_{1})$ and $\psi^{+}(C_{1})$ is smooth by \ref{item:P2-smooth}. The point $\psi^{+}(A_{1})$ is a singular point of $D_{\min}$ other than $\psi^{+}(C_{1}\cap E_{0})$ and the cusps of $\psi^{+}(E_{0})$, so it is a common point of $\psi^{+}(C_{1})$ and $\psi^{+}(E_{0})$. Lemma \ref{lem:beta_flat}\ref{item:centers_on_En} implies that $\Exc\psi^{+}=A_{1}+\Delta_{A_{1}}$, where $\Delta_{A_{1}}$ is a $(-2)$-twig of $D_{0}$ meting $C_{1}$. It follows that $Q_{1}=C_{1}+\Delta_{A_{1}}$, so by Lemma \ref{lem:notation}\ref{item:T^0=C},  $\Delta_{A_{1}}=\Delta_{T_{1}}=[(2)_{t_{1}}]$ and $q_{1}\in \bar{E}$ has multiplicity sequence $(\tau_{1})_{t_{1}+1}$. We have $\tau_{1}\geq 3$, because $q_{1}\in \bar{E}$ is not semi-ordinary. Moreover, $\pi_{0}(A_{1})^{2}\geq 1$, so $A_{1}$ meets $\ltip{\Delta_{T_{1}}}$ and $t_{1}\geq 2$. We obtain  $\psi^{+}(E_{0})\cdot\psi^{+}(C_{1})=t_{1}+\tau_{1}\geq 5$. But $\deg\psi^{+}(E_{0})+\deg\psi^{+}(C_{1})=5$ by \ref{item:P2-deg} and $\deg\psi^{+}(C_{1})\leq 2$ by \ref{item:P2-smooth}, so $\deg\psi^{+}(E_{0})=3$, $\deg \psi^{+}(C_{1})=2$ and hence $t_{1}+\tau_{1}=6$. Now  $(\deg\pi_{0}(A_{1}))^{2}=A_{1}^{2}+t_{1}=t_{1}-1\leq 2$, so $t_{1}=2$ and $\tau_{1}=4$. Thus, $q_{1}\in \bar{E}$ has multiplicity sequence $(4)_{3}$, see Lemma \ref{lem:notation}\ref{item:T^0=C}. Because $\psi^{+}(E_{0})$ has one ordinary cusp, $c=2$ and $q_{2}\in \bar{E}$ is ordinary. This is a contradiction with Lemma \ref{lem:HN-equations}\eqref{eq:genus-degree}.

\ref{item:P2_deg-E} Parts \ref{item:P2-deg},\ref{item:P2-n} give $\deg\psi^{+}(E_{0})\leq \deg D_{\min}-(\#D_{\min}-1)=5-n\leq 3$.

\ref{item:P2-conics} Let $G_{1}$, $G_{2}$ be two conics contained in $D_{\min}$. The pencil generated by them gives a map $g\colon \P^{2}\map \P^{1}$. Put $\ll=D_{\min}-G_{1}-G_{2}$. Part \ref{item:P2-deg} gives $\deg\ll=5-2-2=1$, so $\ll$ is a line.  By \ref{item:P2_nodes}, $\ll$ meets $G_{1}$, $G_{2}$ in different points. We have $n=2$ by \ref{item:P2-n}, so by \ref{item:P2_nodes}, the components of $D_{\min}$ meet in four points. Hence, $\#G_{1}\cap G_{2}=4-\#\ll\cap G_{1}-\#\ll\cap G_{2}\leq 2$. In fact, the equality holds, because otherwise $g|_{X_{\min}\setminus D_{\min}}$ is a $\C^{**}$-fibration, which is impossible by \eqref{eq:assumption}. Hence,  $G_{1}\cap G_{2}=\{p_{1},p_{2}\}$ for some $p_{1}\neq p_{2}$ and $\ll$ is tangent to $G_{1}$, $G_{2}$ off $p_{1},p_{2}$. Let $\ll_{12}$ be the line joining $p_{1}$ with $p_{2}$. The morphism $g|_{\ll}\colon \ll\to \P^{1}$ has degree $2$. We have $(G_{1}\cdot G_{2})_{p_{i}}\neq 2$ for some $i\in \{1,2\}$, for otherwise  $g|_{\ll}$ is ramified at $\ll\cap G_{1}$, $\ll\cap G_{2}$ and $\ll\cap\ll_{12}$, contrary to the Hurwitz formula. The result follows.
\end{proof}

We will now consider the cases $n=0,2,3$ separately.

\begin{prop}\label{prop:n0}
	If $n=0$ then $\bar{E}$ is of type $\Qb$ or $\Qa$ (see Figures \ref{fig:Qa}--\ref{fig:Qb}).
\end{prop}
\begin{proof}
	By Lemma \ref{lem:beta_flat}\ref{item:n0} we have $D_{0}-E_{0}=\Upsilon_{0}+\Delta_{0}^{+}$, so  $(X_{\min},D_{\min})=(\P^{2},\bar{E})$. Lemma \ref{lem:P2_lem}\ref{item:P2-deg} implies that $\bar{E}$ is a quintic with only semi-ordinary cusps. Thus every $q_{j}\in \bar{E}$, $j\in\{1,\dots, c\}$, has multiplicity sequence $(2)_{t_{j}+1}$ for some $t_{j}\geq 0$  (see Example \ref{ex:ordinary_cusp}). Lemma \ref{lem:HN-equations}\eqref{eq:genus-degree} in this case reads as
	\begin{equation}\label{eq:n0}
	\sum_{j=1}^{c}(t_{j}+1)=6.
	\end{equation}
	
We claim that $t_{j}\leq 2$ for every $j\in \{1,\dots,c\}$. Suppose that $t_{1}\geq 3$. Let $\ll_{1}\subseteq \P^{2}$ be the line tangent to $\bar{E}$ at $q_{1}\in \bar{E}$. The number $(\ll_{1}\cdot\bar{E})_{q_{1}}$ is the sum of at least two initial terms of the multiplicity sequence of $q_{1}\in \bar{E}$, and since $(\ll\cdot\bar{E})_{q_{1}}\leq \deg\bar{E}=5$, we get $(\ll_{1}\cdot \bar{E})_{q_{1}}=4$. The proper transform $L_{1}$ of $\ll_{1}$ on $X_{0}$ satisfies $L_{1}\cdot Q_{1}=1$, so it meets the unique component of $\pi_{0}^{-1}(q_{1})$ of multiplicity $4$, that is, the second component of $Q_{1}$. Moreover, $L_{1}^{2}=-1$, $L_{1}\cdot D_{0}=2$ and $L$ meets $D_{0}$ in $Q_{1}$ and $E_{0}$. The sum of $L_{1}$ and the first three exceptional curves over $q_{1}$ supports a fiber of a $\P^{1}$-fibration which restricts to a $\C^{**}$-fibration of $\P^{2}\setminus \bar{E}$; a contradiction with \eqref{eq:assumption}.
	
	Assume that $t_{j}=0$ for all $j\geq 2$. Let $\sigma$ be a blowup at $q_{1}$. Then the pencil of lines through $q_{1}\in\bar{E}$ induces a morphism $\sigma^{-1}_{*}\bar{E}\to \P^{1}$ of degree $\deg\bar{E}-2=3$, ramified at every $\sigma^{-1}(q_{j})$, $j\geq 2$ and, if $t_{1}>0$, also at the point infinitely near to $q_{1}$ on $\sigma^{-1}_{*}\bar{E}$. Hence the Hurwitz formula implies that $c\leq 5$ and $c\leq 4$ if $t_{1}>0$. Since $t_{1}\leq 2$ and $t_{1}+c=6$ by \eqref{eq:n0}, it follows that $t_{1}=2$ and that $c=4$, so $\bar{E}$ is of type $\Qa$, see Figure \ref{fig:Qa}. Note that $E^{2}=-7$ by Lemma \ref{lem:HN-equations}.
	\begin{figure}[htbp]
		\centering
		\begin{minipage}[b]{0.45\textwidth}
			\centering
			\includegraphics[scale=0.4,trim={0 0 0 1cm},clip]{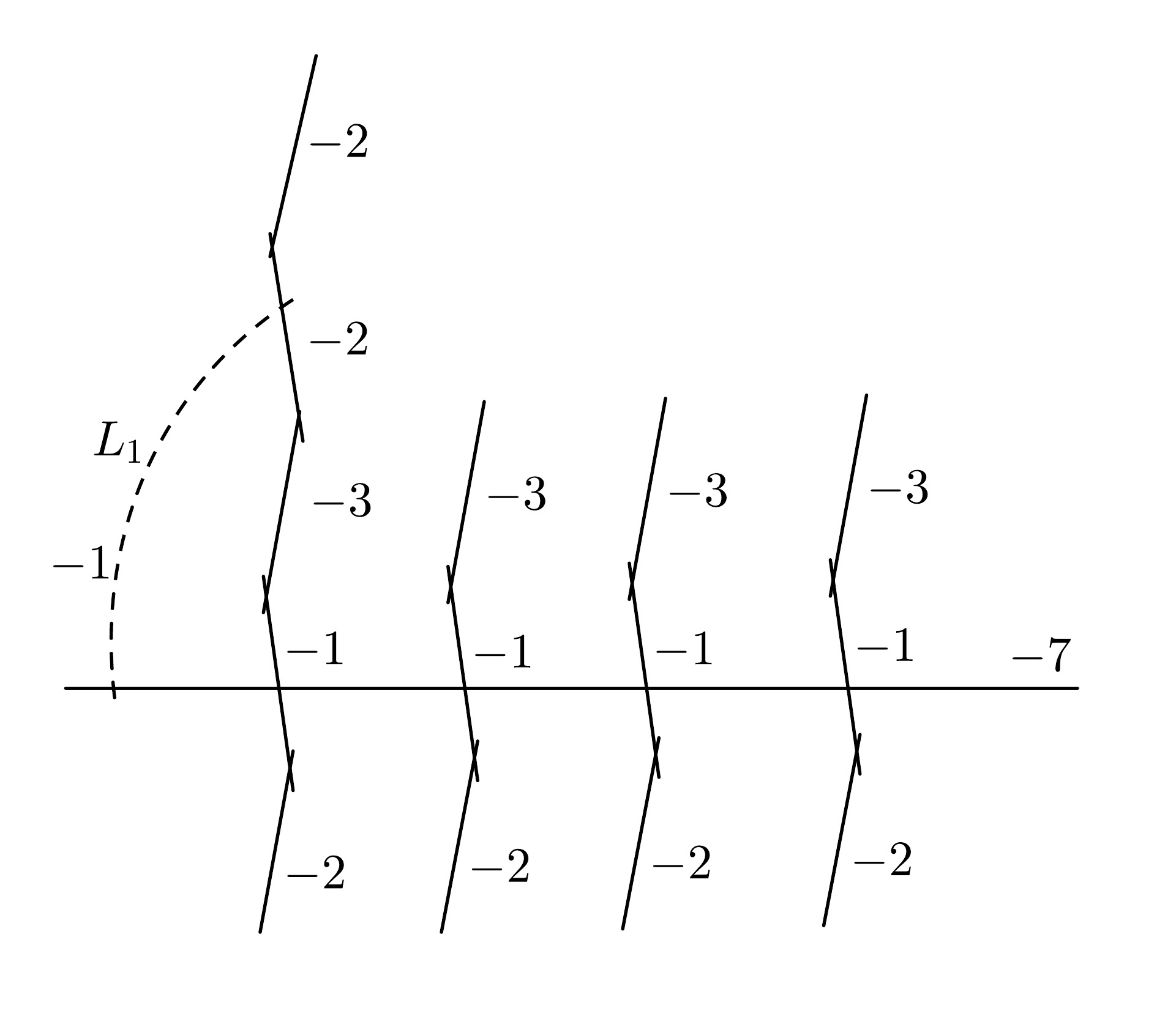}
			\caption{Type $\Qa$.}
			\label{fig:Qa}
		\end{minipage}\hfill
		\begin{minipage}[b]{0.45\textwidth}
			\centering
			\includegraphics[scale=0.4,trim={0 0 0 1cm},clip]{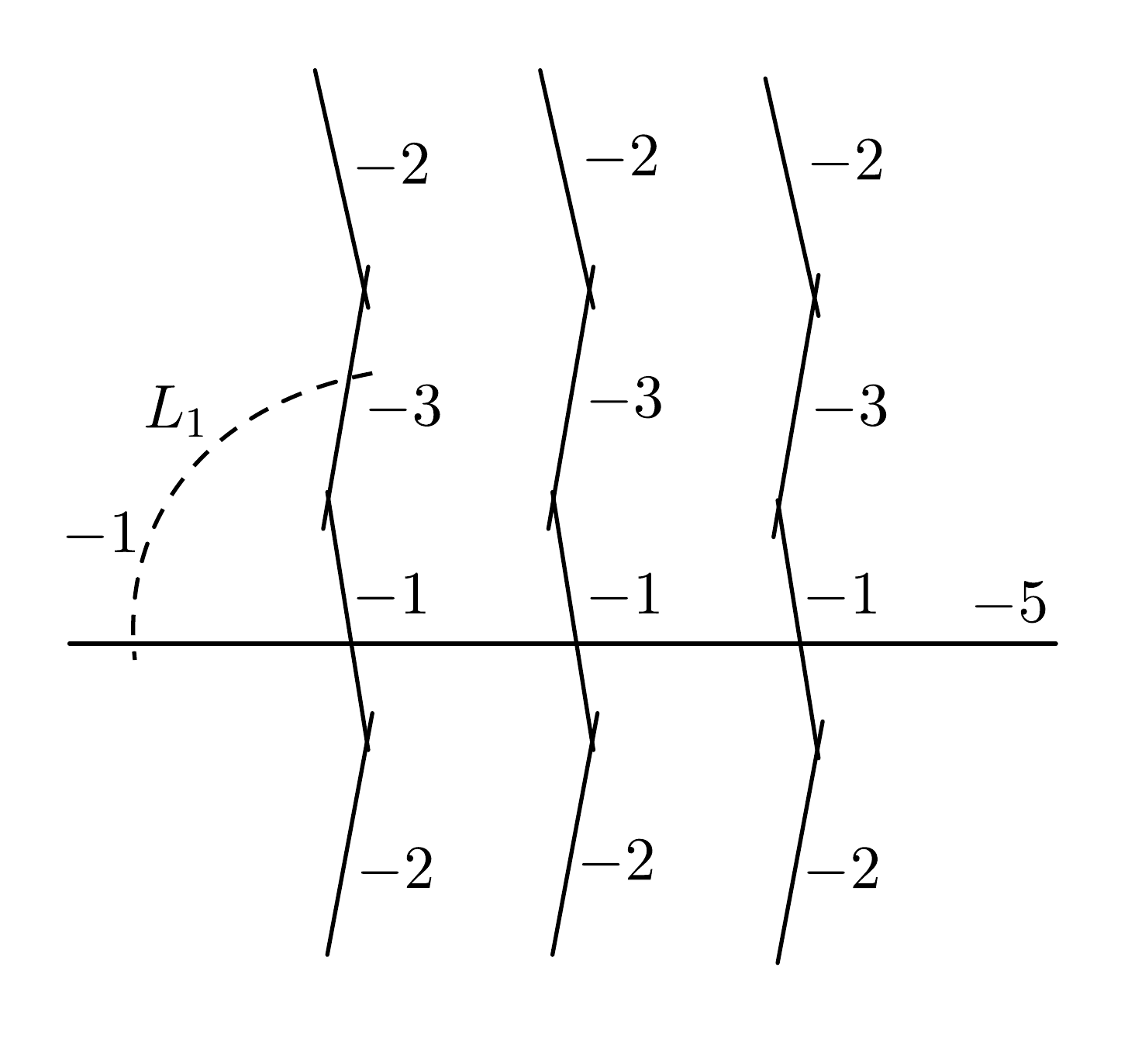}
			\caption{Type $\Qb$.}
			\label{fig:Qb}
		\end{minipage}
	\end{figure}
	
	Therefore, we may assume that at least two numbers $t_{j}$ are non-zero, say $t_{1}\geq t_{2}\geq 1$. Let $\ll_{1},\ll_{2}\subseteq \P^{2}$ be the lines  tangent to $\bar{E}$ at $q_{1}$, $q_{2}$ and let $\ll_{12}\subseteq \P^{2}$ be the line joining $q_{1}$ and $q_{2}$. For $j\in\{1,2\}$, the number $(\ll_{j}\cdot \bar{E})_{q_{j}}$ is the sum of at least two initial terms of the multiplicity sequence of $q_{j}\in \bar{E}$, so, since $\bar{E}$ is a quintic,  $(\ll_{j}\cdot \bar{E})_{q_{j}}= 4$. If $\ll_{j}=\ll_{12}$ then $\ll_{12}\cdot \bar{E}\geq 4+2=6>\deg\bar{E}$, which is false. Hence, $\ll_{1}$, $\ll_{2}$ and $\ll_{12}$ are distinct and $\{p\}\de \ll_{12}\cap\bar{E}\setminus \{q_{1},q_{2}\}$ is a point where $\ll_{12}$ and $\bar{E}$ meet transversally. Denote by $\sigma$ the blowup at $q_{1}$, $q_{2}$ and their infinitely near points on the proper transforms of $\bar{E}$. For $j\in\{1,2\}$ let $q_{j}'$ be the point infinitely near to $q_{j}$ on $\sigma^{-1}_{*}\bar{E}$. The pencil of conics generated by $\ll_{1}+\ll_{2}$ and $2\ll_{12}$ gives a morphism $\sigma^{-1}_{*}\bar{E}\to \P^{1}$  of degree $2\deg\bar{E}-(4+4)=2$, which is ramified at the preimages of $p,q_{3},\dots, q_{c}$ and at $q_{j}'$ for every $j\in \{1,2\}$ such that $t_{j}\geq 2$. By the Hurwitz formula, there are exactly two ramification points. If $t_{1}\geq 2$ then we get $c=2$ and $t_{2}=1$, so $t_{1}=4$ by \eqref{eq:n0}, contrary to our claim. Hence, $t_{1}=t_{2}=1$ and we get $c=3$. Now $t_{3}=1$ by \eqref{eq:n0}, so $\bar{E}$ is of type $\Qb$, see Figure \ref{fig:Qb}. Note that $E^{2}=-5$ by Lemma \ref{lem:HN-equations}.
\end{proof}

By Lemma \ref{lem:P2_lem}\ref{item:P2-n}, we are now left with the cases $n=2$ and $n=3$. By definition, the morphism $\psi$ contracts $n$ curves not contained in $D_{0}$. Denote them by $A$, $A'$ and, if $n=3$, by $A''$. By Lemma \ref{lem:ale_pr-tr}\ref{item:ale_pr-tr} they are almost log exceptional on $(X_{0},\tfrac{1}{2}D_{0})$. By Lemma \ref{lem:P2_lem}\ref{item:P2-semiordinary},\ref{item:P2_deg-E}, $\psi^{+}(E_{0})$ is either a cuspidal cubic or a conic or a line. In the following propositions we treat these cases separately.

				\begin{figure}[htbp]
					\begin{tabular}{c c c}
						\multirow{2}{*}{
							\begin{subfigure}{0.45\textwidth}\centering
								\vspace{-2.5cm}
								\includegraphics[scale=0.25,trim={1cm 0 0 0}, clip]{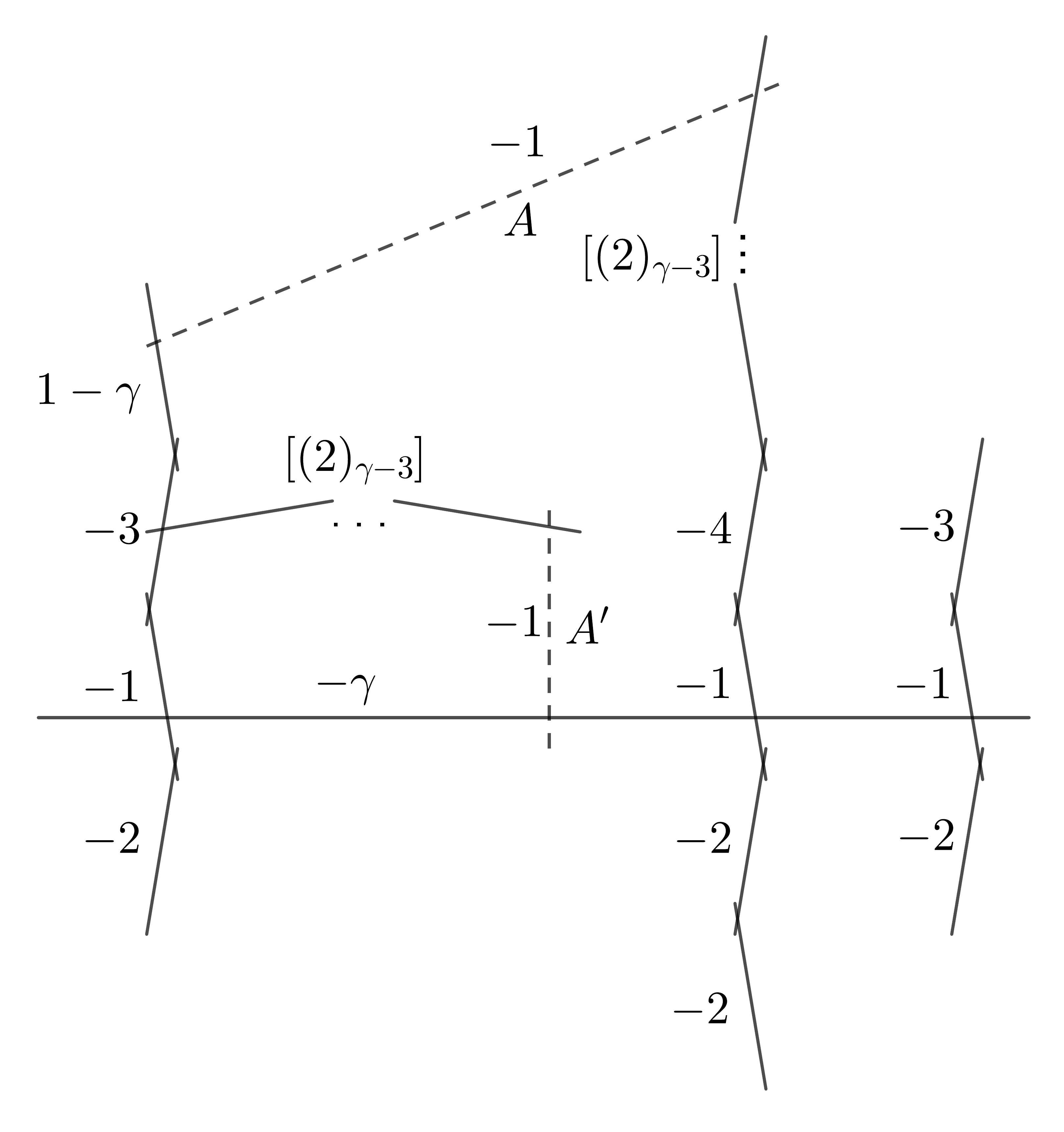}
								\vspace{-1cm}
								\caption{$(X,D)$}
								\vspace{.5cm}
								\begin{flushright} \begin{fmpage}{.5\textwidth} \begin{equation*}\begin{split}
										\boldsymbol{\circ}& = \mbox{ image of } A\\
										\boldsymbol{\diamond}&= \mbox{ image of } A'
										\end{split}\end{equation*} \end{fmpage} \end{flushright}					
							\end{subfigure}						
						}
						&
						$\xrightarrow{\quad \displaystyle{ \psi\circ\psi_{0}}\quad }$
						& 
						\begin{subfigure}{0.35\textwidth}\centering
							\includegraphics[scale=0.35,trim={1cm 0 0 0}, clip]{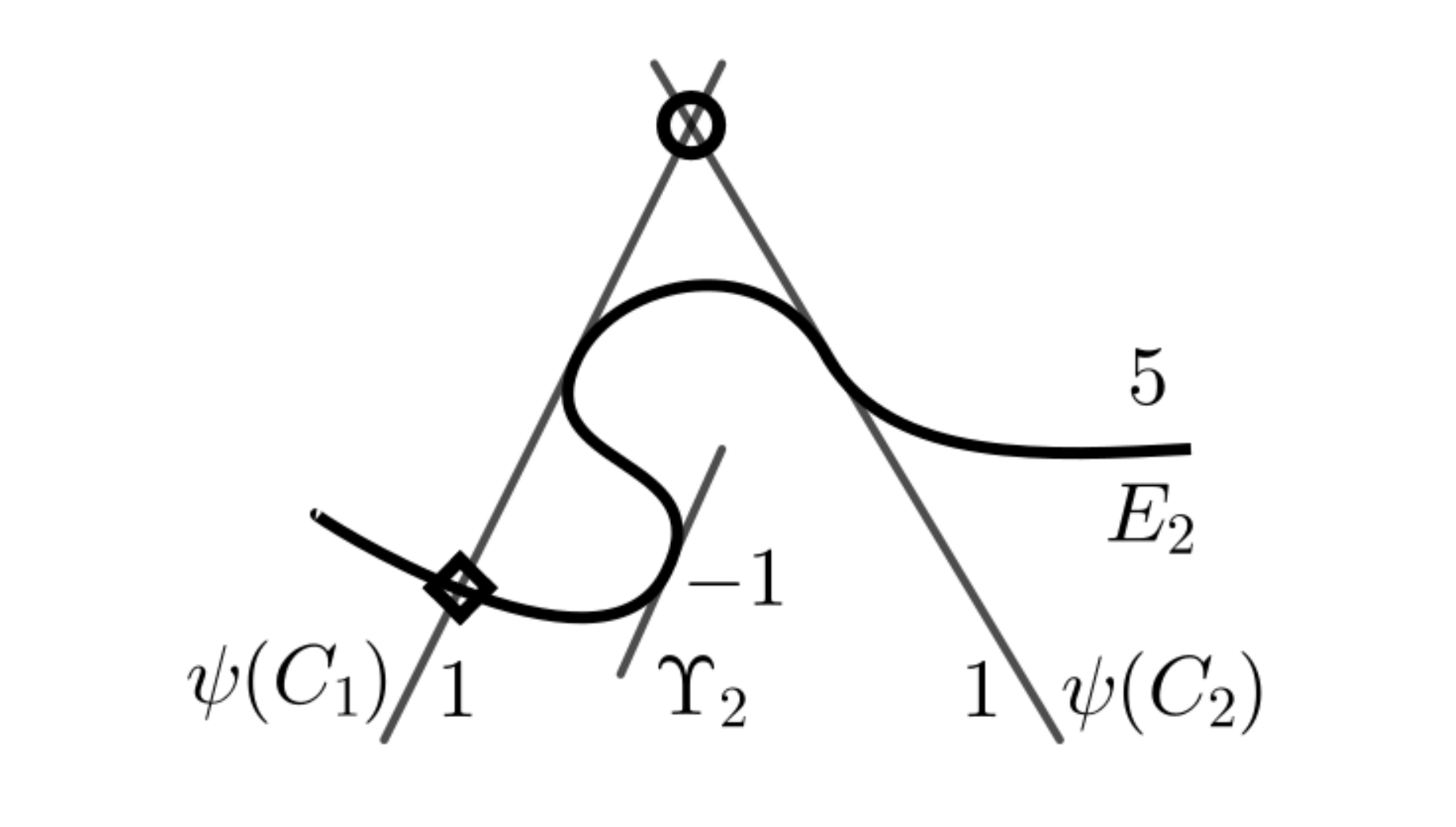}
							\caption{$(X_{2},D_{2})$}
						\end{subfigure}
						\vspace{0.5cm}
						\\
						&&
						$
						\Bigg\downarrow \alpha_{2}
						$
						\\
						&&
						\begin{subfigure}{0.35\textwidth}\centering
							\includegraphics[scale=0.35,trim={1cm 0 0 0}, clip]{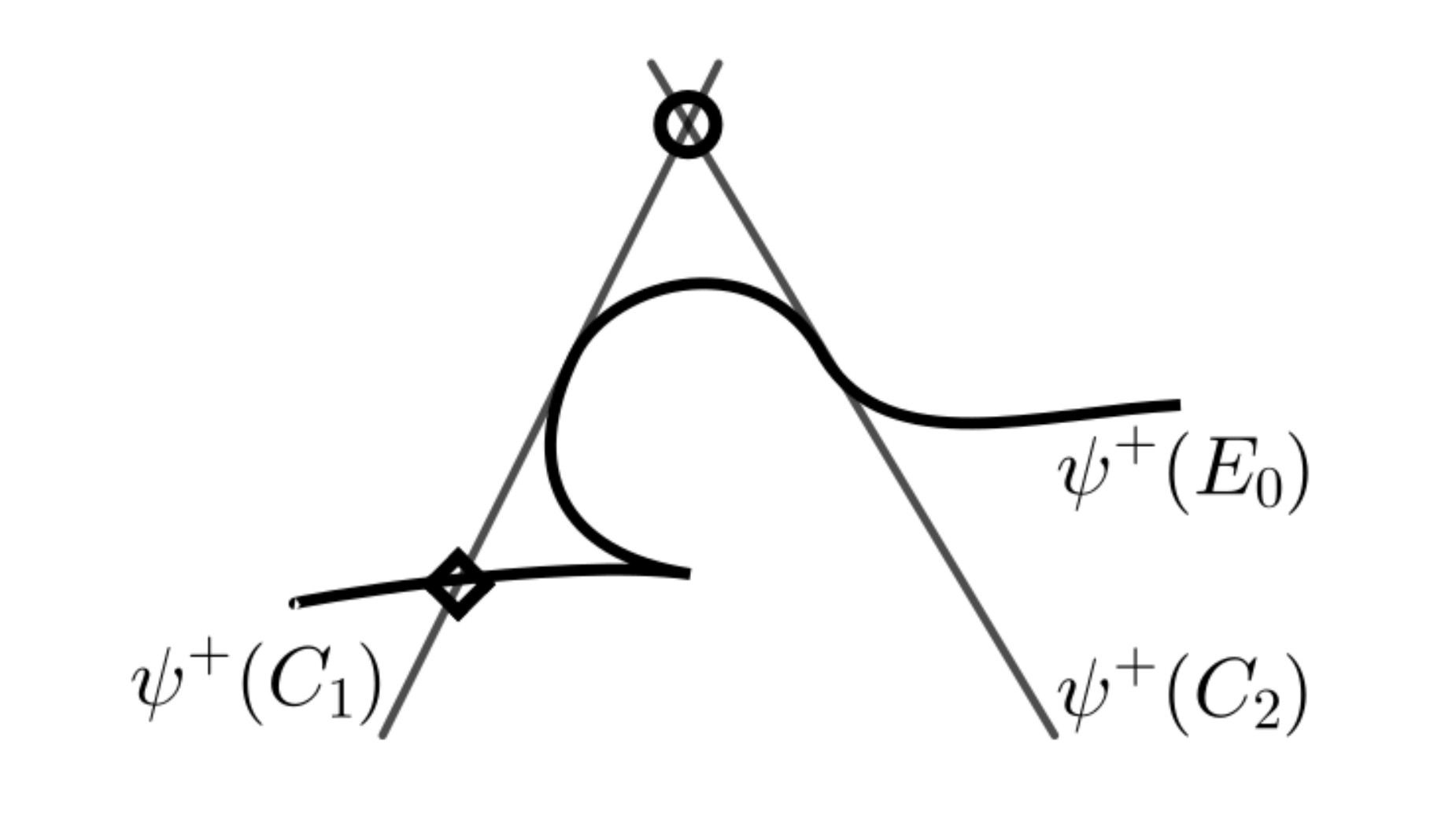}
							\caption{$(X_{\min},D_{\min})$}
						\end{subfigure}
					\end{tabular}
					\caption{Type $\FZb(\gamma)$, $\gamma\geq 4$.}
					\label{fig:FZb}
				\end{figure}			
\begin{prop}\label{prop:FZb} If $n=2$ and $\psi^{+}(E_{0})$ is a cubic then $\bar{E}$ is of type $\FZb$ (see Figure \ref{fig:FZb}).
\end{prop}
\begin{proof}
	Lemma \ref{lem:P2_lem}\ref{item:P2-deg},\ref{item:P2-n} implies that $D_{\min}-\psi^{+}(E_{0})$ is a union of two lines, say $\ll_{1}$, $\ll_{2}$. By Lemma \ref{lem:P2_lem}\ref{item:P2-semiordinary}, $\psi^{+}(E_{0})$ is a cuspidal cubic, so it has a unique singular point, which is an ordinary cusp. Lemma \ref{lem:P2_lem}\ref{item:P2_nodes} gives
	\begin{equation*}
	\#(\ll_{1}+\ll_{2})\cap\psi^{+}(E_{0})=2n-\#\ll_{1}\cap\ll_{2}=3,
	\end{equation*}
	so, say, $\#\ll_{1}\cap \psi^{+}(E_{0})=2$ and  $\#\ll_{2}\cap \psi^{+}(E_{0})=1$. Write $\ll_{1}\cap \psi^{+}(E_{0})=\{p_{1},p_{2}\}$, where  $(\ll_{1}\cdot\psi^{+}(E_{0}))_{p_{1}}=1$ and $(\ll_{1}\cdot\psi^{+}(E_{0}))_{p_{2}}=2$. We claim that 
	\begin{equation}\label{eq:FZb_claim}
		\ll_{1}=\psi^{+}(C_{1}),\quad  \ll_{2}=\psi^{+}(C_{2})\quad \mbox{and}\quad \psi^{+}(A)=\ll_{1}\cap\ll_{2},\quad \psi^{+}(A')=\{p_1\}.
	\end{equation}
	
	Lemma \ref{lem:beta_flat}\ref{item:centers_on_En} implies that $(\psi^{+})^{-1}_{*}\ll_{1}$ meets $E_{0}$, so by Lemma \ref{lem:P2_lem}\ref{item:P2-s_j} $(\psi^{+})^{-1}_{*}\ll_{1}=C_{j}$ for some $j\in \{1,\dots c\}$, say, $j=1$. The point $p_{1}$ is a center of $\psi^{+}$, since otherwise $C_{1}\cdot E_{0}=(\ll_{1}\cdot\psi^{+}(E_{0}))_{p_{1}}=1$, which is false. Say that $\psi^{+}(A')=\{p_{1}\}$. By Lemma \ref{lem:beta_flat}\ref{item:centers_on_En} 
	\begin{equation*}
	(\psi^{+})^{-1}(p_{1})=A'+\Delta_{A'},
	\end{equation*}
	
	where $\Delta_{A'}$ is a $(-2)$-twig meeting $C_{1}$. Since $(\ll_{1}\cdot \psi^{+}(E_{0}))_{p_{1}}=1$,  $A'$ meets $\ftip{\Delta_{A'}}$. It follows that $\psi_{A'}$ touches $C_{1}$ once. But $\psi^{+}(C_{1})^{2}\geq 1=C_{1}^{2}+ 2$, so $\psi^{+}$ touches $C_{1}$ at least twice. Hence, $\psi^{+}(A)\subseteq \psi^{+}(C_{1})$. As a consequence, $\ll_{1}\cap \ll_{2}=\psi^{+}(A)$, so, since $n=2$, $\ll_{2}\cap \psi^{+}(E_{0})$ is not a center of $\psi^{+}$. Thus $(\psi^{+})^{-1}_{*}\ll_{2}\cdot E_{0}=\ll_{2}\cdot \psi^{+}(E_{0})=3$, so  $(\psi^{+})^{-1}_{*}\ll_{2}=C_{2}$. This proves \eqref{eq:FZb_claim}. It follows that
	\begin{equation*}
	\tau_{1}=(\ll_{1}\cdot \psi^{+}(E_{0}))_{p_{2}}=2\quad\mbox{and}\quad\tau_{2}=\ll_{2}\cdot \psi^{+}(E_{0})=3. 
	\end{equation*}

	Since the components of $D_{\min}$ meet transversally at $\psi^{+}(A)$, $\psi^{+}(A')$, these points are not touched by $\alpha_{2}^{-1}$, hence the only center of $\alpha_{2}$ is the ordinary cusp of $\psi^{+}(E_{0})$. Thus 
	\begin{equation*}
	\Exc\psi=\Exc\psi^{+}-\Upsilon_{0}=A+A'+D_{0}-\Upsilon_{0}-(\psi^{+})^{-1}_{*}D_{\min}=A+A'+Q_{1}-C_{1}+Q_{2}-C_{2},
	\end{equation*}
	and hence 
	\begin{equation*}
	\Exc\psi_{A}=\Exc\psi-\Exc\psi_{A'}=(Q_{1}-C_{1}-\Delta_{A'})+A+(Q_{2}-C_{2}).
	\end{equation*}
	
	Lemma \ref{lem:MMP-properties}\ref{item:Exc-psi_i} implies now that $Q_{1}-C_{1}-\Delta_{A'}$ and $Q_{2}-C_{2}$ are zero or twigs of $D_{0}$. Because $Q_{1}$ meets $A$, the cusp $q_{1}\in \bar{E}$ is not semi-ordinary, so since $\tau_{1}=2$, by Lemma \ref{lem:notation}\ref{item:T^0=C} we have $Q_{1}-C_{1}\neq\Delta_{A'}$. Because $Q_{1}$ and $Q_{2}$ contract to smooth points, we obtain $Q_{1}=[(2)_{t_{1}},\#\Delta_{A'}+2,1,(2)_{\#\Delta_{A'}}]$ and $Q_{2}=[(2)_{t_{2}},1]$. Since $\Delta_{A'}\neq 0$, the contractibility of $\Exc\psi_{A}=[\#\Delta_{A'}+2,(2)_{t_{1}},1,(2)_{t_{2}}]$ to a smooth point gives $t_{1}=0$ and $t_{2}=\#\Delta_{A'}\geq 1$.
	
	Thus $Q_{1}=[t_{2}+2,1,(2)_{t_{2}}]$ and  $Q_{2}=[(2)_{t_{2}},1]$. Recall that $\tau_{1}=2$, $\tau_{2}=3$ and, by Lemma \ref{lem:P2_lem}\ref{item:P2-s_j}, $s_{1}=s_{2}=1$. Since $\psi^{+}(E_{0})$ has one ordinary cusp, $c=3$ and $q_{3}\in \bar{E}$ is ordinary. Therefore, $\bar{E}$ is of type $\FZb(\gamma)$, where $\gamma=t_{2}+3\geq 4$, see Figure \ref{fig:FZb}. Note that $E^{2}=-\gamma$ by Lemma \ref{lem:HN-equations}.
\end{proof}

\begin{figure}[htbp]
	\begin{tabular}{c c c}
		\multirow{3}{*}{
			\begin{subfigure}{0.4\textwidth}\centering
				\vspace{-2.5cm}
				\includegraphics[scale=0.25,trim={1cm 0 0 0}, clip]{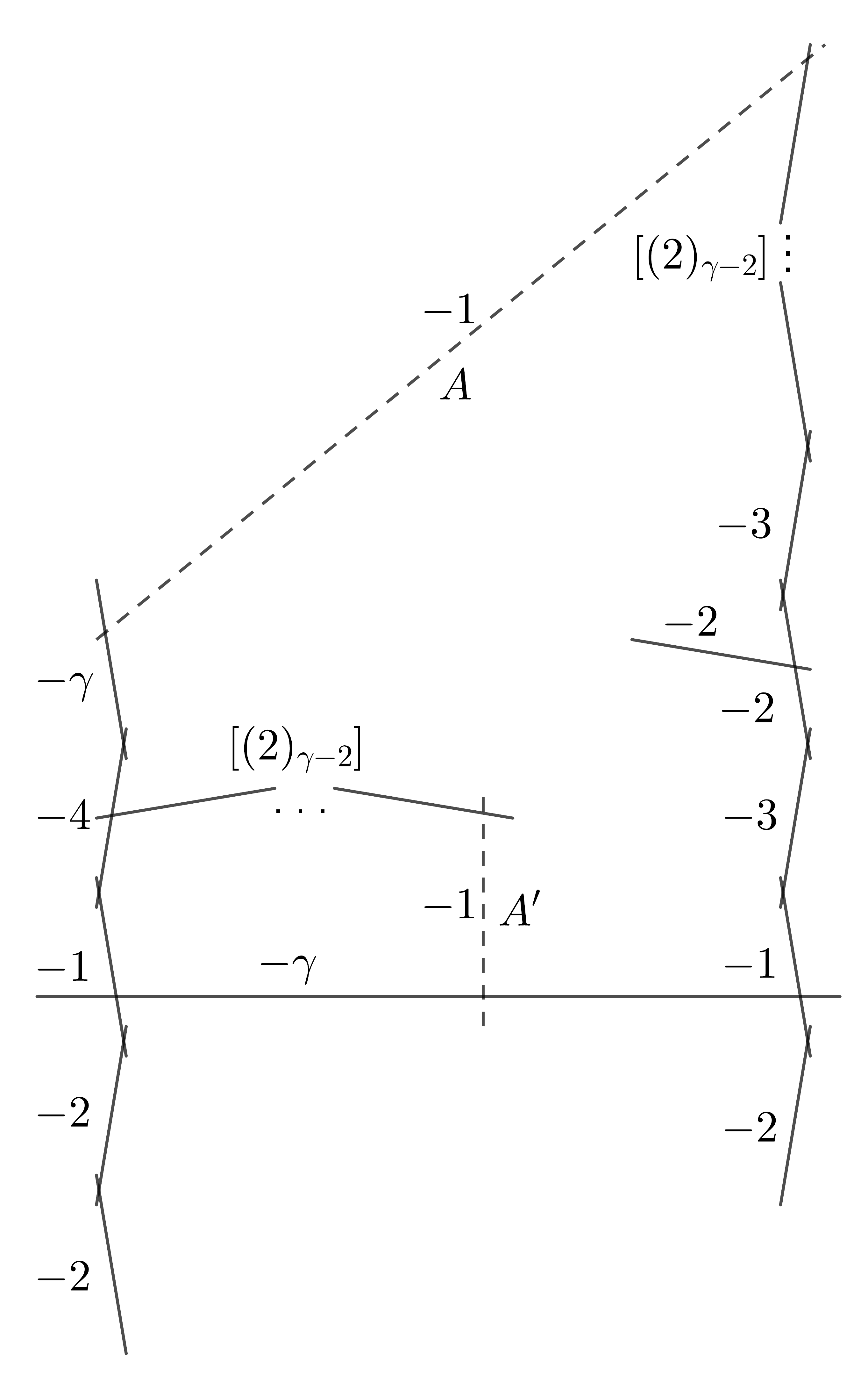}
				\vspace{-1cm}
				\caption{$(X,D)$}
				\vspace{.5cm}
			\end{subfigure}
		}
		&
		$\xrightarrow{\quad \displaystyle{ \psi\circ\psi_{0}}\quad }$
		& 
		\begin{subfigure}{0.35\textwidth}\centering
			\includegraphics[scale=0.22,trim={1cm 0 0 0}, clip]{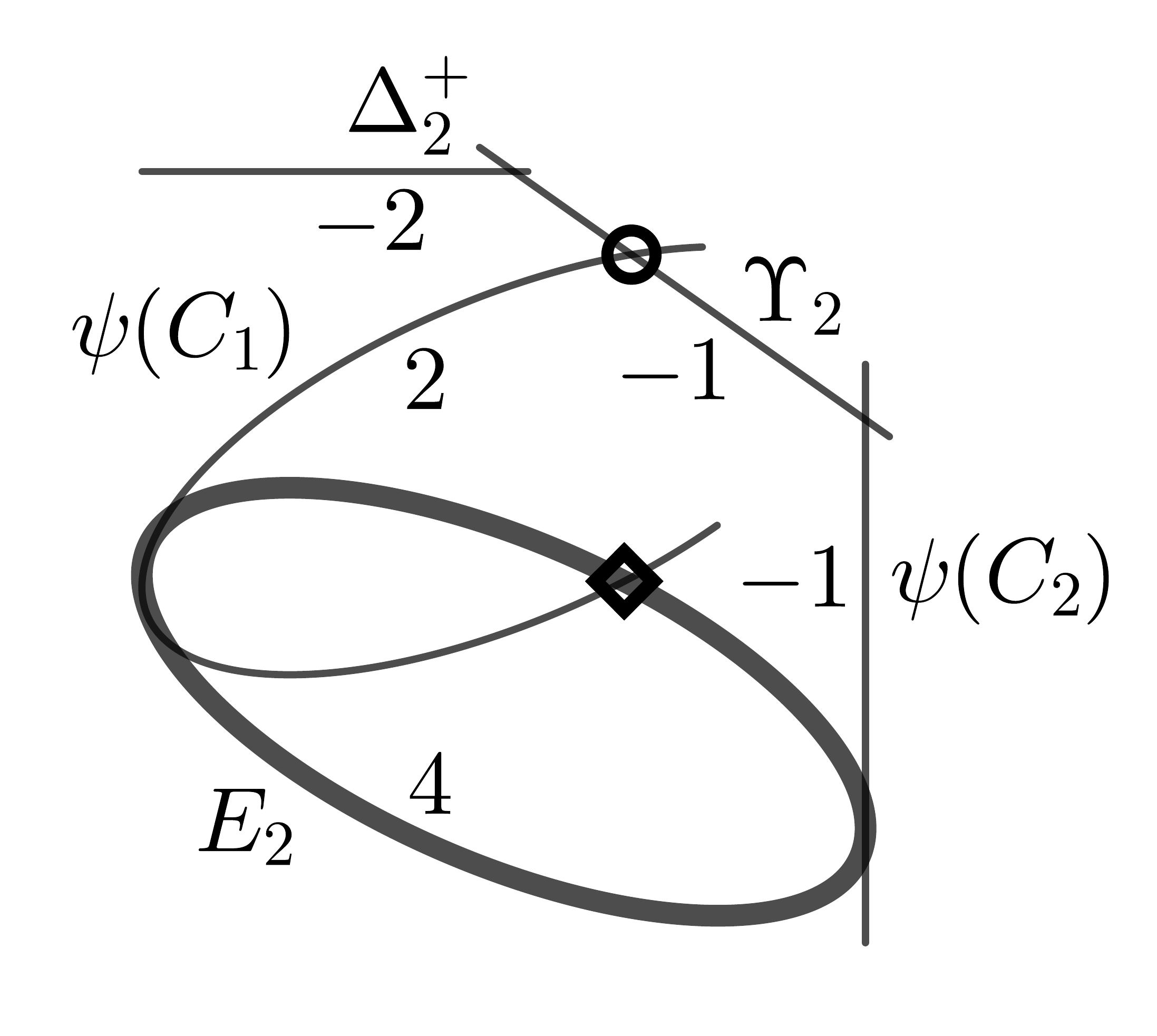}
			\caption{$(X_{2},D_{2})$}
		\end{subfigure}
		\vspace{0.5cm}
		\\
		&&
		$
		\Bigg\downarrow \alpha_{2}
		$
		\\
		&
		\begin{fmpage}{.2\textwidth} \begin{equation*}\begin{split}
			\boldsymbol{\circ} = \mbox{ image of } A\\
			\boldsymbol{\diamond}= \mbox{ image of } A'
			\end{split}\end{equation*} \end{fmpage}				
		&
		\begin{subfigure}{0.35\textwidth}\centering
			\includegraphics[scale=0.22,trim={1cm 0 0 0}, clip]{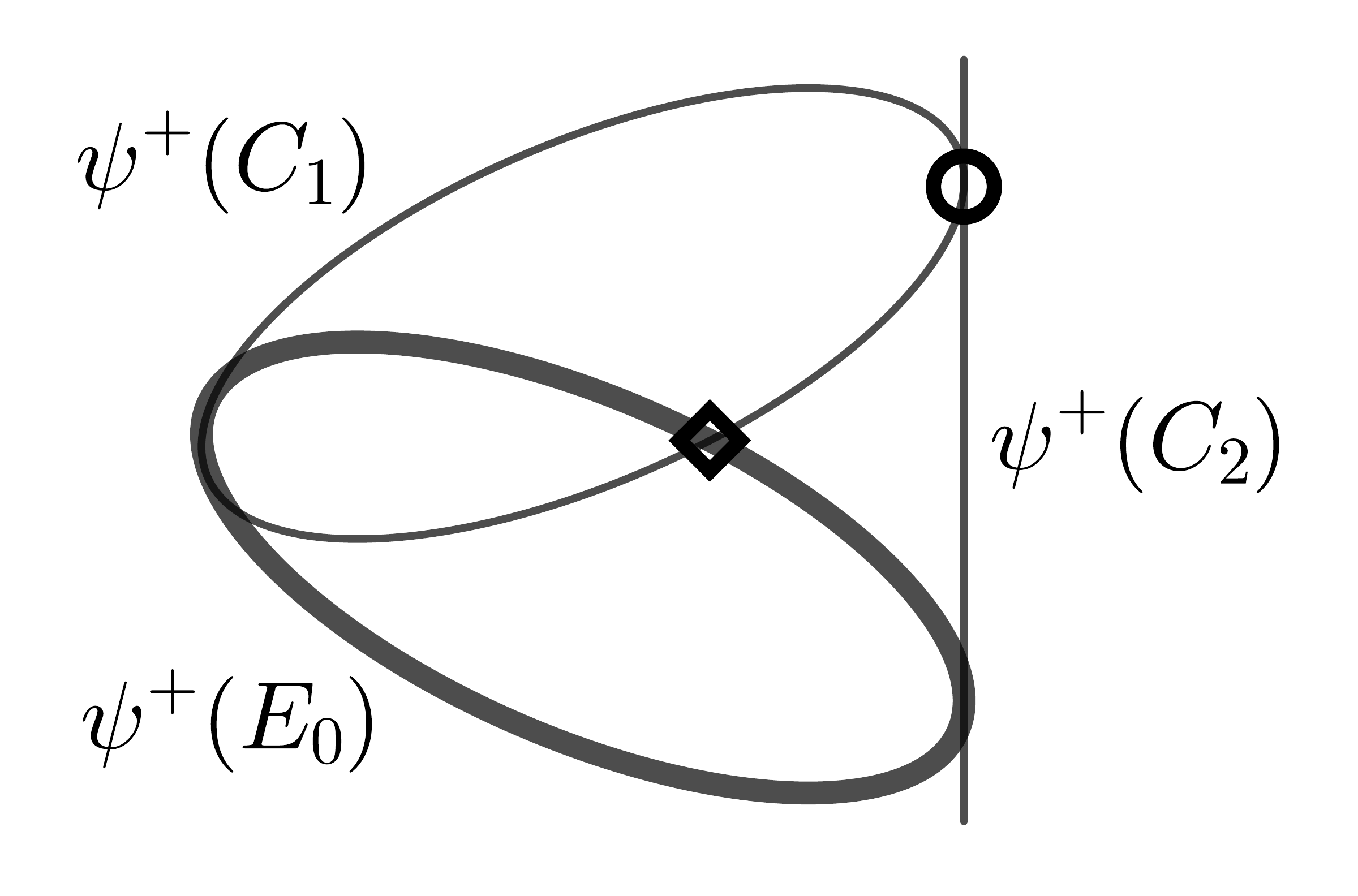}
			\caption{$(X_{\min},D_{\min})$}
		\end{subfigure}
	\end{tabular}
	\caption{Type $\cH(\gamma)$, $\gamma\geq 3$.}
	\label{fig:H}
\end{figure}		
\begin{prop}\label{prop:H} If $n=2$ and $\psi^{+}(E_{0})$ is a conic then $\bar{E}$ is of type $\cH$ (see Figure \ref{fig:H}).
\end{prop}
\begin{proof}
Lemma \ref{lem:P2_lem}\ref{item:P2-deg},\ref{item:P2-n} implies that $D_{\min}-\psi^{+}(E_{0})$ is a union of a line $\ll$ and a conic $\mm$. Hence, $D_{\min}$ is as in Lemma \ref{lem:P2_lem}\ref{item:P2-conics}. We now proceed as in the proof of Proposition \ref{prop:FZb}. Write $\mm\cap \psi^{+}(E_{0})=\{p_{1},p_{2}\}$, where $(\mm\cdot\psi^{+}(E_{0}))_{p_{1}}=1$ and $(\mm\cdot\psi^{+}(E_{0}))_{p_{2}}=3$. We claim that 
 \begin{equation}\label{eq:H_claim}
 \mm=\psi^{+}(C_{1}),\quad \ll=\psi^{+}(C_{2}),\quad\mbox{and}\quad \psi^{+}(A)=\mm\cap\ll,\quad \psi^{+}(A')=\{p_{1}\}.
 \end{equation}
 Lemma \ref{lem:beta_flat}\ref{item:centers_on_En} implies that $(\psi^{+})^{-1}_{*}\mm$ meets $E_{0}$, so by Lemma \ref{lem:P2_lem}\ref{item:P2-s_j} $(\psi^{+})^{-1}_{*}\mm=C_{j}$ for some $j\in \{1,\dots c\}$, say, $j=1$. The point $p_{1}$ is a center of $\psi^{+}$, since otherwise $C_{1}\cdot E_{0}=(\ll_{1}\cdot\psi^{+}(E_{0}))_{p_{1}}=1$, which is false. Say that $\{p_{1}\}=\psi^{+}(A')$. By Lemma \ref{lem:beta_flat}\ref{item:centers_on_En} 
 \begin{equation*}
 (\psi^{+})^{-1}(p_{1})=A'+\Delta_{A'},	
 \end{equation*}
 where $\Delta_{A'}$ is a $(-2)$-twig meeting $C_{1}$. Because $(\mm\cdot \psi^{+}(E_{0}))_{p_1}=1$, $A'$ meets $\ftip{\Delta_{A'}}$. It follows that $\psi_{A'}$ touches $C_{1}$ once. But $\psi^{+}(C_{1})^{2}\geq 1=C_{1}^{2}+ 2$, so $\psi^{+}$ touches $C_{1}$ at least twice. Hence, $\psi^{+}(A)\subseteq \psi^{+}(C_{1})$. As a consequence, $\mm\cap \ll=\psi^{+}(A)$, so, since $n=2$, $\ll\cap \psi^{+}(E_{0})$ is not a center of $\psi^{+}$. Thus $(\psi^{+})^{-1}_{*}\ll\cdot E_{0}=\ll\cdot \psi^{+}(E_{0})=2$, so  $(\psi^{+})^{-1}_{*}\ll=C_{2}$. This proves \eqref{eq:H_claim}. It follows that
 \begin{equation*}
 \tau_{1}=(\mm\cdot \psi^{+}(E_{0}))_{p_{2}}=3,\quad \tau_{2}=\ll\cdot \psi^{+}(E_{0})=2.	
 \end{equation*}

We have $\alpha_{2}^{-1}(\psi^{+}(A))=[1,2]\subseteq \Upsilon_{2}+\Delta_{2}^{+}$, because $(\mm\cdot \ll)_{\psi^{+}(A)}=2$ and $D_{2}-E_{2}$ is snc. The latter inclusion is in fact an equality, because $\psi^{+}(A')$ is not a center of $\alpha_{2}$  and $\Upsilon_{0}+\Delta_{0}^{+}=0$, as $\psi^{+}(E_{0})$ is smooth. Put $U=\psi^{-1}_{*}\Upsilon_{2}$, $\Delta_{U}=\psi^{-1}_{*}\Delta_{2}^{+}$. By Lemma \ref{lem:MMP-properties}\ref{item:Delta_pr-tr} $\Delta_{U}$ is a $(-2)$-tip of $D_{0}$ meeting $U$. We have $\psi(U)\cap\psi(C_{1})=\psi(A)$, because otherwise $U$ meets $C_{1}$ and 
\begin{equation*}
\mm^{2}=\psi(C_{1})^{2}+2=\psi_{A'}(C_{1})^{2}+2=C_{1}^{2}+3=2,
\end{equation*}
which is false, as $\mm^2=4$. The divisor $D_{0}\wedge \Exc\psi$ equals
\begin{equation*}
D_{0}\wedge \Exc\psi^{+}-\psi^{-1}_{*}(\Upsilon_{2}+\Delta_{2}^{+})=D_{0}-(\psi^{+})^{-1}_{*}D_{\min}-U-\Delta_{U}=Q_{1}-C_{1}-U-\Delta_{U}+Q_{2}-C_{2},
\end{equation*} so 
\begin{equation*}
\Exc\psi_{A}=\Exc\psi-\Exc\psi_{A'}=(Q_{1}-C_{1}-\Delta_{A'})+A+(Q_{2}-C_{2}-U-\Delta_{U}).
\end{equation*}
 Lemma \ref{lem:MMP-properties}\ref{item:Exc-psi_i} implies that $V_{1}\de Q_{1}-C_{1}-\Delta_{A'}$ and $V_{2}\de Q_{2}-C_{2}-U-\Delta_{U}$ are zero or twigs of $D_{0}$ meeting $C_{1}$ and $U$, respectively. Because $Q_{1}$ contracts to a smooth point and $\Delta_{A'}\neq 0$, we obtain $V_{1}=T_{1}=[(2)_{t_{1}},\#\Delta_{A'}+2]$. If $t_{1}\neq 0$ then \eqref{eq:line} holds for $j=1$, so $A'$ meets $Q_{2}$ by Lemma \ref{lem:line}\ref{item:tangent}, which is false, because $\{p_{1}\}=\psi^{+}(A')\not\subseteq \ll$. Thus $t_{1}=0$. The contractibility of $Q_{2}$ to a smooth point implies that $V_{2}+[1]+\Delta_{U}$ contracts to a smooth point, so either $V_{2}=0$ or $V_{2}=T_{2}=[(2)_{t_{2}},3]$. Eventually, since $\#\Delta_{A'}>0$, the contractibility of $\Exc\psi_{A}$ to a smooth point gives $V_{2}\neq 0$ and $t_{2}=\#\Delta_{A'}\geq 1$.

Thus $Q_{1}=[t_{2}+2,1,(2)_{t_{2}}]$ and $Q_{2}$ is a fork with maximal twigs $T_{2}=[(2)_{t_{2}},3]$, $\Delta_{U}=[2]$ and $C_{1}=[1]$. Because $Q_{2}$ contracts to a smooth point, $B_{2}^{2}=-2$. Recall that $\tau_{1}=3$, $\tau_{2}=2$ and that,  by Lemma \ref{lem:P2_lem}\ref{item:P2-s_j}, $s_{1}=s_{2}=1$. We have $c=2$ since $\psi^{+}(E_{0})$ is smooth. Therefore, $\bar{E}$ is of type $\cH(\gamma)$, where $\gamma=t_{2}+2\geq 3$, see Figure \ref{fig:H}. Note that $E^{2}=-\gamma$ by Lemma \ref{lem:HN-equations}.
\end{proof}

\begin{figure}[htbp]
	\begin{tabular}{c c c}
		\multirow{2}{*}{
			\begin{subfigure}{0.45\textwidth}\centering
				\vspace{-2cm}
				\includegraphics[scale=0.25,trim={1cm 0 0 0}, clip]{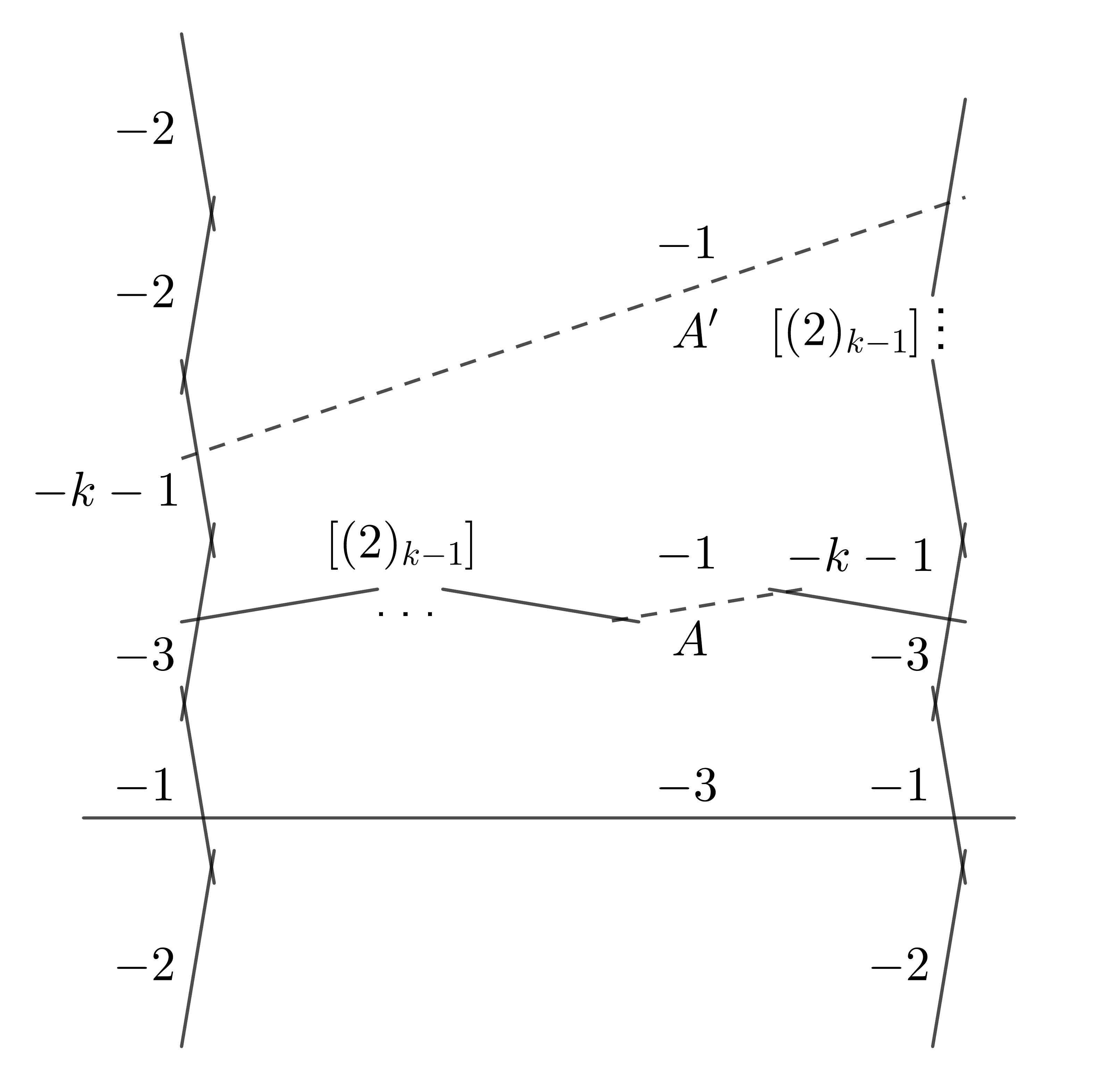}
				\vspace{-.5cm}
				\caption{$(X,D)$}
				\vspace{.5cm}
				\begin{flushright} \begin{fmpage}{.5\textwidth} \begin{equation*}\begin{split}
						\boldsymbol{\circ}& = \mbox{ image of } A\\
						\boldsymbol{\diamond}&= \mbox{ image of } A'
						\end{split}\end{equation*} \end{fmpage} \end{flushright}				
			\end{subfigure}
		}
		&
		$\xrightarrow{\quad \displaystyle{ \psi\circ\psi_{0}}\quad }$
		& 
		\begin{subfigure}{0.35\textwidth}\centering
			\includegraphics[scale=0.22,trim={1cm 0 0 0}, clip]{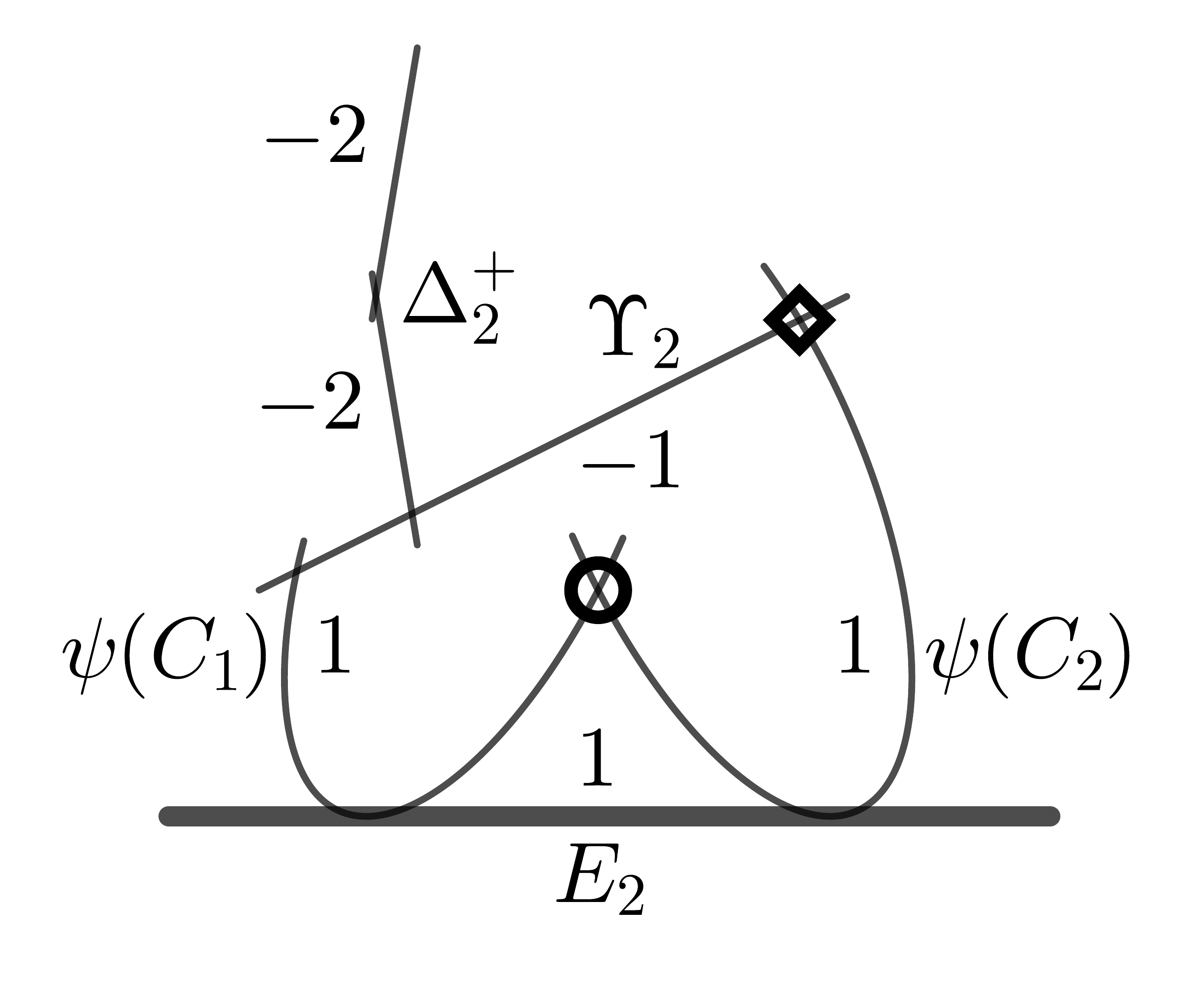}
			\caption{$(X_{2},D_{2})$}
		\end{subfigure}
		\vspace{0.5cm}
		\\
		&&
		$
		\Bigg\downarrow \alpha_{2}
		$
		\\
		&&
		\begin{subfigure}{0.35\textwidth}\centering
			\includegraphics[scale=0.22,trim={1cm 0 0 0}, clip]{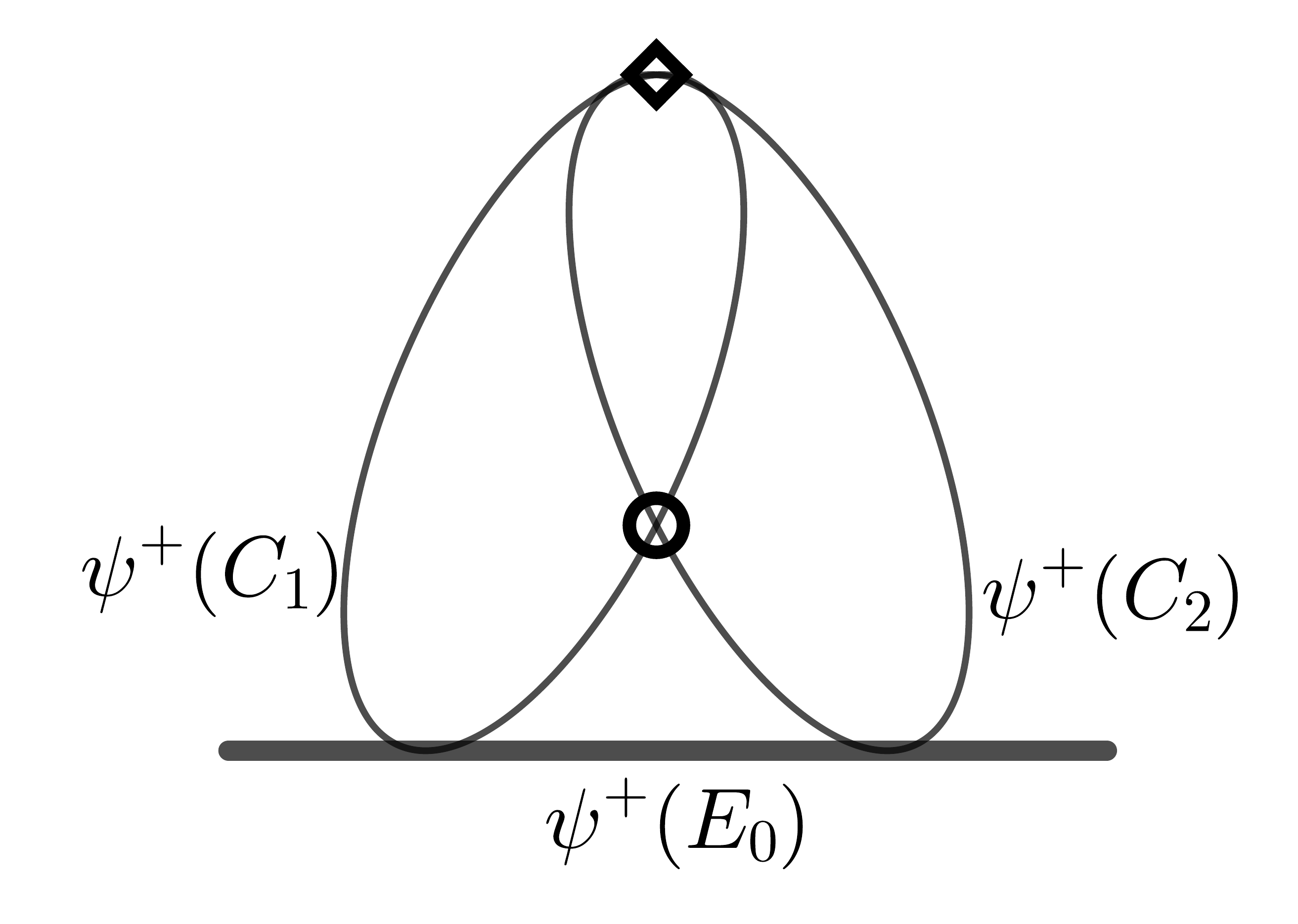}
			\caption{$(X_{\min},D_{\min})$}
		\end{subfigure}
	\end{tabular}
	\caption{Type $\cJ(k)$, $k\geq 2$ for $n=2$ (cf.\ Figure \ref{fig:ex2} for $\cJ(2)$).}
	\label{fig:J}
\end{figure}
\begin{prop}\label{prop:J}
If $n=2$ and $\psi^{+}(E_{0})$ is a line then $\bar{E}$ is of type $\cJ$ (see Figure \ref{fig:J}).
\end{prop}
\begin{proof}
Lemma \ref{lem:P2_lem}\ref{item:P2-deg},\ref{item:P2-smooth},\ref{item:P2-n} implies that $D_{\min}-\psi^{+}(E_{0})$ is a sum of two conics $\mm_{1},\mm_{2}$. Hence, $D_{\min}$ is as in Lemma \ref{lem:P2_lem}\ref{item:P2-conics}. In particular, $\mm_{1}\cap \mm_{2}=\{p_{1},p_{2}\}$, where $(\mm_{1}\cdot \mm_{2})_{p_{1}}=1$ and $(\mm_{1}\cdot \mm_{2})_{p_{2}}=3$. 

Because $D_{2}-E_{2}$ is snc, the point $p_{2}$ is a center of $\alpha_{2}$ and $\alpha^{-1}_{2}(p_{2})=[1,2,2]\subseteq \Upsilon_{2}+\Delta_{2}^{+}$. Let $U$ be the component of $(\psi^{+})^{-1}(p_{2})$ such that $\psi(U)\subseteq \Upsilon_{2}$ and let $\Delta_{U}$ be the connected component of $\psi^{-1}_{*}\Delta_{2}^{+}$ meeting $U$. We have $\Delta_{U}=[2,2]$ by Lemma \ref{lem:MMP-properties}\ref{item:Delta_pr-tr}.

Lemma \ref{lem:P2_lem}\ref{item:P2-semiordinary} implies that $q_{1}\in\bar{E}$ is not semi-ordinary, so $C_{1}\not\subseteq \Exc\psi^{+}$, hence $\psi^{+}(C_{1})=\mm_{j}$ for some $j\in\{1,2\}$. By symmetry we may assume  $\mm_{1}=\psi^{+}(C_{1})$. We claim that $\mm_{2}=\psi^{+}(C_{2})$. Suppose the contrary. Then $(\psi^{+})^{-1}_{*}\mm_{2}\cdot E_{0}<\mm_{2}\cdot\psi^{+}(E_{0})$, so the unique point of  $\psi^{+}(E_{0})\cap \mm_{2}$ is a center of $\psi^{+}$, say, $\psi^{+}(E_{0})\cap \mm_{2}=\psi(A)$. By Lemma \ref{lem:beta_flat}\ref{item:centers_on_En} the preimage on $X_{0}$ of $\psi^{+}(A)$ is a chain $A+\Delta_{A}$ for some $\Delta_{A}\subseteq \Delta_{0}$ meeting $(\psi^{+})^{-1}_{*}\mm_{2}$. Because $(\mm_{2}\cdot\psi^{+}(E_{0}))_{\psi^{+}(A)}> 1$, $A$ meets $\ltip{\Delta_{A}}$ and $\#\Delta_{A}=\mm_{2}\cdot\psi^{+}(E_{0})=2$. The curve $\psi(U)\subseteq \Upsilon_{2}$ meets $D_{2}-\psi_{*}\Delta_{U}$ exactly in two points, belonging to $(\alpha_{2}^{-1})_{*}\mm_{1}=\psi(C_{1})$ and $(\alpha_{2}^{-1})_{*}\mm_{2}$, respectively. One of these points is $\psi(A')$, and the other is not a center of $\psi$ (see Lemma \ref{lem:beta_flat}\ref{item:Ups}). If $\psi(A')\not \subseteq \psi(C_{1})$ then $\mm_{1}^{2}=C_{1}^{2}+3=2$, which is impossible. Hence, $\psi(A')\not \subseteq (\alpha_{2}^{-1})_{*}\mm_{2}$, so $((\psi^{+})_{*}^{-1}\mm_{2})^{2}=\mm_{2}^{2}-3-2=-1$, which implies that  $(\psi^{+})_{*}^{-1}\mm_{2}=C_{2}$. But $(\psi^{+})_{*}^{-1}\mm_{2}$ does not meet $E_{0}$; a contradiction.

Therefore, $\mm_{j}=\psi^{+}(C_{j})$ for $j\in\{1,2\}$. Since $\#(\mm_{j}\cap\psi^{+}(E_{0}))=1$, it follows that 
	$\tau_{j}=(\psi^{+})^{-1}_{*}\mm_{j}\cdot E_{0}=\mm_{j}\cdot\psi^{+}(E_{0})=2$ and that the centers of $\psi^{+}$ are $p_{1}$ and $p_{2}$.

Recall that $\psi(U)\subseteq \Upsilon_{2}$, so $\psi(U)$ meets $D_{2}-\psi_{*}(U+\Delta_{U})$ in two points, namely $\psi(U)\cap \psi(C_{j})$ for $j\in\{1,2\}$. Exactly one of these points is a center of $\psi$. By symmetry we may assume that it is $\psi(U)\cap\psi(C_{2})$. Then $U$ meets $C_{1}$, so $U+\Delta_{U}\subseteq Q_{1}$. We have
\begin{equation*}
D_{0}\wedge \Exc\psi=(Q_{1}-C_{1}-U-\Delta_{U})+(Q_{2}-C_{2}).
\end{equation*}
Lemma \ref{lem:MMP-properties}\ref{item:Exc-psi_i} implies that $Q_{2}$ is a chain and $Q_{1}$ is either a chain or a fork with branching component $U$. Put $k=-U^{2}-1\geq 1$.

Let $A\subseteq \Exc\psi$ be the almost log exceptional curve meeting $T_{2}$. Lemma \ref{lem:MMP-properties}\ref{item:Exc-psi_i} implies that  $W\de \Exc\psi_{A}-A-T_{2}$ is zero or a twig of $D_{0}$. Because $\psi^{+}(A)\subseteq  \mm_{1}$, we have $W\subseteq Q_{1}$. We claim that $W\neq 0$ and $W\neq T_{1}$.  If $W=0$ then $T_{2}\subseteq \Delta_{0}$, so $T_{2}^{0}=C_{1}$ and by Lemma \ref{lem:notation}\ref{item:T^0=C} $q_{2}\in \bar{E}$ is semi-ordinary, which is false. Thus $W\neq 0$. Suppose that $W=T_{1}$. Then $A$ meets $D_{0}$ only in $\ftip{T_{j}}$ for $j\in\{1,2\}$. It follows that $\pi_{0}(A)\cap\bar{E}=\{q_{1},q_{2}\}$, the numbers $(\pi_{0}(A)\cdot \bar{E})_{q_{j}}$ are equal to the multiplicities $\mu_{j}$ of $q_{j}\in \bar{E}$ and $\pi_{0}(A)^{2}=A^{2}+2=1$. Hence, $\deg\bar{E}=\mu_{1}+\mu_{2}$. Because $s_{j}=1$, the multiplicity sequence of $q_{j}$ consists of some number of terms divisible by $\tau_{j}=2$ and the sequence $(1,1)$ at the end. In particular, $2|\mu_{j}$, so $2|\deg\bar{E}$ and Lemma \ref{lem:HN-equations}\eqref{eq:deg^2} implies that $2|E^{2}$.  But $E^{2}=E_{0}^{2}-(\tau_{1}+\tau_{2})=\psi^{+}(E_{0})^{2}-4=-3$; a contradiction. Therefore, $W\neq T_{1}$. 

The contractibility of $Q_{1}$ to a smooth point implies that $W=[(2)_{k-1}]$, so $k\geq 2$, and that $W$ meets $C_{1}$. Since $\psi_{A'}$ does not touch $C_{1}$, we have $\psi_{A}(C_{1})^{2}=\mm_{1}^{2}-3=1=C_{1}^{2}+2$, so $\psi_{A}$ touches $C_{1}$ twice, which implies that $\Exc\psi_{A}=[(2)_{k-1},1,k+1]$. Hence, $T_{2}=[k+1]$. Because $Q_{2}$ contracts to a smooth point, we get  $Q_{2}=[k+1,1,(2)_{k-1}]$. The morphism $\psi_{A'}$ touches $C_{2}$ once, because $\psi_{A'}(C_{2})^{2}=\psi(C_{2})^{2}-1=\mm_{2}^{2}-4=0=C_{2}^{2}+1$. Hence, $\Exc\psi_{A'}=A'+T_{2}'$. It follows that $Q_{1}=\Delta_{U}+U+C_{1}+W=[2,2,k+1,1,(2)_{k-1}]$. 

Recall that for $j\in\{1,2\}$, we have $\tau_{j}=\mm_{j}\cdot\psi^{+}(E_{0})=2$ and $s_{j}=1$ by Lemma \ref{lem:P2_lem}\ref{item:P2-s_j}. Also, $c=2$, because $\psi^{+}(E_{0})$ is smooth. It follows that $\bar{E}$ is of type $\cJ(k)$, see Figure \ref{fig:J}. Note that $E^{2}=-3$ by Lemma \ref{lem:HN-equations}.
\end{proof}

		\begin{figure}[htbp]
			\begin{tabular}{c c c}
				\multirow{2}{*}{
					\begin{subfigure}{0.45\textwidth}\centering
						\vspace{-2cm}
						\includegraphics[scale=0.25,trim={1cm 0 0 0}, clip]{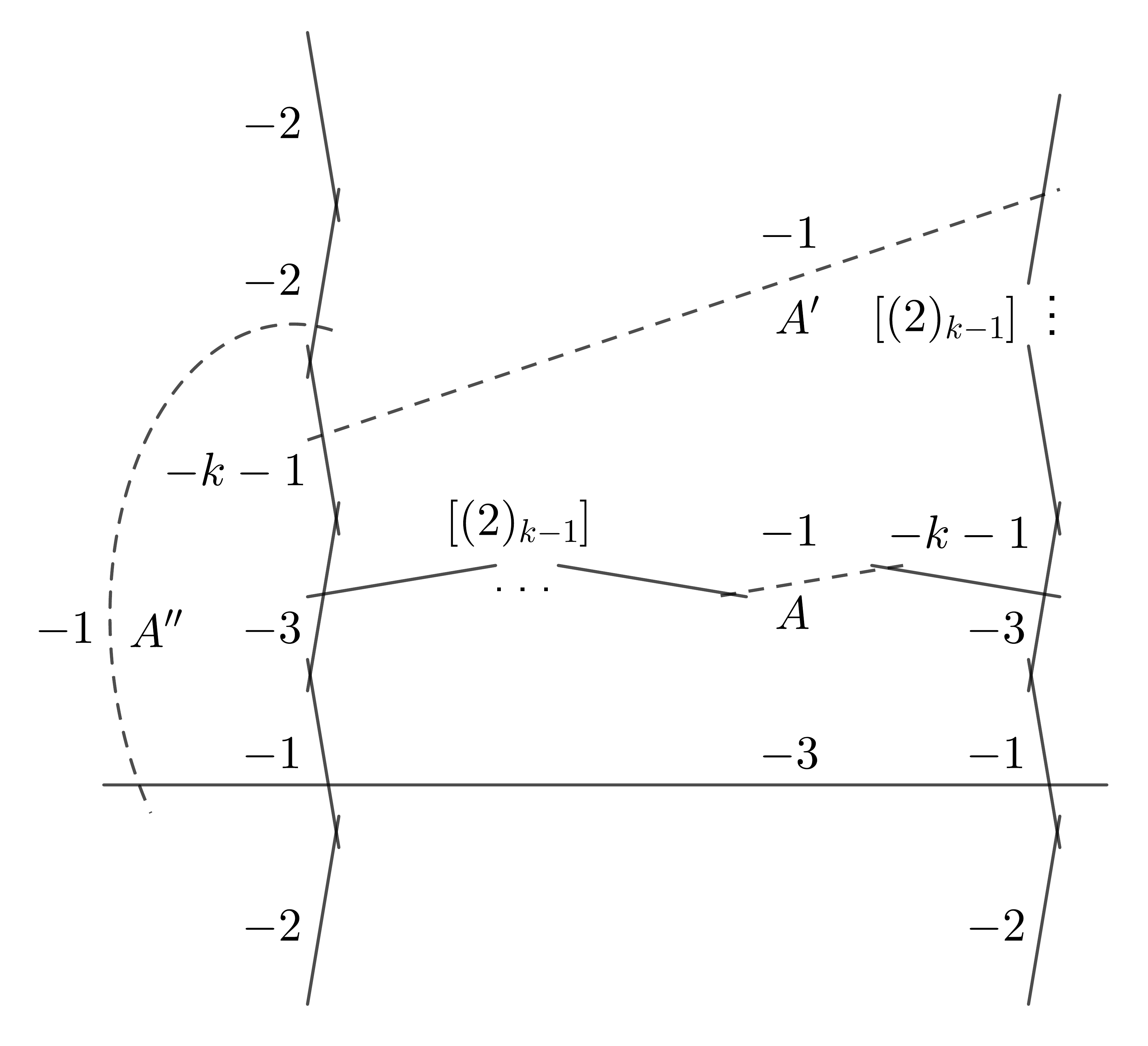}
						\caption{$(X,D)$}
						\vspace{.5cm}
						\begin{flushright} \begin{fmpage}{.5\textwidth} \begin{equation*}\begin{split}
								\boldsymbol{\circ}& = \mbox{ image of } A\\
								\boldsymbol{\diamond}&= \mbox{ image of } A'\\
								\boldsymbol{\scriptstyle{\square}}& = \mbox{ image of } A''
								\end{split}\end{equation*} \end{fmpage} \end{flushright}
					\end{subfigure}
				}
				&
				$\xrightarrow{\quad \displaystyle{ \psi\circ\psi_{0}}\quad }$
				& 
				\begin{subfigure}{0.35\textwidth}\centering
					\includegraphics[scale=0.35,trim={1cm 0 0 0}, clip]{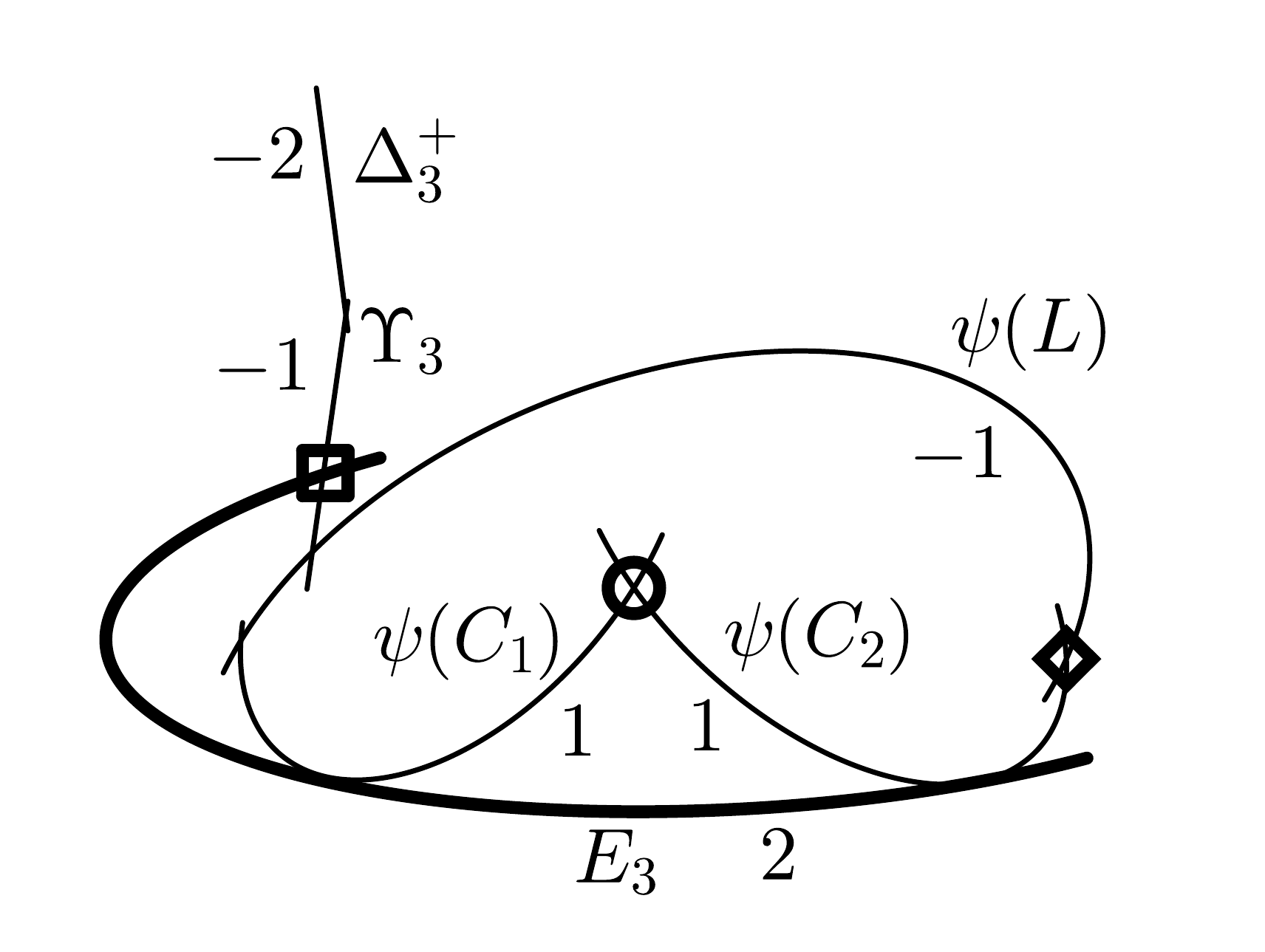}
					\caption{$(X_{3},D_{3})$}
				\end{subfigure}
				\vspace{0.5cm}
				\\
				&&
				$
				\Bigg\downarrow \alpha_{3}
				$
				\\
				&&
				\begin{subfigure}{0.35\textwidth}\centering
					\includegraphics[scale=0.35,trim={1cm 0 0 0}, clip]{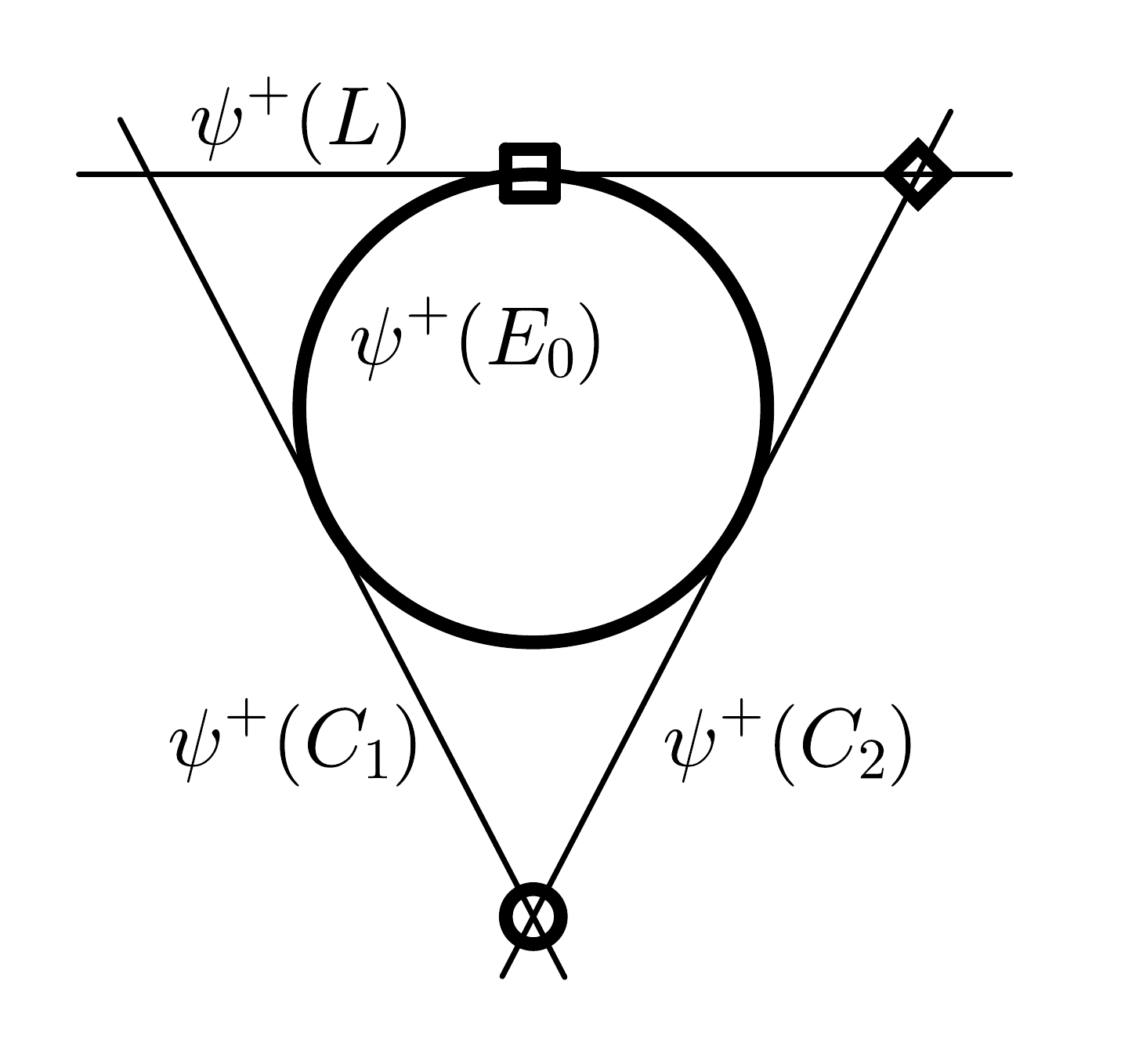}
					\caption{$(X_{\min},D_{\min})$}
				\end{subfigure}
			\end{tabular}
			\caption{Type $\cJ(k)$, $k\geq 2$ for $n=3$ (cf.\ Figure \ref{fig:ex2} for $\cJ(2)$).}
			\label{fig:J_line}
		\end{figure}
\begin{prop}\label{prop:n3}
	If $n=3$ then $\bar{E}$ is of type $\cJ$ (see Figure \ref{fig:J_line}).
\end{prop}
\begin{proof}
	Lemma \ref{lem:P2_lem}\ref{item:P2-deg},\ref{item:P2_nodes} implies that $D_{\min}$ is a conic $\mm$ inscribed in a triangle $\ll_{1}+\ll_{2}+\ll_{3}$. We claim that $\mm=\psi^{+}(E_{0})$. Suppose not. Then  $(\psi^{+})^{-1}_{*}\mm$ is contained in $Q_{1}+\dots + Q_{c}$, so it  meets every other component of $D_{0}-E_{0}$ at most once. Hence the two points $\mm\cap (D_{\min}-\psi^{+}(E_{0})-\mm)$ are centers of $\psi^{+}$. Moreover, since $E_{0}$ contains no point of normal crossings on $D_{0}$, the two points $\psi^{+}(E_{0})\cap (D_{\min}-\psi^{+}(E_{0})-\mm)$ are centers of $\psi^{+}$, too. But then $n\geq 4$; a contradiction.
	
	Suppose that $\ll_{j}=\psi^{+}(C_{j})$ for $j\in\{1,2,3\}$. Then at each center of $\psi^{+}$ two components of $D_{\min}$ meet transversally, so from the definition of $\Upsilon_{3}$ (see Figure \ref{fig:Upsilon}) we infer that  $\psi^{+}=\psi$. Lemma \ref{lem:MMP-properties}\ref{item:Exc-psi_i} implies that $\psi^{*}(D_{n}-E_{n})\redd$ is circular, and hence $(\pi_{0})_{*}\psi^{*}(D_{n}-E_{n})\redd$ is circular, too. The latter equals $\pi_{0}(A)+\pi_{0}(A')+\pi_{0}(A'')$ and consists of three smooth components meeting transversally in three points, so it is again a triangle. It follows that $\pi_{0}$ touches $A$, $A'$, $A''$ twice each, so they meet each $Q_{j}$ only in its first component. Since $\psi^{*}(D_{0}-E_{0})\redd$ is circular, this is possible only if $Q_{j}=C_{j}$  for $j\in\{1,2,3\}$. But then, since $\tau_{j}=\psi^{+}(E_{0})\cdot\psi^{+}(C_{j})=2$, every $q_{j}\in \bar{E}$ is ordinary; a contradiction.

	It follows that some $\ll_{k}$, say $\ll_{3}$, is not the image of $C_{j}$ for any $j\in\{1,\dots, c\}$. We now determine the centers of $\psi^{+}$. By Lemma \ref{lem:P2_lem}\ref{item:P2-s_j}, $L\de (\psi^{+})^{-1}_{*}\ll_{3}$ does not meet $E_{0}$. In particular, $\ll_{3}\cap \mm$ is a center of $\psi^{+}$, say $\ll_{3}\cap \mm=\psi^{+}(A'')$. Lemma \ref{lem:beta_flat}\ref{item:centers_on_En} implies that $(\psi^{+})^{-1}(\ll_{3}\cap \mm)$ is a chain $A''+\Delta_{A''}$, where $\Delta_{A''}\subseteq \Delta_{0}$ meets $L$. Because $(\psi^{+}(L)\cdot \mm)_{\psi^{+}(A'')}=2$, $A''$ meets $\ltip{\Delta_{A''}}$ and $\#\Delta_{A''}=2$. Since the contraction of $A''+\Delta_{A''}$ touches $L$ twice and $L^{2}+2\leq 0<\ll_{3}^{2}$, the line $\ll_{3}$ contains another center of $\psi^{+}$, say $\ll_{2}\cap\ll_{3}=\psi^{+}(A')$. Because $\psi^{+}$ touches $\ll_{1}$, we have $\psi^{+}(A)\subseteq \ll_{1}$. If $\psi^{+}(A)\subseteq \mm$ then arguing as above for $\psi^{+}(A'')\subseteq \ll_{3}$ we obtain that $\ll_{1}$ contains another center of $\psi^{+}$, which is false. It follows that for $j\in\{1,2\}$ the curve $(\psi^{+})^{-1}_{*}\ll_{j}$ meets $E_{0}$, so by Lemma \ref{lem:P2_lem}\ref{item:P2-s_j} we have, say, $\ll_{j}=\psi^{+}(C_{j})$ for $j\in\{1,2\}$. Because $C_{1}$ and $C_{2}$ are disjoint, the common point of their images is a center of $\psi^{+}$, so $\ll_{1}\cap\ll_{2}=\psi^{+}(A)$. Moreover, we get $\tau_{j}=C_{j}\cdot E_{0}=\ll_{j}\cdot \mm=2$ for $j\in\{1,2\}$ and $L\cdot C_{1}=\ll_{3}\cdot \ll_{1}=1$. 	

	To describe the shape of $Q_1$ note first that the point $\psi^{+}(A'')$ is the only center of $\alpha_{3}$. Indeed, since $\psi^{+}(E_{0})$ is smooth, the centers of $\alpha_{3}$ are contained in the image of $\Exc\psi$, and the points $\psi^{+}(A)$, $\psi^{+}(A')$ are not centers of $\alpha_{3}$, because they are nc-points of $D_{\min}$. It follows that 
	\begin{equation*}
	D_{0}\wedge \Exc\psi=(Q_{1}-C_{1}-L-\Delta_{A''})+(Q_{2}-C_{2}).
	\end{equation*}
	Lemma \ref{lem:MMP-properties}\ref{item:Exc-psi_i} implies that $Q_{2}$ is a chain and $Q_{1}$ is either a chain or a fork with branching component $L$. Suppose that $T_{1}\subseteq \Exc\psi$. Then, because $Q_{1}$ contracts to a smooth point, $\Delta_{A''}\subseteq T_{1}'$ and $T_{1}=[(2)_{t_{1}},\#\Delta_{A''}+2]=[(2)_{t_{1}},4]$. Hence $$\pi_{0}(A'')^{2}=(A'')^{2}+2+4(t_{1}+1)=4t_{1}+5.$$ In particular, $t_{1}\neq 0$, so \eqref{eq:line} holds for $j=1$ and Lemma \ref{lem:line}\ref{item:tangent} implies that $A''$ meets $Q_{2}$; a contradiction.  It follows that $\Delta_{A''}\subseteq T_{1}$. This inclusion is strict, because otherwise $Q_{1}=C_{1}+\Delta_{A''}$, which is false. Hence $T_{1}=\Delta_{A''}+L$, so $Q_{1}$ is a chain, and because it contracts to a smooth point, $Q_{1}=[2,2,k+1,1,(2)_{k-1}]$, where $k=-L^{2}-1\geq 1$. We have $k\geq 2$, because otherwise Lemma \ref{lem:notation}\ref{item:T^0=C} implies that $q_{1}\in \bar{E}$ is semi-ordinary, which is false.

	We now describe the shape of $Q_2$. Because $\psi^+(A')\subseteq \ll_3$, the curve $L$ meets $\Exc \psi_{A'}$. The divisor $\Exc \psi^+$ has a negative definite intersection matrix, so $A'\cdot (\Delta_{A''}+A'')=0$. Since $\Delta_{A''}$ is the unique twig of $D_{0}$ meeting $L$, it follows from Lemma \ref{lem:MMP-properties}\ref{item:Exc-psi_i} that $A'$ meets $L$ and hence $\Exc \psi_{A'}-A'$ is a $(-2)$-twig. If $\Exc \psi_{A'}-A'=T_{2}$ then by Lemma \ref{lem:notation}\ref{item:T^0=C} $q_{2}\in \bar{E}$ is semi-ordinary, which is false. Thus $\Exc \psi_{A'}=A'+T_2'$. Eventually, $T_{2}'=[(2)_{k-1}]$, because $\psi_{A'}$ touches $L$ exactly $k$ times. Indeed,
	\begin{equation*}
	\psi_{A'}(L)^{2}-L^{2}=\psi(L)^{2}+k+1=\psi^{+}(L)^{2}+k-1=k.
	\end{equation*}
	The contractibility of $Q_{2}$ to a smooth point implies that $T_{2}=[(2)_{t_{2}},k+1]$. We have $t_{2}=0$, for otherwise \eqref{eq:line} holds for $j=2$ and by Lemma \ref{lem:line}\ref{item:tangent} $A''$ meets $\ftip{T_{1}'}$, which is false. Hence, $Q_{2}=[k+1,1,(2)_{k-1}]$.
	
	Recall that for $j\in\{1,2\}$, $s_{j}=1$ by Lemma \ref{lem:P2_lem}\ref{item:P2-s_j} and $\tau_{j}=C_{j}\cdot E_{0}=\ll_{j}\cdot \mm=2$. We have $c=2$ since $\psi^{+}(E_{0})$ is smooth. Therefore, $\bar{E}$ is of type $\cJ(k)$, see Figure \ref{fig:J_line}. Note that $E^{2}=-3$ by Lemma \ref{lem:HN-equations}.
\end{proof}
		
\begin{rem}[Uniqueness of the process of almost minimalization]\label{rem:J_twice}
	As we have indicated in Example \ref{ex:different_weak}, a rational cuspidal curve of type $\cJ$ is obtained both in Propositions \ref{prop:J} and \ref{prop:n3}, which corresponds to two possible choices of the process $\psi$ of almost minimalization. Our proof shows that for every other type in Definition \ref{def:our_curves} this process (as defined in Section \ref{sec:MMP}) is unique.
\end{rem}

\bigskip
\section{Existence and uniqueness}\label{sec:existence}
In this Section we finish the proof of Theorem \ref{thm:main} by proving that each type in the list of Definition \ref{def:our_curves} is realized by a planar rational cuspidal curve which is unique up to a projective equivalence, and that the complement of such curve is a surface of log general type which satisfies \eqref{eq:assumption}. We begin with a proof of condition \eqref{eq:assumption} in Section \ref{sec:kappa}. The nonexistence of $\C^{**}$-fibrations follows from \cite[Theorem 1.3]{PaPe_Cstst-fibrations_singularities}, because the type of $\bar{E}$ is not listed in loc.\ cit.\ and we have  $\kappa(\P^{2}\setminus \bar{E})=2$ by Lemma \ref{lem:kappa=2}. An independent proof of this fact is given in Section \ref{sec:no_Cstst}. Existence and uniqueness of curves of types $\FZb$, $\FE$ and $\cH$ is either established or can be inferred from results in the literature, see Section \ref{sec:discussion}. This is also true for types $\Qb$ and $\Qa$, but we give an independent geometric argument in Section \ref{sec:q4}. For the remaining types $\cI$ and $\cJ$ we prove the existence and uniqueness in Section \ref{sec:IJ} by reverting the process of almost minimalization constructed in Section \ref{sec:possible_HN-types}.

In Section \ref{sec:1.2=>1.1} we prove Theorem \ref{thm:geometric}, which implies that all curves of types listed in Definition \ref{def:our_curves} are closures of images of some proper injective morphisms $\C^{*}\to \C^{2}$. The latter morphisms, which we call \emph{singular embeddings}, are classified in \cite{BoZo-annuli} under some regularity assumptions. In Remarks \ref{rem:BZ_quintics}, \ref{rem:FZb-FE-H}\ref{item:FZb},\ref{item:FE},\ref{item:H_BZ} and \ref{rem:IJ}\ref{item:BZ_IJ} we explain how to obtain curves as in Definition \ref{def:our_curves} from \cite{BoZo-annuli}.

\medskip

 In Section \ref{sec:possible_HN-types} we were working mostly with the minimal weak resolution $\pi_{0}\colon (X_{0},D_{0})\to(\P^{2},\bar{E})$. Here it will be more convenient to work with the minimal log resolution $\pi\colon (X,D)\to (\P^{2},\bar{E})$. For $j\in \{1,\dots, c\}$ we denote by $Q_{j}'\subseteq D$ the reduced preimage of $q_{j}\in \bar{E}$. Its components are naturally ordered as exceptional divisors of the successive blowups in the decomposition of $\pi^{-1}$, see \eqref{eq:Q-ordering}. In particular, the last component of $Q_{j}'$ is the unique $(-1)$-curve in $Q_{j}'$, which we denote by $C_{j}'$. As in Section \ref{sec:resolutions}, we put $E\de (\pi^{-1})_{*}\bar{E}$. If $\bar{E}$ is of one of the types in Definition \ref{def:our_curves} then the number $E^{2}$, computed using Lemma \ref{lem:HN-equations}, is given in Table \ref{table:nofibrations}. In particular, $E^{2}\leq -3$.  The weighted graph of $D$ is shown in one of the Figures \ref{fig:FE}--\ref{fig:J}.	

\subsection{The Kodaira--Iitaka dimension.}\label{sec:kappa}

\begin{lem}[Complements are of log general type]\label{lem:kappa=2} If $\bar{E}\subseteq \P^2$ is of one of the types listed in Definition \ref{def:our_curves} then $\kappa(\P^{2}\setminus \bar{E})=2$.
\end{lem}
\begin{proof}
Suppose the contrary. We have  $c\geq 2$, so \cite{Wakabayashi-cusp} implies that $\kappa(\P^{2}\setminus\bar{E})\geq 0$ and $c=2$, hence $\bar{E}$ is of type $\cH$, $\cI$ or $\cJ$. By the Iitaka Easy Addition Theorem, $\P^{2}\setminus \bar{E}$ has no $\C^{1}$-fibration. Hence  \cite[Proposition 2.6]{Palka-minimal_models}  implies that $\P^{2}\setminus \bar{E}$ has a $\C^{*}$-fibration and by \cite[Proposition 4.2]{PaPe_Cstst-fibrations_singularities} we can choose one without base points on $X$. Therefore, $X$ has a $\P^{1}$-fibration such that $F\cdot D=2$ for any fiber $F$.

Suppose that $D$ contains some fiber $F$. Because $D$ contains no $0$-curves, \cite[7.5]{Fujita-noncomplete_surfaces} implies that $F\redd=[2,1,2]$ and $F$ meets $D-F\redd$ only in the middle component. The latter equals $C_{j}'$ for some $j\in \{1,2\}$. Because $c=2$, $E$ meets $D-C_{j}'$, so $E\not\subseteq F$. Then $Q_1'$ contains $F\redd$, so it is not negative definite; a contradiction. By Lemma \ref{lem:fibrations-Sigma-chi} the horizontal part of $D$ consists of two $1$-sections and every fiber $F$ has a unique component $L_{F}$ not contained in $D$. We claim that $D\hor=C_{1}'+C_{2}'$. Suppose that  $C_{j}'$ is vertical for some $j\in \{1,2\}$. A fiber of a $\P^{1}$-fibration cannot contain a branching $(-1)$-curve, so since $\beta_{D}(C_{j}')=3$, $C_{j}'$ meets a section in $D$, hence $C_{j}'$ has multiplicity one in a  fiber. Thus $C_{j}'$ is a tip of that fiber, so both sections in $D$ meet $C_{j}'$. In particular, $D\hor\subseteq E+Q_{j}'$. But then $C_{3-j}'$ is vertical and by the same argument we get $D\hor\subseteq E+Q_{3-j}'$, so $D\hor=E$; a contradiction. It follows that $E$ is a component of some fiber $F_{E}$. Since $E\cdot (D-C_{1}'-C_{2}')=0$, by the connectedness of $D$ we get $F_{E}\wedge D=E$. As a consequence, $F_{E}=E+L_{F_{E}}=[1,1]$, so $E^{2}=-1$. This is a contradiction, because $E^{2}\leq -3$.
\end{proof}

\begin{lem}[Existence of special lines, cf.\ Theorem \ref{thm:geometric}]\label{lem:special_lines}
	Let $\bar{E}\subseteq \P^{2}$ be a rational cuspidal curve of one of the types listed in Definition \ref{def:our_curves} or in \cite[Theorem 1.3, cf.\ Table 1]{PaPe_Cstst-fibrations_singularities}. We order the cusps $q_{1},\dots, q_{c}\in \bar{E}$ in such a way that their multiplicity sequences $(\mu_{j},\mu_{j}',\dots )$, $j\in\{1,\dots, c\}$ form a non-increasing sequence in the lexicographic order. Denote by $\ll_{12}$ the line joining $q_{1},q_{2}\in \bar{E}$ and by $\ll_{1}$ the line tangent to $q_{1}\in \bar{E}$. Then 
		\begin{enumerate}
			\item\label{item:l}
			if $\bar{E}$ is of type $\FZa, \FZb, \FE, \cA-\cF, \cH$ or $\cI$ then $(\ll_{12}\cdot \bar{E})_{q_{j}}=\mu_{j}$ for $j\in\{1,2\}$ and $\ll_{12}$ does not meet $\bar{E}\setminus \{q_{1}, q_{2}\}$,
			\item\label{item:t} if $\bar{E}$ is of type $\FZa,\FZb,\FE,\cA-\cD,\cG,\cH,\cJ,\Qb$ or $\Qa$ then $(\ll_{1}\cdot\bar{E})_{q_{1}}=\mu_{1}+\mu_{1}'$ and $\ll_{1}$ meets $\bar{E}\setminus\{q_{1}\}$ once and transversally, that is, $(\ll_{1}\setminus \{q_{1}\})\cdot (\bar{E}\setminus \{q_{1}\})=1$.
		\end{enumerate}	
	\end{lem}
	\begin{proof}
		We check case by case that
			\begin{equation*}
			\mu_{1}+\mu_{2}=\deg\bar{E} \mbox{ for types listed in \ref{item:l} and } \mu_{1}+\mu_{1}'=\deg\bar{E}-1 \mbox{ for types listed in \ref{item:t}}
			\end{equation*}
			(for types listed in \cite[Table 1]{PaPe_Cstst-fibrations_singularities} this was done in Section 4F loc.\ cit). The first equation implies \ref{item:l}. We have $(\ll_{1}\cdot \bar{E})_{q_{1}}<\deg\bar{E}$. Indeed, otherwise $\ll_{1}$ does not meet $\bar{E}\setminus \{q_{1}\}$, so $\ll_{1}\cap (\P^{2}\setminus \bar{E})\cong \C^{1}$, which by Lemma \ref{lem:Qhp_has_no_lines} implies that $\kappa(\P^{2}\setminus \bar{E})<2$, contrary to Lemma \ref{lem:kappa=2}. Because the number $(\ll_{1}\cdot\bar{E})_{q_{1}}$ is the sum of at least two initial terms of the multiplicity sequence of $q_{1}$, the second equation implies \ref{item:t}.
	\end{proof}	

\smallskip
We now study the Kodaira--Iitaka dimension of the divisor $K_{X}+\tfrac{1}{2}D$.

\begin{lem}[A criterion for $\kappa_{1/2}=-\infty$]\label{lem:C3st}
	Let $D$ be a reduced effective divisor on a smooth projective surface $X$. Assume that there is a $\P^{1}$-fibration of $X$ such that for a fiber $F$ we have
	\begin{equation}\label{eq:C3st-condition}
	D \cdot F=4 \mbox{ and there exists a $(-2)$-twig of $D$ with a horizontal component.}
	\end{equation}
	Then $\kappa(K_{X}+\tfrac{1}{2}D)=-\infty$.
\end{lem}
\begin{proof}
	Suppose that $\kappa(K_{X}+\tfrac{1}{2}D)\geq 0$. Write $\mathcal{P}$, $\mathcal{N}$ for the positive and negative part of the Zariski-Fujita decomposition of $K_{X}+\tfrac{1}{2}D$. Let $T=T_{1}+\dots +T_{m}$ be a $(-2)$-twig of $D$, where $T_{k}$ for $1\leq k \leq m$ is the $k$-th component of $T$. We have $T_{1}\cdot (K_{X}+\tfrac{1}{2}D)=-1<0$, so $T_{1}\subseteq \Supp \mathcal{N}$. If for some $1\leq k< m$ we have $T_{1}+\dots+T_{k}\subseteq \Supp \mathcal{N}$ then $$T_{k+1}\cdot \mathcal{N}=T_{k+1} \cdot (K_{X}+\tfrac{1}{2}D)-T_{k+1}\cdot \mathcal{P}=-T_{k+1}\cdot \mathcal{P}\leq 0,$$ so in fact $T_{k+1}\cdot \mathcal{N}<0$, because $T_{k+1}\cdot T_{k}>0$. It follows that $T_{k+1}\subseteq \Supp \mathcal{N}$ and hence by induction $T\subseteq \Supp \mathcal{N}$. Therefore, $$0< F\cdot \mathcal{N}=F\cdot (K_{X}+\tfrac{1}{2}D) - F\cdot \mathcal{P} =-F\cdot \mathcal{P};$$ a contradiction.
\end{proof}

\begin{prop}[Complements are $\C^{***}$-fibered and have $\kappa_{1/2}=-\infty$]\label{prop:finding_C3st}
	Let $\bar{E}\subseteq \P^{2}$ be a rational cuspidal curve of one of the types listed in Definition \ref{def:our_curves}. Then the minimal log resolution $(X,D)$ of $(\P^{2},\bar{E})$ admits a $\P^{1}$-fibration satisfying \eqref{eq:C3st-condition}. In particular,
	\begin{equation*}
	\kappa(K_{X}+\tfrac{1}{2}D)=\kappa(K_{X_{0}}+\tfrac{1}{2}D_{0})=-\infty.
	\end{equation*}
\end{prop}

\begin{proof}
	Assume that $\bar{E}$ is of type $\Qb$ or $\Qa$. Then $q_{1}\in \bar{E}$ has multiplicity $\mu_{1}=2$, so the pullback of the pencil of lines through $q_{1}$ induces a $\P^{1}$-fibration of $X$ with $D\cdot F=\deg\bar{E}-(\mu_{1}-1)=4$ for a fiber $F$. The first component of $Q_{1}'$ is a horizontal $(-2)$-tip of $D$, so this $\P^{1}$-fibration satisfies \eqref{eq:C3st-condition}. 
	
	Assume now that $\bar{E}$ is of one of the remaining types. For $j\in\{1,2\}$ denote  by $U_{j}$ the maximal twig of $D$ containing the first component of $Q_{j}'$ and by $V_{j}$ the unique $(-2)$-twig of $D$ meeting $C_{j}'$. If  $Q_{j}'$ is not a chain, we denote by $B_{j}$ the branching component of $Q_{j}'$ meeting $U_{j}$ and by $T_{j}'$ the unique twig of $D$ meeting $B_{j}$ which is disjoint from $U_{j}+V_{j}$ (cf.\ Notation \ref{not:graphs}).
	
	Consider the types $\FE(\gamma)$, $\FZb(\gamma)$ and $\cH(\gamma)$. By definition, we have $\gamma\geq 5$, $\gamma\geq 4$ and $\gamma\geq 3$, respectively. Let $A,A'$ be the proper transforms on $X$ of the lines $\ll_{12}$, $\ll_{1}$ from Lemma \ref{lem:special_lines}. We have $(A')^{2}=A^{2}=-1$, $A'\cdot D= A\cdot D=2$, $A$ meets $D$ only in the first components of $Q_{1}'$ and $Q_{2}'$ and $A'$ meets $D$ only in $E$ and in the second component of $Q_{1}'$. 
	
	Consider the types $\FE(\gamma)$ and $\cH(\gamma)$  (see Figures \ref{fig:FE}, \ref{fig:H}). Then $T_{2}'=[2]$ is disjoint from $C_{2}'+A$ and $(Q_{2}'-T_{2}')+A+U_{1}$ is a chain $[2,1,3,3,(2)_{\gamma-4},1,\gamma-2]$ if $\bar{E}$ is of type $\FE(\gamma)$ and $[2,1,3,2,3,(2)_{\gamma-2},1,\gamma]$ if $\bar{E}$ is of type $\cH(\gamma)$. This chain supports a fiber $F$ of a $\P^{1}$-fibration of $X$, which meets $D\hor$ only once in $C_{2}'$, once in $B_{2}$ and once in $U_{1}$. Recall from Section \ref{sec:fibrations} that for a vertical curve $C$ we denote by $\mu(C)$ its multiplicity in the respective fiber. Here we have $\mu(C_{2}')=2$ and $\mu(B_{2})=\mu(U_{1})=1$, so $F\cdot D=4$. Because $T_{2}'$ is a horizontal $(-2)$-tip of $D$, this $\P^{1}$-fibration satisfies \eqref{eq:C3st-condition}.
	
	Consider the type $\FZb(\gamma)$ (see Figure \ref{fig:FZb}). Now $$C_{2}'+U_2+A+U_{1}+B_{1}+T_{1}'+A'=[1,4,(2)_{\gamma-3},1,\gamma-1,3,(2)_{\gamma-3},1]$$ supports a fiber $F$ of a $\P^{1}$-fibration of $X$, which meets $D\hor$ only  twice in $C_{2}'$, once in $A'$ and once in $B_{1}$. We have $\mu(C_{2}')=\mu(B_{1})=\mu(A')=1$, so $F\cdot D=4$. Because $\ltip{V_{2}}$ is a horizontal component of a $(-2)$-twig of $D$, this $\P^{1}$-fibration satisfies \eqref{eq:C3st-condition}.
	
	Finally, consider the types $\cI$ and $\J$ (see Figures \ref{fig:I}, \ref{fig:J}). Then $V_{2}+C_{2}'+E+C_{1}'=[2,1,3,1]$ supports a fiber $F$ of a $\P^{1}$-fibration of $X$, which meets $D\hor$ only once in $C_{2}'$ and twice in $C_{1}'$. We have $\mu(C_{2}')=2$ and $\mu(C_{1}')=1$, so $F\cdot D=4$. Because $\ltip{V_{1}}$ is a horizontal component of a $(-2)$-twig of $D$, this $\P^{1}$-fibration satisfies \eqref{eq:C3st-condition}.
	
	The last statement of the lemma follows from Lemma \ref{lem:C3st} and Proposition \ref{prop:MMP}.
\end{proof}

\smallskip
\subsection{Existence and uniqueness for types $\Qb$ and $\Qa$.}\label{sec:q4}

The existence and uniqueness of curves of types $\Qb$ and $\Qa$ follows from the classification of rational planar quintics \cite[Theorem 2.3.10]{Namba_geometry_of_curves}. The proof sketched in loc.\ cit.\ is based on computations which are left as an exercise, so for completeness we give an independent geometric argument. We construct these curves from the Steiner tricuspidal quartic $\bar{C}=\FZa(4,1)$ using quadratic Cremona transformations. We prove their projective uniqueness, too. The existence and uniqueness of $\bar{C}$ itself follows, for example, from \cite[Theorem 3.5]{FLZa-_class_of_cusp}, but since we need its explicit form \eqref{eq:quartic}, we provide a direct argument.

\begin{lem}[The auxiliary quartic $\FZa(4,1)$] \label{lem:quartic} Up to a projective equivalence there exists a unique  pair $(\bar{C},(p_{1},p_{2},p_{3}))$, where $\bar{C}\subseteq \P^{2}$ is a rational quartic with three ordinary cusps  $p_{1}$, $p_{2}$, $p_{3}$. It has a parameterization $\theta\colon \P^{1}\to\P^{2}$ given by
\begin{equation}\label{eq:quartic}
[u:v]\mapsto [u^{2}v^{2}:v^{2}(u-v)^{2}:u^{2}(u-v)^{2}],
\end{equation}
in which case $p_{1}=[1:0:0]$, $p_{2}=[0:1:0]$ and $p_{3}=[0:0:1]$.
\end{lem}
\begin{proof}
Let $\bar{C}\subseteq \P^{2}$ be a quartic with three ordinary cusps $p_{1}$, $p_{2}$, $p_{3}$. Let $\sigma\colon \P^{2}\map \P^{2}$ be the standard quadratic transformation centered at $p_{1}$, $p_{2}$, $p_{3}$. Then $\sigma_{*}\bar{C}$ is a conic tangent to the exceptional lines of $\sigma^{-1}$. Conversely, if $\mm$ is a conic and $\ll_{1},\ll_{2},\ll_{3}$ are distinct lines tangent to $\mm$ then the standard quadratic transformation centered at $\ll_{j}\cap\ll_{k}$ for $j\neq k$ maps $\mm$ onto a quartic with ordinary cusps at the points which are images of $\ll_{j}$. Up to an automorphism of $\P^{2}$  this transformation is inverse to $\sigma$. We obtain a one-to-one correspondence between the classes of projective equivalence of pairs $(\bar{C},(p_{1},p_{2},p_{3}))$ and $(\mm,(\ll_{1},\ll_{2},\ll_{3}))$. The latter is unique. Taking $$\mm=\{[(u-v)^{2}:u^{2}:v^{2}]: [u:v]\in \P^{1}\} \text{\ \ and \ \ }\ll_{1}=\{x=0\},\ \ll_{2}=\{y=0\},\ \ll_{3}=\{z=0\}$$ we get $\bar{C}$ given by \eqref{eq:quartic}.
\end{proof}

\begin{prop}
	\label{prop:Qb}
	Up to a projective equivalence there exists a unique rational cuspidal curve (quintic) of type $\Qb$.
\end{prop}
\begin{proof}
We describe a one-to-one correspondence between classes of projective equivalence of pairs $(\bar{E},q_{1})$ and  $(\bar{C},(p_{1},r_{1}))$, where $\bar{E}\subseteq \P^{2}$ is of type $\Qb$, $q_{1}\in\Sing \bar{E}$, $\bar{C}\subseteq \P^{2}$ is a tricuspidal quartic with cusps $p_{1},p_{2},p_{3}$ and $r_{1}\in \bar{C}$ is a point such that
\begin{equation}\label{eq:Cpr}
\parbox{.9\textwidth}{for $j\in\{2,3\}$ the line $\ll_{j}$ joining $r_{1}$ with $p_{j}$ meets $\bar{C}\setminus \{r_{1},p_{j}\}$ in a (unique) point $r_{j}$ such that $p_{1},r_{2},r_{3}$ are collinear (see Figure \ref{fig:Qb_construction}).}
\end{equation}	
	\begin{figure}[htbp]
		\centering
		\begin{subfigure}{0.25\textwidth}
			\centering
			\includegraphics[scale=0.25]{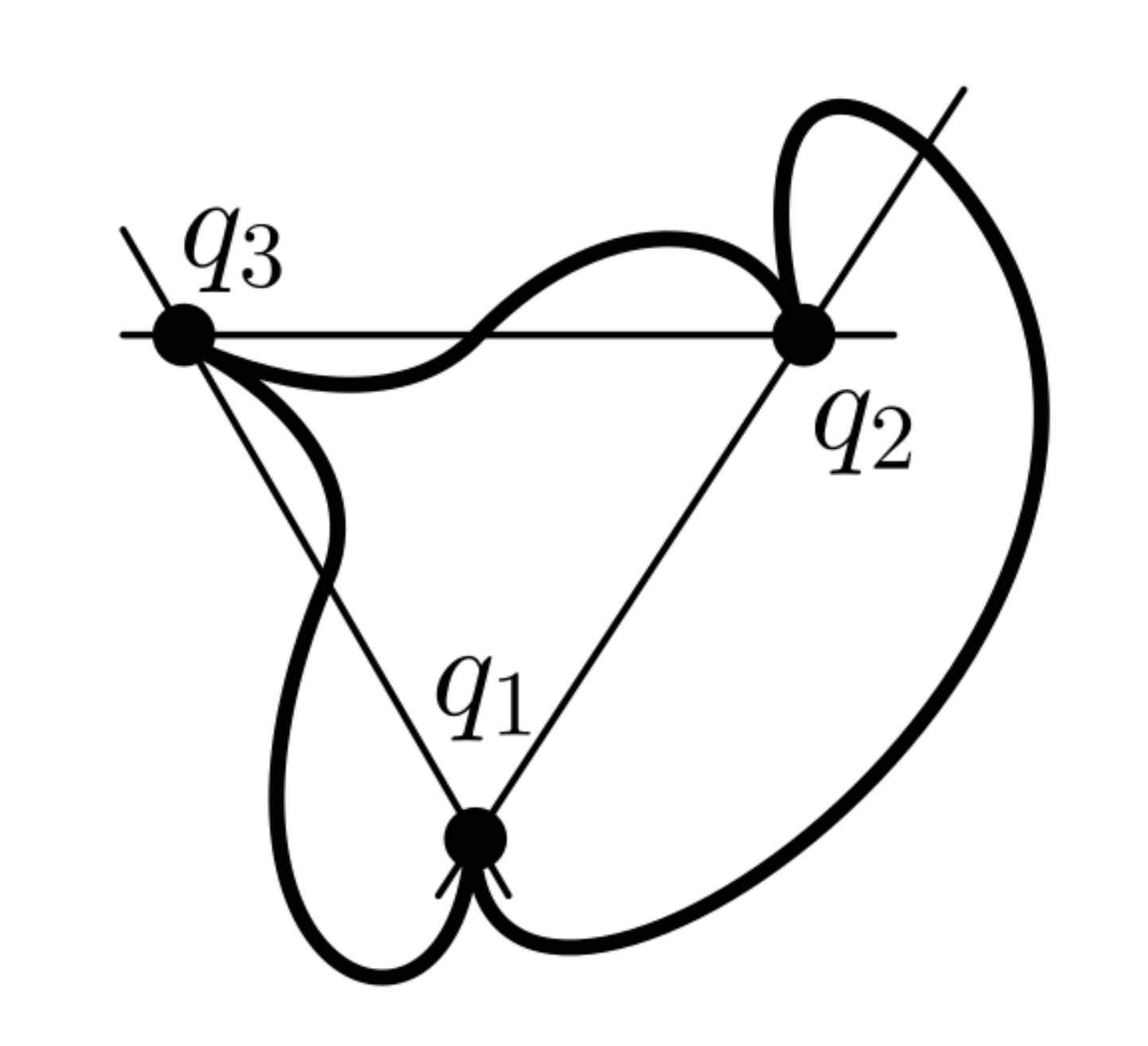}
			\caption{$\Qb$}
		\end{subfigure}$\longleftarrow$\hfill
		\begin{subfigure}{0.35\textwidth}
			\centering
			\includegraphics[scale=0.3]{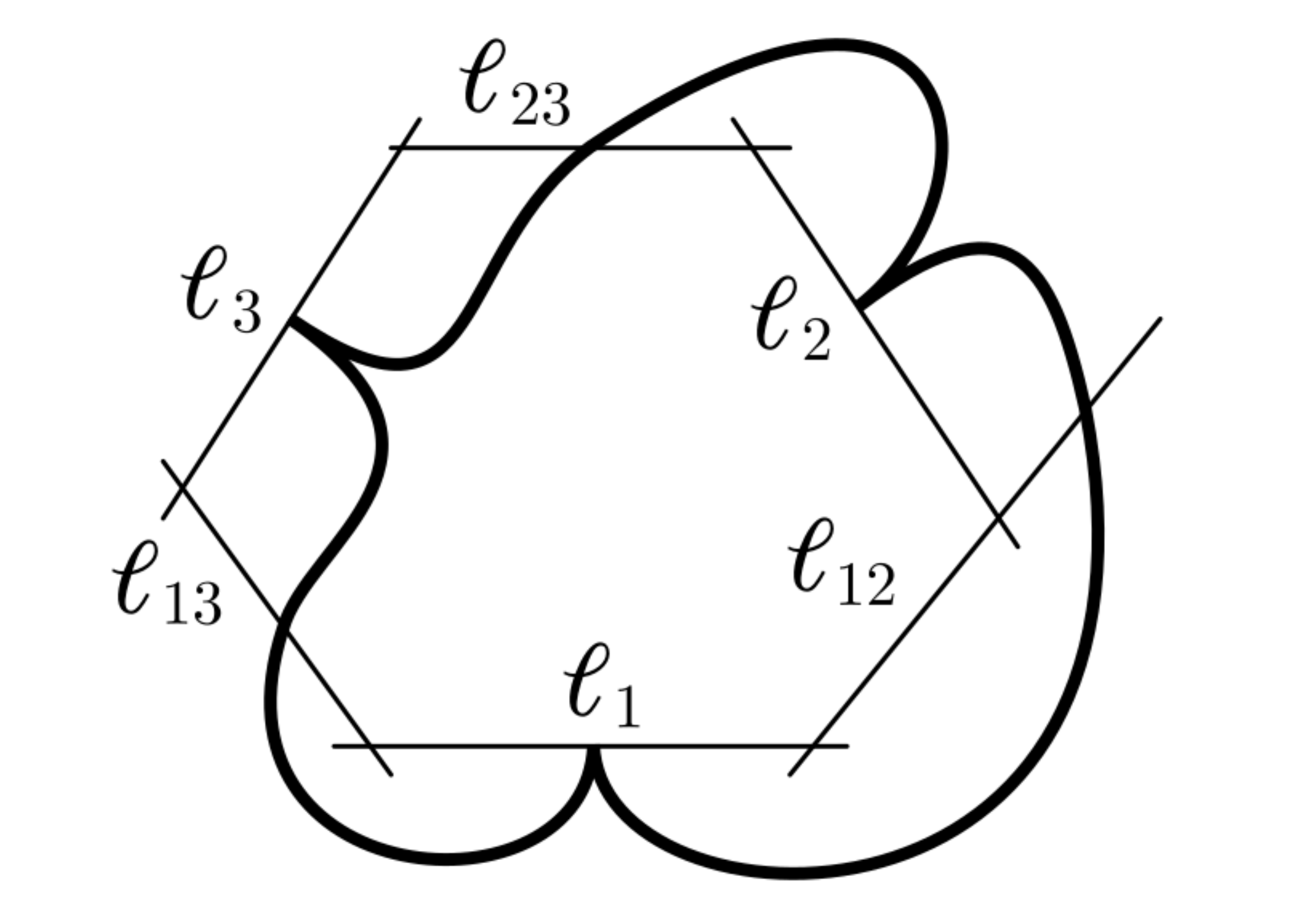}
		\end{subfigure}$\longrightarrow$\hfill
		\begin{subfigure}{0.3\textwidth}
			\centering
			\includegraphics[scale=0.25]{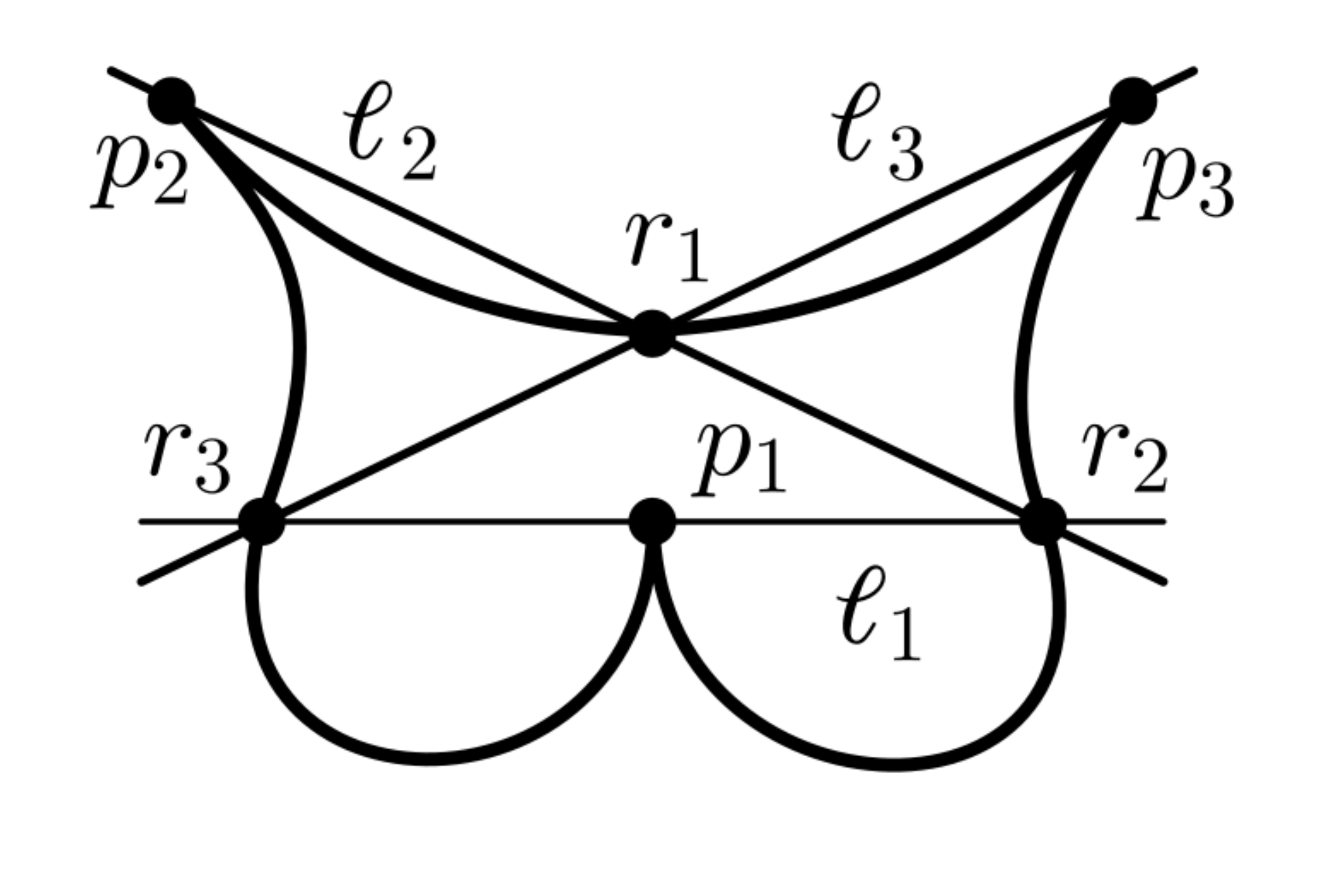}
			\caption{$\FZa(4,1)$}
		\end{subfigure}
		\caption{A construction of the quintic $\Qb$.}
		\label{fig:Qb_construction}
	\end{figure}
	
Assume that $\bar{E}$ is of type $\Qb$. Denote by $q_{1},q_{2},q_{3}$ the cusps of $\bar{E}$ and by $\ll_{jk}$ the line joining $q_{j}$ with $q_{k}$. Let $\sigma$ be the standard quadratic transformation centered at $q_{1},q_{2},q_{3}$. Put $\bar{C}\de \sigma_{*}\bar{E}$. Let $p_{j}$ be the point infinitely near to $q_{j}$ on $\bar{C}$ and let $\ll_{j}$ be the exceptional line of $\sigma^{-1}$ containing $p_{j}$. Put $r_{1}\de \sigma(\ll_{23})$ and $r_{j}\de \sigma(\ll_{1j})$ for $j\in \{2,3\}$. Lemma \ref{lem:special_lines}\ref{item:t} implies that the lines $\ll_{jk}$ are not tangent to $\bar{E}$, so $\ll_{jk}$ meets $\bar{E}\setminus \{q_{j},q_{k}\}$ once and transversally. Therefore, the points $p_{1},p_{2},p_{3},r_{1},r_{2},r_{3}$ lie on $\bar{C}$ and are distinct. The curve $\bar{C}$ is cuspidal, with cusps $p_{1}$, $p_{2}$, $p_{3}$. The multiplicity sequence of $p_{j}\in \bar{C}$ is the multiplicity sequence of $q_{j}\in\bar{E}$ shortened by the initial term, so $p_{j}\in\bar{C}$ is ordinary. It follows that $\bar{C}$ is a tricuspidal quartic. Because $r_{1},p_{j},r_{j}\in \ll_{j}$ for $j\in \{2,3\}$ and $p_{1},r_{2},r_{3}\in\ll_{1}$, the pair $(\bar{C},(p_{1},r_{1}))$ satisfies \eqref{eq:Cpr}.

Conversely, assume that $(\bar{C},(p_{1},r_{1}))$ is as in \eqref{eq:Cpr}. Then the standard quadratic transformation centered at $r_{1},r_{2},r_{3}$ maps $\bar{C}$ to a curve of type $\Qb$. Up to an automorphism of $\P^{2}$ this transformation is inverse to $\sigma$. Therefore, we have a one-to-one correspondence between classes of projective equivalence of pairs $(\bar{E},q_{1})$ and $(\bar{C},(p_{1},r_{1}))$.

By Lemma \ref{lem:quartic} it remains to show that the class of $(\bar{C},(p_{1},r_{1}))$ satisfying \eqref{eq:Cpr} is uniquely determined by the class of $(\bar{C},p_{1})$. Let $(\bar{C},p_{1})$ be as in \eqref{eq:quartic} and let $r_{1}$ be a smooth point of $\bar{C}$, so $r_{1}=\theta[\alpha:1]$ for some $\alpha\in \C\setminus \{0,1\}$. Then for $j\in \{2,3\}$ the lines $\ll_{j}$ joining $r_{1}$ with $p_{j}$ are given by $$\ll_{2}=\{(\alpha-1)^{2}x=z\}\text{\ \ and\ \ } \ll_{3}=\{(\alpha-1)^{2}x=\alpha^{2}y\}.$$ Now $\ll_{j}\cap\bar{C}=\{r_{1},p_{j},r_{j}\}$, where $r_{2}=\theta[2-\alpha:1]$ and $r_{3}=\theta[\alpha:2\alpha-1]$. We check that the points $p_{1}$, $r_{2}$, $r_{3}$ are collinear if and only if $(\alpha-1)^{2}=\alpha$. Then $r_{1}=[1:\alpha^{-1}:\alpha]$ and $\alpha=(3\pm\sqrt{5})/2$. The two possible points $r_{1}$, namely $[2:3+\sqrt{5}:3-\sqrt{5}]$ and $[2:3-\sqrt{5}:3+\sqrt{5}]$, are mapped to each other by the involution $[x:y:z]\mapsto [x:z:y]$ which fixes $(\bar{C},p_{1})$.
\end{proof}
 
\begin{lem}\label{lem:normalization_Aut} For every curve $C\subseteq \P^2$ other than a line, assigning to an automorphism of $(\P^2,C)$ its pullback to the normalization $\nu\:C^\nu\to C$ defines a monomorphism
\begin{equation}\Aut(\P^2,C)\hookrightarrow \Aut(C^\nu,\nu^{-1}(\Sing C)).
\end{equation}
\end{lem}

\begin{proof}Let $\sigma\in \Aut(\P^2,C)$. By the universal property of the normalization, the morphism $\sigma\circ \nu$ factors as $\sigma\circ \nu=\nu\circ\sigma^\nu$ for some $\sigma^\nu\:C^\nu\to C^\nu$. Since $\sigma^\nu$ is birational and lifts $\sigma_{|C}$, it is unique and we have $\sigma^\nu\in \Aut(C^{\nu},\nu^{-1}(\Sing C))$. It follows from the uniqueness that the assignment $\sigma\mapsto\sigma^\nu$ is a homomorphism. If $\sigma^\nu=\id_{C^\nu}$ then $\sigma_{|C}=\id$, so $\sigma=\id$, because $C$ spans $\P^2$.
\end{proof}

\begin{rem}[Properties of $\Qb$]\label{rem:Q3}
Let $\alpha=(3- \sqrt{5})/2$. We argue that $\Qb$ has a parameterization
\begin{equation}\label{eq:Qb_param}
	[u:v]\mapsto [u^{2}v^{2}(u- \alpha v):
	v^{2}(u-v)^{2}((1-\alpha)u+v):
	u^{2}(u-v)^2((\alpha-1)u+v)],
\end{equation}
and that
\begin{equation}
\Aut(\P^{2},\Qb)\cong\Z_{3}.
\end{equation}

The parameterization, call it $\nu\:\P^1\to \bar E$, follows from the construction in the proof of Proposition \ref{prop:Qb}, the three cusps are $[1:0:0]$, $[0:1:0]$ and $[0:0:1]$.

To compute the automorphism group note first that the automorphism $\epsilon([x:y:z])=[\alpha y:z:(\alpha-1)x]$ fixes $\bar E$ and cyclically permutes the cusps $q_1=\nu([1:1])$, $q_2=\nu([0:1])$ and $q_3=\nu([1:0])$, hence $\Z_{3}\cong\langle\epsilon\rangle \subseteq \Aut(\P^{2},\bar{E})$. Suppose that $\sigma \in  \Aut(\P^{2},\bar{E},q_{1})$ and $\sigma\neq \id$. Let $\sigma^\nu$ be as in Lemma \ref{lem:normalization_Aut}. We have $\sigma^{\nu}([u:v])=[v:u]$, so $\sigma([x:y:z])=[x:z:y]$. But this automorphism does not fix $\bar E$, as for instance, the inverse images of the unique points of intersection of $\bar E\setminus \{q_1,q_2\}$ with the lines through $q_1$ and $q_j$, $j\in\{2,3\}$, namely $[1-\alpha:1]$ and $[\alpha-1:1]$, are not mapped to each other; a contradiction.
\end{rem}

\begin{prop}
	\label{prop:Qa}
	Up to a projective equivalence there exists a unique rational cuspidal curve (quintic) of type $\Qa$.
\end{prop}
\begin{proof}
We describe a one-to-one correspondence between classes of projective equivalence of pairs $(\bar{E},q_{2})$ and  $(\bar{C},(p_{1},s_{1}))$, where $\bar{E}\subseteq \P^{2}$ is of type $\Qa$, $q_{2}\in \bar{E}$ is an ordinary cusp, $\bar{C}\subseteq \P^{2}$ is a tricuspidal quartic with cusps $p_{1}$, $p_{2}$, $p_{3}$ and $s_{1}\in \bar{C}$ is a point such that
	\begin{equation}\label{eq:Cps}
	\mbox{the line tangent to $\bar{C}$ at $s_{1}$ is tangent to $\bar{C}\setminus \{s_{1}\}$,}
	\end{equation}	
see Figure \ref{fig:Q4}, that is, this line is \emph{bitangent} to $\bar{C}$. Because $\deg\bar{C}=4$, it meets $\bar{C}$ in exactly two smooth points, with multiplicity $2$.
	\begin{figure}[htbp]
		\centering
		\begin{subfigure}{0.3\textwidth}
			\centering
			\includegraphics[scale=0.25]{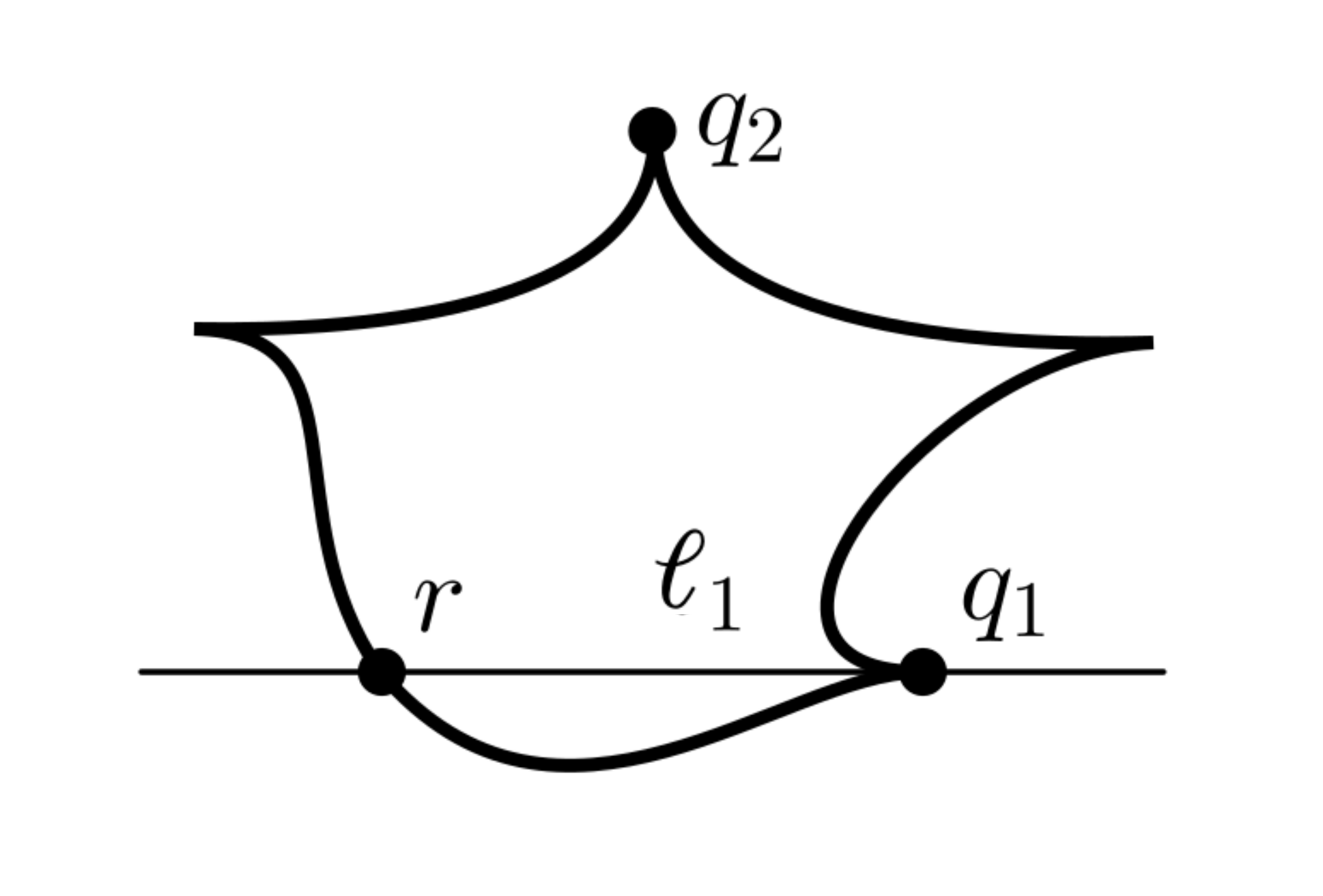}
			\caption{$\Qa$}
		\end{subfigure}$\longleftarrow$\hfill
		\begin{subfigure}{0.3\textwidth}
			\centering
			\includegraphics[scale=0.25]{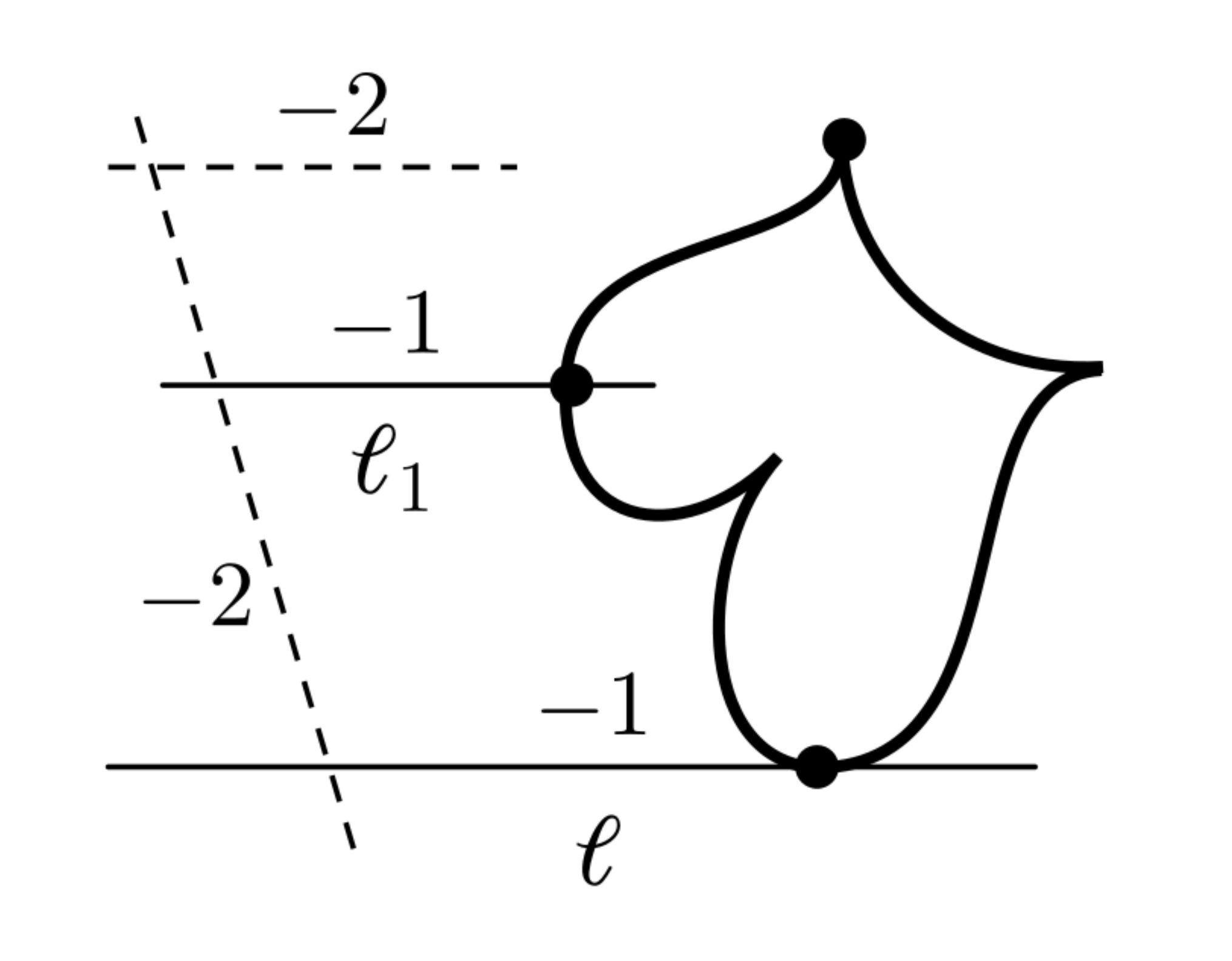}
		\end{subfigure}$\longrightarrow$\hfill
		\begin{subfigure}{0.3\textwidth}
			\centering
			\includegraphics[scale=0.25]{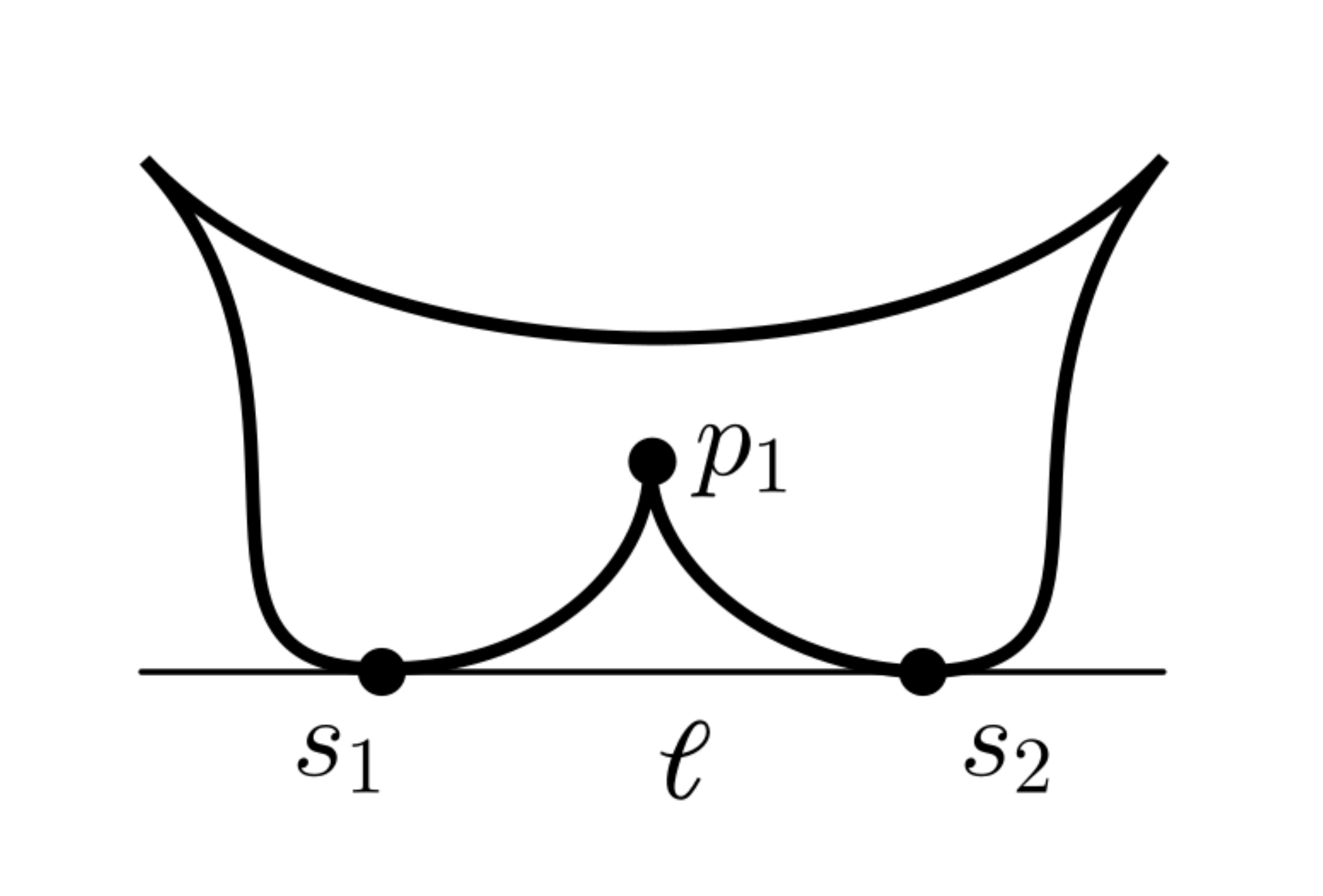}
			\caption{$\FZa(4,1)$}
		\end{subfigure}
		\caption{A construction of the quintic $\Qa$.}
		\label{fig:Q4}
	\end{figure}
	
Assume that $\bar{E}$ is of type $\Qa$. Recall that by Lemma \ref{lem:special_lines}\ref{item:t} the line $\ll_{1}$ tangent to the cusp $q_{1}\in \bar{E}$ with multiplicity sequence $(2,2,2)$ satisfies $(\ll_{1}\cdot\bar{E})_{q_{1}}=4$ and meets $\bar{E}\setminus \{q_{1}\}$ in a unique point, say $r$.  Blow up three times at $q_{1}$ and its infinitely near points on the proper transform of $\bar{E}$. The exceptional divisor is a chain $V_{1}+V_{2}+L=[2,2,1]$, meeting the proper transforms of $\ll_{1}$ and $\bar{E}$ in $V_{2}$ and $L$, respectively. Contract the proper transform of $\ll_{1}$, which is a $(-1)$-curve, and the image of $V_{1}+V_{2}$. Note that this contraction does not touch the proper transform of the germ of $\bar{E}$ at $q_{1}$, which is smooth and meets the image of $L$ with multiplicity $2$. Denote the resulting map by $\sigma\colon \P^{2}\map \P^{2}$ (it is a quadratic transformation with one proper base point, see \cite[Proposition 8.5.2]{AC_Cremona-maps}). Put $\bar{C}\de\sigma_{*}\bar{E}$, $p_{j}\de \sigma(q_{j+1})$ for $j\in \{1,2,3\}$ and $s_{1}\de \sigma(\ll_{1})$. Let $\ll$ be the image of $L$ and let $s_{2}\in \ll$ be the image of the point infinitely near to $q_{1}$, see Figure \ref{fig:Q4}. Then $\ll$ is a line, $s_{1},s_{2}\in \bar{C}$ are smooth and $(\bar{C}\cdot \ll)_{s_{k}}=2$ for $k\in\{1,2\}$. Since $\sigma$ restricts to an isomorphism $\P^{2}\setminus \ll_{1}\to\P^{2}\setminus \ll$, the curve $\bar{C}$ is a tricuspidal quartic with cusps $p_{1}$, $p_{2}$, $p_{3}$. Hence, $(\bar{C},(p_{1},s_{1}))$ satisfies \eqref{eq:Cps}. 

Conversely, let $(\bar{C},(p_{1},s_{1}))$ be as in \eqref{eq:Cps}. The line $\ll$ tangent to $\bar{C}_{1}$ at $s_{1}$ is bitangent to $\bar{C}$. Let $\sigma'\colon \P^{2}\map \P^{2}$ be a composition of three blowups at $s_{1}$ and its infinitely near points on the proper transforms of $\bar{C}$, followed by the contraction of the proper transform of $\ll$ and the images of the first and second exceptional curve. Then $\sigma'$ maps $\bar{C}$ to a curve of type $\Qa$. Up to an automorphism of $\P^{2}$ the map $\sigma'$ is inverse to $\sigma$. Therefore, we have a one-to-one correspondence between  classes of projective equivalence of pairs $(\bar{E},q_{2})$ and $(\bar{C},(p_{1},s_{1}))$.

By Lemma \ref{lem:quartic} it remains to show that the class of $(\bar{C},(p_{1},s_{1}))$ satisfying \eqref{eq:Cps} is uniquely determined by the class of $(\bar{C},p_{1})$. 
Let $(\bar{C},p_{1})$ be as in \eqref{eq:quartic}. Suppose that  $\bar{C}$ has two bitangent lines $\ll$, $\ll'$. Then the four points of tangency are distinct, so the projection from $\ll\cap \ll'$ restricts to a morphism $\bar{C}\to \P^{1}$ of degree $4$ ramified at the four points of $(\ll\cup \ll')\cap \bar{C}$ and at three cusps of $\bar{C}$. This contradicts the Hurwitz formula. Thus $\bar{C}$ has at most one bitangent line, say $\ll$. For $\bar{C}$ as in \eqref{eq:quartic} this is the line $\ll=\{x+y+z=0\}$, meeting $\bar{C}$ with multiplicity $2$ at
\begin{equation*}
s_{1}=\theta[-\zeta:1]=[1:\zeta^{2}:\zeta]\text{\ \ and\ \ } s_{2}=\theta[-\zeta^{2}:1]=[1:\zeta:\zeta^{2}],
\end{equation*}
where $\zeta=\exp(2\pi\imath /3)$. Therefore, the pairs $(\bar{C},(p_{1},s_{k}))$ for $k\in \{1,2\}$ satisfy \eqref{eq:Cps}. The result follows, because the points $s_{1}$, $s_{2}$ are mapped to each other by the involution $[x:y:z]\mapsto [x:z:y]$ which fixes $(\bar{C},p_{1})$.
\end{proof}

\begin{rem}[Properties of $\Qa$]\label{rem:Q4}
We argue that $\Qa$ has a parameterization 
\begin{equation}\label{eq:Qa-par}
[u:v]\mapsto [uv^{4}:v^{2}(u^{3}-v^{3}):u^{2}(u^{3}+2v^{3})], 
\end{equation}
 and that
\begin{equation}\label{eq:Aut(Q4)}
\Aut(\P^{2},\Qa)\cong\Z_{3}.
\end{equation}

The parameterization, call it $\nu\:\P^1\to \bar E$, follows from the construction from the proof of Proposition \ref{prop:Qa}, it is also given in \cite[2.3.10.6]{Namba_geometry_of_curves}. The curve $\bar{E}$ has a cusp with multiplicity sequence $(2,2,2)$ at $q_1=\nu([1:0])$ and ordinary cusps $q_{2+k}=\nu([-\zeta^{k}:\sqrt[3]{2}])$, where $\zeta=\exp(2\pi\imath/3)$ and $k\in \{0,1,2\}$.

To compute the automorphism group we use Lemma \ref{lem:normalization_Aut}. Note that  $\epsilon([x:y:z])=[\zeta x:y:\zeta^{2} z]$ fixes $\bar E$ and cyclically permutes the ordinary cusps, hence $\Z_{3}\cong\langle\epsilon\rangle \subseteq \Aut(\P^{2},\bar{E})$. If $\sigma \in  \Aut(\P^{2},\bar{E},q_{2})$ then $\sigma^\nu$ as in Lemma \ref{lem:normalization_Aut} fixes $[-1:\sqrt[3]{2}]$, $[1:0]$ and $\{[-\zeta:\sqrt[3]{2}], [-\zeta^2:\sqrt[3]{2}]\}$, which is possible only if $\sigma^\nu=\id_{\P^1}$. Thus \eqref{eq:Aut(Q4)} holds.
\end{rem}

\begin{rem}[Other proofs of existence for types $\Qb$, $\Qa$]\label{rem:BZ_quintics} \ 
\begin{enumerate}
	\item Let $\iota_{3},\iota_{4}\colon \C^{*}\to \C^{2}$ be the injective morphisms given, respectively, by (v) and (w) in \cite{BoZo-annuli}. Via an automorphism of $\C^{2}$ they are equivalent to $\bar{E}\setminus \ll_{1}\into \P^{2}\setminus \ll_{1}$, where $\bar{E}$ is of type $\Qb$ or $\Qa$, respectively, and $\ll_{1}$ is the line tangent to $q_{1}\in \bar{E}$.
	Indeed, let $j_{3},j_{4}\colon \C^{2}\into \P^{2}$ be embeddings given by 
	\begin{equation*}
	j_{3}(x,y)=[x:4(1+\sqrt{5})y-x^2:1],\quad  j_{4}(x,y)=[x:8y-x^2:1].
	\end{equation*}
	Then the closures of the images of $j_{3}\circ \iota_{3}$ and $j_{4}\circ\iota_{4}$ are of type $\Qb$ and $\Qa$, respectively. 

	\item\label{item:Moe_quintics} A construction of curves of types $\Qb$ and $\Qa$, similar to ours, that is,  using quadratic transformations applied to simple planar configurations, is given in \cite[Section 6.3]{Moe-cuspidal_MSc}.
\end{enumerate}
\end{rem}

\subsection{Existence and uniqueness for types $\FZb$, $\FE$ and $\cH$.}\label{sec:discussion}
The existence and uniqueness of the curves $\FZb(\gamma)$, $\gamma\geq 4$ is proved in \cite{FlZa_cusps_d-3} by showing that some quadratic transformation maps a curve of type $\FZb(\gamma)$ to a curve of type $\FZb(\gamma+1)$. The existence and uniqueness of Fenske curves $\FE$ is shown in \cite{Fenske_cusp_d-4} by a similar method. 

The type $\cH(\gamma)$ is realized by the closure of the embedding $\mathbf{sq-}(k)$, $k=\gamma-1$ from \cite[6.9.3]{CKR-Cstar_good_asymptote}, given by the formula in Theorem 8.2(iii) loc.\ cit. Indeed, that closure has two points at infinity, which are cusps described by Hamburger--Noether pairs $\binom{4}{4\gamma-2}\binom{2}{3}$ and $\binom{3\gamma-3}{3\gamma}\binom{3}{1}$, so their multiplicity sequences are $((4)_{\gamma-1},2,2,2)$ and $(3\gamma-3, (3)_{\gamma-1})$, respectively (cf.\ \cite[Lemma 2.11]{PaPe_Cstst-fibrations_singularities}). Hence it is of type $\cH(\gamma)$. This proves that curves of type $\cH(\gamma)$ do exist.

Assume that $\bar{E}$ is of type $\cH(\gamma)$, $\gamma \geq 3$. By Lemma \ref{lem:special_lines}\ref{item:l} the line $\ll_{12}$ joining the cusps of $\bar{E}$ does not meet the smooth part of $\bar{E}$, so $\bar{E}\setminus \ll_{12}\subseteq \P^{2}\setminus \ll_{12}$ is a proper embedding $\C^{*}\into \C^{2}$. By Lemma \ref{lem:special_lines}\ref{item:t} the line $\ll_{1}$ tangent to $q_{1}\in \bar{E}$ meets $\bar{E}\setminus \ll_{12}$ once and transversally, so it is a good asymptote for $\bar{E}\setminus \ll_{12}\subseteq \P^{2}\setminus \ll_{12}$. The classification \cite[Theorem 8.2]{CKR-Cstar_good_asymptote} of proper embeddings $\C^{*}\into \C^{2}$ which admit a good asymptote implies that the embedding $\bar{E}\setminus \ll_{12}\subseteq \P^{2}\setminus\ll_{12}$ is unique up to an automorphism of $\P^{2}\setminus \ll_{12}\cong \C^{2}$. To infer the projective uniqueness of $\bar{E}\subseteq \P^{2}$, one needs to show that any automorphism of $(\P^{2}\setminus \ll_{12},\bar{E}\setminus\ll_{12})$ extends to an automorphism of $(\P^{2},\bar{E})$. This is done in \cite[Lemma 4.6]{PaPe_Cstst-fibrations_singularities} for closures of some other embeddings, but the proof for $\cH(\gamma)$ is exactly the same. It relies of the fact that the surface $\P^{2}\setminus (\bar{E}\cup \ll_{12})$, being of log general type, admits a unique minimal log smooth completion, which in turn is uniquely determined by the singularities of $\bar{E}+\ll_{12}$. We leave the details to the reader.

\bigskip 
\begin{rem}[Other proofs of existence for types $\FZb$, $\FE$, and $\cH$]\label{rem:FZb-FE-H}\ 
\begin{enumerate}
\item\label{item:FZb} A curve $\bar{E}$ of type $\FZb(k+2)$ for $k\geq 2$ can be also obtained as the closure of the image of a singular embedding $\C^{*}\to \C^{2}$ given by \cite[(g)]{BoZo-annuli} via the standard embedding $\C^{2}\ni(x,y)\mapsto[x:y:1]\in \P^{2}$. Then the line $\{z=0\}$ at infinity is the line $\ll_{12}$ joining $q_{1},q_{2}\in \bar{E}$. Another way to get this curve is to note that the line $\ll_{1}$ tangent to $q_{1}\in \bar{E}$ meets $\bar{E}$ in two points, so $\bar{E}\setminus \ll_{1}\subseteq\P^{2}\setminus \ll_{1}$ is a singular embedding $\C^{*}\to \C^{2}$. One can check that it is given by \cite[(k)]{BoZo-annuli} via the embedding $(x,y)\mapsto [x:y-\sum_{l=1}^{k}a_{l}x^{k-l+2}:1]$, where $a_{1}=1$ and $a_{l+1}=(-1)^{l}\left(\tbinom{3l-2}{l}+4\tbinom{3l-2}{l-1}\right)-\sum_{r=1}^{l}(-1)^{r}\tbinom{3r}{r}a_{l+1-r}$ for $l\geq 1$.	
\item\label{item:FE} A curve $\bar{E}$ of type $\FE(k+3)$ for $k\geq 2$ can be obtained as the closure of the image of a singular embedding $\C^{*}\to \C^{2}$ given by \cite[(h)]{BoZo-annuli}  via the standard embedding $\C^{2}\ni(x,y)\mapsto[x:y:1]\in \P^{2}$. As in \ref{item:FZb}, the line $\{z=0\}$ at infinity is the line $\ll_{12}$ joining $q_{1},q_{2}\in \bar{E}$. Choosing for the line at infinity the one tangent to $q_{1}\in\bar{E}$, we get another singular embedding $\C^{*}\to \C^{2}$, and one can check that it is given by \cite[(p)]{BoZo-annuli} via the embedding $(x,y)\mapsto [x:y-\sum_{l=1}^{k}a_{l}x^{k-l+2}:1]$, where $a_{1}=1$ and $a_{l+1}=(-1)^{l}\left(\tbinom{4l-2}{l}+3\tbinom{4l-2}{l-1}\right)-\sum_{r=1}^{l}(-1)^{r}\tbinom{4r}{r}a_{l+1-r}$ for $l\geq 1$.
\item\label{item:H_BZ} Similarly, a curve $\bar{E}$ of type $\cH(k+1)$ for $k\geq 2$ can be obtained as the closure of the image of an embedding $\C^{*}\to \C^{2}$ given by \cite[(i)]{BoZo-annuli} via the standard embedding $\C^{2}\ni(x,y)\mapsto[x:y:1]\in \P^{2}$. Then the line at infinity joins the two cusps of $\bar{E}$. Again, one can check that $\bar{E}\setminus \ll_{1}\subseteq \bar{E}\setminus \ll_{1}$, where $\ll_{1}$ is the line tangent to $q_{1}\in\bar{E}$, is a singular embedding $\C^{*}\to \C^{2}$ given by  \cite[(o)]{BoZo-annuli} with parameters $(m,n)$ equal to $(1,k)$, via the embedding $\C^{2}\ni (x,y)\mapsto [y:x-\sum_{l=1}^{k}a_{l}y^{k-l+2}:1]\in \P^{2}$, where $a_{1}=1$ and $a_{l+1}=(-1)^{l}\left(\tbinom{4l-2}{l}-\tbinom{4l-2}{l-1}\right)-\sum_{r=1}^{l}(-1)^{r}\tbinom{4r}{r}a_{l+1-r}$ for $l\geq 1$.
\item\label{item:H_Bodnar} An inductive construction of the series $\cH$ by quadratic transformations was given recently by J.\ Bodnár, see \cite[Theorem 3.1(e)]{Bodnar_type_G_and_J}.		

\item\label{item:H_CKR} Theorem \ref{thm:main} and \cite[Theorem 1.3]{PaPe_Cstst-fibrations_singularities} yield the following characterization. A rational cuspidal curve $\bar{E}\subseteq\P^{2}$ with complement of log general type is of type $\cH(\gamma)$ for some $\gamma\geq 3$ if and only if the surface $\P^{2}\setminus \bar{E}$ admits no $\C^{**}$-fibration and $\bar{E}$ is a closure of the image of a smooth proper embedding $\C^{*}\into\C^{2}$ with a good asymptote. Note that \cite[Theorem 8.2(iii)]{CKR-Cstar_good_asymptote} allows $\gamma=k-1=2$ which we do not, because in this case $\bar{E}$ is of type $\cA(2,2,1)$ from \cite[Theorem 1.3]{PaPe_Cstst-fibrations_singularities}, so $\P^{2}\setminus \bar{E}$ has a $\C^{**}$-fibration.
\end{enumerate}
\end{rem}

\smallskip
\subsection{Existence and uniqueness for types $\cI$ and $\cJ$.}\label{sec:IJ}
To prove the existence and uniqueness of curves of types $\cI$ and $\cJ$ we first show the uniqueness of minimal models $(X_{\min},\tfrac{1}{2}D_{\min})$ of the corresponding log surfaces $(X_{0},\tfrac{1}{2}D_{0})$ which appeared in Section \ref{sec:possible_HN-types}. Recall that $\F_{2}\de\P(\O_{\P^{1}}(2)\oplus \O_{\P^{1}})$ and denote by $\tilde{\Delta}^{-}$ the negative section of $\F_{2}$.

\begin{lem}[Recovering $(Z,D_{Z})$, cf.\ \eqref{eq:theta}]\label{lem:Z}
The following configurations exist and are unique up to an isomorphism of pairs:
\begin{enumerate}
	\item\label{item:Z_J} $(\P^{2},\tilde{C}_{1}+\tilde{C}_{2}+\tilde{E}_{0})$, where $\tilde{C}_{1}$, $\tilde{C}_{2}$ are conics meeting with multiplicities $3$, $1$ and $\tilde{E}_{0}$ is a line tangent to $\tilde{C}_{1}$  and $\tilde{C}_{2}$ off $\tilde{C}_{1}\cap \tilde{C}_{2}$ (see Figure \ref{fig:J}).
	\item\label{item:Z_I} $(\F_{2},\tilde{\Delta}^{-}+\tilde{C}_{1}+\tilde{C}_{2}+\tilde{E}_{0})$, where $\tilde{C}_{1}$, $\tilde{C}_{2}$, $\tilde{E}_{0}$ are $1$-sections such that $\tilde{C}_{1}\cdot \tilde{\Delta}^{-}=1$, $\tilde{C}_{2}\cdot\tilde{\Delta}^{-}=\tilde{E}_{0}\cdot\tilde{\Delta}^{-}=0$, $\tilde{C}_{1}$ meets $\tilde{C}_{2}$ in two points (with multiplicities $1$, $2$) and $\tilde{E}_{0}$ meets each $\tilde{C}_{j}$, $j\in \{1,2\}$ in a unique point off $\tilde{C}_{1}\cap\tilde{C}_{2}$ (with multiplicity $4-j$), (see Figure \ref{fig:I}).
\end{enumerate}
\end{lem}
\begin{proof}
	\ref{item:Z_J} Since up to a projective equivalence there exists a unique conic with an ordered triple of distinct points on it, the triple $(\tilde{C}_{1}+\tilde{E}_{0},p,p')$, where $\tilde{E}_{0}$ is a line tangent to a conic $\tilde{C}_1$ and $p,p'$ are distinct points of $\tilde{C}_1\setminus \tilde{E}_{0}$, is unique up to a projective equivalence, say, $\tilde{C}_{1}=\{x^{2}=yz\}$, $\tilde{E}_{0}=\{z=0\}$, $p=[1:1:1]$ and $p'=[0:0:1]$. The pencil of conics tangent to $\tilde{C}_{1}$ with multiplicity $3$ at $p'$ and passing through $p$ is given by $$\{\lambda (x^{2}-yz)=\mu y(y-x)\}_{[\lambda:\mu] \in \P^{1}}.$$ It contains a unique smooth member tangent to $\tilde{E}_{0}$, namely $\tilde{C}_{2}=\{x^{2}-yz=4y(y-x)\}$ .
	
	\ref{item:Z_I}  Let $\tilde{C}_{1}+\tilde{C}_{2}+\tilde{E}_{0}$ be as in \ref{item:Z_J}. Denote by $p$ the point where $\tilde{C}_{1}$ and $\tilde{C}_{2}$ meet transversally and by $\ll$ the line tangent to $\tilde{C}_{2}$ at $p$. Let $\theta\colon \P^{2}\map \F_{2}$ be a blow up at $p$ and its infinitely near point on the proper transform of $\ll$ followed by the contraction of the latter. Then $\theta_{*}(\tilde{C}_{1}+\tilde{C}_{2}+\tilde{E}_{0})$ is as in \ref{item:Z_I}, where $\theta_{*}\tilde{C}_{1}$, $\theta_{*}\tilde{C}_{2}$ and $\theta_{*}\tilde{E}_{0}$ correspond to $\tilde{C}_{1}$, $\tilde{E}_{0}$ and $\tilde{C}_{2}$, respectively. Conversely, let $\tilde{C}_{1}+\tilde{C}_{2}+\tilde{E}_{0}$ be as in \ref{item:Z_I}, let $p'$ be the point where $\tilde{C}_{1}$ and $\tilde{C}_{2}$ meet transversally and let $F$ be the fiber through $p'$. Let $\eta\colon \F_{2}\map \P^{2}$ be a blowup at $p'$ followed by the contraction of the proper transform of $F+\tilde{\Delta}^{-}$. Then $\eta_{*}(\tilde{C}_{1}+\tilde{C}_{2}+\tilde{E}_{0})$ is as in \ref{item:Z_J}, where $\eta_{*}\tilde{C}_{1}$, $\eta_{*}\tilde{C}_{2}$ and $\eta_{*}\tilde{E}_{0}$ correspond to $\tilde{C}_{1}$, $\tilde{E}_{0}$ and $\tilde{C}_{2}$, respectively. Clearly, $\eta$ and $\theta$ are inverse to each other. Hence, \ref{item:Z_I} follows from \ref{item:Z_J} and from the universal property of blowing up.
\end{proof}

\begin{prop}\label{prop:IJ_existence}
	Planar rational cuspidal curves of types $\cI$ and $\cJ(k)$, $k\geq 2$ exist and are unique up to a projective equivalence.
\end{prop}
\begin{proof}
	For a construction of a curve of type $\cI$ let $(Z,D_{Z})$ be as in Lemma \ref{lem:Z}\ref{item:Z_I}, see Figure \ref{fig:I}. Write $\tilde{C}_{1}\cap\tilde{C}_{2}=\{p,p'\}$  where $(\tilde{C}_{1}\cdot \tilde{C}_{2})_{p}=2$, $(\tilde{C}_{1}\cdot \tilde{C}_{2})_{p'}=1$. Let $\alpha_{2}^{+}\colon X_{2}\to Z$ be the blowup at $p$ and its infinitely near point on the proper transform of $D_{Z}$. Put $\hat{C}_{j}\de (\alpha_{2}^{+})^{-1}_{*}\tilde{C}_{j}$, $j\in \{1,2\}$ and let $\Upsilon_{2}\subseteq X_{2}$ be the last exceptional curve of $\alpha_{2}^{+}$, so $\Upsilon_{2}^{2}=-1$ and  $\Upsilon_{2}+\hat{C}_{1}+\hat{C}_{2}$ is circular. Now let $\psi\colon X_{0}\to X_{2}$ be the composition of blowups over $\hat{C}_{1}\cap\hat{C}_{2}$, $\hat{C}_{1}\cap\Upsilon_{2}$ whose centers are double points of the subsequent preimages of $\Upsilon_{2}+\hat{C}_{1}+\hat{C}_{2}$ (so $\psi^{*}(\Upsilon_{2}+\hat{C}_{1}+\hat{C}_{2})\redd$ is circular, too), such that $\psi^{-1}(\hat{C}_{1}\cap\hat{C}_{2})\redd$ and  $\psi^{-1}(\hat{C}_{1}\cap\Upsilon_{2})\redd$ are chains $[1,2]$ and $[3,2,1,3]$, respectively, meeting $\psi^{-1}_{*}\hat{C}_{1}$ in their first tips.
			
	For a construction of a curve of type $\cJ(k)$, $k\geq 2$ let $(Z,D_{Z})$ be as in Lemma \ref{lem:Z}\ref{item:Z_J}, see Figure \ref{fig:J}. Write $\tilde{C}_{1}\cap\tilde{C}_{2}=\{p,p'\}$  where $(\tilde{C}_{1}\cdot \tilde{C}_{2})_{p}=1$, $(\tilde{C}_{1}\cdot \tilde{C}_{2})_{p'}=3$. Blow up three times at $p'$ and its infinitely near points on the proper transforms of $D_{Z}$ and denote this morphism by $\alpha_{2}^{+}\colon X_{2}\to Z$. As before, put $\hat{C}_{j}\de (\alpha_{2}^{+})^{-1}_{*}\tilde{C}_{j}$, $j\in \{1,2\}$ and denote by  $\Upsilon_{2}\subseteq X_{2}$ the last exceptional curve of $\alpha_{2}^{+}$. Now let $\psi\colon X_{0}\to X_{2}$ be the composition of blowups over $\hat{C}_{1}\cap\hat{C}_{2}$, $\Upsilon_{2}\cap\hat{C}_{2}$ and at double points of the subsequent preimages of $\Upsilon_{2}+\hat{C}_{1}+\hat{C}_{2}$, such that $\psi^{-1}(\hat{C}_{1}\cap\hat{C}_{2})\redd$ and  $\psi^{-1}(\Upsilon_{2}\cap\hat{C}_{2})\redd$ are chains $[(2)_{k-1},1,k+1]$ and $[1,(2)_{k-1}]$, respectively, meeting $\psi^{-1}_{*}\hat{C}_{2}$ in their last tips. 
	
	For both types put $\psi^{+}=\alpha_{2}^{+}\circ \psi$, $E_{0}=(\psi^{+})^{-1}_{*}\tilde{E}_{0}$ and write $((\psi^{+})^{*}D_{Z})\redd=D_{0}+A+A'$, where $A$ and $A'$ are the $(-1)$-curves in the preimages of $p$ and $p'$, respectively. Computing the changes of self-intersection numbers of the components of $D_{Z}$ we infer that connected components of $D_{0}-E_{0}$ contract to smooth points. The resulting surface has Picard rank $$\rho(X_{0})-\#(D_{0}-E_{0})=\rho(X_{0})-\#((\psi^{+})^{*}D_{Z})\redd+3=\rho(Z)-\#D_{Z}+3=1,$$ so it is $\P^{2}$. Looking at the weighted graph of $D_{0}$ we see that the image of $E_{0}$ is a cuspidal curve of type $\cI$ and $\cJ(k)$, respectively. Therefore, such curves do exist.
	
	Let $\bar{E}\subseteq \P^{2}$ be of type $\cI$ or $\cJ(k)$ for some $k\geq 2$. We prove the projective uniqueness of $\bar{E}$ in each case. As before, let $\pi_{0}\colon (X_{0},D_{0})\to(\P^{2},\bar{E})$ be the minimal weak resolution. We use Notation \ref{not:graphs} for the components of $D_{0}$.  By the results of Section \ref{sec:kappa} the surface $\P^{2}\setminus \bar{E}$ satisfies \eqref{eq:assumption}. By Lemma \ref{lem:beta_flat}\ref{item:n0} the pair $(X_{0},\tfrac{1}{2}D_{0})$ is not almost minimal. In Section \ref{sec:possible_HN-types} we have shown that for such curves we can run the almost MMP for $(X_{0},\tfrac{1}{2}D_{0})$ as in Propositions \ref{prop:F2_I} and \ref{prop:J}. In particular, $n=2$ and the proper transforms $A,A'\subseteq X_{0}$ of the contracted almost log exceptional curves are $(-1)$-curves such that $A\cdot D_{0}=A'\cdot D_{0}=2$ and 
\begin{enumerate}
	\item\label{item:AA'_I} if $\bar{E}$ is of type $\cI$ then $A$ meets $\ftip{T_{1}}$ and $\ftip{T_{2}}$, and $A'$ meets $C_{1}$ and the $(-2)$-tip of $Q_{2}$ meeting $C_{2}$, see Figure \ref{fig:I}.
	\item\label{item:AA'_J} if $\bar{E}$ is of type $\cJ$ then $A$ meets $\ftip{T_{1}'}$ and $\ftip{T_{2}}$, and $A'$ meets $\ltip{T_{1}}$ and $\ftip{T_{2}'}$, see {Figure \ref{fig:J}}.
\end{enumerate}
The above properties specify the intersection numbers of $A$ and $A'$ with all the components of $D_{0}$. Since the N\'eron-Severi group $\NS_{\Q}(X_{0})$ is generated freely by the classes of components of $D_{0}$, the numerical classes of $A$ and $A'$ are uniquely determined. Since $A^{2}<0$ and $(A')^{2}<0$, the curves $A$, $A'$ are unique.

Now we prove that a morphism $\psi^{+}=\alpha_{2}^{+}\circ\psi\colon (X_{0},D_{0})\to (Z,D_{Z})$ as in Section \ref{sec:possible_HN-types}, see Figures \ref{fig:I} and \ref{fig:J}, is uniquely determined by $\bar{E}\subseteq \P^{2}$.  We check directly that in our cases $\Exc\psi_{A}$ and $\Exc\psi_{A'}$ are disjoint (see Lemmas \ref{lem:kappa=2} and \ref{lem:ale_pr-tr}\ref{item:Exc-psi_i-disjoint}). Let $\psi\colon (X_0,D_0)\to (X_2,D_2)$ be the contraction of $\Exc\psi_{A}+\Exc\psi_{A'}$. The morphism $\psi$ is uniquely determined by $(X_{0},D_{0}+A+A')$, see Definition \ref{def:psi_A}, hence by $(X_{0},D_{0})$, and in consequence by $\bar{E}\subseteq \P^{2}$. Let $\Upsilon_{2}+\Delta_{2}^{+}\subseteq X_{2}$ be as in Notation \ref{not:MMP} (see Figure \ref{fig:Upsilon}). Let $\alpha_{2}^{+}\colon (X_{2},D_{2})\to (Z,D_{Z})$ be the contraction of $\Upsilon_{2}+\Delta_{2}^{+}$ (in our cases $\Upsilon_{2}$ is a unique $(-1)$-curve in $D_{2}$ and $\Delta_{2}^{+}$ is a unique $(-2)$-twig of $D_{2}$ meeting $\Upsilon_{2}$). We check that for  $\bar{E}$ of type $\cI$ or $\cJ(k)$ the pair $(Z,D_{Z})$ is as in Lemma  \ref{lem:Z}\ref{item:Z_I} or \ref{item:Z_J}, respectively (see Propositions \ref{prop:F2_I} and \ref{prop:J}). Therefore, we have shown that a curve $\bar{E}\subseteq \P^{2}$ of type $\cI$ or $\cJ(k)$ uniquely determines, via $\psi^{+}=\alpha_{2}^{+}\circ\psi$, a pair $(Z,D_{Z})$ as in Lemma \ref{lem:Z}\ref{item:Z_I} or \ref{item:Z_J}, respectively, together with a pair of points $p,p'\in\tilde{C}_{1}\cap\tilde{C}_{2}$ which are centers of $\psi^{+}$.

Conversely, given a pair $(Z,D_{Z})$ as in Lemma \ref{lem:Z}\ref{item:Z_I} or \ref{item:Z_J}, there is a unique sequence of blowups over $\tilde{C}_{1}\cap\tilde{C}_{2}=\{p,p'\}$ such that the weighted graph of the total transform of $D_{Z}$ is the same as that of $D_{0}+A+A'$, where $D_{0}$ is as required for the minimal weak resolution of $\bar{E}\subseteq \P^{2}$ of type $\cI$ or $\cJ(k)$, respectively, and $A+A'$ is as in \ref{item:AA'_I} and \ref{item:AA'_J}. Indeed, since the preimage of each of these points is a chain which has a unique $(-1)$-curve and meets the proper transform of $D_{Z}$ in tips, we see by induction that the center of each blowup is uniquely determined as the common point of a specific pair of components of the preimage of $D_{Z}$. Therefore, we have  a one-to-one correspondence between the isomorphism classes of pairs $(Z,D_{Z})$ and $(X_{0},D_{0})$, and consequently, of pairs $(\P^{2},\bar{E})$ of respective type. Thus the projective uniqueness of $\bar{E}$ follows from the uniqueness of $(Z,D_{Z})$ proved in Lemma \ref{lem:Z}.
\end{proof}

\begin{rem}[An alternative way to obtain $A$ and $A'$]\label{rem:AA'}
In the above proof of the projective uniqueness of $\bar{E}\subseteq \P^{2}$ the key role is played by the curves  $A,A'\subseteq X_{0}$ used to reconstruct the process $\psi$ of almost minimalization. We have obtained them using the description of $\psi$ from Section \ref{sec:possible_HN-types}. Below we sketch how to find them directly using the geometry of $\bar{E}\subseteq \P^{2}$. As we will see, the curve $\pi_{0}(A)$ for type $\cI$ is a line and $\pi_{0}(A)$, $\pi_{0}(A')$ for type $\cJ$ are conics. But for type $\cI$ the curve $\pi_{0}(A')$ is a specific quartic, more difficult to see directly on $\P^{2}$.

\ref{item:AA'_I} (Type $\cI$, see Figure \ref{fig:I}). The curve $A$ is the proper transform of the line $\ll_{12}$ from Lemma \ref{lem:special_lines}\ref{item:l}. In order to construct $A'$ let $\theta$ be the contraction of $A$ and of all new $(-1)$-curves in the subsequent images of $D_{0}$ followed by the blowup at the image of $C_{2}\cap E_{0}$ and its infinitely near point on the proper transform of $E_{0}$. Let $V$ and $C_{2}\s$ be, respectively, the first and the second exceptional curve over $C_{2}\cap E_{0}$. Put $\hat{D}_{0}\de \theta_{*}D_{0}+V+C_{2}\s$. We have $(\theta_{*}E_{0})^{2}=E^{2}+\tau_{1}=0$, so $|\theta_{*}E_{0}|$ induces a $\P^{1}$-fibration, see Figure \ref{fig:I_fibration}. 
\begin{figure}[ht]
	\includegraphics[scale=0.3]{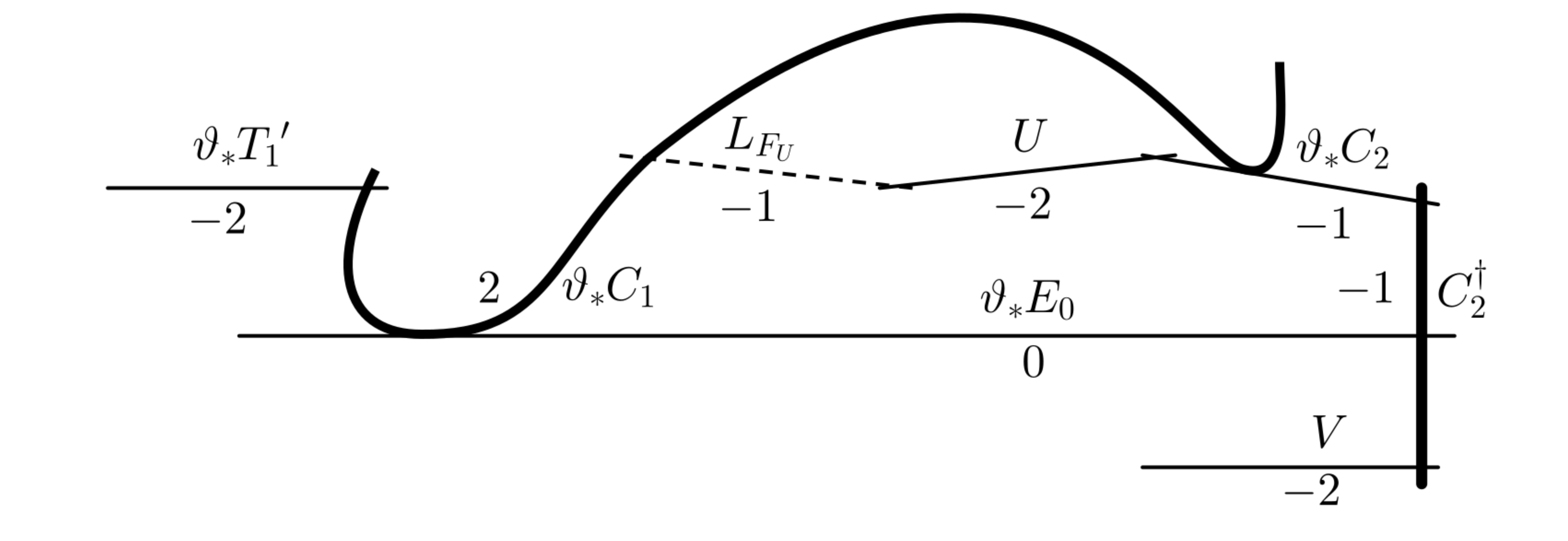}
	\caption{The $\P^{1}$-fibration from Remark \ref{rem:AA'}\ref{item:AA'_I}.}
		\label{fig:I_fibration}
\end{figure}
We use Notation \ref{not:fibrations_h_and_nu} for this $\P^{1}$-fibration. The horizontal part of $\hat{D}_{0}$ consists of a $3$-section $\theta_{*}C_{1}$ and a $1$-section $C_{2}\s$. The divisor $(\hat{D}_{0}-\theta_{*}E_{0})\vert$ has three connected components, namely $V=[2]$, $\theta_{*}T_{1}'=[2]$ and a chain $[2,1]$. The first component, say $U$, of the latter chain is a $(-2)$-tip of $\hat{D}_{0}$. The second one equals $\theta_{*}C_{2}$ and meets $\theta_{*}C_{1}$ and $C_{2}\s$ with multiplicities $2$ and $1$, respectively. In particular, the multiplicity $\mu(\theta_{*}C_{2})$ of $\theta_{*}C_{2}$ in the fiber equals $\mu(\theta_{*}C_{2})\theta_{*}C_{2}\cdot C_{2}\s\leq \theta_{*}E_{0} \cdot C_{2}\s=1$, so  $\mu(\theta_{*}C_{2})=1$. Denote by $F_{V}$ and $F_{U}$ the fibers containing $V$ and $U$, respectively. Lemma \ref{lem:fibrations-Sigma-chi} implies that every fiber $F\neq \theta_{*}E_{0}$ has a unique component, say $L_{F}$, not contained in $(\hat{D}_{0})\vert$. We have $\theta_{*}C_{2}\not\subseteq F_{V}$, because $V$ and $\theta_{*}C_{2}$ meet the same $1$-section $C_{2}\s$. Since $F_{V}\neq V+L_{F_{V}}$, it follows that $F_{V}$ contains $\theta_{*}T_{1}'$, so $F_{U}\wedge (\hat{D}_{0})\vert$ is connected. Because $\mu(\theta_{*}C_{2})=1$, $\theta_{*}C_{2}$ is a tip of $F_{U}$, so the contractibility of $F_{U}$ to a $0$-curve implies that $L_{F_{U}}^{2}=-1$ and that $L_{F_{U}}$ meets $(\hat{D}_{0})\vert$ only in $U$. Moreover, $$L_{F_{U}}\cdot C_{2}\s=1-(F_{U}-L_{F_{U}})\cdot C_{2}\s=0\text{\ \ and\ \ }L_{F_{U}}\cdot \theta_{*}C_{1}=3-(F_{U}-L_{F_{U}})\cdot \theta_{*}C_{1}=1,$$ so $A'\de \theta^{-1}_{*}L_{F_{U}}$ satisfies the required conditions.

\ref{item:AA'_J} (Type $\cJ(k)$, $k\geq 2$, see Figure \ref{fig:J}). Consider the decomposition $\pi_{0}=\theta\circ \eta$, where $\theta$ is a composition of a blow up at $q_{2}$ and three blowups over $q_{1}$ and its infinitely near points on the proper transforms of $\bar{E}$. Let $\ll_{1},\ll_{12}\subseteq \P^{2}$ be as in Lemma \ref{lem:special_lines}. Put  $L_{1}=\theta^{-1}_{*}\ll_{1}$, $L_{12}=\theta^{-1}_{*}\ll_{12}$, denote by $V_{j}$ the last exceptional curve of $\theta^{-1}$ over $q_{j}$, $j\in \{1,2\}$ and put $W=\Exc\theta-V_{1}-V_{2}$. Looking at the multiplicity sequences of $q_{1},q_{2}\in \bar{E}$ listed in Table \ref{table:nofibrations}, we show that for $j\in \{1,2\}$ the curve $\theta^{-1}_{*}\bar{E}$ has a cusp $q_{j}'\in V_{j}\setminus W$ with multiplicity sequence $(2)_{k}$. Recall from Lemma \ref{lem:special_lines}\ref{item:t} that $\ll_{1}$ meets $\bar{E}\setminus \{q_{1}\}$ once and transversally. We have $\deg\bar{E}=\mu_{1}+\mu_{2}+1$ and $\mu_{2}'=2>1$ (see Table \ref{table:nofibrations}), so arguing as in the proof of Lemma \ref{lem:special_lines} we show that $\ll_{12}$ meets $\bar{E}\setminus \{q_{1},q_{2}\}$ once and transversally, too. It follows that $L_{1}\cdot \theta^{-1}_{*}\bar{E}=L_{12}\cdot\theta^{-1}_{*}\bar{E}=1$. Since $W\cdot \theta^{-1}_{*}\bar{E}=0$, the linear system of $L_{1}+W+L_{12}=[1,2,2,1]$ induces a $\P^{1}$-fibration such that $\theta^{-1}_{*}\bar{E}$ is a $2$-section. Lemma \ref{lem:fibrations-Sigma-chi} implies that $L_{1}+W+L_{12}$ is the unique degenerate fiber. In particular, for $j\in \{1,2\}$ the fiber $F_{j}$ passing through $q_{j}'$ is smooth. Now $F_{j}\cdot W=0$, $F_{j}\cdot V_{j'}=1$ for $j'\in \{1,2\}$ and $F_{j}$ meets $\theta^{-1}_{*}\bar{E}$ only in $q_{j}'$, with multiplicity $2$. It follows that $\eta$ touches $\eta^{-1}_{*}F_{j}$ exactly once. Then $(\eta^{-1}_{*}F_{j})^{2}=-1$ and we see that $A\de\eta^{-1}_{*}F_{1}$, $A'\de\eta^{-1}_{*}F_{2}$ satisfy the required conditions.
\end{rem}

\begin{rem}[A new proof of existence and uniqueness of curves $\FZb$, $\FE$ and $\cH$] The above procedure can also be applied to construct rational cuspidal curves of types $\FZb$, $\FE$ and $\cH$ and to give a geometric proof of their projective uniqueness. We sketch the argument, leaving the details to the reader. First, one shows that the following configurations $(Z, D_{Z})$ are unique up to an isomorphism of pairs:
\begin{enumerate}
	\item\label{item:Z_FZb} $(\P^{2},\tilde{E}_{0}+\tilde{C}_{1}+\tilde{C}_{2})$, where $\tilde{E}_{0}$ is a cuspidal cubic and $\tilde{C}_{j}$, $j\in \{1,2\}$ is a line tangent to $\tilde{E}_{0}$ with multiplicity $j+1$, see Figure \ref{fig:FZb}.
	\item\label{item:Z_FE} $(\F_{2},\tilde{\Delta}^{-}+\tilde{C}_{1}+\tilde{C}_{2}+\tilde{E}_{0})$, where $\tilde{C}_{1}$ is a fiber of the unique $\P^{1}$-fibration of $\F_{2}$,
	$\tilde{C}_{2}$ is a $1$-section and $\tilde{E}_{0}$ is a rational cuspidal $2$-section such that $\tilde{C}_{2}+\tilde{E}_{0}$ is disjoint from $\tilde{\Delta}^{-}$, $\tilde{E}_{0}$ is tangent to $\tilde{C}_{1}$ off $\tilde{C}_{2}$ and meets $\tilde{C}_{2}$ with multiplicities $3$, $1$, see Figure \ref{fig:FE}.
	\item\label{item:Z_H} $(\P^{2},\tilde{E}_{0}+\tilde{C}_{1}+\tilde{C}_{2})$, where $\tilde{E}_{0}$, $\tilde{C}_{1}$ are conics meeting with multiplicities $3$, $1$ and $\tilde{C}_{2}$ is a line tangent to $\tilde{E}_{0}$ and $\tilde{C}_{1}$ off $\tilde{E}_{0}\cap \tilde{C}_{1}$, see Figure \ref{fig:H} (cf.\ Lemma \ref{lem:Z}\ref{item:Z_J}).	
\end{enumerate}	
For a construction of curves of type $\FZb$, $\FE$ or $\cH$ let $(Z,D_{Z})$ be as in \ref{item:Z_FZb}, \ref{item:Z_FE} or \ref{item:Z_H}, respectively. As in the proof of Proposition \ref{prop:IJ_existence}, we see that it is possible to choose a morphism $\psi^{+}\colon (X_{0},D_{0}+A+A')\to (Z,D_{Z})$, with weighted graphs as in Figures \ref{fig:FZb}, \ref{fig:FE} and \ref{fig:H}, respectively. Then the connected components of $D_{0}-(\psi^{+})^{-1}_{*}\tilde{E}_{0}$ contract to points on $\P^{2}$ and the image of $(\psi^{+})^{-1}_{*}\tilde{E}_{0}$ is a rational cuspidal curve of the respective type.

To see the projective uniqueness let $\bar{E}\subseteq \P^{2}$ be of type $\FZb$, $\FE$ or $\cH$. Then there exist unique $(-1)$-curves $A,A'\subseteq X_{0}$ such that $A\cdot D_{0}=A'\cdot D_{0}=2$, $A$ meets $\ftip{T_{1}}$ and $\ftip{T_{2}}$ and $A'$ meets $\ftip{T_{1}'}$ and $E_{0}$. Indeed, $A$ and $A'$ can be taken as the proper transforms of the lines $\ll_{12}$ and $\ll_{1}$ from Lemma \ref{lem:special_lines}, respectively. Again, since the numerical classes of $A$, $A'$ are uniquely determined by their intersection numbers with the components of $D_{0}$, we infer from the inequalities $A^{2}<0$, $(A')^{2}<0$ that $A$, $A'$ are unique. We check that $\psi^{+}\colon (X_{0},D_{0})\to (Z,D_{Z})$ defines a bijection between the classes of projective equivalence of $\bar{E}\subseteq \P^{2}$ and the isomorphism classes of pairs $(Z,D_{Z})$. Therefore, in each case $\bar{E}\subseteq\P^{2}$ is unique up to a projective equivalence.
\end{rem}

\begin{rem}[Other proofs of existence for types $\cI$ and $\cJ$]\label{rem:IJ}\ 
	\begin{enumerate}
		\item\label{item:BZ_IJ} A curve of type $\cI$ can also be obtained as the closure of the embedding $\C^{*}\into \C^{2}$ given by \cite[(t)]{BoZo-annuli} via the standard embedding $\C^{2}\ni (x,y)\mapsto [x:y:1]\in \P^{2}$. A curve of type $\cJ(k)$ for $k\geq 2$ can be obtained as a closure of the singular embedding $\C^{*}\to\C^{2}$ given by \cite[(n)]{BoZo-annuli}, with parameters $(l,m,n)$ equal to $(1,k+1,1)$. But for this one needs to use a very specific embedding $\C^{2}\into \P^{2}$, given by
		\begin{equation*}
		(x,y)\mapsto [x-y:4x-(x-y)^2:1].
		\end{equation*}
		\item\label{item:Bodnar_IJ}  M. Zaidenberg informed us that the existence of curves of type $\cJ(k)$ was most likely known to T.\ tom Dieck, who listed their multiplicity sequences in his private correspondence with H.\ Flenner in 1995. An inductive construction of this series by Cremona maps was given recently by J.\ Bodnár \cite[Theorem 3.1(c)]{Bodnar_type_G_and_J}.
		\item\label{item:KP_I} The classification of proper smooth embeddings  $\C^{*}\into \C^{2}$ will be established in a forthcoming article \cite{KoPa-SporadicCstar2}. The most exceptional case is the curve of type $\cI$. Together with  \cite[Theorem 1.3]{PaPe_Cstst-fibrations_singularities} one gets the following characterization. A rational cuspidal curve $\bar{E}\subseteq \P^{2}$ with a complement of log general type is of type $\cI$ if and only if the surface $\P^{2}\setminus \bar{E}$ admits no $\C^{**}$-fibration and $\bar{E}$ is the closure of the image of a smooth proper embedding $\C^{*}\into \C^{2}$ with no good asymptote.
	\end{enumerate}
\end{rem}

\smallskip
\subsection{A direct proof of nonexistence of $\C^{**}$-fibrations.}\label{sec:no_Cstst}

The proof of nonexistence of $\C^{**}$-fibrations of $\P^{2}\setminus \bar{E}$ given in the beginning of Section \ref{sec:existence} relies on the classification result \cite[Theorem 1.3]{PaPe_Cstst-fibrations_singularities}. Here we give a direct proof of this fact.

\setcounter{claim}{0}
\begin{prop}[Complements are not $\C^{**}$-fibered]\label{prop:no_Cstst}
	If $\bar{E}\subseteq \P^{2}$ is a rational cuspidal curve of one of the types listed in Definition \ref{def:our_curves} then $\P^{2}\setminus \bar{E}$ admits no $\C^{**}$-fibration.
\end{prop}
\begin{proof}
	Suppose that $\P^{2}\setminus \bar{E}$ has a $\C^{**}$-fibration. By \cite[Proposition 3.3]{PaPe_Cstst-fibrations_singularities} we can choose one with no base point on $X$, where $\pi\colon (X,D)\to (\P^{2},\bar{E})$ is the minimal log resolution. Then this fibration extends to a $\P^{1}$-fibration $X\to \P^{1}$ such that $F\cdot D=3$ for every fiber $F$. We use Notation \ref{not:fibrations_h_and_nu} and for the components of $D$ we use the notation introduced at the beginning of Section \ref{sec:existence}. In particular, $E=\pi^{-1}_{*}\bar{E}$, $Q_{j}'=\pi^{-1}(q_{j})\redd$ for $j\in \{1,\dots, c\}$ and $C_{j}'$ is the last component of $Q_{j}'$, that is, the unique $(-1)$-curve in $Q_{j}'$. Recall also that for types listed in Definition \ref{def:our_curves} we have $c\in \{2,3,4\}$ and $E^{2}\leq -3$, see Table \ref{table:nofibrations}.
	
	\begin{claim}\label{cl:E+C_j'-not-vertical}
		For every $j\in \{1,\dots, c\}$ the divisor $E+C_{j}'$ is not vertical.
	\end{claim}
	\begin{proof}
	Suppose that $E+C_{j}'$ is contained in some fiber, say, $F_{E}$. Because $\beta_{D}(C_{j}')=3$ and because fibers of $\P^{1}$-fibrations contain no branching $(-1)$-curves, we have $C_{j}'\cdot D\hor\geq 1$, and if the equality holds then $C_{j}'$ is not a tip of $F_{E}$, so its multiplicity $\mu(C_{j}')$ in $F_{E}$ is at least $2$. In any case, we get $\mu(C_{j}')C_{j}'\cdot D\hor\geq 2$. It follows that $C_{j'}'$ is horizontal for $j'\neq j$ and 
	\begin{equation*}
	\mu(C_{j}')C_{j}'\cdot D\hor \leq (F_{E}-E)\cdot D\hor= 3-(c-1)\leq 2,
	\end{equation*}
	so the equalities hold. In particular, $c=2$ and $D\hor$ meets $F_{E}$ only in $C_{j}'+E$. Since $D$ is connected, every connected component of $D\vert$ meets $D\hor$, so the set $F_{E}\cap D$ is connected. As a consequence, $F_{E}\subseteq D$: indeed, if $F_{E}$ has a component $L\not\subseteq D$ then $L\cdot D\hor=0$ and, since $F_{E}$ is a rational tree, $L\cdot D\vert=1$ and $L\cong \P^{1}$, contrary to Lemma \ref{lem:Qhp_has_no_lines}. We have $E\cdot (D\vert-C_{j}')=0$. If $C_{j}'$ is a tip of $F_{E}$ then $F_{E}=C_{j}'+E=[1,1]$ and if $C_{j}'$ is not a tip of $F_{E}$ then $\mu(C_{j}')=2$ and, since $C_{j}'$ is the unique $(-1)$-curve in $F_{E}$, $F_{E}=[2,1,2]$. In both cases we get a contradiction with $E^{2}\leq -3$.
\end{proof}

	\begin{claim}\label{cl:E-horizontal}
	The curve $E$ is horizontal and $D$ contains no fiber.
	\end{claim}
\begin{proof}
	Suppose that $E$ is contained in some fiber, say $F_{E}$. It follows from Claim \ref{cl:E+C_j'-not-vertical} that $C_{1}'+\dots+C_{c}'$ is horizontal, so $D\vert$ has no $(-1)$- or $0$-curves. In particular, $D$ contains no fiber. Recall that $E$ meets $D-E$ only in $C_{1}',\dots, C_{c}'$. As a consequence, $3=F_{E}\cdot D\hor \geq \mu(E) E\cdot D\hor = \mu(E)c\geq 2\mu(E)$, so $\mu(E)=1$ and $(F_{E}-E)\cdot D\hor=3-c$. The connectedness of $D$ gives $$b_{0}((F_{E}-E)\cap D)\leq (F_{E}-E)\cdot D\hor=3-c\leq 1,$$ so $(F_{E}-E)\cap D$ is connected. For every component $L$ of $F_{E}$ not contained in $D\vert$ we have $L\cong \P^{1}$, so $L\cdot D\geq 2$ by Lemma \ref{lem:Qhp_has_no_lines}. Because the fiber $F_{E}$ has no loops, such $L$ is unique and meets $D_{0}$ in $E$ and in $(F_{E}-E)\cap D$. In particular, the latter set is nonempty, so the above inequalities imply that $c=2$. There is no $(-1)$-curve in $D\vert$, so $L$ is a unique $(-1)$-curve in $F_E$. In particular, $\mu(L)\geq 2$. Both connected components of $F_{E}\cap D$ meet sections in $D$, so they contain components of multiplicity $1$. Lemma \ref{lem:singular_P1-fibers}\ref{item:adjoint_chain} implies that $F_{E}$ is a chain of type $[\gamma,1,(2)_{\gamma-1}]$, where $\gamma=-E^{2}\geq 3$, and meets $D\hor$ in tips. Thus $D$ contains a twig $V=[(2)_{\gamma-1}]$ and $L$ meets $\ftip{V}$. Because $c=2$, $\bar{E}$ is of type $\cH$, $\cI$ or $\cJ$. In the first case (see Figure \ref{fig:H}) the existence of $V$ implies that $\gamma=3$ and that $V$ meets $C_{1}'$, so $\pi(L)^{2}=6$; a contradiction. In the second case (see Figure \ref{fig:I}) we get that $V$ meets $C_{1}'$ and $\pi(L)^{2}=10$; a contradiction. In the third case (see Figure \ref{fig:J}) $\ftip{V}$ is either the first component of $Q_{1}'$, or, if $k\geq 3$, the fourth component of $Q_{1}'$ or the second component of $Q_{2}'$. We get $\pi(L)^{2}=0,3,1$, respectively, so the last case holds. Then $\pi(L)$ is a line and $\deg\bar{E}=(\pi(L)\cdot\bar{E})_{q_{2}}+1=2k+3$. But $\deg\bar{E}=4k+1$, so $k=1$; a contradiction. The second assertion follows because the intersection matrix of $D\vert\subseteq D-E$ is negative definite.
\end{proof}

\begin{claim}
	\label{cl:E-section}
	The curve $E$ is a $1$-section.
\end{claim}
\begin{proof}
	Claim \ref{cl:E-horizontal} and Lemma \ref{lem:fibrations-Sigma-chi} imply that $E$ is not a $3$-section. Suppose that $E$ is a $2$-section. Then $H\de D\hor - E$ is a section contained in, say, $Q_{j_0}'$ for some $j_{0}\in\{1,\dots, c\}$. Consequently, $Q_{j}'$ is vertical for every $j\neq j_{0}$. The restriction $p|_{E}\colon E\to \P^{1}$ has degree $2$ and is ramified at every $E\cap C_{j}'$ for $j\neq j_{0}$, so the Hurwitz formula gives $c\leq 3$. For $j\neq j_{0}$ let $F_{j}$ be the fiber containing $Q_{j}'$ and let $\mu_{j}$ be the multiplicity of $q_{j}\in\bar{E}$. Because our $\P^{1}$-fibration factors through the contraction of $Q_{j}'$, say $\pi_{j}$, the divisor $F_{j}-\pi_{j}^{-1}(q_{j})$ is effective, so the projection formula gives $2\leq \mu_{j}=\pi^{-1}_{j}(q_{j})\cdot E\leq F_{j}\cdot E=2$. It follows that $\mu_{j}=2$ and $F_{j}$ meets $E$ only in $Q_{j}'$, so $C_{j'}'\not \subseteq F_{j}$ for $j'\neq j$. Because $\bar{E}$ has at most three cusps and at most one of them is not semi-ordinary, $\bar{E}$ is of type $\Qb$ (see Table \ref{table:nofibrations}) and we may assume that $j_{0}=1$. By Lemma \ref{lem:fibrations-Sigma-chi} for $j\in \{2,3\}$ there is a unique component of $F_{j}$ not contained in $D\vert$, say $L_{j}$. Then $$2\geq \mu(\pi_{j}(L_{j})) \pi_{j}(L_{j})\cdot\pi_{j}(E)\geq \mu(L_{j})\mu_{j}=2\mu(L_{j}),$$ so equalities hold. In particular, $\mu(L_{j})=1$ and $\pi_{j}(L_{j})$ is not tangent to $\pi_{j}(E_{j})$, so $L_{j}$ meets the first component of $Q_{j}'$. The image of $F_{j}-L_{j}$ contains no $(-1)$-curves, so $\pi_{j}(F_{j})=[0]$, hence $L_{j}^{2}=-1$ and $(F_{j})\redd=L_{j}+Q_{j}'$. We have $H\subseteq Q_{1}'=[2,3,1,2]$. The connected components of $Q_{1}-H$ are contained in different fibers, other than $F_{2}$ and $F_{3}$. It follows that $H$ is a tip of $Q_{1}'$, for otherwise one of those connected components is a $(-2)$-curve and the fiber containing it has at least two components not contained in $D\vert$, contrary to Lemma \ref{lem:fibrations-Sigma-chi}. Recall that $L_{2}$ satisfies $L_{2}^{2}=-1$, $L_{2}\cdot D=2$ and $L_{2}$ meets $D$ in $H$ and in the first component of $Q_{2}'$. It follows that $H$ is in fact the first component of $Q_{1}'$. Indeed, otherwise the contractions of $Q_{1}'$ and $Q_{2}'$ touch $L_{2}$ three times and once, respectively, so $\pi(L_{2})^{2}=3$, which is impossible. Therefore, $\pi(L_{2})^{2}=1$, so $\pi(L_{2})$ is a line and $$5=\deg\bar{E}=\bar{E}\cdot \pi(L_{2})=(\bar{E}\cdot \pi(L_{2}))_{q_{1}}+(\bar{E}\cdot \pi(L_{2}))_{q_{2}}=2+2;$$ a contradiction.
\end{proof}

	If for some $j\in \{1,\dots, c\}$ the curve $C_{j}'$ is vertical then we denote by $F_{j}$ the fiber containing it. In this case  $\mu(C_{j}')=1$, because $C_{j}'$ meets $E$, so $C_{j}'$ is a tip of $F_{j}$. We infer that for every $j\in \{1,\dots, c\}$, $(Q_{j}')\hor$ contains $C_{j}'$ or a component of $Q_{j}'$ meeting $C_{j}'$. Since $c\geq 2$, we see that $c=2$ and $D\hor-E$ consists of two $1$-sections, one in each $Q_{j}'$. The curve $\bar{E}\subseteq \P^{2}$ is of type  $\cH$, $\cI$ or $\cJ$. We check that in all three cases (see Figures \ref{fig:H}, \ref{fig:I} and \ref{fig:J}), every $C_{j}'$ meets a unique twig of $D$, say $V_{j}$, which is in fact a $(-2)$-twig, and that the component $B_{1}$ of $Q_{1}'-V_{1}$ meeting $C_{1}'$ is branching in $Q_{1}'$.
	
\begin{claim}\label{cl:F_j}
	For every $j\in \{1,\dots, c\}$, either the curve $C_{j}'$ is horizontal or $F_{j}=C_{j}'+V_{j}+L_{j}=[1,2,\dots,2,1]$, where $L_{j}$ is a unique component of $F_{j}$ not contained in $D\vert$.
\end{claim}
\begin{proof}
	Assume $C_{j}'$ is vertical. Note that $C_{3-j}'\not\subseteq F_{j}$, because otherwise $F_{j}\cdot E\geq 2$, which is false. Since $C_{j}'$ meets two sections in $D\hor$, the connectedness of $D$ implies that the set $F_{j}\cap D$ has at most two connected components. Since $F_{j}$ is a rational tree, Lemma \ref{lem:Qhp_has_no_lines} implies that $\sigma(F_{j})=1$ and $b_{0}(F_{j}\cap D)=2$. Because $\mu(C_{j}')=1$ and $(F_{j}-C_{j}')\wedge D\vert$ contains no $(-1)$-curves, we have $L_{j}^{2}=-1$. Since both connected components of $F_{j}\cap D$ contain (or belong to, in case some of them is a point) a component of multiplicity $1$ in $F_j$, we infer that $F_{j}$ is a chain meeting $D\hor$ in tips. Therefore, $F_{j}$ contains a twig of $D$ meeting $C_{j}'$, namely the $(-2)$-twig $V_{j}$. Because $F_{j}$ contracts to a $0$-curve, we get $F_{j}=C_{j}'+V_{j}+L_{j}$, as claimed.
\end{proof}	
	
	Claim \ref{cl:F_j} implies for instance that the section contained in $Q_{1}'$ is $C_{1}'$ or $B_{1}$. Because $\beta_{Q_{1}'}(B_{1})=3$, in any case we get that $(Q_{1}')\vert$ has at least two connected components with no $(-1)$-curves, lying in different fibers. Lemma \ref{lem:fibrations-Sigma-chi} implies that one of these fibers, say $F$, has a unique component not contained in $D\vert$, say $L_{F}$. From Claim \ref{cl:F_j} we infer that $F\neq F_{1},F_{2}$, so $C_{1}',C_{2}'\not\subseteq F$. Hence,  $L_{F}$ is a unique $(-1)$-curve in $F$ and $(F\wedge D\vert)\cdot E=0$, so $E$ meets $F$ in $L_{F}$. But $\mu(L_{F})\geq 2$ and $E$ is a $1$-section by Claim \ref{cl:E-section}; a contradiction.
\end{proof}

\section{Geometric consequences}\label{sec:corollaries}

\subsection{Proof of Theorem \ref{thm:geometric}.} \label{sec:1.2=>1.1}

\setcounter{claim}{0}
Theorem \ref{thm:main} and \cite[Theorem 1.3]{PaPe_Cstst-fibrations_singularities} describe the possible types of $\bar{E}\subseteq \P^{2}$. In particular, since $\bar{E}$ is not an Orevkov curve, it has at least two cusps. We order those cusps as in Lemma \ref{lem:special_lines}. Let $\ll_{12}$ and $\ll_{1}$ be the lines defined there. 

\begin{claim}\label{cl:l=l_1}
We may assume that $\bar{E}$ is of type $\cG$, $\cJ$, $\Qb$ or $\Qa$.
\end{claim}
\begin{proof}
Assume that $\bar{E}$ is not of one of the above types. By Lemma \ref{lem:special_lines}\ref{item:l} the line $\ll_{12}$ joining $q_{1}$ with $q_{2}$ meets $\bar{E}$ in exactly two points. 
If $\bar{E}$ is of one of the types listed in Theorem \ref{thm:geometric}\ref{item:cor-asymptote} then by Lemma \ref{lem:special_lines}\ref{item:t} the line $\ll_{1}$ tangent to $q_{1}\in \bar{E}$ meets $\bar{E}\setminus\{q_{1}\}$ exactly once and transversally. In particular, $\ll_{1}\neq \ll_{12}$ is another line through $q_{1}$ meeting $\bar{E}$ in exactly two points.  Thus we may assume that $\bar{E}$ is of type $\cE$, $\cF$ or $\cI$. We need to show that part \ref{item:cor-no-asymptote} of Theorem \ref{thm:geometric} holds. Using the parameterization of $\bar{E}$ given by \cite[(t) or (s)]{BoZo-annuli}, see \cite[Remark 4.14]{PaPe_Cstst-fibrations_singularities} and Remark \ref{rem:IJ}\ref{item:BZ_IJ}, we check that any line through $q_{1}$ other than $\ll_{12}$ meets $\bar{E}\setminus \{q_{1}\}$ in at least two points. Hence, $\ll$ is unique. Suppose that $\uu\subseteq \P^{2}$ is a curve such that $\uu\setminus \ll_{12}\cong \C^{1}$ and $(\uu\setminus \ll_{12})\cdot (\bar{E}\setminus \ll_{12})=1$. Then $\uu\cap \ll_{12}=\{q_{j}\}$ for some $j\in \{1,2\}$. If $\uu$ is a line then $(\uu\cdot \bar{E})_{q_{j}}=\deg\bar{E}-1$ is the sum of some number of initial terms of the multiplicity sequence of $q_{j}\in \bar{E}$. But we check directly (see \cite[Table 1]{PaPe_Cstst-fibrations_singularities} and Table \ref{table:nofibrations}) that the latter is impossible, hence $\uu$ is (a conic or a rational unicuspidal curve) tangent to $\ll_{12}$ at $q_{j}$. Because by Lemma \ref{lem:special_lines}\ref{item:t} $\bar{E}$ is not, the number $(\uu\cdot\bar{E})_{q_{j}}=\deg\uu\cdot\deg \bar{E}-1$ is the product of multiplicities of $q_j\in\uu$ and $q_j\in\bar{E}$. Hence, $\deg\uu\cdot\deg \bar{E}-1\leq (\deg\uu-1)(\deg\bar{E}-1)$, so $\deg\uu+\deg\bar{E}\leq 2$; a contradiction.
\end{proof}

Let us recall that a curve of type $\cG(\gamma)$, $\gamma\geq 3$ has degree $2\gamma-1$ and the multiplicity sequences of its cusps are $(\gamma-1)_{4}$ and $(2)_{\gamma-1}$, see \cite[Table 1]{PaPe_Cstst-fibrations_singularities}.

\begin{claim}\label{cl:conics}
	If $\bar{E}$ is of type $\cG$ then there exist unique non-degenerate conics $\mm$ and $\mm'$ such that $(\mm\cdot \bar{E})_{q_{1}}=2\deg\bar{E}-1$, $(\mm'\cdot \bar{E})_{q_{1}}=2\deg\bar{E}-2$ and $(\mm'\cdot \bar{E})_{q_{2}}=2$.
\end{claim}
\begin{proof}
The proof is similar to the one in Remark \ref{rem:AA'}(b). Let $\theta$ be a composition of four blowups at $q_{1}\in \bar{E}$ and its infinitely near points on the proper transforms of $\bar{E}$. Denote by $V$ the last exceptional curve of $\theta$. Lemma \ref{lem:special_lines}\ref{item:t} implies that $(\theta^{-1}_{*}\ll_{1})^{2}=-1$ and  $\theta^{-1}_{*}\ll_{1}\cdot \theta^{-1}_{*}\bar{E}=1$. The divisor $\Exc \theta=[2,2,2,1]$ meets $\theta^{-1}_{*}\ll_{1}$ only in the second component, transversally, and meets $\theta^{-1}_{*}\bar{E}$ only in $V$, with multiplicity $\gamma-1$. Hence, $\theta^{-1}_{*}\ll_{1}+\Exc\theta-V$ supports a fiber $F$ of a $\P^{1}$-fibration such that $\theta^{-1}_{*}\bar{E}$ is a $2$-section. Denote by $F_{1}$, $F_{2}$ the fibers passing through $V\cap\theta^{-1}_{*}\bar{E}$ and $\theta^{-1}(q_{2})$, respectively. Lemma \ref{lem:fibrations-Sigma-chi} implies that $F$ is the unique degenerate fiber, so $F_{1}$ and $F_{2}$ are smooth. This $\P^{1}$-fibration restricts to a morphism $\theta^{-1}_{*}\bar{E}\to \P^{1}$ of degree $2$, ramified at $\theta^{-1}_{*}\ll_{1}\cap \theta^{-1}_{*}\bar{E}$ and at $\theta^{-1}(q_{2})$. The Hurwitz formula implies that $F_{1}$ and $F_{2}$ are not tangent to $\theta^{-1}_{*}\bar{E}$. It follows that $\mm$, $\mm'$ are the required conics if and only if they are images of $F_{1}$ and $F_{2}$, respectively (see \cite[Figure 11]{PaPe_Cstst-fibrations_singularities}, where the proper transforms of $\mm$, $\mm'$ on $X$ are denoted by $L_{F_{1}}$ and $L_{F_{2}}$, respectively). 
\end{proof}

From now on we assume that $\bar{E}$ is as in Claim \ref{cl:l=l_1}. Clearly, there exists at least one $\ll$ through $q_1$ meeting $\bar E$ at two points, namely $\ll=\ll_{1}$. We will show in Claim \ref{cl:l_unique} that this is the only possibility. First, we make the following reduction.

\begin{claim}\label{cl:l_unique-reduction}
If $\bar E$ is not of type $\cG$ and a line $\ll'$ meets $\bar{E}\setminus \ll_{1}$ in one point then $\bar{E}$ is of type $\cJ$ and $\ll'\cap\ll_1\subseteq \bar E\setminus\{q_1\}$.
\end{claim}
\begin{proof}
	By Claim \ref{cl:l=l_1} $\bar E$ is of type $\cJ$, $\Qb$ or $\Qa$. By Lemma \ref{lem:special_lines}\ref{item:t} $\ll_1$ meets $\bar E\setminus\{q_1\}$ once and transversally. Let $r$ be the unique common point of $\ll'$ and $\ll_1$. By Lemma \ref{lem:Qhp_has_no_lines} $\ll'$ meets $\bar{E}$ in at least two points, so $r\in \bar{E}$. Denote by $\mu$ the multiplicity of $r\in \bar{E}$. Let $\theta$ be a blowup at $r$. Then $|\theta^{-1}_{*}\ll'|$ induces a $\P^{1}$-fibration which restricts to a morphism $g\colon \theta^{-1}_{*}\bar{E}\to \P^{1}$ of degree $\deg \bar E-\mu$.  For a point $p\in \theta^{-1}_{*}\bar{E}$ denote by $r_{p}$ the ramification index of $g$ at $p$. For any line through $r$ the sum of $r_{p}$ for $p$ contained in the proper transform of that line equals $\deg g$.
	Hence,
	\begin{equation*}
	\sum_{p\in\theta^{-1}_{*}(\ll_{1}+\ll')}(r_{p}-1)=2\deg g-\#(\theta^{-1}_{*}(\ll_{1}+\ll')\cap \theta^{-1}_{*}\bar{E})\geq 2\deg g-3,
	\end{equation*}
	where the last inequality follows from the assumption. The Hurwitz formula gives
	\begin{equation*}
	\sum_{p\not\in \theta^{-1}_{*}(\ll_{1}+\ll')}(r_{p}-1)\leq 2\deg g - 2 - (2\deg g-3)=1.
	\end{equation*}
	Thus $g$ has at most one ramification point off $\theta^{-1}_{*}(\ll_{1}+\ll')$, and its ramification index is $2$. Hence, $\bar{E}$ has at most one cusp off $\ll_{1}+\ll'$ and this cusp, if exists, has multiplicity $2$. It follows that $\bar{E}$ is of type $\cJ$ or $\Qb$, and (after renaming the cusps $q_2,\ldots, q_c$ if necessary) $q_{2}\in\ll'$. Consider the first case. If $q_{1}\in \ll'$ then $$(\ll'\cdot\bar{E})_{q_{2}}=\deg\bar{E}-(\ll'\cdot\bar{E})_{q_{1}}=4k+1-2k=2k+1$$ is a sum of some number of initial terms of the multiplicity sequence of $q_{2}\in\bar{E}$. But from Table \ref{table:nofibrations} we see that this is impossible, which proves the claim in this case. We are left with the case when $\bar{E}$ is of type $\Qb$. 
	If $\ll'$ is tangent to $\bar E$ at $q_2$ then $(\ll'\cdot\bar{E})_{q_{2}}=4$ and $(\ll'\cdot\bar{E})_{r}=1$, so $\deg g=4$ and $g$ is ramified at $\theta^{-1}(q_{3})$ and with index $4$ at $\theta^{-1}(q_{1})$ and $\theta^{-1}(q_{2})$. Since the latter contradicts the Hurwitz formula, we infer that $\ll'$ is not tangent to $\bar{E}$ at $q_{2}$, so $(\ll' \cdot\bar{E})_{r}=3$. By Lemma \ref{lem:special_lines}\ref{item:t} and Remark \ref{rem:Q3}, any line $\ll''$ through $q_{j}$ for $j\in \{1,2,3\}$  satisfies $(\ll''\cdot\bar{E})_{q_{j}}\in \{2,4\}$. Hence, $r$ is a smooth point of $\bar{E}$. By Remark \ref{rem:Q3} there is an automorphism  $\epsilon\in \Aut(\P^{2},\bar{E})$ such that $\epsilon(q_{2})=q_{1}$. The line $\epsilon(\ll')$ meets $\bar{E}$ only in two points, namely in $q_{1}$ and $\epsilon(r)$, hence $\bar{E}$ has two cusps off $\epsilon(\ll')$. But we have shown above that this is impossible; a contradiction.
\end{proof}

\begin{claim}\label{cl:l_unique}
	If a line $\ll'$ meets $\bar{E}\setminus \ll_{1}$ in one point then $q_{1}\not\in\ll'$ and $q_{2}\in\ll'$.
\end{claim}
\begin{proof}
	Either $\bar E$ is of type $\cG$ or $\bar E$ and $\ll'$ are as in Claim \ref{cl:l_unique-reduction}. In both cases Lemma \ref{lem:Qhp_has_no_lines} implies that $\ll'$ meets $\bar{E}$ in at least two points, so $\ll'\cap\ll_{1}\subseteq \bar{E}$. If $\ll'\cap\bar{E}=\{q_{1},q_{2}\}$ then $\bar E$ is of type $\cG$ and we put $P=\pi^{-1}_{*}(\ll_{1}+\ll'+\mm')$, where $\mm'$ is as in Claim \ref{cl:conics}. In other cases we put $P=\pi^{-1}_{*}(\ll_{1}+\ll')$.   Note that $D+P-\pi^{-1}_{*}\ll'$ is an snc divisor and $P-\pi^{-1}_{*}\ll'$ consists of disjoint $(-1)$-curves.  Let $$\pi'\colon (X',D')\to(X,D+P)$$ be the minimal log resolution, $P'\de (\pi')^{-1}_{*}P$ and let $(X',D')\to (\tilde{X},\tilde{D})$ be the (unique) snc-minimalization of $D'$. We have $\etop(\tilde{X}\setminus\tilde{D})=\etop(X\setminus(D+P))=\etop(X\setminus D)=1$, because $P\setminus D$ is a disjoint union of curves isomorphic to $\C^{*}$. Lemma \ref{lem:Qhp_has_no_lines} and \cite[6.24]{Fujita-noncomplete_surfaces} imply that the negative part of the Zariski--Fujita decomposition of $K_{\tilde{X}}+\tilde{D}$ equals $\Bk\tilde{D}$,	that is, the sum of barks of all maximal twigs of $\tilde{D}$, see Section \ref{sec:log_surfaces}. Put $\ind(\tilde{D})=-(\Bk\tilde{D})^{2}\geq 0$. Then the logarithmic Bogomolov-Miyaoka-Yau inequality reads as 
	\begin{equation}\label{eq:BMY}
	(K_{\tilde{X}}+\tilde{D})^{2}+\ind(\tilde{D})\leq 3.
	\end{equation}
	We will rewrite it using properties of $P$ and $\pi'$.
	
	Since the morphism $X'\to\tilde{X}$ is a composition of blowups with centers at the nodes of successive reduced total transforms of $\tilde{D}$, we have $(K_{\tilde{X}}+\tilde{D})^{2}=(K_{X'}+D')^{2}$. Moreover, $$D'\cdot (K_{X'}+D')=2p_{a}(D')-2=2\#P-2.$$ The non-nc points of $D+P$ are exactly the points of $\pi^{-1}(\ll'\cap\bar{E}\setminus \{q_{1},q_{2}\})$. To see this, recall that $D+P-\pi^{-1}_{*}\ll'$ is snc. Since $\ll'$ is smooth, $\pi^{-1}_{*}\ll'$ meets $D-E$ transversally, off $E$, each component of $D-E$ at most once. By definition, $\pi^{-1}_{*}\ll'$ is disjoint from $P-\pi^{-1}_{*}(\ll_{1}+\ll')$ and the common point of $\pi^{-1}_{*}\ll_{1}$ and $\pi^{-1}_{*}\ll'$, if exists, is contained in $E$, so it is not an nc point of $D-E$. Conversely, if $\pi^{-1}_{*}\ll'$  meets $E$ transversally at some point $r'$ then $r'\in \pi^{-1}_{*}\ll_{1}$, since otherwise $\ll'\cap \ll_{1}=\{q_{1}\}$ and $(\ll'\cdot\bar{E})_{q_{1}}=\deg\bar{E}-1=(\ll_{1}\cdot\bar{E})_{q_{1}}$, which is false. Let $\epsilon\de 2 -\#(\ll'\cap\{q_{1},q_{2}\})$ be the number of non-nc points of $D+P$. Each of these points is resolved by a sequence of blowups such that the centers of all but the first one are the nodes of respective total reduced transforms of $D$, hence $$K_{X'}\cdot(K_{X'}+D'-P')=K_{X}\cdot (K_{X}+D)-\epsilon.$$ By \cite[Lemma 4.3(i)]{Palka-minimal_models} $K_{X}\cdot(K_{X}+D)=h^{0}(2K_{X}+D)$, which in our cases vanishes, because $\kappa(K_{X}+\tfrac{1}{2}D)=-\infty$ (see \cite[Lemma 4.4]{PaPe_Cstst-fibrations_singularities} and Proposition \ref{prop:finding_C3st}). Therefore, \eqref{eq:BMY} reads as
	\begin{equation}\label{eq:BMY_1}
	K_{X'}\cdot P'+2\#P+\ind(\tilde{D})\leq 5+\epsilon.
	\end{equation}
	
	By the definition of a bark, the number $\ind(\tilde{D})$ equals the sum of coefficients of $\ftip{W}$ in $\Bk_{\tilde{D}}(W)$, taken over all maximal twigs $W$ of $\tilde{D}$. Hence \cite[II.3.3.4]{Miyan-OpenSurf} gives
	\begin{equation*}
	 \ind(\tilde D)=\sum_{W}\frac{d(W-\ftip{W})}{d(W)},
	\end{equation*}
	where $d(W)$ is as in Section \ref{sec:log_surfaces}. If $W$ is a $(-2)$-twig then the respective summand is $\tfrac{\#W}{\#W+1}\geq \tfrac{1}{2}$. We claim that $\ind(\tilde{D})>1.$ To see this, note that the map $X\map \tilde{X}$ does not touch the maximal $(-2)$-twigs of $D$ meeting $C_{1}'$, $C_{2}'$ (where $C_{j}'$ is the last component of $Q_{j}'\de \pi^{-1}(q_{j})\redd$, see the beginning of Section \ref{sec:existence}). If $q_{1}\not\in \ll'$ then it does not touch the first component of $Q_{1}'$, whose image becomes a maximal $(-2)$-twig of $\tilde{D}$. If $q_{1}\in \ll'$ then, since  $(\ll'\cdot \bar{E})_{q_{1}}<(\ll_{1}\cdot\bar{E})_{q_{1}}=\deg\bar{E}-1$, $\ll'$ is tangent to $\bar{E}$ at the other point of $\ll'\cap\bar{E}$, so the image of its resolution contains a twig of $\tilde{D}$. This proves $\ind(\tilde{D})>1$. Now \eqref{eq:BMY_1} gives
	$K_{X'}\cdot P'+2\#P\leq 3+\epsilon.$ 
	Let $L'$, $L_{1}$ and $M'$ be the proper transforms on $X'$ of $\ll'$, $\ll_{1}$ and $\mm'$, respectively. They are smooth and rational, so
	\begin{equation}\label{eq:BMY_2}
	-3-\epsilon \leq (L')^{2}+L_{1}^{2}+(P'\wedge M')^{2},
	\end{equation}
	where $P'\wedge M'=M'$ or $0$ depending on whether $\ll'\cap \bar{E}=\{q_{1},q_{2}\}$ or not. 
	
	Suppose that $q_{1}\in \ll'$. By Claim \ref{cl:l_unique-reduction} the curve $\bar{E}$ is of type $\cG(\gamma)$ for some $\gamma\geq 3$. Since the centers of $\pi'$ belong to $\pi^{-1}_*\ll'$, $\pi'$ does not touch $L_{1}$, so $L_{1}^{2}=-1$. Since $\ll'$ is not tangent to $\bar E$ at $q_1$, the contraction of $Q_{1}'$ touches $\pi^{-1}_{*}\ll'$ exactly once. Consider the case $q_{2}\in \ll'$. Then $\epsilon=0$, $P'\wedge M'=M'$ and $\pi'=\id$. We have $(\ll'\cdot\bar{E})_{q_{2}}=\deg\bar{E}-(\ll'\cdot\bar{E})_{q_{1}}=\gamma$, so the contraction of $Q_{2}'$ touches $L'$ exactly $\tfrac{1}{2}\gamma$ times, hence $(L')^2=-\tfrac{1}{2}\gamma$. It follows from \eqref{eq:BMY_2} that $\gamma\leq 2$; a contradiction. Consider the case $q_{2}\not\in\ll'$. Then $\epsilon=1$, $P'\wedge M'=0$ and  $\pi'$ touches $L'$ exactly $\deg\bar{E}-(\ll'\cdot\bar{E})_{q_{1}}=\gamma$ times, so $(L')^2=-\gamma$ and hence $\gamma=3$ by \eqref{eq:BMY_2}. The maximal twigs of $\tilde{D}$ are: the $(-2)$-tips meeting the images of $C_{1}'$, $C_{2}'$, the image of the other twig of $D$ contained in $Q_{2}'$, which is of type $[2,3]$, and the maximal $(-2)$-twig contained in the preimage of $\ll'\cap \bar{E}\setminus \{q_1\}$. Hence, $\ind(\tilde{D})>2$, which is in contradiction with \eqref{eq:BMY_1}.
	
	Hence, $q_{1}\not\in \ll'$. Suppose that  $q_{2}\not\in \ll'$. Then $\epsilon=2$. We have $P'\wedge M'=0$ and $\pi'$ touches $L_{1}$ exactly once, so $L_{1}^{2}=-2$ and the inequality \eqref{eq:BMY_2} reads as $(L')^2\geq -3$. But $\pi$ does not touch $\pi^{-1}_{*}\ll'$ and $\pi'$ touches $L'$ exactly $\deg\bar{E}$ times, so $-3\leq (L')^{2}=1-\deg\bar{E}\leq -4$; a contradiction.
\end{proof}

Claim \ref{cl:l_unique} implies in particular that $\ll$ as in Theorem \ref{thm:geometric} is unique, equal to $\ll_1$. Assume that $\uu\subseteq \P^{2}$ is a curve such that $\uu\setminus \ll_{1}\cong \C^{1}$ and $(\uu\setminus \ll_{1})\cdot (\bar{E}\setminus \ll_{1})=1$. By Claim \ref{cl:conics} it remains to show that $\bar{E}$ is of type $\cG$ and $\uu=\mm$.

\begin{claim}\label{cl:U-singular}
	The curve $\uu$ is tangent to $\bar{E}$ at $q_{1}$. We may assume that it is singular.
\end{claim}
\begin{proof}
	Write $\uu\cap\ll_{1}=\{r\}$. Since $\uu\cdot\bar{E}\geq 2$, we have $r\in \bar{E}$. By Claim \ref{cl:l_unique}, $\uu$ is not a line, so either $\uu$ is a conic or $r\in \uu$ is its unique singular point, a cusp. Denoting by $\ll_{r}$ the line tangent to $\uu$ at $r$ we get $\ll_{1}\cdot \uu=(\ll_{1}\cdot \uu)_{r}\leq (\ll_{r}\cdot \uu)_{r}\leq \ll_{r}\cdot \uu$, so the equalities hold and we have $\ll_{r}=\ll_{1}$. Suppose   that $r\neq q_{1}$.  By Lemma \ref{lem:special_lines}\ref{item:t}, $\ll_{1}$ is not tangent to $\bar{E}$ at $r$, so neither is $\uu$. Then the number $(\uu\cdot\bar{E})_{r}=\deg\uu\cdot\deg \bar{E}-1$ is the product of multiplicities of $r\in\uu$ and $r\in\bar{E}$, hence $\deg\uu\cdot\deg \bar{E}-1\leq (\deg\uu-1)(\deg\bar{E}-1)$, so $\deg\uu+\deg\bar{E}\leq 2$; a contradiction. Thus $\uu$ is tangent to $\bar{E}$ at $q_{1}$.

	Assume that $\uu$ is a conic. Then $(\uu\cdot\bar{E})_{q_{1}}=2\deg\bar{E}-1$, so the latter is a sum of some number of initial terms of the multiplicity sequence of $q_{1}\in\bar{E}$. Because $\uu$ is smooth, at most two of these summands are distinct. We check directly that this is possible only if $\bar{E}$ is of type $\cG$. Then Claim \ref{cl:conics} implies that $\uu=\mm$.
\end{proof}

By Claim \ref{cl:U-singular} we may, and will, assume that $\uu$ is a rational unicuspidal curve meeting $\ll_{1}$ only in the cusp $q_1\in\uu$. Let $\pi_{\uu}\colon (X_{\uu},D_{\uu})\to(\P^{2},\uu)$ be the minimal log resolution. 

\begin{claim}\label{cl:theta}
	We have a decomposition $\pi=\pi_{\uu}\circ \xi$ and $\xi$ touches $\pi^{-1}_{*}\uu$.
\end{claim}
\begin{proof}
Let $U$ and $E_\uu$ be the proper transforms on $X_\uu$ of $\uu$ and $\bar E$, respectively. The curve $\uu$ is of Abhyankar--Moh--Suzuki type in the sense of \cite{Tono-equations_cusp_curves}. By Theorem 1.1(iv) loc.\ cit.\ there exist integers $s\geq 0$, $k_{1},\dots,k_{s+1}\geq 1$ such that putting $d_{i}=(k_{i+1}+1)\cdot\ldots\cdot(k_{s+1}+1)$ for $i\in \{0,\dots,s\}$, we have $\deg\uu=d_{0}$ and $q_{1}\in \uu$ has multiplicity sequence
\begin{equation}\label{eq:Tono_AMS}
(k_{1}d_{1},(d_{1})_{2k_{1}},\dots, k_{s}d_{s},(d_{s})_{2k_{s}},k_{s+1},(1)_{k_{s+1}})
\end{equation}
(including all $1$'s in the end). Lemma \ref{lem:HN-equations}\eqref{eq:3deg} gives
\begin{equation*}
U^{2}=3d_{0}-3\sum_{i=1}^{s}k_{i}d_{i}-2k_{s+1}-2=3d_{0}-3\sum_{i=1}^{s}(d_{i-1}-d_{i})-2d_{s}=d_{s}\geq 0.
\end{equation*}

Suppose first that $U$ does not meet $E_{\uu}$ on $\Exc\pi_{\uu}$. Then $U\cdot E_{\uu}=1$ and after $U^{2}\geq 0$ blowups over  $U\cap E_{\uu}$ the linear system of the proper transform of $U$ induces a $\P^{1}$-fibration which restricts to a $\C^{*}$-fibration of $\P^{2}\setminus \bar{E}$. This gives $\kappa(\P^{2}\setminus \bar{E})\leq 1$; a contradiction.

The minimal log resolutions of the cusps $q_j\in \bar E$, $j\in \{2,\dots, c\}$ do not touch $\uu$, so it is enough to show that the common point of $E_\uu$ and $\Exc\pi_{\uu}$ is not a point of normal crossings of $E_\uu+\Exc\pi_{\uu}$. Suppose it is. Since $\pi_\uu$ is minimal, the last $(-1)$-curve in $\Exc \pi_\uu$ meets two components of $D_\uu-U$. It follows that $q_{1}\in \bar{E}$ and $q_{1}\in \uu$ have the same multiplicity sequences. The last two terms (except the $1$'s) in the multiplicity sequence of $q_{1}\in \bar{E}$ are equal (see Table \ref{table:nofibrations}), so in \eqref{eq:Tono_AMS} we have $s\geq 1$ and $k_{s+1}=1$, hence $d_s=2$ and the sequence is $(2k_{s},(2)_{2k_{s}},1,1)$ if $s=1$ or ends with
\begin{equation*}
(k_{s-1}d_{s-1},(d_{s-1})_{2k_{s-1}},2k_{s},(2)_{2k_{s}},1,1),\ \  d_{s-1}=2(k_s+1)
\end{equation*}
if $s\geq 2$. Looking at Table \ref{table:nofibrations} we see that this is possible only if $k_{s}=1$ and $s=1$. In this case $\bar{E}$ is of type $\Qa$ and $\deg \uu=d_{0}=4$. As a consequence,  $(\bar{E}\cdot\uu)_{q_{1}}=\deg\bar{E}\cdot\deg\uu-1=19$. We have $I(q_{1})=3\cdot 4+2\cdot 1=14$ (see Lemma \ref{lem:HN-equations}), so $U$ meets $E_{\uu}$ on $\Exc\pi_{\uu}$ with multiplicity $19-I(q_{1})=5$. It follows that the minimal log resolution $\pi'\colon (X',D')\to (X_\uu,D_\uu+E_\uu)$ is a composition of minimal log resolutions of the cusps $\pi_{\uu}^{-1}(q_{j})\in E_\uu$, $j\in\{2,3,4\}$ with five blowups over the point of tangency of $U$ and $E_\uu$ on $\Exc \pi_\uu$. We argue as in the proof of Claim \ref{cl:l_unique}. Because $\uu\setminus \bar{E}\cong \C^{*}$, we have 
\begin{equation*}
\etop(X'\setminus D')=\etop (\P^{2}\setminus (\bar{E}+\uu))=\etop(\P^{2}\setminus \bar{E})=1.
\end{equation*}
Put $U'=(\pi')^{-1}_*U$. Since each of the five blowups touches the image of $U'$, we have  $(U')^{2}=(\deg\uu)^{2}-I(q_{1})-5=-3$, so $D'$ is snc-minimal and $K_{X'}\cdot U'=-2-(U')^{2}=1$. The center of each of those blowups is a node of the respective preimage of $D$, so as before we obtain $K_{X'}\cdot(K_{X'}+D'-U')=K_{X}\cdot(K_{X}+D)=0$. This gives $(K_{X'}+D')^{2}=K_{X'}\cdot U'+D'\cdot(K_{X}'+D')=1$. The twigs of $D'$ are exactly the proper transforms of the twigs of $D$, so $\ind(D')=3\cdot\left(\tfrac{1}{3}+\tfrac{1}{2}\right)+\tfrac{1}{2}+\tfrac{5}{7}>2$. This contradicts the log BMY inequality (see \eqref{eq:BMY}).
\end{proof}

Since $\uu$ is singular and $\pi_\uu$ is a minimal log resolution, the unique $(-1)$-curve $V$ in $D_\uu-U$ meets $U$ and two other components of $D_\uu$. By Claim \ref{cl:theta}, $E_\uu$ meets $\Exc \pi_\uu$ at the point $V\cap U$, which is the center of the next blowup in the decomposition of $\pi$. It follows that $Q_{1}'$ is not a chain, so $\bar{E}$ is of type $\cJ(k)$ for some $k\geq 2$. The divisor $D_\uu$ is snc, so the proper transform of $U$ on $X$ meets a tip of $Q_{1}'$, the same as $C_{1}'$ (see Figure \ref{fig:J_line}). As a consequence, the multiplicity sequence of $q_{1}\in \uu$ equals $(k)_{3}$, we have $\deg \uu=\uu\cdot \ll_1=(\uu\cdot\ll_{1})_{q_{1}}=2k$ and $$\uu\cdot \bar E-1=(\uu\cdot\bar{E})_{q_{1}}=3\cdot 2k\cdot k+k\cdot 2\cdot 1+1=6k^{2}+2k+1.$$ Then $2k\cdot(4k+1)=\deg\uu\cdot\deg\bar{E}=6k^{2}+2k+2$, so $k=1$; a contradiction. \qed

\smallskip
\subsection{Log deformations and the proof of Theorem \ref{thm:geom_conseq}.}

We recall some results on logarithmic deformations. Let $T$ be an snc divisor on a smooth projective surface $V$. The deformation theory of the pair $(V,T)$ is described in terms of the cohomology of the logarithmic tangent sheaf $\lts{V}{T}$, that is, the sheaf of those $\O_{V}$-derivations which preserve the ideal sheaf of $T$. The number $h^{1}(\lts{V}{T})$ is the number of moduli for log deformations of the pair $(V,T)$ and $H^{2}(\lts{V}{T})$ is the space of obstructions for extending infinitesimal deformations, see \cite[Lemma 1.1]{FZ-deformations}. We need the following lemma.

\begin{lem}[Properties of $\lts{V}{T}$, \cite{FZ-deformations}]\label{lem:rig}  
	Let $S$ be a smooth surface and let $(V,T)$ be some log smooth completion of $S$. Then
	\begin{enumerate}
		\item\label{item:rig_bl} The number $h^{2}(\lts{V}{T})$ depends only on $S$. 
		\item\label{item:rig_surg} If $L\subseteq T$ is a $(-1)$-curve then $h^{i}(\lts{V}{(T-L)})=h^{i}(\lts{V}{T})$ for $i\geq 0$. 
		\item\label{item:rig_C**} If $V$ admits a $\P^{1}$-fibration such that $F\cdot T\leq 3$ for a fiber $F$ then $H^{2}(\lts{V}{T})=0$.
		\item\label{item:rig_chi} $\chi(\lts{V}{T})=K_{V}\cdot (K_{V}+T)+2\chi(\O_{V})-\etop(V)+r-\sum_{i=1}^{r}p_{a}(T_{i})$,  where $T_{1},\dots, T_{r}$ are the components of $T$.
		\item\label{item:rig_chi_QHP} If $S$ is $\Q$-acyclic then $\chi(\lts{V}{T})=K_{V}\cdot (K_{V}+T)$.
	\end{enumerate}
\end{lem}
\begin{proof}
	\ref{item:rig_bl}, \ref{item:rig_surg}, \ref{item:rig_C**} are shown in \cite[Lemma 1.5(5) and Propositions 1.7(3), 6.2]{FZ-deformations}.
	
		\ref{item:rig_chi} Since $T$ is snc, the exact sequence \cite[Proposition 1(2)]{Kawamata_deformations} reads as  
		\begin{equation*}
		0\to \lts{V}{T}\to \mathcal{T}_{V} \to \bigoplus_{i=1}^{r}\mathcal{N}_{T_{i}/V}\to 0,
		\end{equation*}
		so $\chi(\lts{V}{T})=\chi(\mathcal{T}_{V})-\sum_{i=1}^{r}\chi(\mathcal{N}_{T_{i}/V})$. We have $c_{2}(\mathcal{T}_{V})=\etop(V)$ and $-c_{1}(\mathcal{T}_{V})=c_{1}(\Omega^{1}_{V})=c_{1}(\bigwedge^{2}\Omega_{V}^{1})=K_{V}$, hence using the Riemann-Roch theorem and the Noether formula we get $\chi(\mathcal{T}_{V})=2\chi(\O_{V})+K_{V}^{2}-\etop(V).$ Since $\mathcal{N}_{T_{i}/V}\cong \O_{T_{i}}(T_{i}^{2})$, the Riemann-Roch theorem and the adjunction formula give
		\begin{equation*}
		\chi(\mathcal{N}_{T_{i}/V})=\chi(\O_{T_{i}}(T_{i}^{2}))=T_{i}^{2}-p_{a}(T_{i})+1=p_{a}(T_{i})-K_{V}\cdot T_{i}-1.
		\end{equation*}
		
		\ref{item:rig_chi_QHP} \cite[Lemma 1.3(5)]{FZ-deformations}. If $V\setminus T$ is $\Q$-acyclic then $V$ is rational, the components of $T$ are rational and freely generate $H_{2}(V,\Q)$. We obtain $\chi(\O_{V})=1$ and $\etop(V)=2+r$. 
		Thus \ref{item:rig_chi_QHP} follows from \ref{item:rig_chi}. 
\end{proof}

In \cite{FZ-deformations} Flenner and Zaidenberg made the following conjecture (cf. \cite[1.3]{Zaidenberg-open_problems}).

\begin{conjecture}[The Strong Rigidity Conjecture]\label{conj:rig}
Let $S$ be a $\Q$-acyclic surface of log general type. Then for a minimal log smooth completion $(X,D)$ of $S$ we have $H^{i}(\lts{X}{D})=0$ for $i\geq 0$.	
\end{conjecture}

As discussed in \cite[Conjecture 2.6]{Palka-minimal_models}, Negativity Conjecture \ref{conj} implies the Weak Rigidity Conjecture, which asserts that $\chi(\lts{X}{D})=0$. A posteriori, from our classification we can deduce the following stronger result for $\Q$-acyclic surfaces which are complements of planar rational cuspidal curves.

\begin{prop}[Negativity implies Strong Rigidity] \label{prop:rig}
	Let $\bar{E}\subseteq \P^{2}$ be a rational cuspidal curve such that $\P^{2}\setminus \bar{E}$ is of log general type. If $\bar{E}\subseteq \P^{2}$ satisfies  Negativity Conjecture \ref{conj} then it satisfies  Strong Rigidity Conjecture \ref{conj:rig}.
\end{prop}

\begin{proof}
	 Since $X\setminus D$ is of log general type, by \cite[Theorem 11.12]{Iitaka_AG} $\Aut(X,D)$ is finite, so $(X,D)$ has no infinitesimal automorphisms, that is, $H^{0}(\lts{X}{D})=0$. Hence by Lemma \ref{lem:rig}\ref{item:rig_chi_QHP}
	 \begin{equation*}
	 h^{2}(\lts{X}{D})-h^{1}(\lts{X}{D})=\chi(\lts{X}{D})=K_{X}\cdot (K_{X}+D)=h^{0}(2K_{X}+D)\geq 0,
	 \end{equation*}
	 where the last equality follows from \cite[Lemma 4.3(i)]{Palka-minimal_models}. Hence it suffices to show that
	 \begin{equation}\label{eq:rig}
	 H^{2}(\lts{X}{D})=0.
	 \end{equation}
	
	If $\P^{2}\setminus \bar{E}$ has a $\C^{**}$-fibration then it extends to a $\P^1$-fibration of some blowup of $X$, in which case \eqref{eq:rig} follows from Lemma \ref{lem:rig}\ref{item:rig_bl},\ref{item:rig_C**}.  Therefore, by Theorem \ref{thm:main} we may assume that $\bar{E}\subseteq \P^{2}$ is of one of the types listed in Definition \ref{def:our_curves}. For types $\Qb$, $\Qa$ or $\FZb$ the equality \eqref{eq:rig} follows from \cite[Corollary 2.4]{FlZa_cusps_d-3}, so it remains to consider the types $\FE$, $\cH$, $\cI$ and $\cJ$.  We use the notation from the beginning of Section \ref{sec:existence}, in particular, for $j\in \{1,\dots, c\}$ we denote by $C_{j}'$ the unique $(-1)$-curve in $Q_{j}'=\pi^{-1}(q_{j})\redd$.
	
	Consider the types $\FE$ and $\cH$ (see Figures \ref{fig:FE}, \ref{fig:H}). Let $A,A'$ be the proper transforms on $X$ of the lines $\ll_{12}$, $\ll_{1}$ from Lemma 4.2. We have $(A')^{2}=A^{2}=-1$ and $A'\cdot D= A\cdot D=2$. Moreover, $A$ meets $D$ only in the first components of $Q_{1}'$ and $Q_{2}'$ and $A'$ meets $D$ only in $E$ and in the second component of $Q_{1}'$. For $j\in \{1,2\}$ denote by $B_{j}$ the branching component of $Q_{j}'$ and by $T_{j}'$ the $(-2)$-twig of $Q_{j}'$ meeting $B_{j}$ (cf. Notation \ref{not:graphs}). Let $\psi'\colon (X,D+A+A')\to (X',D')$ be an snc-minimalization. 
	By Lemma \ref{lem:rig}\ref{item:rig_bl},\ref{item:rig_surg}
	\begin{equation*}
	h^{2}(\lts{X}{D})=h^{2}(\lts{X}{(D+A+A')})=h^{2}(\lts{X'}{D'})=h^{2}(\lts{X'}{(D'-\psi' (B_{1}))}),
	\end{equation*}
	where the last equality holds since in both cases $\psi' (B_{1})$ is a  $(-1)$-curve. We have
	\begin{equation*}
	\psi'_{*}(Q_{2}'-T_{2}')=[2,1,2]\mbox{ for type } \FE\mbox{ and }\psi'_{*}(Q_{2}'-T_{2}')=[1,3,1,2]\mbox{ for type }\cH.
	\end{equation*}
	Each of these chains supports a fiber $F$ of a $\P^{1}$-fibration of $X'$ which is met by $(D'-\psi' (B_{1}))\hor$ once in $\psi' (C_{2}')$ and once in $\psi' (B_{2})$. We have $\mu(\psi' (C_{2}'))=2$ and $\mu(\psi' (B_{2}))=1$ (where $\mu(C)$ denotes the multiplicity of $C$ in $F$), so $F\cdot (D'-\psi' (B_{1}'))=3$. Hence  \eqref{eq:rig} follows from Lemma \ref{lem:rig}\ref{item:rig_C**}.
	
	We are left with the types $\cI$ and $\cJ$ (see Figures \ref{fig:I}, \ref{fig:J_line}). Let $V_{1}$ be the $(-2)$-twig of $D$ meeting $C_{1}'$.  As before, Lemma \ref{lem:rig}\ref{item:rig_surg} gives
	\begin{equation*}
	h^{2}(\lts{X}{D})=h^{2}(\lts{X}{(D-C_{2}')}).
	\end{equation*}
	Consider the type $\cI$. Then $V_{1}+C_{1}'+E=[2,2,1,3]$ supports a fiber $F$ of a $\P^{1}$-fibration of $X$ which is met $(D-C_{2}')\hor$ only once, in $C_{1}'$. Since $\mu(C_{1}')=3$, we have $F\cdot(D-C_{2}')=3$ and again we deduce \eqref{eq:rig} from Lemma \ref{lem:rig}\ref{item:rig_C**}. 
	Consider the type $\cJ$. Let $A''$ be the proper transform of the line $\ll_{1}$ from Lemma 4.2. Then $(A'')^{2}=-1$, $A''\cdot D=2$ and $A''$ meets $D$ only in $E$ and in the second component of $Q_{1}'$. Now $V_{1}+C_{1}'+E+A''=[2,1,3,1]$ supports a fiber $F$ of a $\P^{1}$-fibration of $X$ which is met by $(D-C_{2}')\hor$ only once in $C_{1}'$ and once in $A''$. Since $\mu(C_{1}')=2$ and $\mu(A'')=1$, we have $F\cdot (D-C_{2}')=3$ and \eqref{eq:rig} follows again from Lemma \ref{lem:rig}\ref{item:rig_C**}.
\end{proof}
\smallskip 

\begin{proof}[Proof of Theorem \ref{thm:geom_conseq}]\

\ref{item:fibr} If $\kappa\de \kappa(\P^{2}\setminus \bar{E})=-\infty$ then by \cite[Theorem 3.1.3.2]{Miyan-OpenSurf} $\P^{2}\setminus \bar{E}$ has a $\C^{1}$-fibration. If $\kappa\in \{0,1\}$ then \cite[Proposition 2.6]{Palka-minimal_models} implies that $\P^{2}\setminus \bar{E}$ has a $\C^{*}$-fibration, which by \cite[Proposition 4.2]{PaPe_Cstst-fibrations_singularities} can be chosen without base points on $X$. Assume $\kappa=2$. If $\P^{2}\setminus \bar{E}$ has a $\C^{**}$-fibration then by \cite[Proposition 4.2]{PaPe_Cstst-fibrations_singularities} it has one with no base point on $X$. In the other case, Proposition \ref{prop:finding_C3st} gives a $\C^{***}$-fibration of $\P^{2}\setminus \bar{E}$ with no base point on $X$. Because $\P^{2}\setminus \bar{E}$ is rational, the base of each of the above fibrations is   rational, and hence an open subset of $\P^{1}$. The components of $D$ are linearly independent in $\NS_{\Q}(X)$, so after possibly resolving a base point in case of a $\C^{1}$-fibration, the preimage of $D$ contains at most one fiber. Hence the base is in fact $\P^{1}$ or $\C^{1}$.

\ref{item:c>=4} If $\bar{E}$ has at least three cusps then $\kappa=2$ by \cite{Wakabayashi-cusp}, so \ref{item:c>=4} follows from \cite[Theorem 1.2]{PaPe_Cstst-fibrations_singularities} and Theorem \ref{thm:main}; see \eqref{eq:Qa-par} for the parameterization. 

\ref{item:C1_or_C*} If $\kappa=2$ then \ref{item:C1_or_C*} follows from Theorem \ref{thm:geometric}. Assume $\kappa\leq 1$. If  $\bar{E}$ has at least two cusps then by \cite{Wakabayashi-cusp,Tsunoda_cusps_complement_kod_0} it has exactly two and  $\kappa=1$. Then by \cite[Theorem 4.1.1]{Tono_doctoral_thesis} there is a line meeting $\bar{E}$ in exactly one point. 

\ref{item:rig} The case $\kappa=2$ is treated in Proposition \ref{prop:rig}, so we may assume that $\kappa\leq 1$. By \cite[Proposition 2.6]{Palka-minimal_models} some log resolution of $(X',D')\to \P^2\setminus \bar E$ has a $\P^1$-fibration such that $F\cdot D'\leq 2$ for a fiber $F$. Then \ref{item:rig} follows from Lemma \ref{lem:rig}\ref{item:rig_bl},\ref{item:rig_C**}.
\end{proof}

\bigskip

In Table \ref{table:nofibrations} below we summarize numerical data for rational cuspidal curves satisfying \eqref{eq:assumption}. For each cusp we list both its multiplicity sequence and a sequence of \emph{Hamburger-Noether pairs}. We write the latter in the standard way in the sense of  \cite[Section 2D]{PaPe_Cstst-fibrations_singularities}. The sequence of Hamburger-Noether pairs is a convenient replacement of the sequence of multiplicities or of the Puiseux pairs (see \cite{Russell_HN_pairs} for a detailed treatment). It is more directly related to the geometry of the resolution. Relations between those sequences are explained in \cite[Lemmas 2.11 and 5.1]{PaPe_Cstst-fibrations_singularities}.

\begin{landscape}
	\thispagestyle{empty}
	
	\begin{small}
	\begin{table}	\addcontentsline{toc}{section}{TABLE: Classification (del Pezzo case)}
		{\renewcommand{\arraystretch}{2}
			\begin{tabular}{>{$}r<{$}|>{$}c<{$}|>{$}c<{$}|>{$}c<{$}|>{$}c<{$}|>{$}c<{$}|>{$}c<{$}|c}
				& c & \mbox{degree} & -E^{2} & \mbox{(standard) HN pairs } & \mbox{multiplicity sequences} & \mbox{parameters} & \mbox{references} \\ \hline
				\Qb & 3 & 5 & 5 & \binom{5}{2},\quad \binom{5}{2}, \quad \binom{5}{2} &  (2,2),\quad (2,2),\quad (2,2) & & \cite[2.3.10.8]{Namba_geometry_of_curves} \\ \hline
				\Qa & 4 & 5 & 7 & \binom{7}{2},\quad \binom{3}{2}, \quad \binom{3}{2}, \quad \binom{3}{2} &  (2,2,2),\quad (2),\quad (2),\quad (2) & & \cite[2.3.10.6]{Namba_geometry_of_curves} \\ \hline
				\FE(\gamma) & 3 & 3\gamma-5 & \gamma & \binom{3\gamma-6}{3\gamma-9}\binom{3}{1}, \quad \binom{4\gamma-10}{4}\binom{2}{1}, \quad \binom{3}{2} &  (3(\gamma-3),(3)_{\gamma-3}),\quad  ((4)_{\gamma-3},2,2),\quad (2) & \gamma\geq 5 & \cite{Fenske_cusp_d-4}, \cite[(h)]{BoZo-annuli} \\ \hline
				\FZb(\gamma) & 3 & 2\gamma-1 & \gamma & \binom{2\gamma-2}{2\gamma-4}\binom{2}{1},\quad \binom{3\gamma-5}{3}, \quad \binom{3}{2} &  (2(\gamma-2),(2)_{\gamma-2}),\quad ((3)_{\gamma-2}),\quad (2) & \gamma\geq 4 & \cite{FlZa_cusps_d-3}, \cite[(g)]{BoZo-annuli} \\ \hline
				\cH(\gamma) & 2 & 3\gamma+1 & \gamma & \binom{3\gamma}{3\gamma-3}\binom{3}{1}, \quad \binom{4\gamma-2}{4}\binom{2}{3}&   (3(\gamma-1),(3)_{\gamma-1}),\quad ((4)_{\gamma-1},2,2,2) & \gamma\geq 3 & \cite[(ii.3)]{CKR-Cstar_good_asymptote}, \cite[(i)]{BoZo-annuli} \\ \hline 
				\cI & 2 & 14 & 3 &   \binom{15}{6}\binom{3}{1},\quad \binom{12}{8}\binom{4}{2}\binom{2}{1}& (6,6,3,3), \quad (8,4,4,2,2) & & \cite[(t)]{BoZo-annuli}, \cite[(c)]{KoPa-SporadicCstar2}  \\ \hline
				\cJ(k) & 2 & 4k+1 & 3 & \binom{6k+2}{2k}\binom{2}{1}, \quad  \binom{2k+2}{2k}\binom{2}{1} &  (2k,2k,2k,(2)_{k}),\quad  (2k,(2)_{k}) & k\geq 2 & \cite[IV]{tomDieck_letter}, 
				\cite[(c)]{Bodnar_type_G_and_J}, Section \ref{sec:IJ} \\ \hline
			\end{tabular}
		}
		\vspace{2em}
		\caption{Numerical data for rational cuspidal curves satisfying \eqref{eq:assumption}, that is, satisfying  Negativity Conjecture \ref{conj} and having a complement of log general type which does not admit a $\C^{**}$-fibration.}
		\label{table:nofibrations}
	\end{table}
	\end{small}
\end{landscape}

\bibliographystyle{amsalpha}
\bibliography{bibl2018}

\end{document}